\documentclass[12pt,leqno]{book}
\usepackage{amsmath,amssymb,amsfonts,latexsym,amsthm} 
\usepackage{graphics}                 
\usepackage{color}                    

\usepackage[all]{xy}
\newtheorem{theorem}{Theorem}[section]
\newtheorem{lemma}[theorem]{Lemma}
\newtheorem {proposition}[theorem]{Proposition}
\newtheorem {corollary} [theorem] {Corollary}
\theoremstyle{definition}
\newtheorem{definition}[theorem]{Definition}
\theoremstyle{definition}
\newtheorem{example}[theorem]{Example}
\newtheorem{exercise}[theorem]{Exercise}
\theoremstyle{remark}
\newtheorem{remark}[theorem]{Remark}

\newtheorem{e}{\bf Problem}[chapter]

\def\ra{{\rightarrow}}

\def\norm{{\|-\|}}

\def\c{{\mathbb{C}}}
\def\r{{\mathbb{R}}}
\def\z{{\mathbb{Z}}}
\def\n{{\mathbb{N}}}
\def\t{{\mathbb{T}}}
\def\q{{\mathbb{Q}}}

\def\pr{{^{\prime}}}
\def\s{{^{\ast}}}
\def\inv{{^{-1}}}

\def\cs{{$C^{\ast}$}}
\def\ss{{$\ast$}}
\def\K{{$K$}}
\def\KK{{$KK$}}
\def\ws{{weak-$\s$\,}}
\def\sub{{\subseteq }}
\def\inner{{\langle - , - \rangle }}
\def\lan{{\langle}}
\def\ran{{\rangle}}
\def\nt{{\|^2}}

\def\mm{{\mathfrak{m}}}
\def\mmm{{\mathcal{M}}}
\def\ggg{{\mathcal{G}}}
\def\fff{{\mathcal{F}}}

\def\ff{{\varphi}}

\def\d{{\delta}}
\def\la{{\lambda}}
\def\si{{\sigma}}
\def\ep{{\varepsilon}}
\def\om{{\omega}}
\def\ga{{\gamma}}
\def\th{{\theta}}
\def\oom{{\Omega}}

\makeindex

\title{Lectures on $C^\ast$-algebras}

\author{Vahid Shirbisheh  \\
{\small\em \copyright \  Draft date \today }}

\date{ }
\begin{document}
\maketitle
 \addcontentsline{toc}{chapter}{Contents}
\pagenumbering{roman}
\tableofcontents

\newpage
\pagenumbering{arabic}
\chapter{Introduction}
\label{ch:introduction}

These notes are mainly based on a course given by the author in Fall 2008. The title of the course was ``topics in functional analysis'', but with a very flexible syllabus mainly about operator algebras. Therefore at the time, we decided to focus only on one topic which was ``\cs-algebras''. We mainly followed Bruce Blackadar's book \cite{blackadaroa} in the course. Meanwhile, we had to refer to other books on \cs-algebras and operator algebras for more details. Therefore we also added many topics, results, examples, details and exercises from other sources. These additional sources are mentioned in these lectures from time to time, but to do them justice we have to name a few of the most important of them; \cite{kadison-r-1, li, murphy, pedersen1, pedersen2, rudinfunctional, takesaki1}. This mixture of sources for the course made us to design the order and depth of the topics differently than other books. Besides, since students attending the course had different background, we had to give full proofs for every statement and explain many details from measure theory and functional analysis as well as the theory of \cs-algebras itself. So, the result was a very self contained series of lectures on \cs-algebra. Hoping that this level of details would help beginners, we decided to prepare these notes in an organized and standard form. During rewriting these notes, we frequently were tempted to add more materials to the original lectures. Although most of the time, we managed to control this temptation, we have added some new topics in order to make the whole notes more consistent and useful. For instance, Sections \ref{sec:loneofG}, \ref{sec:spectralcompact}, \ref{sec:Gelfandduality}, \ref{sec:compactoperators}, and \ref{sec:vonneumann} were not part of the original course. On the other hand, we presented GNS construction fully in the course, but it is not given in these notes. Hopefully, a chapter on states, representations and GNS construction will be added to the present notes in the near future.

The order and list of the topics covered in these notes are as follows: Chapter \ref{ch:banachalgebras} begins with elements of Banach algebras and some examples. We also devote a section to detailed study of Banach algebras of the form $L^1(G)$, where $G$ is a locally compact group. Afterwards, we discus spectrum of elements of Banach algebras. In Section \ref{sec:spectralcompact}, we study basics of the spectral theory of compact operators on Banach spaces. The first chapter is concluded with a section on the holomorphic functional calculus in Banach algebras. Chapter \ref{ch:gelfandduality} is mainly about the Gelfand transform and its consequences. So, the Gelfand transform on commutative Banach algebras and \cs-algebras is discussed in Section \ref{sec:Gelfandtrans}, the continuous functional calculus is presented in Section \ref{sec:confunctionalcal}, and finally the Gelfand duality between commutative \cs-algebras and locally compact and Hausdorff topological spaces is studied in Section \ref{sec:Gelfandduality}. We begin our study of abstract \cs-algebras in Chapter \ref{ch:basics}. Positivity in \cs-algebras, approximate units, ideals of \cs-algebras, hereditary \cs-subalgebras and multiplier algebras are the main topics covered in this chapter. Finally, these notes end in Chapter \ref{ch:Hilbertspaces}, where we present various topics concerning the \cs-algebra $B(H)$ of bounded operators on a Hilbert space $H$. We begin this chapter with presenting necessary notions and materials about Hilbert spaces. Elementary topics about bounded operators on Hilbert spaces are discussed in Section \ref{sec:boundedoperators}. We discuss three important examples of concrete \cs-algebras in Section \ref{sec:concreteexamples} including the reduced group \cs-algebra of a locally compact group $G$. Three locally convex topologies on the \cs-algebra $B(H)$, specifically the strong, weak and strong-$\s$ operator topologies are discussed in Section \ref{sec:topologies}. The Borel functional calculus in $B(H)$ is presented in Section \ref{sec:Borelfunctionalcalculus}. Projections in $B(H)$ and the polar decomposition of elements of $B(H)$ are studied in Section \ref{sec:projections}. In Section \ref{sec:compactoperators}, \cs-algebras of compact operators are studied briefly. Finally, the von Neumann bicommutant theorem is presented in Section \ref{sec:vonneumann}.

Although we have tried to present every topic as easy and self contained as possible, we have left many little details to readers in the form of exercises amongst the main part of the text. We also added some exercises at the end of each chapter. In order to distinguish between these two groups of the exercises, we named the latter group ``problems''. We only used a limited number of references to prepare these notes, but we give a long list of books related to the subject. We hope this list helps student and beginners to find complementary topics related to \cs-algebras.

We welcome any suggestions and comments related to these notes, especially regarding possible mistakes, typos or suggesting new examples, exercises and/or topics. Please, send your comments to shirbisheh@gmail.com or shirbisheh@yahoo.com.


\chapter{Banach algebras and spectral theory}
\label{ch:banachalgebras}

\cs-algebras are a special type of Banach algebras. Therefore many fundamental facts about Banach algebras are usually applicable in the theory of \cs-algebras too. Besides, some \cs-algebras are obtained from some Banach algebras, for instance, the reduced group \cs-algebra, see Example \ref{exa:reducedgroupcsalgebra}. Therefore we devote this chapter to the study of several topics in Banach algebras which are relevant to the theory of \cs-algebras.

In Section \ref{sec:topalgebras}, we gather basic definitions and facts concerning Banach algebras  and give some examples of Banach algebras and \cs-algebras. A detailed study of the Banach algebra $L^1(G)$ associated to a locally compact group $G$ is given in Section \ref{sec:loneofG}. Although the materials presented in this section are not necessary for the basic theory of \cs-algebras, we include this section for several reasons: First, $L^1(G)$ appears naturally in applications of the theory of \cs-algebras in harmonic analysis. Secondly, $L^1(G)$ motivates some constructions in \cs-algebras. And finally, it provides us with many examples of Banach algebras which are neither commutative nor the algebras of bounded operators on some Banach spaces. The spectrum of an element of a Banach algebra is introduced and studied in Section \ref{sec:spectrum}. The algebra of compact operators on a Banach space is another general example of Banach algebras. The spectral theory of compact operators is much richer than the spectral theory of general elements of Banach algebras and it is used in the study of \cs-algebras of compact operators on Hilbert spaces. Therefore we devote Section \ref{sec:spectralcompact} to a detailed study of this topic. Finally, in Section \ref{sec:holomorphic}, we discuss the holomorphic functional calculus in Banach algebras. It is a useful theory which enables us to construct new elements in a Banach algebra by applying certain holomorphic functions defined over the spectrum of an element of Banach algebra. We include this section , because we also discuss the continuous functional calculus in \cs-algebras in Section \ref{sec:confunctionalcal} and the Borel functional calculus in the \cs-algebra $B(H)$ of bounded operators on a Hilbert space $H$ in Section \ref{sec:Borelfunctionalcalculus}. Thus all the three major functional calculi related to \cs-algebras are covered.

Before starting our study of Banach algebras, we recall some well known theorems from functional analysis. Their proofs can be found in standard texts on functional analysis such as \cite{folland-ra, pedersen2, rudinfunctional}.

\begin{theorem}
\label{thm:uniboundedness}
[Uniform boundedness theorem] Let $E$ and $F$ be two Banach spaces. Given a subset $\Sigma\sub B(E,F)$, if the set $\{\|Tx\|; T\in \Sigma \}$ is bounded for every $x\in X$, then $\Sigma$ is bounded, that is the set $\{\|T\|; T\in \Sigma \}$ is bounded.
\end{theorem}
\begin{theorem}
\label{thm:openmapping}
[Open mapping theorem] Every onto bounded linear map $T:E\ra F$ between two Banach spaces is open.
\end{theorem}
\begin{theorem}
\label{thm:closedgraph}
[Closed graph theorem] A linear map $T:E\ra F$ between two Banach spaces is bounded if and only if its graph is a closed subset of $E\times F$.
\end{theorem}
\begin{theorem}
\label{thm:banachalaoghlo}
[The Banach-Alaoglu theorem] Assume $O$ is a neighborhood of $0$ in a topological vector space $V$. The subset $\{ \rho \in V\s; |\rho(x)|\leq 1 \, \text{for all }\, x\in O \}\sub V\s$ is \ws compact. In particular the closed unit ball of $V\s$ is \ws compact.
\end{theorem}
For the proof of the following two theorems see Theorem 3.6 and 3.7 of \cite{rudinfunctional}.
\begin{theorem}
\label{thm:hahnbanach}
[The Hahn-Banach theorem] Assume $X$ is a locally convex topological vector space and $M$ is a subspace of $X$.  Every bounded linear functional of $M$ can be extended to a bounded linear functional on $X$.
\end{theorem}
A subset $Y$ of a complex vector space $X$ is called {\bf balanced} if $\alpha Y\sub Y$ for every $\alpha\in \c$ such that $|\alpha|\leq 1$.
\begin{theorem}
\label{thm:ccblanced}
Assume $B$ is a convex, closed and balanced set in a locally convex space $X$ and $x_0\in X-B$. Then there exists $\rho\in X\s$ such that $|\rho(x)|\leq 1$ for all $x\in B$ and $\rho(x_0)>1$.
\end{theorem}
For the proof of the following proposition see Corollary 1.2.12 of \cite{kadison-r-1}.
\begin{proposition}
\label{prop:locallyconvexfunctional} Assume $Y$ is a closed convex subset of a locally convex topological vector space $X$. For every $x\in X-Y$, there exists a continuous linear functional $\rho\in X\s$ and a real number $b$ such that $Re \rho(x)> b$ and $Re \rho(y)\leq b$ for all $y\in Y$.
\end{proposition}

\section{Basics of Banach algebras}
\label{sec:topalgebras}

In this section, we recall basic definitions of topological algebras, normed algebras, Banach algebras, involutive algebras, and \cs-algebras. We also give many elementary examples for these algebras. Afterwards, we explain some methods for adding a unit element to a Banach algebra or a \cs-algebra. We also present some basic facts about invertible elements in Banach algebras. Finally, we take a closer look at the Banach algebra $B(E)$ of bounded operators on a Banach space $E$ and introduce two important two sided ideal of this algebra; the algebra $K(E)$ of compact operators on $E$ and the algebra $F(H)$ of finite rank operators on $E$. It is also shown that $K(E)$ is a Banach algebra itself.

\begin{definition}
\begin{itemize} \item[(i)] A {\bf topological vector space} is a vector space endowed with a topology such that both the scalar multiplication and the addition are continuous maps.
\item[(ii)] A {\bf topological algebra} is a topological vector space $A$ with a {\bf jointly continuous multiplication}, that is the multiplication $A \times A \ra A$ is a continuous map.
\item[(iii)] A {\bf normed algebra} is a normed space  $(A, \norm)$ with a {\bf sub-multiplicative} multiplication, that is
\[
\|ab\|\leq \|a\| \|b\|, \quad \forall a,b\in A.
\]
\item[(v)] A normed algebra $(A, \norm)$ is called a {\bf Banach algebra} if
$A$ is complete with respect to its norm.
\end{itemize}
\end{definition}
The key point in topological algebras is that the multiplication is always assumed to be jointly continuous. Let $A$ be a ring or an algebra. We denote the algebra of $n\times n$ matrices with entries in $A$ by $M_n(A)$.

\begin{example}
\label{ex:topalg}
\begin{itemize}
\item[(i)] Endow $M_n(\c)$ with the Euclidean topology of $\c^{n^2}$. Then the matrix multiplication is jointly continuous. Therefore $M_n(\c)$ with
this topology is a topological vector space.
\item[(ii)] Let $E$ be a normed space. The {\bf norm operator} on the {\bf algebra $B(E) $ of bounded linear operators on $E$} is defined as follows
\begin{eqnarray*}
\|T\|&:=& \sup \{ \|Tx\|; x\in E, \|x\|= 1\}\\
&=& \sup \{ \|Tx\|; x\in E, \|x\|\leq 1\}\\
&=& \sup \{\frac{\|Tx\|}{\|x\|}; x\in E, x\neq 0\}\\
&=& \inf \{ k ; \|Tx\|\leq k \|x\|, \, \forall x\in E \}.
\end{eqnarray*}
It is easy to see that $\|TS\|\leq \|T\|\|S\|$ for all $T, S\in B(E)$, namely $B(E)$ with the operator norm is a normed algebra.  When $E$ is a Banach space, $B(E) $ is a Banach algebra. If $A$ is a normed (resp. Banach) algebra, $A^n:=A\oplus \cdots \oplus A$ ($n$ copies of $A$) with the norm defined by
\[
\|(x_1,\cdots,x_n)\|:=\max \{ \|x_i\|; i=1,\cdots, n\}, \quad \forall (x_1,\cdots,x_n)\in A^n
\]
is a normed (resp. Banach) algebra. Therefore $M_n(A)$ with the operator norm is a normed (resp. Banach) algebra if $A$ is a normed (resp. Banach) algebra.
\item[(iii)] Let $X$ be a topological space. The {\bf algebra $Bd(X)$ of bounded complex  functions over $X$} equipped with the norm;
\[
\|f\|_{\sup}:=\sup\{|f(x)|; x\in X\}
\]
is a Banach algebra. Some of the subalgebras of $Bd(X)$ are
\begin{itemize}
\item[(a)] the {\bf algebra $C_b(X)$ of continuous and bounded functions},
\item[(b)] the {\bf algebra $C_c(X)$ of continuous and compact support functions}, and
\item[(c)] the {\bf algebra $C_0(X)$ of continuous functions vanishing at infinity}. A continuous function $f$ is called vanishing at infinity if $f^{-1}([\epsilon,\infty[)$ is compact for all $\epsilon>0$.
\end{itemize}
One checks that $C_b(X)$ and $C_0(X)$ are Banach algebras. However, $C_c(X)$ is not complete, and so it cannot be a Banach algebra, unless $X$ is compact. When $X$ is compact, the above subalgebras of $Bd(X)$ are the same as the {\bf algebra $C(X)$ of complex continuous functions on $X$}.
\item[(iv)] Let $(X, \mu)$ be a measure space. For every measurable complex function $f$ on $X$, define
\[
\|f \|_\infty:=\inf\{ a\geq 0 ; \mu(\{x\in X; |f(x)|>a \})=0\}.
\]
It is called the {\bf essential supremum of $|f|$}. One checks that $\|-\|_\infty$ is a semi-norm on the space \[
L^\infty(X,\mu)=L^\infty(X):=\{ f:X\ra \c; f \, \text{is measurable and }\, \|f\|_\infty<\infty\}.
\]
To obtain a norm, we consider the quotient of $L^\infty(X)$ module the subspace of all null functions with respect to $\mu$ and denote this quotient again by $L^\infty(X)$. Then $\|-\|_\infty$ is a norm on $L^\infty(X)$ and $L^\infty(X)$ equipped with this norm and multiplication of functions is a Banach algebra, see also Theorem 6.8 of \cite{folland-ra}. The elements of $L^\infty(X)$ are called {\bf essentially bounded complex function on $X$.}

When $\mu$ is the counting measure, we denote $L^\infty(X)$ by $\ell^\infty(X)$ and the essential supremum is just the supremum of $|f|$.

When $X$ is a topological space and $(X,\mu)$ is a {\bf Borel measure space}, namely the domain of $\mu$ as a $\si$-algebra is generated by open subsets of $X$ (or more generally contains all open subsets of $X$), $\|f\|_\infty=\|f\|_{\sup}$ for every continuous complex function on $X$. Therefore in this case, Banach algebras $C_b(X)$ and $C_0(X)$ are Banach subalgebras of $L^\infty(X)$ and $C_c(X)$ is just a subalgebra of $L^\infty(X)$.
\end{itemize}
\end{example}

\begin{exercise}
Check the details of the above examples.
\end{exercise}

\begin{remark}
\label{rem:linfinity}
It is worthwhile to note that, for given $f\in L^\infty(X)$, the set $\{x\in X; |f(x)|> \|f\|_\infty\}$ is a null set. The following equality proves this:
\[
\{x\in X; |f(x)|> \|f\|_\infty\}=\bigcup_{n=1}^\infty \left\{x\in X; |f(x)|> \|f\|_\infty +\frac{1}{n} \right\}.
\]
\end{remark}
\begin{definition} Let $A$ be an algebra. An {\bf involution} over $A$ is a map $^\ast: A \ra A$ satisfying the following conditions for all $x,y\in A$ and $\lambda \in\c$:
\begin{itemize}
\item[(i)] $(x\s)\s=x$,
\item[(ii)] $(x+y)\s=x\s+y\s$,
\item[(iii)] $(\la x)\s=\overline{\la} x\s$,
\item[(iv)] $(xy)\s=y\s x\s$.
\end{itemize}
When $A$ is a normed algebra, we also assume
\begin{itemize}
\item[(v)] $\|x\s \|=\|x \|$.
\end{itemize}
An algebra $A$ equipped with an involution $\ast$ is called an {\bf
involutive algebra} and is denoted, as an ordered pair, by
$(A,\ast)$. {\bf Involutive normed algebras} and {\bf involutive
Banach algebras} are defined similarly and are denoted by $(A,
\norm, \ast)$. A subalgebra of an involutive algebra is called an
{\bf involutive subalgebra} or a {\bf $\ast$-subalgebra} if it is
closed under the involution.
\end{definition}

\begin{example}
\label{ex:invalg}
\begin{itemize}
\item[(i)] The conjugation map is an involution over $\c$.
\item[(ii)] We denote the algebra of polynomials of two variables $z$ and
$\overline{z}$ with coefficients in $\c$ by $\c[z,\bar{z}]$. We define an involution on this algebra by mapping coefficients of a polynomial to their
complex conjugates and $z$ to $\overline{z}$ and vice versa.
\item[(iii)] Let $A$ be an  involutive algebra. Then $M_n(A)$ is an involutive algebra with the involution defined by
\[
(a_{ij})\mapsto (a_{ji}^\ast), \quad \forall (a_{ij})\in M_n(A).
\]
\item[(iv)] Back to Example \ref{ex:topalg}(iii), the algebra $Bd(X)$ and its subalgebras are involutive normed algebras with the involution defined by
\[
f\s(x):=\overline{f(x)}, \qquad \forall f\in Bd(X), x\in X.
\]
\item[(v)] Back to Example \ref{ex:topalg}(iv), the map $f\s(x):=\overline{f(x)}$ for all $f\in L^\infty(X)$ defines an involution on $L^\infty(X)$.
\item[(vi)] Let $H$ be a Hilbert space with an inner product $\langle \,- , - \rangle$. In Corollary \ref{cor:adjointoperator}, we will show that the {\bf algebra $B(H)$ of bounded operators on $H$} has a unique involution such that
\[
\langle Tx,y\rangle=\langle x,T\s y\rangle, \quad \forall x,y\in H,
T\in B(H).
\]
\end{itemize}
\end{example}

\begin{definition}
An involutive Banach algebra $(A, \norm, \ast)$ is called a {\bf
\cs-algebra} if
\begin{equation}
\label{csnorm} \|x\s x\|=\| x\|^2, \quad\forall x\in A.
\end{equation}
We call the above identity the {\bf \cs-identity}. A norm satisfying this identity is called a  {\bf\cs-norm}.
\end{definition}

One should note that the definition of a \cs-norm does not require completeness of $A$. In other words, we may consider \cs-norms on involutive algebras which are not necessarily complete. Sometimes, these norms are called {\bf pre-\cs-norms} and the normed algebras equipped with them are called {\bf pre-\cs-algebras}.

\begin{example}
\begin{itemize}
\item[(i)] Back to Examples \ref{ex:topalg}(iii) and \ref{ex:invalg}(iv), Banach algebras $C_0(X)$ and $C_b(X)$ are \cs-algebras for all topological spaces $X$.
\item[(ii)] Back to Example and \ref{ex:invalg}(vi), the algebra $B(H)$ is a \cs-algebra, see Proposition \ref{prop:involutionoperator}. If $H$ is finite dimensional, i.e. $H=\c^n$, and is equipped with the ordinary inner product;
\[
\langle x,y \rangle:=\sum_{i=1}^n x_i\overline{y_i},
\]
for all $x=(x_1,\cdots,x_n), y=(y_1,\cdots, y_n)$ in $H$, then $B(H)=M_n(\c)$ with the operator norm and the involution
\[
(a_{ij})\mapsto (\overline{a_{ji}}).
\]
\item[(iii)] If $(X,\mu)$ is a measure space, then the Banach algebra $L^\infty(X)$ is a \cs-algebra.

\item[(iv)] A norm closed involutive subalgebra of a \cs-algebra is a \cs-algebra, and is called a {\bf \cs-subalgebra}.
\end{itemize}
\end{example}
\begin{exercise}
Prove the content of the above examples.
\end{exercise}

Let $A$ be a \cs-algebra and let $S$ be a subset of $A$. The smallest \cs-subalgebra of $A$ containing $S$ is called the {\bf \cs-subalgebra generated by $S$} and is denoted usually by $C\s(S)$. A similar terminology is also used for the Banach (resp. involutive) subalgebra generated by a subset in a Banach (resp. involutive) algebra.

Another easy construction on \cs-algebras, which is needed here, is the direct sum of finitely many \cs-algebras. For $i=1,\cdots, n$, let $A_i$ be \cs-algebras. We define the following involution and norm on the algebraic direct sum $\oplus_{i=1}^n A_i$:
\begin{eqnarray*}
(a_1,\cdots, a_n)\s &=& (a_1\s,\cdots, a_n\s),\\
\|(a_1,\cdots,a_n)\|&=& \max \{ \|a_i\|; i=1,\cdots,n \},
\end{eqnarray*}
for all $(a_1,\cdots,a_n)\in \oplus_{i=1}^n A_i$.

\begin{proposition} The {\bf direct sum} $\oplus_{i=1}^n A_i$ with the above norm and involution is a \cs-algebra.
\end{proposition}
\begin{proof}
We only show the \cs-identity. For all $(a_1,\cdots,a_n)\in \oplus_{i=1}^n A_i$ we have
\begin{eqnarray*}
\|(a_1,\cdots, a_n) (a_1,\cdots, a_n)\s\| &=& \max\{\|a_ia_i\s\|; i=1,\cdots, n\}\\
&=&\max\{\|a_i\|^2; i=1,\cdots, n\}\\
&=&(\max\{\|a_i\|; i=1,\cdots, n\})^2\\
&=& \|(a_1,\cdots, a_n) \|^2
\end{eqnarray*}
\end{proof}

The above examples are fundamental classes of \cs-algebras. In Chapter \ref{ch:gelfandduality}, we shall see that every commutative \cs-algebra is of the form $C_0(X)$ for some topological space $X$. This is the essence of the {\bf Gelfand duality}. Also, every \cs-algebra is isomorphic to a \cs-subalgebra of $B(H)$ for some Hilbert space $H$. This is the main goal of {\bf GNS construction}. Moreover, every finite dimensional \cs-algebra is isomorphic to a direct sum of finitely many \cs-algebras of the form $M_n(\c)$ for some natural numbers $n$.

\begin{definition} Let $A$ be an algebra. An element $a\in A$ is called a {\bf left} (resp. {\bf right}) {\bf unit of $A$} if $ab=b$ (resp. $ba=b$) for all $b\in A$. If $A$ has a left unit $a_1$ and a right unit $a_2$, then $a_1=a_1a_2=a_2$ and this unique element of $A$ is called the {\bf unit of $A$} and usually is denoted by $1_A$ (or simply by 1). In this case, $A$ is called {\bf unital}.
\end{definition}
\begin{exercise}
\label{exe:norm1}
Let $(A,\norm)$ be a unital normed algebra. Show that $\|1\|\geq 1$. If $A$ is a \cs-algebra, then show that $\|1\|= 1$.
\end{exercise}

Given a Banach algebra $(A,\norm)$, for every real number $r\geq 1$, $(A,r\norm)$ is a Banach algebra too. Thus the norm of the unit is not necessarily 1 in unital Banach algebras. However, the norm of an arbitrary Banach algebra can be replaced by another norm so that the new norm of the unit to be 1.

\begin{proposition}
\label{prop:nnormalization} Let $(A,\norm)$ be a unital Banach algebra. Then there exists a norm $\norm_o$ on $A$ such that
\begin{itemize}
\item[(i)] The norms $\norm$ and $\norm_o$ are equivalent on $A$,
\item[(ii)] $(A,\norm_o)$ is a Banach algebra,
\item[(iii)] $\| 1\|_o=1$.
\end{itemize}
\end{proposition}
\begin{proof}
We embed $A$ into $B(A)$ by left multiplication;
\[
L:A\ra B(A), \quad \text{where}\quad L_x(y):=xy, \quad \forall x,y
\in A.
\]
We define the norm $\norm_o$ on $A$ to be the restriction of the
operator norm of $B(A)$ to the image of $A$, that is
\[
\|x\|_o:=\|L_x\|=\sup\{\|xy\|;y\in A, \|y\|\leq 1\}\quad \forall
x\in A.
\]
For $\|y\|\leq 1$, we have $\|xy\|\leq \| x\|\|y\|\leq\|x\|$. This
shows that $\|x\|_o\leq \|x\|$. On the other hand, we have
\[
\frac{\|x\|}{\|1\|}=\frac{\|x1\|}{\|1\|}\leq\sup\{\frac{\|xy\|}{\|y\|};
y\in A, y\neq 0\}=\|x\|_o.
\]
This shows that $\|x\|\leq \|1\|\|x\|_o$ for all $x\in A$ and completes the proof of (i). It follows from (i) that $A$ is a closed subalgebra of $B(A)$, so it is a Banach algebra with the new norm $\norm_o$. Part (iii) is clear from the definition.
\end{proof}
Using the above proposition, we can always assume that the norm of the unit equals 1 in a unital Banach algebra. We shall see that the \cs-norm on a  \cs-algebra $A$ is unique and there is no way to replace it with another \cs-norm unless we change $A$ as well.

Many notions on Banach algebras and \cs-algebras are defined when they are unital. Now, we explain the process of adding unit to a non-unital Banach or \cs-algebra. For Banach algebras, the condition of being non-unital is superfluous and the unitization process can be applied to Banach algebras that are already unital too. But for unital \cs-algebras, we have to use another unitization process. Let $A$ be an involutive Banach algebra. Set $A_1:=A\times \c$ and define the ordinary operations by
\begin{eqnarray*}
(x,\lambda)(y,\mu)&:=&(xy+\lambda y+\mu x, \lambda \mu),\\
(x,\lambda)\s&:=& (x\s, \overline\lambda),\\
\|(x,\lambda)\|&:=& \|x\|+|\lambda|
\end{eqnarray*}
for all $x,y\in A$ and $\lambda, \mu\in \c$. The algebra $A_1$ is called the {\bf  Banach algebra unitization of $A$}.

\begin{exercise}
Show that $A_1$ with the above structure is a unital Banach algebra with the unit $(0,1)$ and the map $A \ra A_1$ is an isometry. Show also that the image of $A$ under this map, which is also shown by $A$, is a closed two sided ideal of $A_1$.
\end{exercise}
For the definition of an isometry see Definition \ref{def:homomorphisms}(v). Given a \cs-algebra $A$, $A_1$ with the above norm is not a \cs-algebra. In fact, one easily checks that the \cs-identity, Equality \ref{csnorm}, does not hold for $x=\left(\left( \begin{array}{ll} 0 &1 \\0&0\end{array}\right), 1\right)\in M_2(\c)\times \c$. In the following, we define another norm on $A\times \c$ which makes it a \cs-algebra.
\begin{proposition}
\label{prop:csunitization}
Let $A$ be a non-unital \cs-algebra. Consider the map $\iota:A\times\c\ra B(A)$, $(x,\lambda)\mapsto L_x+\lambda I$, where $I$ is the identity map on $A$. Then the image of $A\times\c$ under $\iota$ which is denoted by $\tilde{A}$ is a \cs-algebra with the operator norm and the involution defined by
\[
\iota(x,\lambda)\s:=\iota(x\s,\overline{\lambda}), \quad \forall
x\in A, \lambda \in \c.
\]
\end{proposition}
\begin{proof}
First, we show that $\iota$ is injective. For $x\in A$, one notes $\|\iota(x,0)\|= \|L_x\|=\|x\|_o$, where $\norm_o$ is the operator norm defined in the proof of Proposition  \ref{prop:nnormalization}. Hence $\iota(x,0)=0$ if and only if $L_x=0$. But $L_x(x\s)=xx\s$ and $\|xx\s\|=\|x\s\|^2=\|x\|^2\neq 0$ if $x\neq 0$. For $\lambda\neq 0$, if $\|\iota(x,\lambda)\|=0$, then $xy+\lambda y=0$ for all $y\in A$. Substituting $y$ with $y/\lambda$, we get $y=(-x/\lambda)y$, namely $-x/\lambda$ is a left unit in $A$, and consequently $(-x/\lambda)\s$ is a right unit in $A$. This contradicts with the assumption that $A$ is non-unital.

Now, we note that the inclusion $A\hookrightarrow \tilde{A}$ is an isometry, because, for every $x\in A$, we have
\[
\|x\|=\frac{\|xx\s\|}{\|x\s\|}\leq \|\iota(x,0)\|=\sup_{\|y\|\leq1}\|xy\|\leq\|x\|.
\]
Hence $A$ is a Banach subspace of $\tilde A$ of codimension 1. By Proposition 2.1.8 \cite{pedersen2}, $\tilde A$ is a Banach space as
well. We only need to show the \cs-identity. Fix $x\in A$ and $\lambda\in \c$. For $0<t<1$, there is
$y\in A$ such that $\|y\|\leq 1$ and we have

\begin{eqnarray*}
t^2\|\iota(x,\lambda)\|^2 &\leq& \|(x+\lambda I) y\|^2\\
&=&\|y\s (x+\la I)\s(x+\la I)y\|\\
&\leq& \|y\s\|\| (x+\la I)\s(x+\la I)y\|\\
&\leq& \| \iota((x,\la)\s(x,\la))y\|\\
&\leq& \| (\iota(x,\la))\s\iota(x,\la))\|\\
&\leq& \|(\iota(x,\la))\s\|\|\iota(x,\la))\|.
\end{eqnarray*}
By letting $t\ra 1$, we get $\|\iota(x,\lambda)\|\leq \|\iota(x,\lambda)\s\|$. The converse of this inequality is proved similarly. Thus we have
\[
\|\iota(x,\lambda)\|^2 \leq\|
(\iota(x,\la))\s\|\|\iota(x,\la))\|=\|\iota(x,\lambda)\|^2,
\]
which proves the \cs-identity.
\end{proof}

\begin{remark}
\label{rem:unitalcsunitization}
When a \cs-algebra $A$ is already unital, we set $\tilde{A}:=A\oplus \c$, where the right hand side is the direct sum \cs-algebra of $A$ and $\c$. Then $(1,1)$ is the unit element of $\tilde{A}$.
\end{remark}

\begin{exercise}
Assume $A$ is a unital \cs-algebra. Find an algebraic isomorphism from $\tilde{A}=A\oplus \c$ onto $A_1=A\times \c$ which sends the unit element of $\tilde{A}$ to the unit element of $A_1$.
\end{exercise}

\begin{definition} Let $A$ be a non-unital (resp. unital) \cs-algebra. The \cs-algebra $\tilde{A}$ defined in Proposition \ref{prop:csunitization} (resp. Remark \ref{rem:unitalcsunitization}) is called the {\bf \cs-unitization of $A$}.
\end{definition}
Although having a unit element is an advantage for a Banach algebra or a \cs-algebra, there is a weaker notion in these algebras that facilitate many proofs, which use unit elements, in non-unital Banach algebras and \cs-algebras. A net $\{a_i\}$ in a Banach algebra $A$ is called an {\bf approximate unit} if $\|a_i\|\leq 1 $ for all $i$ and $\|aa_i-a\|\ra 0$ and $\|a_i a-a\|\ra 0$ as $i\ra \infty$. We will prove the existence of an approximate unit for certain Banach algebras in Section \ref{sec:loneofG}. However, there are some Banach algebras which admit no approximate units. For example, take a Banach algebra $A$ and change its multiplication into zero. Then it is still a Banach algebra and has no approximate unit. On the contrary, every \cs-algebra possesses an approximate unit. A basic version of this notion for \cs-algebras will be defined in Chapter \ref{ch:gelfandduality} and we will prove the existence of different types of approximate units for \cs-algebras in Chapter \ref{ch:basics}.

\begin{definition}
Let $A$ be a unital algebra. For $a\in A$, we say $b\in A$ is a {\bf left} (resp. {\bf right}) {\bf inverse of $a$} if $ba=1$ (resp. $ab=1$). If $a$ has a left inverse $a_1$ and a right inverse $a_2$, then $a_1=a_1aa_2=a_2$ and this unique element of $A$ is called the {\bf inverse of $a$} and is denoted by $a^{-1}$. In this case, $a$ is called {\bf invertible}. The group of all invertible elements of $A$ is denoted by $A^\times$.
\end{definition}
\begin{proposition}
Let $A$ be a unital Banach algebra. If $\|x-1\|<1$, then $x$ is invertible and
\[
x\inv = \sum_{n=0}^\infty (1-x)^n,
\]
where $a^0:=1$ for all $0\neq a \in A$.
\end{proposition}
\begin{proof}
\begin{eqnarray*}
x \sum_{n=0}^m (1-x)^n&=& (1-(1-x))  \sum_{n=0}^m (1-x)^n\\
&=&1-(1-x)^{m+1}.
\end{eqnarray*}
By letting $m\ra \infty$ and using the fact that $\lim_{m\ra \infty} (1-x)^{m+1}=0$, we see that the series $\sum_{n=0}^\infty (1-x)^n$ is a right (and similarly left) inverse of $x$.
\end{proof}

\begin{proposition}
\label{prop:invelements}
Let $A$ be a unital Banach algebra. The group $A^\times$ is an open subset of $A$. In fact, if $a\in A^\times$ and
$\|x-a\|<1/\|a^{-1}\|$ then $x\in A^\times$ and we have
\[
x\inv = \sum_{n=0}^\infty a\inv(1-xa^{-1})^n.
\]
\end{proposition}
\begin{proof}
Consider the inequality $\|xa^{-1} -1\|=\|(x-a)a^{-1}\|\leq \|x-a\| \|a^{-1}\| <1$ and apply the previous proposition for $xa^{-1}$.
\end{proof}

\begin{corollary}
\label{cor:inversion}
Let $A$ be a unital Banach algebra. The inversion map $x\mapsto x^{-1}$ is continuous in $A^\times$. Therefore $A^\times$ is a topological group.
\end{corollary}
\begin{proof}
Let $a$ be a fixed invertible element of $A$. For all $x\in A$ such that $\|x-a\|<1/\|a^{-1}\|$, we have
\begin{eqnarray*}
\|x\inv -a\inv \|&=& \|\sum_{n=1}^\infty (1-xa^{-1})^n a\inv \|\\
&\leq& \|a\inv\| \sum_{n=1}^\infty \|(1-xa^{-1})^n\| \\
&\leq& \|a\inv\| \sum_{n=1}^\infty \|a\inv\|\|(a-x)\|\|(1-xa^{-1})\|^{n-1}\\
&=&\|a\inv\|^2\|a-x\| \sum_{n=0}^\infty \|(1-xa^{-1})\|^{n},
\end{eqnarray*}
where the latter series is a geometric series and convergent because $\|(1-xa^{-1})\|<1$. Therefore the right hand side of the above inequality is dominated by a constant coefficient of $\|a-x\|$. This proves the continuity of inversion at $a$.
\end{proof}
In Example \ref{ex:topalg}, we observed how the algebra of bounded operators on a Banach space turns out to be a Banach algebra. This type of Banach algebras are very important in the theory of \cs-algebras. It is because every \cs-algebra can be thought of as a \cs-subalgebra of $B(H)$ for some Hilbert space. Therefore we explain some more details here as well as in the exercises.
The closed unit ball in a Banach space $E$ is denoted by $(E)_1$.
\begin{definition}
\label{def:compact}
Let $E$ and $F$ be two Banach spaces and let $T:E\ra F$ be a linear map (not necessarily bounded). It is called {compact} if the image of $(E)_1$ under $T$ is {\bf relatively compact} in $F$, namely $\overline{T((E)_1)}$ is compact. A bounded operator $T$ is called a {\bf finite rank operator} if the dimension of its image is finite. The collection of all compact linear maps (resp. finite rank operators) from $E$ into $F$ is denoted by $K(E,F)$ (resp. $F(E,F)$). When $E=F$, we use the notation $K(E)$ and $F(E)$, respectively.
\end{definition}

Since $(E)_1$ is relatively compact, it is norm bounded, and consequently every compact linear map is bounded. On the contrary, being bounded is part of the definition of a finite rank operator.

\begin{proposition}
\label{prop:compactoperators}
Let $E$ and $F$ be two Banach spaces. Then the following statements are true:
\begin{itemize}
\item [(i)] An operator $T:E\ra F$ is compact if and only if every bounded sequence $\{x_i\}$ in $E$ has a subsequence $\{x_{i_j}\}$ such that $\{T(x_{i_j})\}$ is convergent in $F$.
\item [(ii)] The set $K(E)$ is a closed two sided ideal of $B(E)$, and so it is a Banach algebra.
\item [(iii)] The set  $F(E)$ of finite rank operators and its norm closure are subalgebras of $K(E)$.
\end{itemize}
\end{proposition}
\begin{proof}
\begin{itemize}
\item [(i)] Using the fact that a subset $K$ of a metric space $X$ is compact if and only if every bounded sequence has a convergent subsequence, it is easy to see the statement.
\item [(ii)] Using the above item, it is easy to see that $K(E)$ is a two sided ideal of $B(E)$. We only show that it is a closed subspace of $B(E)$. Let $\{T_n\}$ be a sequence in $K(E)$ convergent to some element $T\in B(E)$. Let $\{x_i\}$ be a bounded sequence in $E$. Assume $M$ is a positive number such that $\|x_i\|\leq M$ for all $i$. For $n\in \n$, choose an increasing function $f_n:\n \ra \n$ such that $T_n(x_{\ff_n(i)})$ is convergent in $E$, where $\ff_n=f_n o f_{n-1}o\cdots of_1$. Therefore $T_n(x_{\ff_m(i)})$ is convergent for all $m\geq n$. Define $f:\n \ra \n$ by $f(i):=\ff_i(i)$. It is an increasing function and $T_n(x_{f(i)})$ is convergent for all $n\in \n$. For given $\ep>0$, pick $n$ such that $\|T_n-T\| <\frac{\ep}{2M}$. Let $y\in E$ be the limit of $T_n(x_{f(i)})$ in $E$. Choose $i_0\in \n$ such that $\|T_n(x_{f(i)}) -y\|< \frac{\ep}{2}$ for all $i\geq i_0$. Then we have
\begin{eqnarray*}
\|T(x_{f(i)})-y\|&\leq& \|T(x_{f(i)})-T_n(x_{f(i)})\|+\|T_n(x_{f(i)})-y\|\\
&<&\frac{\ep\|x_{f(i)}\|}{2M} +\frac{\ep}{2}\leq \ep.
\end{eqnarray*}
    This shows that the subsequence $\{T(x_{f(i)}) \}$ is convergent to $y\in E$, and so $T$ is a compact operator.
\item [(iii)] Again, it is easy to see that the set of all finite rank operators is an ideal in $B(E)$. The rest of the statement follows from the above item if we show every finite rank operator $T$ is compact. Let $R(T)$ denote the image of $T$. Then $R(T)$ is homeomorphic to a copy of $\c^n$, and so has the Heine-Borel property, namely every closed and bounded subset of $R(T)$ is compact. Since $T$ is bounded, the image of the unit ball under $T$ is bounded. Therefore its closure is compact.
\end{itemize}
\end{proof}

To verify that an operator whether $T\in B(E)$ is invertible, one can use the following proposition:

\begin{proposition}
\label{prop:invoperator}
Let $T\in B(E)$ be a bounded operator on a Banach space. Then $T$ is invertible if and only if it is bijective.
\end{proposition}
\begin{proof}
Let $S$ be the algebraic inverse of $T$. Then the graph of these two maps are related as follows:
\[
Graph(S)=\{(a,S(x));x\in E \} = \{(T(y), y); y\in E\}.
\]
The right hand side of the above inequality is closed in $E\oplus E$ because of the continuity of $T$, and so is the left hand side. By the closed graph theorem $S$ is bounded, see Theorem \ref{thm:closedgraph}.
\end{proof}

We shall continue our study of operators on Banach spaces in Section \ref{sec:spectralcompact}.
\section{$L^1(G)$}
\label{sec:loneofG}

In this section, we study an important class of Banach algebras associated to topological groups. These Banach algebras have natural generalizations for other algebraic and topological classes of objects such as topological semi-groups, groupoids, rings, Hecke pairs, which will be discussed later. Since groups and actions of groups on other mathematical objects are commonplace in mathematics, these Banach algebras and their generalizations appear in a number of applications of the theory of \cs-algebras, particularly, in noncommutative geometry and harmonic analysis. Therefore we decided to introduce them in the very beginning of the book to prepare the reader for the complementary discussions which will appear in the upcoming chapters. Historically, these Banach algebras have also inspired some of the developments of the theory of \cs-algebras. For example, the Gelfand transform is considered as the generalization of the Fourier transform. For the sake of briefness, we skip some of the elementary technicalities, mostly from general topology and measure theory. The interested reader can find them in any standard text book of harmonic analysis such as \cite{deitmar-echterhoff, folland-ha}.

\begin{definition}
A group $G$ equipped with a topology is called a {\bf topological group} if the group multiplication $G\times G\ra G$, $(g,h)\mapsto gh$ and the inversion map $G\ra G$, $g\mapsto g\inv$ are both continuous maps. It is called a {\bf locally compact group}, briefly an {\bf LCG}, if its topology is locally compact and Hausdorff.
\end{definition}

Every group with the discrete topology is an LCG. These examples of groups are called {\bf discrete groups}. In fact, the discrete topology is the only topology which makes a finite group into an LCG, because of the Hausdorffness.

\begin{example}
\begin{itemize}
\item [(i)] If $G$ is an abelian LCG, it is called a {\bf locally compact abelian group}. The examples of these groups include $\r$ with summation and with ordinary topology, $\t:=\{ z\in \c; |z=1\}$ with multiplication and topology inherited from $\c$, finite abelian groups, and their products and subgroups, for instance $\r^n$, $\t^n$, $\q$, and so on.
\item [(ii)] $\r^4$ as the quaternion group with the ordinary topology is an LCG, so is its subgroup $S^3:=\{ x\in \r^4; \|x\|=1 \}$.
\item [(iii)] Let $F$ be a {\bf topological field}, that is a field with a topology such that $(F,+)$ and $(F^\times, .)$ are locally compact abelian groups. Then the {\bf $n$th order general linear group of $F$};
    \[
    GL_n(F):=\{ g\in M_n(F); det(g)\neq 0\}
    \]
    with the topology inherited from $M_n=\r^{n^2}$ is an LCG. One of the most important subgroup of $GL_n(F)$ is the {\bf special Linear group}, denoted by $SL_n(F)$, consisting of those elements of $GL_n(F)$ whose determinants are $1$.
\item [(iv)] Finally, we should mention {\bf profinite} groups. A profinite group is an inverse limit of direct system of finite groups equipped with the inverse limit topology. The set of examples of these groups includes all Galois groups of Galois extensions. Because of the complicated nature of these groups, there are several important and challenging  conjectures and theories around these groups which deserve an operator algebraic approach towards them. To give an explicit example, consider $\hat{\z}:=\lim_{\leftarrow} \z/n\z$, which is the absolute Galois group of every finite field $\mathbb{F}_q$. For more details on these groups, we refer the reader to \cite{fried-jarden}.
\end{itemize}
\end{example}

In the next statements, we summarize some of the elementary definitions and properties of locally compact groups that we need in our discussion of $L^1(G)$.  Let $E$ be a subset of a group $G$ and $g\in G$. The set $\{ge;e\in E\}$ is denoted by $gE$. The set $Eg$ is defined similarly. If $F$ is another subset of $G$, then $EF:=\{ef; e\in E \text{\, and\,} f\in F \}=\cup_{e\in E} eF=\cup_{f\in F} Ef$. For the proof of the following lemma see Lemma 1.1.2 of \cite{deitmar-echterhoff}.

\begin{lemma} Let $G$ be an LCG.
\begin{itemize}
\item [(i)] For $s\in G$, the translation maps $g\ra sg$ and $g\ra gs$, as well as the inversion map $g\ra g\inv$ are homeomorphisms of $G$.
\item [(ii)] If $U$ is a neighborhood of unit, then $U\inv:=\{u\inv ; u\in U\}$ is a neighborhood of the unit too. Therefore $V=U\cap U\inv$ is a {\bf symmetric neighborhood of unit}, that is $V=V\inv$.
\item [(iii)] For a given neighborhood $U$ of unit, there is a neighborhood $V$ of unit such that $V^2\subseteq U$.
\item [(iv)] If $A,B\subseteq G$ are compact, then $AB$ is compact.
\item [(v)] If $A,B\subseteq G$ and at least one of them is open, then $AB$ is open.
\end{itemize}
\end{lemma}

\begin{lemma}
\label{lem:uniformcon}
Every function $f\in C_c(G)$ is {\bf uniformly continuous}, namely, for every $\ep>0$, there is a neighborhood $U$ of unit such that $gh\inv\in U$ or $g\inv h\in U$ imply that $|f(g)-f(h)|<\ep$.
\end{lemma}
See Lemma 1.3.6 of \cite{deitmar-echterhoff} for the proof of the above lemma. A function $f:G\ra \c$ is called {\bf symmetric} if $f(g)=f(g\inv)$ for all $g\in G$. Given a function $f$, the formula $f^s(g):=f(g)+f(g\inv)$ defines a symmetric function which possesses most of the properties of $f$. For example if $f$ is compact support or continuous, so is $f^s$.

A reader not acquainted with measure theory is advised to consult with \cite{folland-ra}, or similar text books, before reading the rest of this section. Let $(X, \mathcal{A})$ be a measurable space, that is  $\mathcal{A}$ is a $\si$-algebra on a set $X$. If $X$ is a topological space and $\mathcal{A}$ is generated by all open subsets of $X$, then $\mathcal{A}$ is called the {\bf Borel $\si$-algebra of $X$}. A measure $\mu : \mathcal{A}\ra [0,\infty]$ is called a {\bf Borel measure} if $\mathcal{A}$ contains the Borel $\si$-algebra and it is called {\bf locally finite} if for every point $x\in X$ there exists an open set $U$ containing $x$ such that $\mu(U)<\infty$. In this section, we always assume that $\mathcal{A}$ is the completion of the Borel $\si$-algebra. Therefore a function $f:X\ra \c$ is called {\bf measurable} if $f$ is {\bf Borel measurable}, i.e. $f\inv (U)\in \mathcal{A}$ for all open subset $U\subseteq \c$ and moreover $\mu(f\inv (E))=0$ for every subset $E$ of $\c$ whose Lebesgue's measure is zero, i.e. $E$ is a {\bf null set}. A measurable function $f:X\ra \c$ is called {\bf integrable with respect to $\mu$} if $\int_X |f(x)|d\mu(x) < 0$.

\begin{definition} Let $\mu$ be a locally finite Borel measure on $(X, \mathcal{A})$. Then it is called an {\bf outer Radon measure} if the following two conditions hold:
\begin{itemize}
\item [(i)] For all $E\in \mathcal{A}$, we have
\[
\mu(E)=\inf\{ \mu(U); U \text{\, is open and \,} E\subseteq U\}.
\]
\item [(ii)] For all $E\in \mathcal{A}$ such that either $E$ is open or $\mu(E)< \infty$, we have
\[
\mu(E)=\sup\{ \mu(K); K \text{\, is compact and \,} K\subseteq E\}.
\]
\end{itemize}
\end{definition}

\begin{definition}
A non-zero outer Radon measure $\mu$ on a locally compact group $G$ is called a {\bf Haar measure on $G$} if it is {\bf left invariant}, that is $\mu(gE)=\mu(E)$ for all measurable set $E\subseteq G$ and $g\in G$.
\end{definition}

The existence of a Haar measure on an arbitrary LCG is stated in the following theorem, which is usually proved in harmonic analysis texts, see for instance Theorem 1.3.4 of \cite{deitmar-echterhoff}.

\begin{theorem}
Let $G$ be a locally compact group. Then there exist a Haar measure $\mu$ on $G$. Every measure $\nu$ on $G$ is a Haar measure if and only if it is a multiplication of $\mu$ by a positive real number.
\end{theorem}

When $G$ is discrete, the counting measure is a Haar measure, and so every Haar measure on a discrete group is a positive multiple of the counting measure.

\begin{example}
Recall that the Lebesgue measure on $\r$ is a complete measure $m$ on $\r$ such that $m([a,b])=b-a$ for all $a,b\in [-\infty, +\infty]$. Theorems 1.18 and 1.21 of \cite{folland-ra} state that $m$ is a Haar measure for the group $(\r, +)$.
\end{example}

Recall that a subset of a topological space is called {\bf $\si$-compact} if it can be covered by the union of a sequence of compact sets. For the proof of the following corollaries, we refer the reader to Page 10 of \cite{deitmar-echterhoff}.

\begin{corollary}
\label{cor:haar}
Let $G$ be an LCG with a Haar measure $\mu$.
\begin{itemize}
\item [(i)] Every non-empty open set has strictly positive measure.
\item [(ii)] Every compact set has finite measure.
\item [(iii)] Let $f$ be a continuous positive function on $G$ such that $\int_G f(g)d\mu(g)=0$. Then $f\equiv 0$, namely $f$ equals zero $\mu$-almost every where.
\item [(iv)] Let $f$ be an integrable function on $G$ with respect to $\mu$. Then the support of $f$ is contained in a $\si$-compact open subgroup of $G$.
\end{itemize}
\end{corollary}

The following two propositions show how the topological structure of an LCG is related to the properties of its Haar measures.

\begin{proposition}
\label{prop:gdiscrete} Let $G$ be an LCG with a Haar measure $\mu$ and unit element $e$. The, the following statements are equivalent:
\begin{itemize}
\item [(i)] There exists $g\in G$ such that $\mu(\{g\})\neq 0$.
\item [(ii)] We have $\mu(\{e\})\neq 0$.
\item[(iii)] The Haar measure is a strictly positive multiple of counting measure.
\item [(iv)] The topology of $G$ is discrete.
\end{itemize}
\end{proposition}
\begin{proof} We only show that (ii) and (iii) implies (iv). The rest of implications are easy and left as an exercise. Let $K$ be a compact neighborhood of $e$. Then there is an open set $U$ such that $e\in U\subseteq K$. By the above corollary, we have  $0<\mu(U)<\infty$. By (iii), $U$ has to be a finite set. Since the topology is Hausdorff, for all $g\in U$ the singleton $\{g\}$ must be an open set. Therefore all singletons of elements of $G$ must be open. In other words the topology of $G$ is discrete.
\end{proof}

\begin{proposition} Let $G$ be an LCG with a Haar measure $\mu$. Then $G$ is compact if and only if $\mu(G)<\infty$.
\end{proposition}
\begin{proof} Prove it as an exercise or read the proof in Page 21 of \cite{deitmar-echterhoff}.
\end{proof}

Given $f:G\ra \c$ and $g\in G$, we define two new functions $L_g(f)$ and $R_g(f)$ from $G$ into $\c$ by
\[
L_g(f)(h):=f(g\inv h)\quad \text{and}\quad  R_g(f)(h):=f(hg), \qquad \forall h\in G.
\]
They are respectively called {\bf left} and {\bf right translations of $f$ by $g$}. These maps are bijective over $C_c(G)$, $C_0(G)$, etc.

\begin{lemma}
\label{lem:contrans}
Let $G$ be an LCG with a Haar measure $\mu$. Then for every $f\in C_c(G)$, the function $h\mapsto \int_G f(gh)d\mu(g)$ is continuous on $G$.
\end{lemma}
\begin{proof}
To prove that the above function is continuous at an arbitrary point $h_0\in G$, one can replace $f$ by $R_{h_0}(f)$ and show that the function $h\mapsto \int_G R_{h_0}(f)(gh) d\mu(g) $ is continuous at the unit element $e\in G$. So we prove this simple case instead. Let $K$ be the support of $f$ and let $V$ be a compact symmetric neighborhood of $e$. For $s\in V$, one easily sees that $KV$ contains the support of $R_h(f)$. Since $L_{g\inv }(f)$ is uniformly continuous, for given $\ep>0$, there exists a symmetric neighborhood $W$ of $e$ such that $|f(gh)-f(g)|<\frac{\ep}{\mu(KV)}$ for all $h\in W$. Therefore for $h\in W\cap V$, we have
\begin{eqnarray*}
\left|\int_G [f(gh)-f(g)]d\mu(g)\right| &\leq& \int_{KV} |f(gh)-f(g)|d\mu(g)\\
&<& \frac{\ep}{\mu(KV)} \mu(KV) =\ep.
\end{eqnarray*}
This completes the proof.
\end{proof}

Since a Haar measure on an LCG $G$ is determined up to a positive multiple, we sometimes call it ``the'' Haar measure of $G$. This (sort of) uniqueness of a Haar measure leads us to the definition of the modular function of an LCG. Let $\mu$ be a Haar measure on an LCG $G$. Given $g\in G$, define $\mu_g(E):=\mu(Eg)$ for all measurable set $E\subseteq G$. It is easy to see that $\mu_g$ is a Haar measure on $G$ as well. Thus there is a positive real number $\Delta(g)$ such that $\mu_g=\Delta(g)\mu$.

\begin{definition}
The function $\Delta:G\ra ]0,\infty[$, defined in the above, is called the {\bf modular function of $G$}. Moreover, $G$ is called {\bf unimodular} if $\Delta$ is identically equal to the constant function $1$.
\end{definition}

Obviously, if $G$ is either a locally compact abelian group or a discrete group, then it is a unimodular group. In fact, every compact group is unimodular too, which will be proved later.

The set of all integrable functions on an LCG $G$ with respect to a Haar measure $\mu$ is denoted by $L^1(G)$, that is
\[
L^1(G):=\{f:G\ra \c; \|f\|_1:=\int_G |f(g)| d\mu(g) <\infty\}.
\]
The norm defined in the above formula is called the {\bf $L^1$-norm} and $L^1(G)$ is a Banach space according to the this norm, see Theorem 6.6 of \cite{folland-ra}.

\begin{lemma}
\label{lem:righttrans}
Let $G$ be an LCG with a Haar measure $\mu$. For $f\in L^1(G)$ and $g\in G$, we have $R_g(f)\in L^1(G)$ and
\[
\int_G R_g(f)(h)d\mu(h) =\Delta(g\inv) \int_G f(h)d\mu(h).
\]
\end{lemma}
\begin{proof}
When $f$ is a characteristic function the statement is clear. The general case follows from the usual approximation argument.
\end{proof}

\begin{theorem} Let $G$ be an LCG with a Haar measure $\mu$ and the modular function $\Delta$.
\begin{itemize}
\item [(i)] Let $\r^\times_+$ denote group of positive real numbers with multiplication. The modular function $\Delta: G\ra \r^\times_+$ is a continuous group homomorphism.
\item [(ii)] If $G$ compact, then it is unimodular.
\end{itemize}
\end{theorem}
\begin{proof}
\begin{itemize}
\item [(i)] Let $E\subseteq G$ be a measurable set such that $0<\mu(E)<\infty$. Then for every $g,h\in G$, one computes $\Delta(gh)\mu(E)=\mu(Egh)=\Delta(h)\mu(Eg)= \Delta(h)\Delta(g)\mu(E)$. Thus $\Delta(gh)= \Delta(g) \Delta(h)$, namely $\Delta$ is a group homomorphism. Choose $f\in C_c(G)$ such that $c=\int_G f(g)d\mu(g) \neq 0$. By Lemma \ref{lem:righttrans}, we have
    \[
    \Delta(h) = 1/c \, \int_G f(gh\inv)d\mu(g).
    \]
    The right hand side as a function on $h$ is continuous by Lemma \ref{lem:contrans}, so is $\Delta$.
\item [(ii)] By (i), when $G$ is compact, the image of $\Delta$ is a compact subgroup of $\r_+^\times$. But the only compact subgroup of $\r^\times_+$ is the trivial subgroup $\{1\}$. This means that $\Delta \equiv 1$.
\end{itemize}
\end{proof}
\begin{remark}
Let $\mu$ be a complex Radon measure on $G$. See Section \ref{sec:Borelfunctionalcalculus} for details. Define $I_{\mu}:C_0(G) \ra \c$ by $I_\mu (f) := \int_G f(g) d\mu(g)$. By the Riesz representation theorem, see Theorem \ref{thm:rieszrep}, the map $\mu\mapsto I_\mu$ is an isomorphism between the vector space $M(G)$ of complex Radon measures on $G$ and the dual vector space $C_0(G)\s$. By this correspondence, a left invariant measure $\mu$ is mapped to a functional that is unchanged by the left translation. In other words, $\mu$ is a left invariant Radon measure if and only if $I_\mu(f) = I_\mu (L_g(f))$ for all $g\in G$.
\end{remark}

\begin{lemma}
\label{lem:intinvsub}
Let $G$ be an LCG with a Haar measure $\mu$ and the modular function $\Delta$. Then
\[
\int_G f(g\inv) \Delta(g\inv) d\mu(g) =\int_G f(g)d\mu (g), \qquad \forall f\in L^1(G).
\]
\end{lemma}
\begin{proof} Regarding the correspondence explained in the above remark, we define another Haar measure by using $\mu$. Then we show that it is the same as $\mu$, and as a consequence, we obtain the desired result. For all $f\in C_c(G)$, define $I(f):=\int_G f(g\inv) \Delta(g\inv)d\mu (g)$. Then by Lemma \ref{lem:righttrans}, for all $s\in G$, we have
\begin{eqnarray*}
I(L_s(f))&=& \int_G f(s\inv g\inv ) \Delta(g\inv) d\mu(g)\\
&=& \int_G f((gs)\inv ) \Delta(g\inv) d\mu(g)\\
&=& \int_G f((gs)\inv ) \Delta((gs\inv s)\inv) d\mu(g)\\
&=& \Delta(s\inv)\int_G f(g\inv ) \Delta(sg\inv ) d\mu(g)\\
&=& \int_G f(g\inv ) \Delta(g\inv ) d\mu(g)\\
&=& I(f).
\end{eqnarray*}
This shows that the measure associated to $I$ is left invariant, and consequently, a Haar measure. Therefore there is a $c>0$ such that $I(f)=c\int_G f(g)d\mu(g)$. We need to show that $c=1$ to complete the proof. For given $\ep>0$, choose a symmetric neighborhood $V$ of unit such that $|1-\Delta(s)|<\ep$ for all $s\in V$. Choose $f\in C_c(V)$ such that it is positive, not identically zero, and symmetric. Then we have
\begin{eqnarray*}
|1-c|\int_Gf(g)d\mu(g) &=& \left| \int_G f(g)d\mu(g) -I(f)\right| \\
&\leq & \int_G | f(g) - f(g\inv) \Delta(g\inv)| d\mu(g)\\
&=& \int_V f(g) | 1- \Delta(g\inv) | d\mu(g)\\
&=& \ep \int_G f(g) d\mu(g).
\end{eqnarray*}
Thus $|1-c|\leq \ep$, where $\ep>0$ is arbitrary. Hence $c=1$.
\end{proof}

There is an interesting multiplication formula over $L^1(G)$ which makes it a Banach algebra. For every $f, k\in L^1(G)$, define $f\ast k:G\ra \c$ by
\[
f\ast k (g):=\int_G f(h)k(h\inv g)d\mu(h), \qquad \forall g\in G.
\]

\begin{proposition}
With the above notation, $f\ast k$ belongs to $L^1(G)$ and the above formula defines an associative multiplication called {\bf convolution product}. Moreover, $L^1(G)$ is a Banach algebra with this multiplication.
\end{proposition}

\begin{remark}
\label{rem:sigmacompact}
In the following proof, we use the Fubini-Tonelli theorem several times, see Theorem 2.37 of \cite{folland-ra}. It requires that the measure space $(G, \mu)$ to be $\si$-finite. But we can only show that the support of all functions in the following integrals are contained in $\si$-compact sets. As $\mu$ is a Haar measure and by using Parts (ii) and (iv) of Corollary \ref{cor:haar}, one easily sees that the Fubini-Tonelli theorem holds in this case too.
\end{remark}
\begin{proof}
Define $\alpha:G\times G \ra G$ by $(h,g)\mapsto (h,h\inv g)$ and $\psi: G\times G\ra \c$ by $\psi(h,g):=f(h)k(h\inv g)$. It is clear that $\psi = (f\times k)\, o\, \alpha$. We know that $(f \times k)$ is a measurable function. On the other hand, $\alpha$ is continuous, and so it is Borel measurable too. Hence $\psi$ is Borel measurable as well. To show $\psi$ is actually measurable we need to show that the preimage of every null set in $G\times G$ under $\alpha$ is a null set again. But this follows if we show that the following equality holds for every measurable function $\ff:G\times G \ra G\times G$:
\[
\int_{G\times G}\ff(h,g) d(\mu\times \mu) (h,g)=\int_{G\times G}\ff(h,h\inv g) d(\mu\times \mu) (h,g).
\]
Since $h$ and $k$ are measurable we can use the Fubini-Tonelli theorem. Now by using the fact that $\mu$ is a Haar measure, one easily can check the above equality.

Let $S(f)$ and $S(k)$ denote the supports of $f$ and $k$, respectively. Then the support of $\psi$ is contained in $S(f)\times S(f) S(k)$ which is a $\si$-compact set. Therefore again using Fubini-Tonelli theorem, we have
\begin{eqnarray*}
\|f\ast k\|_1&\leq& \int_G \int_G |f(h)k(h\inv g)|d\mu(h)d\mu(g)\\
&=& \int_G \int_G |f(h)k(h\inv g)|d\mu(g)d\mu(h)\\
&=& \int_G \int_G |f(h)k(g)|d\mu(g)d\mu(h)\\
&=& \int_G |f(h)|dh \int_G |k(g)|d\mu(g)\\
&=&\|f\|_1 \, \|k\|_1 <\infty.
\end{eqnarray*}
Besides the above inequality  which shows $L^1(G)$ is a Banach algebra, the above computation also shows that the function $\psi(g, . )$ is integrable for almost every $g\in G$ and $f\ast k$ is a measurable function. Other algebraic properties of $L^1(G)$, such as the associativity of the multiplication, follow from straightforward computations which are left as an exercise to the reader.
\end{proof}

A generalization of the above proposition is given in Proposition \ref{prop:younginequality}. In fact, $L^1 (G)$ is an involutive Banach algebra. With the above notations, the involution on $L^1(G)$ is defined by
\[
f\s (g) := \Delta(g\inv )\overline{f(g\inv)}, \quad \forall f\in L^1(G), \, \forall g\in G.
\]

\begin{exercise}
Check all axioms of involution for the above $\ast$-operation.
\end{exercise}

It is tempting to see if $L^1 (G)$ is a \cs-algebra. The answer is ``no'', unless $G$ is the trivial group.

\begin{proposition}
\label{prop:triviall1}
With the above notions, $L^1 (G)$ is a \cs-algebra if and only if $G$ is the trivial group of one element. In this case $L^1 (G) \simeq \c$.
\end{proposition}
For a proof of this proposition, we refer the reader to Proposition 2.6.2 of the second edition of  \cite{deitmar-echterhoff} published in 2014\footnote{The proof given for Proposition \ref{prop:triviall1} in the first version of these notes was wrong and therefore it was omitted. I sincerely thank Bat-Od Battseren for pointing me to the mistake.}.

The following theorem and Theorem \ref{thm:l1unital} illustrate how algebraic properties of the Banach algebra $L^1(G)$ reveal some of the algebraic and topological features of $G$.

\begin{theorem} Let $G$ be an LCG. The algebra $L^1(G)$ is commutative if and only if $G$ is abelian.
\end{theorem}
\begin{proof}
Let $L^1(G)$ be commutative. Then for every $f, k \in L^1(G)$ and $g\in G$, we have
\begin{eqnarray*}
0&=&f\ast k(g) - k\ast f(g)\\
&=& \int_G f(h)k(h\inv g) d\mu(h)  - \int_G k(h)f(h\inv g)  d\mu(h).
\end{eqnarray*}
By replacing $h$ with $gh$ and then using Lemma \ref{lem:intinvsub} in the first integral, we get
\begin{eqnarray*}
0&=& \int_G  f(gh\inv)k(h) \Delta(h\inv) d\mu(h)  - \int_G k(h)f(h\inv g)  d\mu(h) \\
&=& \int_G k(h) \left(f(gh\inv) \Delta(h\inv) - f(h\inv g)\right)  d\mu(h).
\end{eqnarray*}
Since this is valid for every $h$, we conclude that $f(gh\inv) \Delta(h\inv) - f(h\inv g)=0$ for all $f\in C_c(G)$ and $g,h\in G$. By setting $g=e$, we conclude that $\Delta= 1$. Therefore $f(gh\inv) = f(h\inv g)$ for all $f\in C_c(G)$ and $g,h\in G$. This implies that $G$ is abelian. The converse direction is easy to check.
\end{proof}
Before stating the next theorem, we need to introduce a notion in the Banach algebra $L^1(G)$ which allows us to use some of the advantages of the unit element even when $L^1(G)$ is not unital. This is actually a net which acts as the unit in the limit. This technique is a powerful idea which also appears in the context of \cs-algebras under the name of approximate identity.

Convergence in topological spaces that are not necessarily metric spaces relies on the notion of ``nets''. Since this concept is going to appear frequently in the future, we give the detailed definition here.

\begin{definition} Let $J$ be a set.
\begin{itemize}
\item [(i)] A {\bf partial order on $J$} is a binary relation $\leq$ such that,  for all $a,b, c\in J$, we have
\begin{itemize}
\item [(a)] $a\leq a$, (it is {\bf reflexive}),
\item [(b)] $a\leq b$ and $b\leq a$ implies that $a=b$, (it is {\bf anti-symmetric}),
\item [(c)] $a\leq b$ and $b\leq c$ implies that $a\leq c$, (it is {\bf transitive}).
\end{itemize}
Then the pair $(J,\leq)$ is called a {\bf partially ordered set}.
\item [(ii)] A partially ordered set $(J,\leq)$ is called a {\bf directed set} if for every $a,b\in J$ there is $c\in J$ such that $a\leq c$ and
$b\leq c$, namely for every two elements of $J$ there is an upper bound.
\item [(iii)] Let $(J,\leq)$ and $(I, \sqsubseteq)$ be two directed sets. A map $\ff:J\ra I$ is called {\bf strictly cofinal} if for every $i_0\in I$ there is some $j_0\in J$ such that  $j_0\leq j$ implies $i_0\sqsubseteq \ff(j)$.
\end{itemize}
\end{definition}

For example, the collection of all subsets (resp. open subsets) of a set (resp. topological space) $S$ equipped with the relation $\subseteq$ is a directed set. The same is true if one considers the converse of inclusion, i.e. $\supseteq$, as the relation.

\begin{definition} Let $X$ be a topological space.
\begin{itemize}
\item [(i)] A {\bf net in $X$} is a function $\alpha:J\ra X$, where $(J, \leq)$ is directed set. Often, $\alpha(j)$ is denoted simply by $\alpha_j$ for $j\in J$ and the net $\alpha$ is denoted by $(\alpha_j)_{j\in J}$ or simply by $(\alpha_j)$.
\item [(ii)] With the above notation, the net $(\alpha_j)$ is called {\bf convergent} to a point $x\in X$ if for every neighborhood $V$ of $x$, there is $j_0\in J$ such that $j_0\leq j$ implies $\alpha_j\in V$.
\item [(iii)] A {\bf subnet} of $\alpha$ is a net $\beta:I\ra X$ together with a strictly cofinal map $\ff:I\ra J$ such that $\beta=\alpha \ff$.
\end{itemize}
\end{definition}

Most statements about sequences in metric spaces have generalizations for nets in topological spaces. For example, a map $f:X\ra Y$ between two topological space is continuous if and only if a net convergent to a point, say $x$, is mapped to a net convergent to $f(x)$ by $f$, see Proposition A.6.4 of \cite{deitmar-echterhoff}. For the proof of the following proposition see Proposition A.6.6 of \cite{deitmar-echterhoff}:
\begin{proposition}
\label{prop:subnetcompact}
A topological space $X$ is compact if and only if every net in $X$ has a convergent subnet.
\end{proposition}
\begin{definition}
Let $G$ be an LCG with a Haar measure $\mu$. A {\bf Dirac net on $G$} is a net $(f_j)$ in $C_c(G)$ such that
\begin{itemize}
\item $f_j\geq 0$ and $\int_G f_j(g)d\mu(g)=1$ for all $j$
\item the support of $f_j$'s shrink to the unit element of $G$, namely, for every neighborhood $V$ of the unit there is $j_0$ such that $j_0\leq j$ implies $supp(f_j)\subseteq V$,
\item and $f_j$ is symmetric for all $j$.
\end{itemize}
\end{definition}
\begin{remark}
\label{rem:diracnet}
In order to construct a Dirac net on an LCG $G$ equipped with a Haar measure $\mu$, consider the directed set $(\mathfrak{U}, \supseteq)$ of all symmetric compact neighborhoods of the unit with inclusion. For given $U\in \mathfrak{U}$, by Urysohn's Lemma, see Theorem 3.1 of \cite{munkeres}, there exists a continuous function $f_U:G\ra [0,1]$ such that $f_U(e)=1$ and $supp(f_U)\subseteq U$. We replace $f_U$ by $f_U^s$ to get a symmetric function and then divide it by $\int_G f_U(g) d\mu(g)$. We denote the function just obtained again by $f_U$ and it is straightforward to check that $(f_U)$ is a Dirac net on $G$.
\end{remark}
\begin{lemma}
\label{lem:contranslp} Let $G$ be an LCG with a Haar measure $\mu$. For given $1\leq p <\infty$  and $f\in L^p(G)$, the maps $g\mapsto L_g(f)$ and $g\mapsto R_g(f)$ are continuous maps from $G$ into $L^p(G)$.
\end{lemma}
\begin{proof}
We first prove this for the case that $f\in C_c(G)$. Let $K$ be the support of $f$ and let $U_0$ be a compact symmetric neighborhood of $e$. Then the support of $L_g(f)$ is contained in $U_0K$ for all $g\in U_0$. Let $\ep>0$. By Lemma \ref{lem:uniformcon}, there exists a neighborhood $U$ of $e$ such that $U\subseteq U_0$ and $\|L_g(f) - f\|_{\sup} < \frac{\ep}{\mu(U_0K)^{1/p}}$ for all $g\in U$. Then we have
\[
\|L_g(f)-f\|_p=\left( \int_G |f(g\inv h ) -f(h)|^p d\mu(h)\right)^{1/p}<\ep.
\]
For general $f\in L^p(G)$, choose $k\in C_c (G)$ such that $\|f-k\|_p<\ep/3$. Also, choose a neighborhood $U$ of $e$ such that $\|k-L_g(k)\|_p<\ep/3$ for all $g\in U$. Then for all $g\in U$, we have
\[
\|f-L_g(f)\|_p\leq \|f-k\|_p +\|k-L_g(k)\|_p + \|L_g(k)-L_g(f)\| <\ep.
\]
In the last step, we used the fact that $\|L_g(k)-L_g(f)\|_p=\|f-k\|_p$ which follows from the left invariance of the Haar measure. The proof for $R_g$, instead of $L_g$, is similar to the above argument except in the very last step. For the last step, one can use Lemma \ref{lem:righttrans} to show that $\|R_g(f) - R_g(k)\|_p = (\Delta(g\inv))^{1/p} \|f-k\|_p$. Since $\Delta$ is continuous and $g$ varies in $U\subseteq U_0$, where $U_0$ is compact, one can easily find a similar estimation to show that $\|f-L_g(f)\|_p\ra 0$ when $U$ shrinks to $e$.
\end{proof}
In the above lemma, we used the following proposition. see Proposition 7.9 of \cite{folland-ra}.
\begin{proposition}
\label{prop:ccxdense} If $\mu$ is a Radon measure on a locally compact and Hausdorff topological space $X$, then $C_c(X)$ is dense in the Banach space $L^p(X,\mu)$ for all $1\leq p<\infty$.
\end{proposition}
\begin{lemma}
\label{lem:diracnet} Let $G$ be an LCG with a Haar measure $\mu$. Let $(f_j)$ be a Dirac net on $G$ and let $f\in L^1(G)$. Then the nets $(f_j\ast f)$ and $(f\ast f_j)$ converge to $f$ in $L^1(G)$. Moreover, if $f$ is continuous, then both the convolution products exist, and $(f_j\ast f)(g)$ and $(f\ast f_j)(g)$ converge to $f(g)$ for all $g\in G$.
\end{lemma}
\begin{proof} One computes
\[
\|f_j\ast f -f \|_1 = \int_G \left| \int_G f_j(h)f(h\inv g )d\mu(h) - f(g) \right| d\mu(g).
\]
Using the fact that $f(g)= f(g)\int_G f_j (h)d\mu(h)=\int_G f(g)f_j (h)d\mu(h)$, we have
\begin{eqnarray*}
\|f_j\ast f -f \|_1 &=& \int_G \left| \int_G f_j(h)\left( f(h\inv g ) - f(g)\right) d\mu(h) \right| d\mu(g)\\
&\leq & \int_G  \int_G f_j(h)\left| f(h\inv g ) - f(g) \right| d\mu(g) d\mu(h)\\
&=& \int_G  f_j(h) \| L_h(f)(g ) - f(g) \|_1  d\mu(h)\\
&=& \int_{supp(f_j)}  f_j(h) \| L_h(f)(g ) - f(g) \|_1  d\mu(h).
\end{eqnarray*}
By Lemma \ref{lem:contranslp}, this integral goes to zero when $j$ tends to infinity. Similarly, one computes $\|f\ast f_j -f \|_1\ra 0$ when $j\ra \infty$.

Assume $f$ is continuous and fix $g\in G$. For given $\ep>0$, by continuity of $f$, there exists a neighborhood $U$ of $e$ such that $gh \in U$ implies that $|f(h\inv)-f(g)|<\ep$, see Lemma \ref{lem:uniformcon}. By the definition of a Dirac net, there is $j_0$ such that $j_0 \leq j$ implies that $supp(f_j)\subseteq U$. Therefore for $j_0 \leq j$, we have

\begin{eqnarray*}
|f_j\ast f -f |&\leq & \int_G f_j(h) \left| f(h\inv g) -f(g)  \right| d\mu (h) \\
&=& \int_G f_j(gh) \left| f(h\inv ) - f(g)  \right| d\mu (h)< \ep. \\
\end{eqnarray*}
\end{proof}

Now, we describe $L^1(G)$ when $G$ is a discrete group.

\begin{remark}
\label{rem:uncountable sum}
We need to explain the meaning of an uncountable summation, say $\sum_{s\in S} a_s$, where $S$ is an uncountable set and all terms of this summation belong to a (complex or real) topological vector space $B$. Let $(\mathfrak{F},\subseteq)$ be the directed set of all finite subsets of $S$ with inclusion. For every $F\in \mathfrak{F}$, define $x_F=\sum_{s\in F} a_s$. Then $(x_F)$ is a net in $B$. We say that the summation $\sum_{s\in S} a_s$ is {\bf convergent} if the net $(x_F)$ is convergent. The {\bf absolutely convergent summations} are defined similar to the ordinary absolutely convergent series.
\end{remark}

Assume $(X,\mu)$ is a measure space. When $\mu$ is the counting measure, we use $\ell^p(X)$ in lieu of $L^p(X,\mu)$ for all $1\leq p\leq \infty$ and drop $\mu$ from the notation.  The Banach space $\ell^1(X)$ is the set of all absolutely convergent summations indexed by $X$. In particular, when $G$ is a discrete group, we have
\[
\ell^1(G)=\{ \sum_{g\in G}\la_g ; \sum_{g\in G} |\la_g| <\infty\}.
\]

By describing $\ell^1(G)$ as above, one easily sees that $\c G$, the group algebra of $G$, can be considered as a dense subalgebra of $\ell^1(G)$. We recall the definition of $\c G$ now. Consider the complex vector space generated by elements of $G$. This vector space becomes an algebra called the {\bf group algebra of $G$} if we extend the group multiplication linearly to all its elements.  The explicit formula for the multiplication of this algebra is as follows:
\[
\left(  \sum_{i=1}^n a_i g_i \right) \left( \sum_{j=1}^m b_j g_j \right)= \sum_{i=1}^n \sum_{j=1}^m  a_ib_j g_ig_j,
\]
where $a_i , b_j\in \c$ and $g_i, g_j\in G$ for all $i, j$.

To embed $\c G$ into $\ell^1(G)$, we send every element $g\in G$ to the characteristic function of the singleton $\{ g\}$, which we denote it by $\delta_g$, and extend this map linearly to whole $\c G$. Clearly, it is a linear injection. We only have to show that it is actually an algebraic homomorphism. We check this only for the product of two arbitrary elements of the basis of $\c G$. For all $g_1, g_2, h \in G$, we have
\begin{eqnarray*}
\d_{g_1} \ast \d_{g_2}(h)&=& \sum_{s\in G} \d_{g_1} (s) \d_{g_2}(s\inv h) \\
&=&  \d_{g_2}(g_1\inv h) \\
&=& \left\{ \begin{array}{ll} 1 & g_2=g_1\inv h \\ 0 & \text{otherwise}
\end{array} \right.\\
&=& \d_{g_1 g_2} (h).
\end{eqnarray*}
In fact, the image of $\c G $ in $\ell^1(G)$ is exactly $C_c(G)$ with convolution product. The above observation leads us to two easy, but important, conclusions; first, $\c G$ is a dense subalgebra of $\ell^1(G)$ and secondly, $\d_e$ is the unit element of $\ell^1(G)$. However, for a general LCG $G$, $L^1(G)$ is unital only if $G$ is discrete.
\begin{theorem}
\label{thm:l1unital}
Let $G$ be an LCG with a Haar measure $\mu$. The Banach algebra $L^1(G)$ is unital if and only if $G$ is discrete.
\end{theorem}
\begin{proof}
Let $k$ be the unit of $L^1(G)$ and let $(f_j)$ be a Dirac net on $G$. For given $\ep>0$, by Lemma \ref{lem:diracnet}, there is $j_0$ such that $j_0<j$ implies $\|f_j \ast k- k\|_1 <\ep$ or equivalently $\|f_j - k\|_1 <\ep$. Since $\ep$ is arbitrary and the support of $f_j$ shrinks as $j\ra \infty$, $supp(k)=\{e\} \cup E$, where $\mu(E)=0$. But $k\neq 0$, so $\mu(\{e\})>0$. This implies that $G$ is discrete by Proposition \ref{prop:gdiscrete}. The converse follows from the above discussion.
\end{proof}

We conclude this section by introducing the Fourier transform briefly. Let $G$ be a locally compact abelian group with a Haar measure $\mu$. A {\bf character on $G$} is a continuous group homomorphism from $G$ into the group $\t$ of complex numbers of absolute value $1$. The set of all characters of $G$ is denoted by $\hat{G}$ and it has a natural group structure as follows:
\begin{eqnarray*}
(\rho_1 \rho_2)(g)&:=&\rho_1(g)\rho_2(g),\quad \forall \rho_1, \rho_2\in \hat{G}, \quad \forall g\in G\\
(\rho\inv)(g)&:=&    (\rho(g))\inv,     \quad \forall \rho\in \hat{G}, \quad \forall g\in G.
\end{eqnarray*}

The group $\hat{G}$ with compact open topology is an LCG. The structure of this group and its relation to the structure of $G$ is discussed in harmonic analysis. The key role is played by a mapping $L^1(G) \ra C_0(\hat{G})$ named the {\bf Fourier transform} defined as follows:
\begin{eqnarray*}
f&\mapsto & \hat{f}\\
\hat{f}(\rho)&:=& \int_G f(g)\overline{\rho(g)} \d\mu(g).
\end{eqnarray*}
The idea is to represent elements of the rather complicated Banach algebra $L^1(G)$ as elements of the more simple Banach algebra $C_0(\hat{G})$. This technique motivates a number of ideas in representation theory of groups as well as the theory of \cs-algebras, for example, see the Gelfand transform in Section \ref{sec:Gelfandtrans}.

\section{The spectrum of elements of a Banach algebra}
\label{sec:spectrum}

\begin{definition} Let $A$ be a unital complex algebra. For $a\in A$, the {\bf spectrum of $a$ in $A$} is defined and denoted as follows:
\[
\sigma_A(a):=\{ \lambda\in \mathbb{C} ; a-\lambda 1 \notin A^\times \}.
\]
The complement of $\sigma_A(a)$ in $\c$ is called the {\bf resolvent of $a$ in $A$} and is denoted by $Res_A(a)$. If $A$ is non-unital, the spectrum of an element $a\in A$ is defined by
\[
\si_A(a):=\si_{A_1}((a,0))\cup \{0\}.
\]
\end{definition}
To simplify the notation, we denote $\la 1\in A$ by $\la$ for all $\la\in\c$. Also, when there is no risk of confusion, we drop $A$ from the notation of the spectrum and the resolvent of an element $a\in A$ and shortly write $\si(a)$ and $Res(a)$.

\begin{example}
\begin{itemize}
\item[(i)] Let $a\in M_n(\c)$. Then the spectrum of $a$ in $M_n(\c)$ is the set of all eigenvalues of $a$.
\item[(ii)] Let $X$ be a compact topological space and let $f\in C(X)$. Then $\si_{C(X)}(f)= f(X)$.
\end{itemize}
\end{example}
\begin{exercise}
Verify the statements in the above example.
\end{exercise}
\begin{proposition}
\label{prop:speccomm}
If $A$ is a unital algebra and $a,b \in A$, then
\[
\si_A(ab)\cup \{0\}=\si_A(ba)\cup \{0\}.
\]
\end{proposition}
\begin{proof}
Let $0\neq\la\in Res_A(ab)$ and set $u:=(ab-\la)\inv$. Hence $abu=uab=1+\la u$, and from this we obtain
\begin{eqnarray*}
(ba-\la)(bua-1)&=&\la\\
(bua-1)(ba-\la)&=&\la.
\end{eqnarray*}
Thus $ba-\la$ is invertible, and so $\la\in Res_A(ba)$.
\end{proof}
\begin{definition}
Let $A$ be a Banach algebra. For every $a\in A$, the {\bf spectral radius of $a$ in $A$} is defined and denoted as follows:
\[
r_A(a):=\sup\{|\la|; \la\in\si_A(a)\}.
\]
\end{definition}

To simplify the notation, sometimes the spectral radius of $a$ in $A$ is denoted simply by $r(a)$. Later, we shall show that the spectrum of an element $a$ of a Banach algebra is not empty and consequently $r(a)\geq 0$. In the following proposition, we find an upper bound for $r(a)$.
\begin{proposition}
\label{prop:spradius1}
Let $A$ be a Banach algebra and let $a\in A$. Then $r(a)\leq \|a\|$ and $\si_A(a)$ is a compact subset of $\c$.
\end{proposition}
\begin{proof}
We can assume that $A$ is unital. Let $\la$  be an element of $\c$ such that $\|x\|< |\la|$. Then
\[
\|1-(1-x/\la )\|=\|x/\la\|<1.
\]
Thus $1-x/\la$ is invertible and so $x-\la$ is invertible, namely, $\la\in Res_A(x)$. This shows that $r(a)\leq \|a\|$, and so $\si_A(a)$ is a bounded subset of $\c$. Next, we note that the map $\ff: \c\ra A$ defined by $\la\mapsto a-\la$ is continuous. Thus the set $Res_A(a)=\ff\inv(G(A))$ is open in $\c$, and so $\si_A(a)$ is closed. Therefore $\si_A(a)$ is compact.
\end{proof}
\begin{exercise}
\label{ex:3}
Let $A$ be an algebra and $a\in A$. If $\la\in \si(a)$, then show that $\la^n\in\si(a^n)$ for all $n\in\n$.
\end{exercise}

In the rest of this chapter, we use some facts from the theory of holomorphic (analytic) vector valued functions of one complex variable. This theory is similar to the elementary theory of complex functions and the interested reader can find more details about it in Section III.4 of \cite{ds1}.
\begin{proposition}
\label{prop:spradius2}
Let $A$ be a Banach algebra. For every $a\in A$, the sequence $\|a^n\|^{1/n}$ converges to $r(a)$.
\end{proposition}
\begin{proof} We prove the following inequalities:
\[
 \limsup_{n\ra \infty} \|a^n\|^{1/n} \leq r(a)\leq \liminf_{n\ra \infty} \|a^n\|^{1/n}.
\]

Let $\la\in \si(a)$. Then it follows from the above exercise and Proposition \ref{prop:spradius1} that $|\la|^n\leq \|a^n\|$, for all $n\in \n$. Hence $r(a)\leq \|a^n\|^{1/n}$, for all $n\in \n$, which implies that the right hand side inequality.

For $\lambda \in \c$ with $|\la |>r(a)$, we claim the series $\sum_{n=0}^\infty \frac{a^n}{\la^{n+1}}$ is absolutely convergent. Thus by the $n$th root test, we must have $\limsup_{n\ra \infty} \left( \frac{\|a^n\|}{\la^{n+1}} \right)^{1/n} \leq 1$. This clearly implies the left hand side inequality.

To prove the above claim, we first note that the function $f(\mu):=\frac{1}{a-\mu}$ is holomorphic on $Res(a)$ because of the following discussion. If $\mu_0\in Res(a)$ and $|\mu -\mu_0|< \|a-\mu_0\|$, then $\|(a-\mu_0)-(a-\mu )\|<\|a-\mu_0\|$, which implies $\|1- \frac{a-\mu}{a-\mu_0} \|\leq 1$. Hence $\frac{a-\mu}{a-\mu_0}$ is invertible. By computing its inverse and after some simplifications, we obtain
\[
f(\mu)= \sum_{n=0}^{\infty} (a-\mu)^n (f(\mu_0))^{n+1}.
\]
This power series is convergent over the open neighborhood $\{ \mu; |\mu -\mu_0|< \|a-\mu_0\| \}$ of $\mu_0$ and shows that $f$ is holomorphic in this neighborhood.

Secondly, we observe that, for $\la\in \c$ such that $|\la|> \| a\|$, the series $\sum_{n=0}^{\infty} \frac{-a^n}{\la^{n+1}}$ is absolutely convergent in norm to $f(\la)$ and this convergence is uniformly over any neighborhood like $\{ \mu; |\mu|\geq \|a\|+\epsilon \}$ for some $\epsilon>0$. Therefore this series is the Laurent expansion of $f$ around $\infty$. Since $f$ is holomorphic in $Res(a)$, and so in the neighborhood $\{ \mu; |\mu|> r(a) \}$, the above series is absolutely convergent for every $\la$ in this latter neighborhood as we claimed.
\end{proof}

\begin{proposition}
Let $A$ be a unital Banach algebra. Then $\si_A(a)$ is non-empty for all $a\in A$.
\end{proposition}

\begin{proof}
Given $a\in A$, let $f$ be as above. For every $\ff\in A\s$, define $f_{\ff}:=\ff f : Res(a) \ra \c$. Then $f_{\ff}$ is holomorphic. Hence if $\si(a)$ is empty, then $f_{\ff}$ is entire. On the other hand, it is easy to see that $\lim_{\la\ra \infty} f_{\ff} (\la) =0$, which implies that $f_{\ff}$ is bounded. Therefore by the Liouville theorem, see Theorem 10.23 of \cite{rudinreal}, $f_{\ff}$ is a constant function. Moreover, $f_\ff$ has to be the zero function because of the above limit. Since this holds for all $\ff\in A\s$, we conclude $f(\la)=0$, for all $\la\in\c$. But this contradicts with the fact that values of $f$ are inverses of some elements of $A$.
\end{proof}

\begin{corollary}
\label{cor:gelfandmazur} [Gelfand-Mazur] If a Banach algebra $A$ is a division ring, then it is isomorphic to $\c$.
\end{corollary}

\begin{proof} Since $A$ is a division ring, it is unital and so it contains a copy of $\c$. For an arbitrary element $a\in A$, there is some $\la\in \c$ such that $a-\la$ is not invertible in $A$ and so it has to be zero because $A$ is a division ring. Hence $a\in \c$.
\end{proof}

\begin{proposition}
\label{prop:subalgspec1}
Let $A$ be a closed unital subalgebra of a unital Banach algebra $B$, i.e. $1_B\in A$. Then for every $a\in A$, we have
\[
\partial \si_A(a)\subseteq \si_B(a)\subseteq \si_A(a),
\]
where $\partial \si_A(a)$ denotes the boundary of  $\si_A(a)$ in $\c$.
\end{proposition}
\begin{proof}
Every invertible element of $A$ is invertible in $B$ too. This implies the right hand side inclusion. It follows from \ref{prop:spradius1} that $\si_A(a)$ is a closed subset of $\c$, so $\partial \si_A(a)\subseteq \si_A(a)$. Given $\la\in\partial \si_A(a)$, let $\{\la_n\}$ be a sequence in $Res_A(a)$ convergent to $\la$. Then $a-\la_n \ra a-\la$ in $A$ and so in $B$. If $a-\la$ is invertible in $B$, then by continuity of inversion, see Corollary \ref{cor:inversion}, we obtain $(a-\la_n)^{-1} \ra (a-\la)^{-1}$. Now, since $A$ is closed and the sequence $\{(a-\la_n)^{-1}\}$ is in $A$, its limit, namely $(a-\la)^{-1}$ belongs to $A$, that is $\la\notin \si_A(a)$. This is a contradiction. Therefore $\la\in \si_B(a)$. This proves the left hand side inclusion.
\end{proof}

\section{The spectral theory of compact operators}
\label{sec:spectralcompact}

In this section, $E$ and $F$ are two Banach spaces and we are often dealing with compact operators in $K(E,F)$ or $K(E)$.  For every operator $T\in B(E,F)$, we denote the kernel of $T$ by $N(T)$ and the image of $T$ by $R(T)$.

\begin{definition}
Given $T\in B(E, F)$, a complex numbers $\la$  is called an {\bf eigenvalue} of $T$ if $T-\la $ is not one-to-one. The set of all eigenvalues of $T$ is denoted by $e(T)$. For every $\la\in e(T)$, the {\bf eigenspace of $\la$} is $N(T-\la)$ and every element of the eigenspace of $\la$ are called an {\bf eigenvector of $\la$}.
\end{definition}
Clearly, $e(T)\sub \si(T)$ for all $T\in B(E)$. We shall show that every non-zero $\la\in \si(T)$ is an eigenvalue of $T$ as well provided that $T\in K(E)$.
\begin{proposition}
\label{prop:compactopt1}
\begin{itemize}
\item [(i)] Assume $T\in K(E,F)$ and $R(T)$ is closed, then $T$ is finite rank.
\item [(ii)] Assume $T\in K(E)$ and $0\neq \la\in \c$, then $dim N(T-\la)<\infty$.
\item [(iii)] If $E$ is infinite dimensional and $T\in K(E)$ then $0\in \si(T)$.
\end{itemize}
\end{proposition}
\begin{proof}
\begin{itemize}
\item [(i)] If $R(T)$ is closed, then $R(T)$ is complete, because $F$ is complete. Hence the map $T:E\ra R(T)$ is open by the open mapping theorem, see Theorem \ref{thm:openmapping}. Therefore the image of every ball in $E$ under $T$ is an open set in $R(T)$ whose closure is compact. This means $R(T)$ is a locally compact topological vector space. Hence by Theorem 1.22 of \cite{rudinfunctional}, $R(T)$ is finite dimensional.
\item [(ii)] For every $x\in N(T-\la)$, we have $(T-\la)Tx=T(T-\la)x=0$. Therefore The map $T|_{N(T-\la)}: N(T-\la)\ra N(T-\la)$ is well defined. Since $\la\neq 0$, this map is onto as well. By (i), since $T|_{N(T-\la)}$ is compact and $N(T-\la)$ is closed, the image of this map, which is $N(T-\la)$, is finite dimensional.
\item [(iii)] If $0\notin \si(T)$, then $T$ is invertible. By an argument similar to Part (i), one can show that $R(T)=E$ is locally compact, and so finite dimensional. But this is a contradiction. Hence $0\in \si(T)$.
\end{itemize}
\end{proof}

\begin{lemma}
\label{lem:complementedsub}
Let $M$ be a closed subspace of a topological vector space $X$.
\begin{itemize}
\item [(i)] If $X$ is locally convex and $dim M<\infty$, then there exist a closed subspace $N$ of $X$ such that $X=M\oplus N$.
\item [(ii)] If $dim(X/M)<\infty$, then there exist a closed subspace $N$ of $X$ such that $X=M\oplus N$.
\end{itemize}
\end{lemma}
\begin{proof}
\begin{itemize}
\item [(i)] Let $\{e_1,\dots, e_n\}$ be a basis for $M$ and let $\{\alpha_1,\cdots, \alpha_n\}$ be its dual basis. Applying the Hahn-Banach theorem \ref{thm:hahnbanach}, we extend $\alpha_i$ to a (bounded) linear functional on $X$ for all $1\leq i\leq n$. Set $N:=\cap_{i=1}^n N(\alpha_i)$. It is straightforward to check that $X=M\oplus N$.
\item [(ii)] It is an easy linear algebra exercise that there is a finite dimensional $N$ of $X$ such that $X=M\oplus N$. Since $N$ is finite dimensional, it is closed in $X$.
\end{itemize}
\end{proof}

\begin{exercise}
\label{exe:projectionsbanach}
Prove that if $M$ and $N$ are two closed subspace of $E$ such that $E=M\oplus N$, then the projections maps $\pi_1:E\ra M$ and $\pi_2: E\ra N$ are bounded operators, (Hint: use the closed graph theorem). Therefore if we equip $M\oplus N$ with the norm $\|m+n\|:=\|m\|+\|n\|$ for all $m\in M$ and $n\in  N$, then $\pi_1 +\pi_2: E\ra M\oplus N$ is a bounded isomorphism with a bounded inverse.
\end{exercise}
\begin{exercise}
\label{exe:sumofcompactoperrators}
Assume $H$ is a Banach space and $T:E\ra H$ and $S:E\ra F$ are compact operators. Show that $T+S:E\ra H\oplus F$ defined by $x\mapsto Tx+Sx$ is a compact operator, where the norm on $H\oplus F$ is defined as the above example.
\end{exercise}
\begin{definition}
An operator $T\in B(E, F)$ is called {\bf bounded below} if there is an $\ep>0$ such that $\ep \|x\|\leq \|Tx\|$ for all $x\in E$.
\end{definition}
\begin{exercise}
Let $T\in B(E,F)$ be a bounded below operator. Then $R(T)$ is closed.
\end{exercise}
\begin{proposition}
\label{prop:compactclosedrange}
For every $T\in K(E)$ and $\la\neq 0$, the subspace $R(T-\la)$ is closed in $E$.
\end{proposition}
\begin{proof}
By Proposition \ref{prop:compactopt1}(ii), $N(T-\la)$ is a finite dimensional closed subspace of $E$, and consequently by Lemma \ref{lem:complementedsub}, there exists a closed subspace $M$ of $E$ such that $E=N(T-\la)\oplus M$. Let $S:M\ra E$ be the restriction of $T-\la$ to $M$. Then $S$ is bounded, one-to-one and $R(S)=R(T-\la)$. Since, $E$ is complete and $S$ is continuous, in order to prove that $R(S)$ is closed, it is enough to show that $S$ is bounded below. If it is not bounded below, then there is a sequence $\{x_n\}$ in $M$ such that $Sx_n\ra 0$ and $\|x_n\|=1$ for all $n\in \n$. Since $T$ is a compact operator, there is a subsequence of $\{T(x_n)\}$, say $\{T(x_{n_i})\}$, converging to some point $x_0\in E$. Since $S=T-\la$ on $M$ and $Sx_{n_i}\ra 0$, we have
\[
\lim_{i\ra \infty} \la x_{n_i}= \lim_{i\ra \infty} (Tx_{n_i} - S x_{n_i})= x_0.
\]
This implies $x_0\in M$ and $Sx_0=\lim_{i\ra \infty} \la Sx_{n_i} = 0$. Since $S$ is one-to-one, $x_0=0$. But, this contradicts with $\|x_{n_i}\|=1$.
\end{proof}
\begin{lemma}
\label{lem:subspacenormed}
Let $X$ be a normed vector space and let $M$ be a subspace of $X$. Assume $M$ is not dense in $X$. For every $r>1$, there exists $x\in X$ such that $\|x\|<r$ and $\|x-y\|\geq1$ for all $y\in M$.
\end{lemma}
\begin{proof}
Since $M$ is not dense in $X$, the quotient space $X/\overline{M}$ is at least one dimensional. Using the quotient norm, it is clear that one can find $x_1\in X$ such that $\inf\{\|x_1 - y \|; y\in M\}=1$. Hence there exists $y_1\in M$ such that $\|x_1 - y_1\|<r$. Set $x:=x_1 - y_1$.
\end{proof}
\begin{proposition}
\label{prop:eigenvaluefinite}
For given $T\in K(E)$ and $r>0$, set
\[
e_r(T):=\{ \la\in e(T); |\la|>r\}.
\]
Then we have
\begin{itemize}
\item [(i)] $R(T-\la)\neq E$ for all $\la\in e_r(T)$, and
\item [(ii)] $e_r(T)$ is finite.
\end{itemize}
\end{proposition}
\begin{proof} We first describe a hypothesis which leads to a contradiction. Afterwards, we shall show that the failure of either one of (i) or (ii) implies our hypothesis, and so a contradiction. Assume there exist a sequence of closed subspaces $M_n$ of $E$ and a sequence of scalars $\la_n\in e_r(T)$ such that the following conditions hold:
\begin{itemize}
\item [(a)] $M_1\varsubsetneq M_2\varsubsetneq M_3 \varsubsetneq \cdots$.
\item [(b)] $T(M_n)\sub M_n$, for all $n\in\n$.
\item [(c)] $(T-\la_n)(M_{n+1})\sub M_n$, for all $n\in\n$.
\end{itemize}
By Lemma \ref{lem:subspacenormed}, for every $n\geq 2$, there exists $y_n\in M_n$ such that
\begin{equation}
\label{eqn:yninmn}
\|y_n\|< 2 \quad \text{and}\quad
\|y_n-x\|\geq 1, \qquad \forall x\in M_{n-1}.
\end{equation}
Then for $n>m\geq 2$, we define
\[
z_{m,n}:=Ty_m -(T-\la_n)y_n.
\]
Conditions (b) and (c) imply that $z_{m,n}\in M_{n-1}$. Hence by (\ref{eqn:yninmn}), we have
\[
\|Ty_n-Ty_m\|=\|\la_n y_n - z_{m,n}\|= |\la_n|\|y_n - \la_n\inv z_{m,n}\|\geq |\la_n|>r.
\]
This shows that the sequence $\{ Ty_n\}$ has no convergent subsequence and this contradicts with $T$ being a compact operator.

Assume (i) is false, namely $R(T-\la_0)=E$ for some $\la_0\in e_r(T)$. Set $S:=T-\la_0$ and define $M_n:=N(S^n)$ for all $n\in\n$. Since $\la_0$ is an eigenvalue of $T$, there exists $0\neq x_1\in M_1$. Since $R(S)=E$ for all $n\in \n$, one can inductively find $x_{n+1}\in M_{n+1}-M_n$ such that $Sx_{n+1}=x_n$. Then $S^n x_{n+1}=x_1\neq 0$, but  $S^{n+1} x_{n+1}=0$. This proves Condition (a) in the above. Condition (b) follows from the fact that $ST=TS$. Set $\la_n:=\la_0$ for all $n\in\n$. Then Condition (c) holds already.

Assume (ii) is false, then there exist a sequence $\{\la_n\}$ of distinct elements of $e_r(T)$. For $n\in \n$, pick a non-zero eigenvector $e_n$ of $\la_n$ and let $M_n$ be the subspace generated by $\{e_1,\cdots, e_n\}$. Conditions (a) and (b) follow immediately from the definition of $M_n$. For every $n\in \n$ and $x = \alpha_1e_1 + \cdots + \alpha_{n+1} e_{n+1}\in M_{n+1}$, we have $(T-\la_{n+1})x=\alpha_1(\la_1 - \la_{n+1})e_1+\cdots+\alpha_n(\la_n- \la_{n+1})e_n\in M_n$. This shows that Condition (c) holds too.
\end{proof}

In the rest of this section, the adjoint of an operator $T\in B(X,Y)$ between two topological vector space is the map $T\s\in B(Y\s, X\s)$ defined by $T\s \rho:= \rho T$ for all $\rho\in Y\s$, where $X\s$ (resp. $Y\s$) is the {\bf dual space of $X$} (resp. $Y$), that is the vector space of all continuous linear functionals on $X$ (resp. $Y$). We recall that the locally convex topology on $X$ induced by semi-norms of the form $x\mapsto |\rho(x)|$, where $\rho\in X\s$ is called the {\bf weak topology of $X$}. Similarly, the locally convex topology on $X\s$ induced by semi-norms of the form $\rho\mapsto |\rho(x)|$, where $x\in X$, is called the {\bf \ws topology of $X\s$}. When $X$ is a normed space, $X\s:=B(X,\c)$ is equipped with the operator norm, and so it is a normed space as well. We review some of the properties of dual spaces in the following exercise:
\begin{exercise}
\label{exe:dualspaces} Assume $X$ is a normed space.
\begin{itemize}
\item [(i)] Using Theorem 3.3 of \cite{rudinfunctional}, show that, for every $x_0\in X$, there exists $\rho\in X\s$ such that $\rho(x_0)=\|x_0\|$ and $|\rho (x)|\leq \|x\|$ for all $x\in X$.
\item [(ii)] Using Part (i), for every $x\in X$, show that
\[
\|x\|=\sup \{ |\rho (x)|; \rho\in X\s, \|\rho\|\leq 1 \}.
\]
\item [(iii)] Let $X^{\ast \ast}:=(X\s)\s$ be the {\bf double dual of $X$}. Define $\th:X\ra X^{\ast \ast}$ by $\th(x):=x^{\ast \ast}$, where $x^{\ast  \ast}(\rho):=\rho (x)$ for all $x\in X$ and $\rho\in X\s$. Show that $\|x\|=\|x^{\ast \ast}\|$, and therefore $\th$ is an isometry.
\item [(iv)] Let $Y$ be another normed space. For every $T\in B(X,Y)$, show that
\[
\|T\|=\sup \{|\rho(Tx)|; x\in X, \|x\|\leq 1, \rho\in Y\s, \|\rho\|\leq 1 \}.
\]
Conclude that $\|T\|=\|T\s\|$.
\item [(v)] Prove that the weak topology of $X$ is the weakest topology on $X$ for which every linear functional $\rho\in X\s$ is continuous. Similarly, prove that the \ws topology is the weakest topology on $X\s$ for which every element of $\th(X)$ is continuous.
\item [(vi)] Assume $X$ is a Banach space. Show that $\th(X)$ is a closed subspace of $X^{\ast \ast}$. Prove that the members of $\th(X)$ are exactly those linear functionals on $X\s$ that are continuous with respect to the \ws topology of $X\s$. In other words, the dual space of the locally convex topological vector space $X\s$ with \ws topology is exactly $\th(X)$.
\end{itemize}
\end{exercise}

A subset $Y$ of  metric space $(X,d)$ is called {\bf totally bounded} if, for every $\ep>0$, $Y$ lies in a union of finitely many balls of radius $\ep$. The reader can find the proof of the next theorem in Page 394 of \cite{rudinfunctional}.

\begin{theorem}
\label{thm:ascoli} [The Arzel\`{a}-Ascoli theorem] Let $X$ be a compact space. Assume $A$ is a subset of $C(X)$ such that it is
\begin{itemize}
\item [(i)] {\bf pointwise bounded}, namely $\{|f(x)|; f\in A\}<\infty$ for all $x\in X$, and
\item [(ii)] {\bf equicontinuous}, namely, for every $\ep>0$ and $x\in X$, there is a neighborhood $U$ of $x$ such that $x'\in U$ implies $|f(x)-f(x')|<\ep$ for all $f\in A$.
\end{itemize}
Then $A$ is totally bounded in $C(X)$.
\end{theorem}
\begin{corollary}
\label{cor:ascoli}
Let $A$ be as described in the Arzel\`{a}-Ascoli theorem. Then every sequence in $A$ has a convergent subsequence.
\end{corollary}
Since the topology of $C(X)$ is induced by the supremum norm, we can rephrase this corollary by saying that every sequence in $A$ has a uniformly convergent subsequence.
\begin{proof}
Since $C(X)$ is complete, the closure of $A$ is complete and totally bounded. This implies that the closure of $A$ is compact, see Theorem 45.1 in \cite{munkeres}.
\end{proof}
\begin{example}
\label{exa:compactkernel}
Let $E$ be the Banach space $C([0,1])$.
\begin{itemize}
\item [(i)] For every $K\in C([0,1]\times [0,1])$, we define a compact operator $T_K\in K(E)$ as follows: For given $f\in E$, we define
\[
T_Kf(s):=\int_0^1 K(s,t)f(t) dt, \qquad \forall s\in [0,1].
\]
In order to show that $T_Kf\in E$, for every $s,s'\in [0,1]$, we compute
\begin{eqnarray*}
|T_Kf(s)-T_Kf(s')|&=&\left|\int_0^1 (K(s,t)-K(s',t))f(t)dt\right|\\
&\leq & \int_0^1 |K(s,t)-K(s',t)||f(t)|dt\\
&\leq & \sup_{t\in [0,1]} |K(s,t)-K(s',t)|\|f(t)\|_{\sup}.
\end{eqnarray*}
Since $[0,1]\times [0,1]$ is compact, $K$ is uniformly continuous. In particular, for every $\ep>0$, there exists a $\d>0$ such that $|s-s'| <\d$ implies
\[
|T_Kf(s)-T_Kf(s')|\leq \sup_{t\in [0,1]} |K(s,t)-K(s',t')| \|f(t)\|_{\sup}<\ep\|f(t)\|_{\sup}.
\]
This shows that $T_Kf$ is continuous on $[0,1]$. It also shows that $T_K((E)_1)$ is equicontinuous. On the other hand, for every $f\in (E)_1$, we have
\[
|T_Kf(s)|\leq \int_0^1 |K(s,t)f(t)|dt \leq \|K\|_{\sup} \|f\|_{\sup}\leq \|K\|_{\sup}.
\]
This shows that $T_K((E)_1)$ is pointwise bounded. Therefore by Corollary \ref{cor:ascoli}, every sequence in  $T_K((E)_1)$ has a convergent subsequence. Hence $T_K$ is a compact operator. This operator is called an {\bf integral operator associated with $K$} and the continuous function $K$ is called the {\bf kernel of $T_K$}.
\item [(ii)] For every $f\in E$, define
\[
Vf(s):=\int_0^s f(t)dt, \qquad \forall s\in [0,1].
\]
Clearly, $Vf$ is continuous, and so $V$ defines a linear map from $E$ into $E$. Next, for every $s,s'\in [0,1]$, we have
\[
|Vf(s)-Vf(s')|= \left|\int_s^{s'} f(t)dt\right|\leq |s-s'|\|f\|_{\sup}.
\]
This shows that $V((E)_1)$ is equicontinuous and pointwise bounded. Therefore $V$ is a compact operator on $E$. This operator is called the {\bf Volterra integral operator on $E$}.
\end{itemize}
\end{example}
\begin{proposition}
\label{prop:adjointcompact}
An operator $T\in B(E,F)$ is compact if and only if $T\s\in B(F\s,E\s)$ is compact.
\end{proposition}
\begin{proof}
Assume $T\in B(E,F)$ is a compact operator. Let $\{y^\ast_n\}$ be sequence in the unit ball of $F\s$. Since $\|y^\ast_n\|\leq 1$ for all $n\in \n$, this sequence is equicontinuous as a family of functions on $F$. Let $X$ be the closure of $T(E_1)$, where $E_1$ is the closed unit ball of $E$. Then $X$ is compact, and so, for every $x\in X$, the set $\{|y^\ast_n x|; n\in \n\}$ is bounded. By Corollary \ref{cor:ascoli}, the sequence $\{y^\ast_n\}$ has a uniformly convergent subsequence, say $\{y^\ast_{n_i}\}$, on $X$. Now, for every $i,j \in \n$, we compute
\begin{eqnarray*}
\|T\s y^\ast_{n_i} -T\s y^\ast_{n_j}\|&=& \sup_{x\in E_1} |T\s y^\ast_{n_i} -T\s y^\ast_{n_j}(x)|\\
&=&\sup_{x\in E_1} |(y^\ast_{n_i} - y^\ast_{n_j})(Tx)\|\\
&=&\|y^\ast_{n_i}-y^\ast_{n_j}\|_{\sup} \ra 0, \quad \text{as} \, i,j\ra\infty,
\end{eqnarray*}
where $\|-\|_{\sup}$ in the last term is the norm of $C(X)$. This shows that the sequence $\{T\s y^\ast_{n_i}\}$ is Cauchy, and since $F\s$ is complete, it is convergent. Therefore $T\s$ is a compact operator. The converse is proved similarly.
\end{proof}
\begin{exercise}
Complete the proof of the above proposition.
\end{exercise}
Assume $X$ is a topological vector space, $M$ is a subset of $X$ and $N$ is a subset of $X\s$. We define the {\bf annihilator of $M$} as follows:
\[
M^\perp:= \{\rho\in X\s; \rho(x)=0, \forall x\in M\}.
\]
Clearly, $M^\perp$ is a subspace of $X\s$ even when $M$ is not a subspace of $X$. Furthermore, it is straightforward to show that $M^\perp$ is closed in \ws topology. Similarly, we define the {\bf annihilator of $N$} as follows:
\[
^\perp N:= \{x\in X; \rho(x)=0, \forall \rho \in N\}.
\]
One easily checks that $^\perp N$ is a closed subspace of $X$.

\begin{proposition}
\label{prop:annihilators}
Assume $X$ is a Banach space, $M$ is a subspace of $X$ and $N$ is a subspace of $X\s$.
\begin{itemize}
\item [(i)] $^\perp(M^\perp)$ is the norm closure of $M$ in X.
\item [(ii)] $(^\perp N)^\perp$ is the \ws closure of $N$ in $X\s$.
\end{itemize}
\end{proposition}
\begin{proof}
\begin{itemize}
\item [(i)] It is easy to see that $^\perp(M^\perp)$ is norm closed and contains $M$, and so $\overline{M}\sub ^\perp(M^\perp)$. Let $x_0\in X-\overline{M}$. Then using the Hahn-Banach theorem, one can find a linear functional $\rho\in X\s$ such that $\rho(x_0)\neq 0$ and $\rho(x)=0$ for all $x\in M$. Since $\rho\in M^\perp$, we conclude that $x_0\notin \,^\perp(M^\perp)$. Thus $^\perp(M^\perp)\sub \overline{M}$.
\item [(ii)] Due to the fact that $(^\perp N)^\perp$ is a \ws closed subspace containing $N$, it contains the \ws closure of $N$ as well. Let $\rho_0\in X\s - \widetilde{N}$, where $\widetilde{N}$ denotes the \ws closure of $N$. Similar to Part (i), by applying the Hahn-Banach theorem to $X\s$, equipped with \ws topology, we find a linear functional $\alpha$ on $X\s$ such that $\alpha$ is continuous with respect to the \ws topology of $X\s$, $\alpha(\rho_0)\neq 0$ and $\alpha(\rho)=0$ for all $\rho\in N$. By Exercise \ref{exe:dualspaces}(vi), there exists some $x\in X$ such that $\th(x)=\alpha$. Hence $\rho_0(x)\neq 0$ and $\rho(x)=0$ for all $\rho\in N$. Therefore $\rho_0\notin (^\perp N)^\perp$. This shows that $(^\perp N)^\perp \sub \widetilde{N}$.
\end{itemize}
\end{proof}
\begin{remark}
It follows from the Hahn-Banach theorem that the elements of $E\s$ separate points of $E$, namely, for every $x\in E$, there exists $\rho\in E\s$ such that $\rho(x)\neq 0$. Similarly, the elements of $E$ separate points of $E\s$.
\end{remark}
\begin{lemma}
\label{lem:NRannihil}
Let $T\in B(E,F)$. Then the following statements hold:
\begin{itemize}
\item [(i)] $N(T\s)=R(T)^\perp$.
\item [(ii)] $N(T)=\,^\perp R(T\s)$.
\item [(iii)] $N(T\s)$ is \ws closed in $F\s$.
\item [(iv)] $R(T)$ is dense in $F$ if and only if $T\s$ is one-to-one.
\item [(v)] $T$ is one-to-one if and only if $R(T\s)$ is \ws dense in $X\s$.
\end{itemize}
\end{lemma}
\begin{proof}
 \begin{itemize}
\item [(i)] $\rho\in N(T\s)\Leftrightarrow T\s \rho =0 \Leftrightarrow \rho T(x)=0,\, \forall x\in E \Leftrightarrow\rho\in R(T)^\perp$.
\item [(ii)] $x\in N(T)\Leftrightarrow Tx =0 \Leftrightarrow \rho (Tx)=0,\, \forall \rho \in F\s \Leftrightarrow T\s \rho(x)=0, \forall \rho\in F\s$ $\Leftrightarrow x\in \,^\perp R(T\s)$.
\item [(iii)] It follows from Part (i).
\item [(iv)] It follows from Part (i), Proposition \ref{prop:annihilators}(i), and the above remark.
\item [(v)] It follows from Part (ii), Proposition \ref{prop:annihilators}(ii), and the above remark.
\end{itemize}
\end{proof}
\begin{proposition}
\label{prop:rangecloseddual} For every $T\in B(E,F)$, the following statements are equivalent:
\begin{itemize}
\item [(i)] $R(T)$ is closed in $F$.
\item [(ii)] $R(T\s)$ is \ws closed in $E$.
\item [(iii)] $R(T\s)$ is norm closed in $E$.
\end{itemize}
\end{proposition}
\begin{proof}
Assume (i) holds. By Propositions \ref{lem:NRannihil}(ii), we have $N(T)^\perp = (^\perp R(T\s))^\perp$, and by \ref{prop:annihilators}(ii), $(^\perp R(T\s))^\perp$ is the \ws closure of $R(T\s)$. Thus $N(T)^\perp$ is the \ws closure of $R(T\s)$. Therefore to prove (ii), it is enough to show $N(T)^\perp\sub R(T\s)$. For $0\neq \rho\in N(T)^\perp$, define $\rho':R(T)\ra \c$ by $\rho'(Tx):=\rho(x)$. One easily checks that $\rho'$ is well defined because $\rho\in N(T)^\perp$. Since $R(T)$ is closed, and consequently complete, by open mapping theorem, $T:E\ra R(T)$ is open. Therefore for every $\ep>0$, there is $\d>0$ such that $\|Tx\|<\d$ implies that $\|x\|\leq \ep/\|\rho\|$. Now, for every $x\in E$ such that $\|Tx\|<\d$, we have
\[
|\rho'(Tx)|=|\rho(x)|\leq \|\rho \|\|x\|\leq \|\rho\| \frac{\ep}{\|\rho\|}=\ep.
\]
This shows that $\rho'$ is continuous. By the Hahn-Banach theorem, $\rho'$ has an extension $\Lambda:F\ra \c$. Then for every $x\in E$, we have $T\s\Lambda(x)=\Lambda(Tx)=\rho'(Tx)=\rho(x)$. Hence $\rho\in R(T\s)$.

Clearly, (iii) follows from (ii).

Assume (iii) holds. Let $Z$ denote the norm closure of $R(T)$. Define $S\in B(E,Z)$ by $Sx=Tx$ for all $x\in E$. By Proposition \ref{lem:NRannihil}(iv), $S\s\in B(Z\s, E\s)$ is one-to-one.  For every $\rho \in F\s$ and $\rho'\in Z\s$ such that $\rho|_Z=\rho'$, we have
\[
T\s \rho (x)=\rho(Tx)=\rho'(Tx)= \rho'(Sx)= S\s \rho' (x), \quad \forall x\in E,
\]
and therefore $T\s\rho=S\s \rho'$. Now, due to the fact that every bounded linear functional on $Z$ has an extension to whole $F$, this shows $R(T\s)=R(S\s)$. Hence $R(S\s)$ is closed, and consequently complete. Therefore we can apply the open mapping theorem to the bijective operator $S\s:Z\s\ra R(S\s)$ and conclude that its inverse is bounded too. This means that there is a constant $\d>0$ such that $\d\|\rho\| \leq \|S\s \rho\|$ for all $\rho \in Z\s$. Now, it follows from this inequality and the following lemma that $S(X)=Z$. Hence $Z=R(S)=R(T)$, and therefore $R(T)$ is norm closed.
\end{proof}

\begin{lemma}
\label{lem:openfourequivalent}
Let $T\in B(E,F)$ and let $U$ and $V$ be the open unit balls in $E$ and $F$, respectively. Then the following four statements are equivalent:
\begin{itemize}
\item [(i)] There is $\d>0$ such that $\d\|\rho\| \leq \|T\s \rho\|$ for all $\rho \in F\s$. In other words $T\s$ is bounded below.
\item [(ii)] There is $\d>0$ such that $\d V\sub \overline{T(U)}$.
\item [(iii)] There is $\d>0$ such that $\d V\sub T(U)$.
\item [(iv)] $T(E)=F$.
\end{itemize}
Moreover, the same $\d$ works for all first three conditions.
\end{lemma}
\begin{proof}
Assume (i) holds. Since $\overline{T(U)}$ is convex, closed and balanced, by Theorem \ref{thm:ccblanced}, for every $y_0\in F- \overline{T(U)}$, one can find $\rho\in F\s$ such that $|\rho(y)|\leq 1$ for all $y\in \overline{T(U)}$ and $\rho(y_0)>1$. Hence for all $x\in U$, we have $|T\s \rho(x)|= |\rho(Tx)|\leq 1$, and so $\|T\s \rho\|\leq 1$. Using (i), we obtain
\[
\d<|\d\rho(y_0)|\leq \d \|y_0\| \|T\s\rho\|\leq \|y_0\|.
\]
Therefore $y\in \overline{T(U)}$ for all $y\in \d V$. Hence (i) implies (ii).

Next, assume (ii) holds, then $\overline{\d V}\sub \overline{T(U)}$ for some $\d>0$. For every $y_1\in V$, let $\{\ep_n\}$ be a sequence of strictly positive real numbers such that $\sum_{n=1}^\infty \ep_n < 1-\|y_1\|$. This implies that $\ep_n\ra 0$ as $n\ra \infty$. Find $x_1\in E$ such that $\|x_1/\d\|\leq \|y_1\|$ and $\|y_1-T(x_1/\d)\|\leq \ep_1$. For $n\geq 2$, set $y_n:=y_{n-1} -T(x_{n-1}/\d)$ and find $x_n\in E$ such that
\[
\|x_n/\d\|\leq \|y_n\| \quad \text{and}\quad \|y_n-T(x_n/\d)\|<\ep_n.
\]
By this construction, we get two sequences $\{x_n\}\sub E$ and $\{y_n\}\sub V$ such that
\[
\|x_{n+1}/\d \| \leq \|y_{n+1}\|= \| y_n - T(x_n/\d)\|<\ep_n, \quad \forall n\in \n.
\]
Therefore $y_n\ra 0$ as $n\ra \infty$ and we have
\[
\sum_{n=1}^\infty \|x_n/\d\|\leq \|x_1/\d\|+ \sum_{n=1}^\infty \ep_n \leq  \|y_1\|+ \sum_{n=1}^\infty \ep_n<1.
\]
Since $E$ is complete, $\sum_{n=1}^\infty x_n/\d$ is convergent to some $x\in E$ and $\|x\|<1$, see Problem \ref{e:5-2}. Now, we compute
\[
Tx=\sum_{n=1}^\infty T(x_n/\d) = \sum_{n=1}^\infty (y_n-y_{n+1})= y_1-\lim_{n\ra \infty}y_{n+1}=y_1.
\]
This shows that $y_1\in T(U)$ and proves (iii).

Assume (iii) holds. For every $\rho\in F\s$, we have
\begin{eqnarray*}
\|T\s \rho\|&=& \sup \{|T\s\rho (x)|; x\in U\}\\
&=& \sup \{|\rho (Tx)|; x\in U\}\\
&\geq& \sup \{|\rho (y)|; y\in \d V\}=\d \|\rho\|.
\end{eqnarray*}
This proves (i).

Finally, we note that (iii) clearly implies (iv), and by open mapping theorem, (iv) implies (iii).
\end{proof}

\begin{exercise}
\label{exe:ontoadjoint}
For given $T\in B(E,F)$, show that $T$ is onto if and only if $T\s$ is one-to-one and $R(T\s)$ is norm closed.
\end{exercise}

\begin{lemma}
\label{lem:anhilatorsubspace}
Let $E$ be a locally convex topological space and let $M_0$ be a closed subspace of $X$.
Then we have
\[
\dim X/M_0\leq \dim M_0^\perp.
\]
\end{lemma}
\begin{proof}
For every positive integer $k\leq \dim X/M_0$, there are vectors $x_1,\cdots,x_k$ in $X$ such that if we set $M_i:=\lan x_1,\cdots ,x_i\ran\oplus M_0$ for all $1\leq i \leq k$, then every $M_i$ is closed by Theorem 1.42 of \cite{rudinfunctional} and $M_0\varsubsetneq M_1\varsubsetneq M_1\varsubsetneq \cdots\varsubsetneq M_k$. Applying the Hahn-Banach theorem, there are $k$ linear functionals $\rho_1,\cdots, \rho_k$ on $X$ such that $\rho_i x_i=1$ and $\rho_i\in M_{i-1}^\perp$. Since $x_1, \cdots, x_k$ are linearly independent, so are $\rho_1,\cdots, \rho_k$. This implies the desired inequality.
\end{proof}
\begin{theorem}
\label{thm:compactoptspec}
Let $T\in K(E)$.
\begin{itemize}
\item [(i)] For every non-zero $\la\in \c$, the following four numbers are equal and finite:
\begin{eqnarray*}
\alpha&:=& \dim N(T-\la),\\
\beta&:=& \dim E/R(T-\la),\\
\alpha\s&:=& \dim N(T\s-\la),\\
\beta\s&:=& \dim E\s /R(T\s-\la).
\end{eqnarray*}
\item [(ii)] For every non-zero $\la\in \si(T)$, $\la$ is an eigenvalue of both $T$ and $T\s$.
\item [(iii)] The spectrum of $T$ is at most countable and its only possible limit point is $0$.
\end{itemize}
\end{theorem}
\begin{proof}
\begin{itemize}
\item [(i)] Set $S:=T-\la$. By Proposition \ref{prop:compactclosedrange}, $R(S)$ is norm closed and, by Lemma \ref{lem:NRannihil}(i), $R(S)^\perp=N(S\s)$. Therefore applying Lemma \ref{lem:anhilatorsubspace} for $R(S)^\perp\sub E$, we obtain
\begin{equation}
\label{eqn:a1}
\beta\leq \alpha\s.
\end{equation}
Since $R(S)$ is norm closed, $R(S\s)$ is \ws closed by Proposition \ref{prop:rangecloseddual}. Considering $R(S\s)$ as a closed subspace of $E\s$ in \ws topology, it follows from Exercise \ref{exe:dualspaces}(vi) that $R(S\s)^\perp =\,^\perp R(S\s)$. On the other hand, by Lemma \ref{lem:NRannihil}(ii), we have $\,^\perp R(S\s) = N(S)$. Therefore if we apply Lemma \ref{lem:anhilatorsubspace} for the closed subspace $R(S\s)$, we obtain
\begin{equation}
\label{eqn:a2}
\beta\s \leq \alpha.
\end{equation}
Now, we want to prove
\begin{equation}
\label{eqn:a3}
\alpha \leq \beta.
\end{equation}
By Proposition \ref{prop:compactopt1}(ii), $\alpha$ is finite. If $\alpha>\beta$, then $\beta$ is finite too and Lemma \ref{lem:complementedsub} implies that there are closed subspaces $M,N\sub E$ such that
\begin{equation}
\label{eqn:a5}
N(S)\oplus M=E=R(S) \oplus N.
\end{equation}
Using the first equality in (\ref{eqn:a5}), for every $x\in E$, there are unique $x_1\in N(S)$ and $x_2\in M$ such that $x=x_1 + x_2$. Define $\pi:E\ra N(S)$ by $\pi x :=x_1$. By Exercise \ref{exe:projectionsbanach}, $\pi$ is bounded. Also, since $\dim N=\beta<\alpha<\infty$, there is a bounded and onto linear map $K:N(S)\ra N$ such that $K x_0=0$ for some $0\neq x_0\in N(S)$. Clearly, $K$ is a compact operator. Therefore $\Phi:E\ra E$ defined by $\Phi x:=Tx +K\pi x$ is a compact operator, see Exercise \ref{exe:sumofcompactoperrators}. One easily sees that $\Phi-\la=S+K\pi$. For every $x\in M$, $\pi x=0$ and $(\Phi -\la)x=Sx$. Hence $(\Phi -\la)(E)=R(S)$. For every $x\in N(S)$, $\pi x=x$ and $(\Phi -\la)x=Kx$. Hence $(\Phi -\la)(N(S))=K(N(S))= N$. Therefore $E=R(S)\oplus N\sub R(\Phi-\la)$. However, $(\Phi -\la)x_0=Kx_0=0$ shows that $\la$ is an eigenvalue of $\Phi$ and since $\Phi$ is a compact operator, Proposition \ref{prop:eigenvaluefinite}(i) implies that $R(\Phi-\la)\subsetneq E$. This contradiction proves \ref{eqn:a3}.

The inequality
\begin{equation}
\label{eqn:a4}
\alpha\s \leq \beta\s.
\end{equation} follows from Inequality (\ref{eqn:a3}) and the fact that $T\s$ is compact too, see Proposition \ref{prop:adjointcompact}. Finally, Inequalities (\ref{eqn:a1}), (\ref{eqn:a2}), (\ref{eqn:a3}), and (\ref{eqn:a4}) show that $\alpha=\beta=\alpha\s=\beta\s<\infty$ and complete the proof of (i).
\item [(ii)] Assume a complex number $\la\neq 0$ is not an eigenvalue of $T$. Then by (i), we have
\[
\dim N(T-\la)=\dim E/R(T-\la)=0.
\]
Therefore $T-\la$ is one-to-one and onto. By Proposition \ref{prop:invoperator}, $T-\la$ is invertible and so $\la\notin \si(T)$.
\item [(iii)] By (ii) and Proposition \ref{prop:eigenvaluefinite}(ii), the set of all $\la\in \si(T)$ such that $\la>r$ is finite for all $r>0$. Therefore the only possible limit point of $\si(T)$ is $0$ and $\si(T)$ is at most countable.
\end{itemize}
\end{proof}

\section{The holomorphic functional calculus}
\label{sec:holomorphic}

It is often useful to construct new elements in Banach algebras and operator algebras with certain properties similar to what one used to work with in elementary calculus. In operator algebras, various functional calculi provide us with practical methods to apply certain functions to operators to construct new operators. {\bf Polynomial functional calculus} is the most simple functional calculus which works in two levels; Banach algebras and operator algebras.

Let $P(z,\bar{z})$ be a complex polynomial of two variables $z$ and $\bar{z}$ and let $T\in B(H)$ be a bounded operator on a Hilbert space $H$. Then by substituting $z$ and $\bar{z}$ with $T$ and $T\s$, respectively, we obtain a new operator which formally we denote it by $P(T,T\s)$ and it behaves the same way as the original polynomial $P$. For example, the values of the polynomial $P(z\bar{z}) = z\bar{z}=|z|^2$ are always positive (non-negative) real numbers, and similarly, the operator constructed by $P$ over an operator $T\in B(H)$, i.e. $P(T,T\s)=TT\s$ is also a positive operator , as it will be explained later.

Let $\c[z,\bar{z}]$ denote the algebra of complex polynomials of two variables $z$ and $\bar{z}$. Sending $z$ to $\bar{z}$ defines an involution on this algebra. Then for given $T\in B(H)$, the polynomial functional calculus is a {\bf $*$-homomorphism} (a homomorphism between two involutive algebras which preserves the involution) from $\c[z,\bar{z}]$ into $B(H)$. Clearly this definition works for any involutive algebra including arbitrary \cs-algebras. Since the target of this functional calculus is an operator algebra with an involution, we had to consider two variables $z$ and $\bar{z}$ in the domain of the mapping to be able to define an involution. However to define polynomial functional calculus in Banach algebras, one variable is sufficient. Let $\c[z]$ be the algebra of complex polynomials. Given an element $a$ in a Banach algebra $A$, we define $\c[z]\ra A$ by sending a polynomial $P(z)$ to $P(a)$. This is the polynomial functional calculus over an element $a\in A$ and it is a homomorphism of algebras.

As for generalizing this functional calculus, it is tempting to consider convergent power series, like $\sum_{n=0}^\infty c_n z^n$, and define the functional calculus for them. It works when the series $\sum_{n=0}^\infty c_n \|a\|^n$ is convergent in the Banach algebra, for example, the exponential of $a\in A$ is defined by $e^a=\sum_{n=0}^\infty  a^n/n!$. The exponential function is an entire function, that is why it can be applied to every element of a Banach algebra. More general holomorphic functions can only be applied to those elements whose spectrum is contained in the domain of the function.

Let $A$ be a unital Banach algebra and $a\in A$. Let $f$ be a holomorphic function on an open set $U_f$ containing $\si(a)$. Assume $\ga=\sum_{j=1}^k \ga_j$ is a finite collection of smooth simple closed curves in $U_f$ whose interiors contains $\si(a)$. The existence of such a finite collection of curves follows from the fact that $\si(a)$ is compact. We also assume these curves are oriented positively with respect to the spectrum, namely if a piece of the spectrum, say $X$, is surrounded by $\ga_1$, then a person walking on $\ga_1$ always has $X$ on his left. In the proof of Proposition \ref{prop:spradius2}, we proved that the map $\la\ra (\la-a)\inv$ is a holomorphic function on $Res(a)$. Hence for every $\ff\in A\s$, the function $\ff_f: \la\mapsto f(\la)\ff((\la-a)\inv)$ is a holomorphic function on $Res(a)$ which contains the image of $\ga$.  Now we can define $F:A\s\ra\c$ by
\[
F(\ff):=\frac{1}{2\pi i} \sum_{j=1}^k \int_{\ga_j} f(\la)\ff((\la-a)\inv) d\la
\]
Clearly, we have $|F(\ff)|\leq \frac{l \|\ff\|}{2\pi}  \sup\{ |f(\la)| \|(\la-a)\inv\| \}$, where $l$ is the sum of the lengths of all curves in $\ga$. This shows that $F$ is a bounded functional over $A\s$.

Let $P$ denote a partition of the curve $\ga_1$ with points $\la_0, \la_1,\cdots, \la_n=\la_0\in Image(\ga_1)$. Since the function $\la\mapsto f(\la)(\la-a)\inv$ is continuous, the limit of the Riemann sum
\[
\sum_{i=1}^{n} f(\la) (\la_i- a)\inv (\la_i-\la_{i-1})
\]
exists in $A$ when $P$ varies in the set of partitions of $\ga_1$ such that $\max_i |\la_i-\la_{i-1} |\ra 0$. Let us denote this limit by $b_1$. The same argument holds for the rest of the curves $\ga_2, \cdots, \ga_k$. We denote the similar limits by $b_2, \cdots, b_k$. By continuity of $\ff$, it is clear that $F(\ff)=\ff( \frac{1}{2\pi i} \sum_{j=1}^k b_j)$, namely $F\in A$ and is equal to $b:=\frac{1}{2\pi i} \sum_{j=1}^k b_j$. By Cauchy's theorem, see Theorems 7.47 and 7.49 of \cite{ponnusamy-silverman}, $b$ does not depend on the curve $\ga=\sum_{j=1}^k \ga_j$. It also follows from Cauchy's theorem that if the interior of any of the curves $\ga_1, \cdots, \ga_k$ does not intersect $\si(a)$ then the corresponding integral is zero. We denote $b$ by $f(a)$ and symbolically write
\[
f(a)=\frac{1}{2\pi i} \int_\ga \frac{f(\la)}{\la-a} d\la.
\]
Here, one final remark is necessary. When $U_f$ consists of several connected components. For every component one can choose a single curve enclosing the part of $\si(a)$ that lies in that component. Therefore to keep our arguments easy to follow, we usually assume there is only one curve around the spectrum.

Let $f$ and $g$ be two holomorphic functions with domains $U_f$ and $U_g$. Then the domains of $fg$, $f+g$ and any scalar multiplication of $f$ and $g$  contain $U_f \cap U_g$. Therefore the set of all holomorphic functions whose domains contain $\si_A(a)$ is an algebra called the {\bf algebra of holomorphic functions over} $\si_A(a)$ and is denoted by $H_A(a)$, or simply $H(a)$.
\begin{definition}
The mapping $H(a)\ra A$, $f\mapsto f(a)$ is called the {\bf holomorphic functional calculus over $a$}.
\end{definition}

\begin{theorem}
\label{thm:holfunc}
With the above notation, the holomorphic functional calculus is an algebra homomorphism which maps the constant function $1$ to $1_A\in A$ and maps the identity function  $z\mapsto z$ to $a\in A$.
\end{theorem}
\begin{proof}
It is straightforward to check that the holomorphic functional calculus is linear. Let $f, g, U_f$ and $U_g$ be as described in the above. Let $\ga_1$ and $\ga_2$ be two smooth simple closed curves in $U_f\cap U_g$ enclosing $\si(a)$. We can assume that $\ga_2$ lies in the interior of $\ga_1$. Then we have
\begin{eqnarray*}
f(a)g(a) &=& \left( \frac{1}{2\pi i} \int_{\ga_1} f(\la) (\la -a)\inv d\la\right)\left( \frac{1}{2\pi i} \int_{\ga_2} g(\mu) (\mu -a)\inv d\mu \right)\\
&=& \frac{-1}{4\pi^2} \int_{\ga_2} \int_{\ga_1} f(\la)g(\mu)  (\la -a)\inv (\mu -a)\inv  d\la d\mu \\
&=& \frac{-1}{4\pi^2} \int_{\ga_2} \int_{\ga_1} f(\la)g(\mu) \frac{1}{\la-\mu} \left(  (\mu -a)\inv -(\la -a)\inv \right) d\la d\mu \\
&=& \frac{-1}{4\pi^2} \int_{\ga_2} \int_{\ga_1} \frac{f(\la)g(\mu)}{\la-\mu} (\mu -a)\inv d\la d\mu \\
&+& \frac{1}{4\pi^2} \int_{\ga_2} \int_{\ga_1} \frac{f(\la)g(\mu)}{\la-\mu} (\la -a)\inv d\la d\mu.
\end{eqnarray*}
After changing the order of integration, the second term equals
\[
\frac{1}{4\pi^2} \int_{\ga_1} \left( \int_{\ga_2} \frac{g(\mu)}{\la-\mu}d\mu \right) f(\la)(\la -a)\inv d\la.
\]
Since $\frac{g(\mu)}{\la-\mu}$ is holomorphic in the interior of $\ga_1$, by Cauchy's theorem, the integral $\int_{\ga_2} \frac{g(\mu)}{\la-\mu}d\mu$ is zero for all $\la\in \ga_1$. Therefore the second term is  zero, and so we have
\[
f(a)g(a)= \frac{1}{2\pi i} \int_{\ga_2} \left( \frac{1}{2\pi i} \int_{\ga_1} \frac{f(\la)}{\la-\mu} d\la \right)g(\mu) (\mu - a )\inv d\mu
\]
By Cauchy's integral formula, see Theorem 8.1 of \cite{ponnusamy-silverman}, $f(\mu)=\frac{1}{2\pi i} \int_{\ga_1} \frac{f(\la)}{\la-\mu} d\la$. Hence we obtain
\[
f(a)g(a)= \frac{1}{2\pi i} \int_{\ga_2} f(\mu)g(\mu) (\mu - a )\inv d\mu=(fg)(a).
\]
This shows two facts: First, the holomorphic functional calculus preserves the multiplication. Secondly, its image is always a commutative subalgebra of $A$, because $fg=gf$ as two holomorphic function.

Let $f$ be the constant function $1$ over $\c$. Then for every $a\in A$, we have
\[
f(a)= \frac{1}{2\pi i} \int_\ga (\la-a)\inv d\la,
\]
where one can choose $\ga$ to be a circle centered at the origin of a reduce greater than $\|a\|$. Then for every $\la\in \ga$, $\sum_{n=0}^\infty \frac{a^n}{\la^{n+1}}$ converges uniformly to $(\la-a)\inv$. Hence we have
\begin{eqnarray*}
f(a)&=& \frac{1}{2\pi i} \int_\ga \sum_{n=0}^\infty \frac{a^n}{\la^{n+1}} d\la\\
&=& \frac{1}{2\pi i} \sum_{n=0}^\infty \int_\ga  \frac{a^n}{\la^{n+1}} d\la\\
&=& 1_A,
\end{eqnarray*}
where the last step follows from the following lemma. A similar computation shows that $f(a) = a$ whenever $f$ is the identity function.
\end{proof}
\begin{exercise}
Prove the last sentence of the above proof.
\end{exercise}
\begin{lemma}
Let $C$ be a circle centered at the origin of radius $r$. Then
\[
\int_C \frac{1}{z^{n+1}} dz =\left\{ \begin{array}{ll} 2\pi i & n=0\\ 0 & n\neq 0\end{array} \right.
\]
\end{lemma}
\begin{proof}
It is a straightforward computation if we apply the change of variable $z=re^{it}$.
\end{proof}
\begin{theorem}
\label{thm:specmap}
[The holomorphic spectral mapping theorem] Let $A$ be a unital Banach algebra and $a\in A$. If $f$ is a holomorphic function over a neighborhood of $\si(a)$, then $\si(f(a))=f(\si(a))$. Moreover, if $g$ is a holomorphic function over a neighborhood of $\si(f(a))$, then $gof(a)=g(f(a))$.
\end{theorem}
\begin{proof}
Set $b:=f(a)$ and let $\mu \notin f(\si(a))$. Since the function $k(z):= \frac{1}{z-\mu}$ is holomorphic over $\c-\{\mu \}$, the function $h(\la):= \frac{1}{f(\la)-\mu}$ is holomorphic over a neighborhood of $\si(a)$. Then we have $(b-\mu) h(a) = h(a)(b-\mu) = 1$. Thus $b-\mu$ is invertible, and so $\mu\notin \si(f(a))$. This shows that $\si(f(a))\subseteq f(\si(a))$.

Let $\mu=f(\la_0)$ for some $\la_0 \in \si(a)$. Since $f$ is holomorphic over a neighborhood, say $U_f$, of $\si(a)$, there is a holomorphic function $l$ over $U_f$ such that $f(\la) - \mu = (\la-\la_0) l(\la)$ for all $\la\in U_f$. Thus
\[
f(a) - \mu = (a-\la_0)l(a)=l(a)(a-\la_0).
\]
Since $a-\la_0$ is not invertible, $f(a)-\mu$ cannot be invertible either. This show $f(\si(a))\subseteq \si(f(a))$.

Now, let $\ga_1$ and $\ga_2$ be two smooth simple closed curves in $\c$ such that $\ga_2$ lies inside of the domain of $g$ enclosing $\si(f(a))$ and $\ga_1$ lies in the domain of $f$ and encloses $f\inv (\ga_2)$. Then we have
\begin{eqnarray*}
g o f(a)&=& \frac{1}{2\pi i} \int_{\ga_1} g o f (\la) (\la-a)\inv d\la\\
&=& \frac{-1}{4\pi^2}  \int_{\ga_1}  \left( \int_{\ga_2} g(\mu) (\mu - f (\la))\inv d\mu \right) (\la-a)\inv d\la\\
&=& \frac{-1}{4\pi^2}  \int_{\ga_2} g(\mu) \left( \int_{\ga_1}  (\mu - f (\la))\inv (\la-a)\inv d\la \right)  d\mu\\
&=& \frac{-1}{4\pi^2}  \int_{\ga_2} g(\mu)  (\mu - f (a))\inv   d\mu\\
&=& g(f(a)).
\end{eqnarray*}
\end{proof}

\begin{exercise}
\label{ex:4} Let $A$ be a unital Banach algebra.
\begin{itemize}
\item[(i)] For given $a\in A$, show that the definition of exponential map of $a$ by holomorphic functional calculus is equivalent to its definition by the series $\sum_{n=0}^\infty \frac{a^n}{n!}$.
\item[(ii)] If $a$ and $b$  are two elements of $A$ such that $ab=ba$, then show that $e^{a+b}=e^a e^b$, and conclude that
$e^{-a}=(e^a)\inv$.
\end{itemize}
\end{exercise}

By the above exercise, the image of the exponential map lies in $A^\times$. Moreover, for every $a\in A$, it defines a one parameter subgroup in $A^\times$ by $\rho_a:\r\ra A^\times$, $t\mapsto e^{ta}$.

\begin{proposition}
Let $A$ be a unital Banach algebra and $U\subseteq \c$ be an open set. Then the set
\[
E_U:=\{a\in A; \si_A(a)\subseteq U \}
\]
is an open subset of $A$.
\end{proposition}
\begin{proof}
Set $U^c:=\c -U$. For $a\in E_U$, the function $U^c\ra A$, $\la\mapsto (a-\la)\inv$, is continuous and $\lim_{\la\ra \infty}\|(a-\la)\inv\|=0$. Hence $\sup_{\la\in U^c} \|(a-\la)\inv\|< \infty$, and so  $\delta:= \inf_{\la\in U^c} \frac{1}{\|(a-\la)\inv\|}>0$. If $\|a-b\|\leq \delta$, then $\|(a-\la)-(b-\la)\|< \frac{1}{\|(a-\la)\inv\|}$ for all $\la\in U^c$. Thus by Proposition \ref{prop:invelements}, $(b-\la)$ is invertible for all $\la\in U^c$. In other words, $\si_A(b)\subseteq E$, or equivalently, $b\in E_U$.
\end{proof}

\section{Problems}

\begin{e}
\label{e:2-1}
Find an example for a linear map $T$ between two Banach spaces such that the rank of $T$ is finite, but $T$ is not bounded.
\end{e}

\begin{e}
\label{e:2-2}
Let $A$ be a unital involutive Banach algebra. Prove that $\si(a\s)=\overline{\si(a)}$ for all $a\in A$.
\end{e}
\begin{e}
\label{e:2-3}
Let $B$ be a unital Banach subalgebra of a unital Banach algebra $A$. For every $b\in B$, show that $r_A(b)=r_B(b)$.
\end{e}
\begin{e}
\label{e:2-4}
Assume $A$ is an involutive Banach algebra. Show that every left unit in $A$ is a unit element of $A$.
\end{e}
\begin{e}
\label{e:2-5}
Find an example for a unital Banach algebra $A$ with two elements $a,b\in A$ such that $ab=1$ but $ba\neq 1$. (Hint: Consider $A=B(\ell^2(\n))$, the bounded operators on $\ell^2(\n)$.)
\end{e}
\begin{e}
\label{e:2-6}
Find a sequence $\{e_n\}$ of non-negative integrable functions on $\r$ such that $supp(e_n)\sub [-1/n, 1/n]$ and $\int_\r e_n(x) dx=1$ for all $n\in \n$. Show that $\{ e_n\}$ is an approximate unit for the Banach algebra $L^1(\r)$.
\end{e}
\begin{e}
\label{e:2-7}
Let $\{c_n\}$ be a sequence of complex numbers. Define an operator $T:B(\ell^2(\n))\ra B(\ell^2(\n))$ by $T(f)(n):=c_nf(n)$ for all $f\in B(\ell^2(\n))$ and $n\in \n$. Prove the following statements:
\begin{itemize}
\item [(i)] $T$ is a well defined bounded operator if and only if $\{c_n\}$ is bounded.
\item [(ii)] $T$ is an invertible operator if and only if there is a $\ep>0$ such that $|c_n|>\ep$ for all $n\in \n$. When $T$ is invertible, describe its inverse.
\item [(iii)] $T$ is a compact operator if and only if $c_n\ra 0$.
\item [(iv)] $\si(T)=\{c_n; n\in \n\}$.
\item [(v)] For $n\in \n$, let $T_n$ be the operator defined by the sequence $\{c_{n,m}\}_{m\in \n}$, where
\[
c_{n,m}:=\left\{ \begin{array}{ll} 1& m\leq n\\0& m>n\end{array} \right.
\]
Then the sequence $\{T_n\}$ is an approximate unit for $K(\ell^2(\n))$.
\end{itemize}
\end{e}
\begin{e}
\label{e:2-8}
Let $A$ be a unital Banach algebra such that $\|1\|=1$.
\begin{itemize}
\item [(i)] For given $a\in A$ such that $\|a\|\leq 1$, show that there is a continuous path $\ga:[0,1]\ra A^\times$ such that $\ga(0)=1$ and $\ga(1)=(1-a)\inv$.
\item [(ii)] Show that, for every $a\in A^\times$, there is $\ep>0$ such that for every $b\in A$ satisfying $\|a-b\|<\ep$ there is a continuous path in $A^\times$ connecting $a$ to $b$.
\item [(iii)] Define
\[
G_0:=\left\{ \prod_{i=1}^n \alpha_i; n\in \n, \forall i, \alpha_i=(1-x)\inv \,\text{or} \, \alpha_i=1-x, \, \text{for some } \, \|x\|<1\right\}.
\]
Prove that $G_0$ is an open subgroup of $A^\times$.
\item [(iv)] Prove that $G_0$ is the connected component of $A^\times$ containing $1$.
\item [(v)] Prove that $G_0$ is normal in $A^\times$, and consequently, $A^{\times}/G_0$ is a discrete group.
\end{itemize}
\end{e}
\begin{e}
\label{e:2-9}
Assume $E$ is a Banach space. Prove that $F(E)$ is not closed in $B(E)$, unless $E$ is finite dimensional.
\end{e}
\begin{e}
\label{e:2-10}
Let $E$ and $F$ be two Banach spaces.
\begin{itemize}
\item [(i)] Show that $T\in F(E,F)$ if and only if $T\s\in F(F\s, E\s)$.
\item [(ii)] Show that $T$ is an isometric isomorphism if and only if $T\s$ is an isometric isomorphism.
\item [(iii)] Let $\si$ and $\tau$ be the \ws topologies on $E\s$ and $F\s$, respectively. Prove that $S:(E\s, \si)\ra (F\s, \tau)$ is a continuous linear map if and only if there is $T\in B(E,F)$ such that $S=T\s$.
\item [(iv)] For $T\in B(E,F)$, assume $R(T)$ is closed. Show that
\begin{itemize}
\item [(a)] $\dim N(T)=\dim E\s/R(T\s)$,
\item [(b)] $\dim N(T\s)= \dim F/R(T)$.
\end{itemize}
   This generalizes the equalities $\alpha = \beta\s$ and $\alpha\s=\beta$ in Theorem \ref{thm:compactoptspec}.
\end{itemize}
\end{e}
\begin{e}
\label{e:2-11}
Set $X:=L^1([0,1])$. Then $X\s=L^\infty([0,1])$, see Theorem 6.15 of \cite{folland-ra}. Describe the \ws topology of $X\s$ using integrations on $[0,1]$. For given $T\in B(X, E)$, prove that if $R(T\s)$ contains all continuous functions on $[0,1]$, then $T$ is one-to-one.
\end{e}
\begin{e}
\label{e:2-12}
Let $E$ be a Banach space, $S\in B(E)$ and $T\in K(E)$. Prove that $S(1-T)=1$ if and only if $(1-T)S=1$. In this case, show that $1-(1-T)\inv$ is a compact operator.
\end{e}
\begin{e}
\label{e:2-13}
Let $E$ and $F$ be Banach spaces. Prove that a subset $A\sub B(E,F)$ is equicontinuous if and only if there exists $M<\infty$ such that $\|T\|\leq M$ for all $T\in A$.
\end{e}
\begin{e}
\label{e:2-14}
Let $E$ and $F$ be Banach spaces. Use the Arzel\`{a}-Ascoli theorem, \ref{thm:ascoli}, and Corollary \ref{cor:ascoli}  to show that $K(E,F)$ is closed in $B(E,F)$.
\end{e}
\begin{e}
\label{e:2-15}
Let $E$ be a Banach space. Show that if $T\in K(E)$ is an idempotent, i.e. $T^2=T$, then $T\in F(E)$.
\end{e}
\begin{e}
\label{e:2-16}
Let $E$ and $F$ be two Banach spaces and $T\in B(E,F)$. Show that if $R(T\s)=N(T)^\perp$, then $R(T)$ is closed.
\end{e}
\begin{e}
\label{e:2-17}
Let $E$ and $F$ be two Banach spaces. Prove that the subset of all onto operators in $B(E,F)$ is open.
\end{e}
\begin{e}
\label{e:2-18}
Let $E$ be a Banach space and let $T\in B(E)$. Prove that $\la\in \si(T)$ if and only if there exists a sequence $\{x_n\}$ in $E$ such that $\|x_n\|=1$ for all $n\in \n$ and $\lim_{n\ra \infty} \|Tx_n -\la x_n\|=0$.
\end{e}
\begin{e} Assume $E$ is a Banach space, $T\in K(E)$, and $\la\neq 0$. Set $S:=T-\la$.
\label{e:2-19}
\begin{itemize}
\item [(i)] Prove that there is some $n\in \n$ such that $N(S^n)=N(S^{n+1})$.
\item [(ii)] For $n\in \n$ satisfying (i), prove that $N(S^n)=N(S^{n+k})$ for all $k\in \n$.
\item [(iii)] Let $n$ be the smallest natural number satisfying (i). Prove that $\dim N(S^n)$ is finite and $E=N(S^n) \oplus R(S^n)$. Moreover, show that the restriction of $S$ to $R(S^n)$ is a bijective mapping from $R(S^n)$ onto itself.
\end{itemize}
\end{e}
\begin{e}
\label{e:2-20}
Let $X$ be a compact subset of $\c$ and let $a\in C(X)$ be the identity map, i.e. $a(z)=z$ for all $z\in \c$. Show that the holomorphic functional calculus over $a$, i.e. the map $H(a)\ra C(X)$, is actually the inclusion of the algebra of holomorphic functions on a neighborhood of $X$ into the algebra of continuous functions on $X$.
\end{e}


\chapter{The Gelfand duality}
\label{ch:gelfandduality}

We begin our journey in abstract \cs-algebras with a thorough study of commutative \cs-algebras in this chapter. Our main goal here is to show that every commutative \cs-algebra $A$ can be realized as the \cs-algebra of continuous functions at infinity over some locally compact and Hausdorff topological space, denoted by $\oom(A)$, which is intrinsically associated to $A$. This topological space is nothing but the space of all characters of $A$, or in other words, the space of all non-zero continuous homomorphisms from $A$ into $\c$. Every such a homomorphism is also a functional on $A$, and therefore $\oom(A)$ is a subset of $A\s$, the dual space of $A$. This realization gives us a way to topologize $\oom(A)$ by inducing the \ws topology of $A\s$. Also, the natural inclusion of $A$ into $A^{\ast \ast}$, the double dual space of $A$, suggests considering every element of $A$ as a continuous function over $\oom(A)$. This is the core of the famous Gelfand transform which, generally, maps a given commutative Banach algebra $A$ into the \cs-algebra $C_0(\oom(A))$. As one notices the original setting of the Gelfand transform goes beyond the theory of \cs-algebras and includes Banach algebras. But the theory finds its edge when it is restricted to \cs-algebras, because in this case, the Gelfand transform is always an isometric isomorphism. Besides many applications of the Gelfand transform in the theory of commutative Banach algebras and specifically commutative \cs-algebras, it also gives rise to a very important tool in the abstract theory of \cs-algebras, namely the continuous functional calculus.

In Section \ref{sec:confunctionalcal}, we study the continuous functional calculus and some of its corollaries. For instance, we show that every injective \ss-homomorphism between two \cs-algebra is an isometry. This shows that the \cs-norm of a \cs-algebra is unique. In other words, the analytical structure of a \cs-algebra is closely related to its algebraic structure. It is also proved that the continuous functional calculus is consistent with the holomorphic functional calculus. More applications of the continuous functional calculus are given in Chapter \ref{ch:basics}.

We return to the Gelfand theory in Section \ref{sec:Gelfandduality} and prove that the correspondence between commutative \cs-algebras and locally compact and Hausdorff topological spaces is actually an equivalence of categories provided that we choose the sets of morphisms carefully.

\section{The Gelfand transform}
\label{sec:Gelfandtrans}
In this section, we assume $A$ is a Banach algebra, and by an ideal, we mean a two sided ideal.

\begin{definition} An ideal $\mm$ of $A$ is called {\bf proper} if $\mm \neq A$. A proper ideal $\mm$ of $A$ is called a {\bf maximal ideal of $A$} if it is not contained in any other proper ideal of $A$.
\end{definition}

\begin{exercise} Let $A$ be unital.
\label{ex:ex2-1}
\begin{itemize}
\item[(i)] Assume $\mm$ is a proper ideal of $A$. Show that it contains no invertible element of $A$.
\item[(ii)] Show that if $\mm$ is an ideal (resp. a proper ideal) of $A$, then so is its closure.
\item[(iii)] Show that every maximal ideal of $A$ is closed.
\end{itemize}
\end{exercise}
\begin{definition}
\label{def:homomorphisms}
Let $A$ and $B$ be two Banach algebras.
\begin{itemize}
\item [(i)] A {\bf homomorphism} from $A$ into $B$ is a continuous linear map $\ff:A\ra B$ such that $\ff(ab)=\ff(a)\ff(b)$ for all $a, b \in A$.
\item [(ii)] A homomorphism which is not necessarily continuous is called an {\bf algebraic homomorphism}.
\item [(iii)] When $A$ and $B$ are involutive, $\ff$ is called a {\bf $\ast$-homomorphism} if $\ff(a\s)=\ff(a)\s$ for all $a\in A$.
\item [(iv)] An {\bf isomorphism} from $A$ into $B$ is a one-to-one and onto homomorphism such that its inverse is also continuous.
\item [(v)] An {\bf isometry} from $A$ into $B$ is a homomorphism $\ff$ which preserves the norm, that is $\|\ff(a)\|=\|a\|$ for all $a\in A$. When an isomorphism preserves the norm, it is called an {\bf isometric isomorphism}.
\end{itemize}
\end{definition}
One notes that every isometry is one-to-one.
\begin{proposition}
Let $\mm$ be a closed ideal of $A$. Then $A/\mm$ equipped with the quotient norm is a Banach algebra and the quotient map $\pi: A \ra A/\mm$ is a homomorphism.
\end{proposition}
\begin{proof}
The quotient map is clearly an algebraic homomorphism. Also we know from functional analysis that $A/\mm$ is a Banach space and $\pi$ is continuous, see Theorem 1.41 of \cite{rudinfunctional}. In fact, $\|\pi\|\leq 1$. We only need to show that the quotient norm is sub-multiplicative. For given $a, b\in A$ and for every $\ep>0$, by the definition of the quotient norm, there exist $m,n\in \mm$ such that $\|a+m\|\leq \|\pi(a)\|+\ep$ and $\|b+n\|\leq\|\pi(b)\|+\ep$. Since $(a+m)(b+n)\in ab+\mm$, we have
\begin{eqnarray*}
\|\pi(a)\pi(b)\|&=&\|\pi((a+m)(b+n))\|\\
&\leq& \|(a+m)(b+n)\|\\
&\leq& \|a+m\|\|b+n\|\\
&\leq& \|\pi(a)\|\|\pi(b)\| +\ep(\|\pi(a)\|+\|\pi(b)\|+\ep).
\end{eqnarray*}
This holds for every $\ep$, and so it implies the desired inequality.
\end{proof}
\begin{definition} Let $A$ be commutative and let $\mathfrak{m}$ be an ideal of $A$. An element $e\in A$ is called a {\bf unit modulo} $\mathfrak{m}$ if its image  $\bar{e}$ in $A/\mathfrak{m}$ is the unit of $A/\mathfrak{m}$, i.e. $ea-a\in \mathfrak{m}$ and $ae-a\in \mathfrak{m}$ for all $a\in A$. An ideal $\mathfrak{m}$ of $A$ is called {\bf regular} if there exists a unit modulo $\mathfrak{m}$ in $A$.
\end{definition}

\begin{lemma}
Let $A$ be as above. If $\mathfrak{m}$ is a proper ideal of $A$ and $e\in A$ is a unit modulo $\mathfrak{m}$, then $\inf\{ \|e-a\|; a\in \mathfrak{m} \}\geq 1$.
\end{lemma}
\begin{proof}
If $\|e-a \|<1$ for some $a\in \mm$, then the power series $b:=\sum_{n=1}^\infty (e-x)^n$ converges and we have $b=(e-a)b+e-a=eb-ab+e-a$. Hence $e=(ab+a)-(eb-b)\in \mm$. By the definition of $e$, we deduce that $\mm = A$, which is a contradiction.
\end{proof}
Using the above lemma, one can repeat Parts (ii) and (iii) of Exercise \ref{ex:ex2-1} for non-unital commutative Banach algebras too.
\begin{proposition}
Every regular and proper ideal of a commutative Banach algebra $A$ is contained in a regular maximal ideal of $A$.
\end{proposition}
\begin{proof}
This statement follows from a routine argument based on Zorn's lemma and is left to the reader.
\end{proof}

\begin{corollary}
Let $A$ be unital and commutative. For every non-invertible element $a\in A$, there exists a maximal ideal of $A$ containing $a$.
\end{corollary}

Let $A$ be commutative and let $\mm$ be a regular maximal ideal of $A$. The quotient algebra $A/\mm$ is a field and by the Gelfand-Mazur theorem, Corollary \ref{cor:gelfandmazur}, it has to be isomorphic to the field $\c$ of complex numbers. Since the only algebra automorphism of $\c$ is identity, the isomorphism $A/\mm\ra \c$ is uniquely determined by $\mm$. Therefore by combining this isomorphism with the quotient map $A\ra A/\mm$, we get a homomorphism uniquely defined by $\mm$. In other words, we correspond a homomorphism $\omega_\mm :A\ra \c$ to each regular maximal ideal $\mm$ in $A$. There is a reverse correspondence too. To each non-zero homomorphism $\omega:A\ra\c$, we simply associate its kernel, $\mm_\omega :=ker\omega$. One notes that since $\c$ is a simple and unital algebra, $\mm_\omega$ is a regular maximal ideal of $A$. This discussion is summarized in the following proposition.

\begin{proposition} Assume $A$ is a commutative Banach algebra. Let $\mmm(A)$ denote the set of all regular maximal ideals of $A$ and let $\Omega(A)$ denote the set of all non-zero homomorphisms from $A$ into $\c$. Then the mappings $\mmm(A)\ra \Omega(A)$, $\mm\mapsto \omega_\mm$, and $\Omega(A)\ra \mmm(A)$, $\omega\mapsto \mm_\omega$, are inverse of each other.
\end{proposition}
\begin{proof}
By the definition, it is clear that $\mm_{\omega_\mm}=ker(\omega_\mm)=\mm$ for every regular maximal idea $\mm$ in $A$. For $\omega\in\Omega(A)$, the mapping $A/Ker\omega \ra \c$ defined by $[a]\mapsto \omega(a)$ for all $a\in A$ is an isomorphism. Thus by combining this isomorphism with the quotient map $A\ra A/Ker\omega$, we obtain a homomorphism from $A$ into $\c$, which is exactly $\omega$. This means $\omega_{\mm_\omega}=\omega$.
\end{proof}

One notes that if a homomorphism $\ff: A\ra \c$ is non-zero, then $\ff(1_A)=1$ and $\ff(a)\neq 0$ for all $a\in A^\times$. Now, we are going to prove that these two conditions are sufficient for a linear map  $\ff: A\ra \c$ to be a homomorphism.
\begin{lemma}
\label{lem:hom}
Let $f$ be an entire function over $\c$ and let $f(0)=1$, $f^\prime (0)=0$, and $0< |f(\la)|< e^{|\la|}$ for all $\la\in \c$. Then $f(\la)=1$ for all $\la\in \c$.
\end{lemma}
\begin{proof}
Since $f$ has no zero,  $\frac{ f^\prime (z)}{f(z)}$ is an entire function, and consequently, it has an anti-derivative, say $g$. Considering the aforementioned initial conditions of $f$, we have $f=exp(g)$, $g(0)=g^\prime (0)=0$, and $Re|g(\la)| \leq |\la|$ for all $\la\in\c$. For given positive real number $r$, this inequality implies that $|g(\la)|\leq |2r -g(\la) |$ for all $|\la|\leq r$. Then it is clear that the function
\[
h_r(\la) := \frac{r^2 g(\la)}{\la^2 (2r - g(\la))}
\]
is holomorphic in the domain $\{\la; 0<|\la|<2r  \}$. It is also holomorphic at $0$, because $(2r - g(0))\neq 0$ and the Taylor expansion of $g$ at zero can be divided by $\la^2$. Therefore $h_r(\la)$ is holomorphic in the disc $\{\la; |\la|<r  \}$. Now, by the maximum modulus theorem, see Theorem 8.59 of \cite{ponnusamy-silverman}, either $g$ is constant or it attains its maximum at the boundary of this disc. In the latter case, we note that $|h_r(\la)|\leq 1$ for all $|\la|=r$. Thus $|h_r(\la)|\leq 1$ for all $|\la| \leq 1$. If we fix $\la$ and let $r\ra \infty$, then we must have $g(\la)=0$. This is true for every $\la\in \c$ and implies the desired conclusion.
\end{proof}

\begin{proposition} [Gleason, Kahane, Zelazko]
\label{prop:hom1}
Let $A$ be a unital Banach algebra, (not necessarily commutative). Let $\ff:A\ra \c$ be a linear map such that $\ff(1_A)=1$ and $\ff(a)\neq 0$ for all $a\in A^\times$. Then $\ff$ is a homomorphism.
\end{proposition}
Note that the continuity of $\ff$ is not a part of the hypothesis, but it follows from the proof.
\begin{proof}
Let $N$ denote the kernel of $\ff$.  Since $A=N\oplus \c$ as a vector space, for arbitrary $a,b\in A$, we can find $x,y\in N$ such that $a=x+ \ff(a)$ and $b=y+\ff(b)$. Thus we have $ab=xy + x\ff(b) +y\ff(a) +\ff(a)\ff(b)$. Since $\ff(1)=1$, by applying $\ff$ to the both sides of this identity, we obtain $\ff(ab)=\ff(xy) + \ff(a)\ff(b)$. Therefore to prove that $\ff$ is an algebraic homomorphism, it is enough to show that $\ff(xy)=0$ for all $x,y\in N$. First, we claim that $\ff(x^2)=0$ for all $x\in N$.

Since $N$ contains no invertible elements, $\|1-x\|\geq 1$ for all $x\in N$. Hence $|\ff(\la-x)|= |\la|\leq \|\la-x \|$ for all $x\in N$ and $\la\in \c$. This shows that $\ff$ is continuous and its norm is less than or equal $1$.

Now, fix $x\in N$. Without loss of generality, we can assume that $\|x\|=1$. Define
\[
f(\la):=\sum_{n=0}^\infty \frac{\ff(x^n) \la^n}{n!}, \qquad \forall \la\in \c.
\]
Since $|\ff(x^n)|\leq \|x^n\|\leq \|x\|^n =1$, $f$ is entire and $|f(\la)|\leq exp(|\la|)$ for all $\la\in \c$. Also, $f(0)=\ff(1)=1$ and $f^\prime (0)=\ff(x)=0$. Moreover, one observes that $f(\la)=\ff(exp(\la x))$ for all $\la\in\c$. But the image of the exponential map lies in $A^\times$, so $f(\la)\neq 0$ for all $\la\in \c$. Now, by Lemma \ref{lem:hom}, $f^{\prime \prime} (0)=0$ , and consequently, $\ff(x^2)=0$, which proves our claim.

Setting $a=b$ in the equation $\ff(ab)=\ff(xy) + \ff(a)\ff(b)$, we get $\ff(a^2)=\ff(a)^2$ for all $a\in A$. By replacing $a$ with $a+b$ in this equation, we obtain $\ff(ab+ba)=2 \ff(a)\ff(b)$. This shows that if $x\in N$ and $y$ is an arbitrary element of $A$, then $xy+yx\in N$. Applying this fact to the identity
\[
(xy-yx)^2 +(xy+yx)^2=2(x(yxy)+ (yxy)x),
\]
we conclude $(xy-yx)^2\in N$ for all $x\in N$. This amounts to $\ff(xy-yx)=0$. By adding this equation to $\ff(xy+yx)=0$, we get $2\ff(xy)=0$ whenever $x\in N$, and this completes the proof.
\end{proof}

Let $(A\s)_1$ denote the (norm) closed unit ball of the dual space $A\s$, namely the set of all bounded linear maps from $A$ into $\c$ of norm less than or equal to 1.
\begin{proposition}
\label{prop:spec1} Let $A$ be a commutative Banach algebra. Then
\begin{itemize}
\item[(i)] $\Omega(A)\subset (A\s)_1$, and
\item[(ii)] $\Omega(A)$ is a locally compact set in the weak-$\s$ topology of $A\s$. Moreover, $\oom(A)$ is compact if $A$ is unital.
\end{itemize}
\end{proposition}
\begin{proof}
\begin{itemize}
\item[(i)] This follows from Proposition \ref{prop:hom1}, but we also give a simple proof. Given $\omega\in \Omega(A)$ and for every $a\in A$ and $n\in\n$, we have $|\omega(a)|=|\omega(a^n)|^{1/n}\leq \|\omega\|^{1/n} \|a^n\|^{1/n}$. Thus $|\omega(a)|\leq \lim_{n\ra \infty}\|\omega\|^{1/n} \|a^n\|^{1/n}=r(a)\leq \|a\|$. Therefore $\|\omega\|\leq1$.
\item[(ii)] Set $\Omega\pr(A):=\Omega(A)\cup \{0\}$. Let $\{ \omega_i\}$ be a net in $\Omega\pr(A)$ convergent to a functional $\omega_0\in (A\s)_1$ in the weak-$\s$ topology. For every $a,b \in A$, we have
\[
\om_0(ab)=\lim_{i}\om_i(ab)=\lim_i \om_i(a) \om_i(b)= \lim_i \om_i(a) \lim_i \om_i(b)= \om_0(a)\om_0(b).
\]
This shows that $\Omega\pr(A)$ is closed in weak-$\s$ topology. On the other hand, by the Banach-Alaoglu's theorem, see Theorem \ref{thm:banachalaoghlo}, $A\s_1$ is compact in \ws topology. Hence $\Omega\pr(A)$ is compact, and consequently $\Omega(A)$ is a locally compact subset of $A\s_1$.

    If $A$ is unital, then $\om(1)=1$ for all $\om\in\Omega(A)$. Hence 0 is an isolated point in $\Omega\pr(A)$, which means $\Omega(A)$ is compact.
\end{itemize}
\end{proof}

\begin{remark}
\begin{itemize}
\item [(i)] In the above proof, we actually proved that if $X\subseteq \Omega(A)$ is closed and $0$ does not belong to its boundary in the \ws topology of $A\s$, then $X$ is compact.
\item [(ii)] One also notes that \ws topology on $A\s$ is a locally convex topology defined by a separating set of functionals, and so it is Hausdorff, see Theorem 3.10 of \cite{rudinfunctional}. Therefore $\Omega(A)$ is always Hausdorff.
\end{itemize}
\end{remark}
\begin{definition} Let $X$ be a locally compact Hausdorff topological space. Consider a point outside of $X$  and denote it by $\infty$. Set $X^\infty :=X \cup \{\infty\}$. To topologize $X^\infty$, we define the collection of all open subsets of $X^\infty$ to be all sets of the following types:
\begin{itemize}
\item [(i)] $U$, where $U$ is an open subset of $X$,
\item [(i)] $O\subseteq X^\infty$, where $X^\infty -O$ is a compact subspace of $X$.
\end{itemize}
Then the topological space $X^\infty$ is compact and is called the {\bf one-point} (or {\bf Alexandrov}) {\bf compactification of $X$}. The point $\infty$ is usually called the {\bf point at infinity of $X$}.
\end{definition}

For example, one easily observes that $\t = \{z\in \c ; |z|=1 \}$ is homeomorphic to the one point compactification of $\r$. More generally, $S^n=\{x\in \r^{n+1}; \|x\|=1\}$ is homeomorphic to the one-point compactification of $\r^n$. When $X$ is already compact, one notes that $X^\infty$ is nothing but the disjoint union of $X$ with the one point set $\{ \infty\}$, (discuss both the unital and non-unital cases).

\begin{exercise}
\label{ex:ommprime}
Let $A$ be a commutative Banach algebra and let $\Omega\pr (A)$ be as the proof of Proposition \ref{prop:spec1}. Show that $\Omega\pr (A)$ is the one-point compactification of $\Omega(A)$.
\end{exercise}

The following propositions illustrates the relationship of one-point compactification of a topological space $X$ and unitization of the commutative Banach algebra $C_0(X)$.

\begin{proposition}
\label{prop:compactificationunit} Let $X$ be a locally compact and Hausdorff topological space. There is a canonical isomorphism between $C_0(X)_1$ and $C(X^\infty)$.
\end{proposition}
\begin{proof}
We extend every $f\in C_0(X)$ to $X^\infty$ by defining $f(\infty)=0$ and denote it again by $f$. Let $e\in C(X^\infty)$ be the constant function 1. Define $\iota: C_0(X)_1\ra C(X^\infty)$ by $(f,\la) \mapsto f+ \la e$ for all $(f, \la) \in C_0(X)_1$. It is an easy exercise to check that $\iota $ is an algebraic isomorphism. For $(f, \la) \in C_0(X)_1$, we have
\[
\|\iota(f,\la)\|_{\sup} = \sup_{x\in X^\infty} |f(x)+\la| \leq  \sup_{x\in X} |f(x)|+|\la|=\|(f,\la)\|.
\]
This shows that $\iota$ is continuous. On the other hand, for every $x\in X$ we have $|f(x)|\leq |\la| +|f(x)+\la|$. Thus we get
\begin{eqnarray*}
\|(f, \la)\| &=&  \sup_{x\in X} |f(x)|+|\la| \\
&\leq & \sup_{x\in X^\infty} |f(x)+\la | + 2|\la| \\
&\leq & \sup_{x\in X^\infty} |f(x)+\la | + 2|f(\infty) +\la| \\
&\leq & 3 \sup_{x\in X^\infty} |f(x)+\la | \\
&=&  3 \|\iota(f,\la)\|_{\sup}.
\end{eqnarray*}
This proves that the inverse of $\iota$ is continuous as well.
\end{proof}
We note that $\iota $ is not an isometry, but its restriction to $C_0(X)$ is. The reason for this phenomenon is that the norm on $C_0(X)_1$ is not a $C\s$-norm, but the norm on $C(X^\infty)$ is. We shall come back to this issue in Exercise \ref{exe:compactificationunit}.

\begin{definition} Let $A$ be a commutative Banach algebra.
\begin{itemize}
\item [(i)] The topological space $\Omega(A)$ is called the {\bf spectrum of $A$} and its elements are called {\bf characters of $A$}.
\item [(ii)] For every $a\in A$, we define $\hat{a}\in \Omega(A)$ by $\hat{a}(\om):=\om(a)$ for all $\om\in\Omega(A)$. The mapping $\ggg:A\ra C_0(\Omega(A))$ sending $a$ to $\hat{a}$ is called the {\bf Gelfand transform}.
\end{itemize}
\end{definition}

\begin{theorem}
\label{thm:gelfand1}
Let $A$ be a commutative Banach algebra and let $a\in A$.
\begin{itemize}
\item[(i)] The Gelfand transform is an algebraic homomorphism.
\item[(ii)] If $A$ is unital, then $\si(a)=\hat{a}(\Omega(A))$, otherwise $\si(a)=\hat{a}(\Omega(A))\cup \{0\}$.
\item[(iii)] $r(a)=\|\hat{a}\|_{\sup}$.
\item[(iv)] The Gelfand transform is continuous. In fact, $\|\hat{a}\|_{\sup}\leq \|a\|$ for all $a\in A$.
\end{itemize}
\end{theorem}
\begin{proof}
\begin{itemize}
\item[(i)] Since the \ws topology on $A\s$ is defined by functionals $\{\hat{a}; a\in A\}$, all functionals, and consequently their restrictions to $\Omega(A)$, are continuous in this topology. On the other hand, for a real number $M>0$ and for $a\in A$, the set $\{\om\in\Omega(A); |\hat{a}(\om)| \geq M\}$ is closed in the \ws topology. Thus it is compact, because it does not contains the zero functional. Therefore $\hat{a}\in C_0(\Omega(A))$, for all $a\in A$, in other words, $\ggg$ is well-defined. It is straightforward to see that $\ggg$ is linear and multiplicative.
\item[(ii)] Let $A$ be unital and let $a\in A$. If $\la\in \si(a)$, then $a-\la$ is not invertible. Thus there is a maximal ideal, say $\mm$, containing $a-\la$. This means $\om_\mm (a-\la)=0$. In other words, $\hat{a}(\om_\mm)=\om_\mm(a)=\la$. Conversely, if $\la=\om(a)=\hat{a}(\om)$ for some $\om\in \Omega(A)$, then $a-\la$ belongs to the maximal ideal $\mm_\om$. Hence it is not an invertible element, that is $\la\in\si(a)$. The case that $A$ is non-unital follows from the above case.
\item[(iii)] It is clear from part (ii) and the definition of the norm $\norm_{\sup}$.
\item[(iii)] It is clear from part (iii) and Proposition \ref{prop:spradius1}.
\end{itemize}
\end{proof}

Assume $A$ is a commutative Banach algebra. To every $\ff\in \Omega(A)$, we associate $\jmath(\ff):=\ff_1\in \Omega(A_1)$ by defining $\ff_1(a,\la):= \ff(a)+ \la$ for all $(a,\la)\in A_1$. Clearly, $\ff_1$ is a non-zero homomorphism. But there is still one non-zero homomorphism in $\Omega(A_1)$ that is not obtained in this way. It is the homomorphism $\ff_\infty: A_1 \ra \c$ defined by $\ff_\infty(a,\la):= \la$.

\begin{exercise}
Show that $\Omega(A_1)=\jmath(\Omega(A))\cup \{ \ff_\infty\}$.
\end{exercise}
\begin{lemma}
With the above notation, the inclusion $\jmath:\Omega(A) \ra \Omega(A_1)$ is a homeomorphism  onto its image.
\end{lemma}
\begin{proof}
Let $(\ff_i)$ be a net in $\Omega(A)$ convergent to $\ff$ in \ws topology. Then for every $(a,\la)\in A_1$, we have $ \ff_{i1}(a,\la)=\ff_i(a) +\la$ which is obviously convergent to $\ff(a)+\la =\ff_1(a,\la)$. Thus $\ff_{i1}\ra \ff_1$  in \ws topology, and so $\jmath$ is continuous. A similar argument works for the continuity of the converse map.
\end{proof}

\begin{exercise}
With the above notation, show that if $A$ is not unital, then $\Omega(A_1)$ is the one point compactification of $\jmath(\Omega(A))$. Describe what happens when $A$ is unital.
\end{exercise}

\begin{lemma}
\label{lem:alghombanach}
Let $\psi:A\ra B$ be an algebraic homomorphism between two commutative Banach algebras such that $\psi\s(\ff):=\ff \psi\neq 0$ for all $\ff\in \Omega(B)$. Then $\psi\s :\oom(B)\ra \oom(A)$ is continuous. Moreover, $\psi\s$ is a homeomorphism whenever it is bijective.
\end{lemma}
\begin{proof}
If $(\ff_i)$ is a convergent net in the \ws topology of $\oom(B)$, then $(\ff_i o \psi)$ is a convergent net in the \ws topology of $\oom(A)$ too. Hence $\psi\s$ is continuous. Now, assume $\psi\s$ is bijective. When $B$ is unital, $\oom(B)$ is compact and consequently, $\psi\s$ is a homeomorphism. When $B$ is not unital, we consider the canonical extension $\psi_1: A_1\ra B_1$ of $\psi$ defined by  $\psi_1(a,\la):=(\psi(a),\la)$. Then $\psi_1\s:\oom(B_1)\ra \oom(A_1)$ defined by $\psi_1\s(\ff_1)=\ff_1 o \psi_1$ is bijective and continuous and consequently a homeomorphism. Using the above exercise and the fact that $\psi\s$ is the restriction of $\psi_1\s$ to $\oom(B)$, one concludes that $\psi\s$ is a homeomorphism.
\end{proof}

\begin{definition} Assume $A$ is a commutative Banach algebra. The kernel of the Gelfand transform $\ggg$ is called the {\bf radical of $A$}. If it is $\{0\}$, then $A$ is called a {\bf semi-simple commutative Banach algebra}.
\end{definition}

\begin{remark}
\label{rem:radical}
Let $a$ be an element of the radical of a commutative Banach algebra $A$. Then $\si_A(a)=\{0\}$ by Theorem \ref{thm:gelfand1}. Since $C_0(\oom(A))$ is a \cs-algebra, the converse is also true, see Problem \ref{e:3-2}. In other words, the radical of a commutative Banach algebra $A$ consists of all elements $a\in A$ such that $\si_A(a)=\{ 0 \}$.
\end{remark}

The Gelfand transform is especially useful for semi-simple commutative Banach algebras. Assume $A$ is such an algebra. The Gelfand transform gives rise to a faithful representation of $A$ as a subalgebra of $C_0(\oom(A))$. If we require also that the image of the Gelfand transform of $A$, i.e. $\hat{A}:=\{ \hat{a}; a\in A\}$, to be dense in $C_0(\oom(A))$, we need to impose another condition. A Banach algebra equipped with this extra condition is called symmetric. Fortunately, every commutative \cs-algebra is both symmetric and semi-simple.

\begin{definition}
An involutive and commutative Banach algebra $A$ is called {\bf symmetric} if $\ff(a\s)=\overline{\ff(a)}$ for all $\ff\in \Omega(A)$ and $a\in A$.
\end{definition}

\begin{proposition}
\label{prop:dense-sym}
Let $A$ be a symmetric Banach algebra. Then $\hat{A}$ is dense in $C_0(\oom(A))$.
\end{proposition}
\begin{proof}
Since $A$ is symmetric, $\hat{A}$ is closed under complex conjugation. Also, $\hat{A}$ separates the points of $\oom(A)$, because, for every $\ff_1\neq \ff_2\in \oom(A)$, there exists some $a\in A$ such that $\ff_1(a)\neq\ff_2(a)$ and so $\hat{a}(\ff_1)\neq \hat{a}(\ff_2)$. Finally, we note that since every $\ff\in\oom(A)$ is non-zero, there is some $a\in A$ such that $\hat{a}(\ff)=\ff(a)\neq 0$. Now, the Stone-Weierstrass theorem implies that $\hat{A}$ is dense in $C_0(\oom(A))$, see Theorem A.10.1 of \cite{deitmar-echterhoff}.
\end{proof}

To show every \cs-algebra is symmetric, we need some more definitions.

\begin{definition} Let $A$ be an involutive Banach algebra.
\begin{itemize}
\item[(i)] An element $a\in A$ is called {\bf self-adjoint} (or {\bf hermitian}) if $a=a\s$. The set of all self adjoint elements of $A$ is denoted by $A_h$.
\item[(ii)] An element $a\in A$ is called {\bf normal} if $aa\s=a\s a$.
\item[(iii)] An element $a\in A$ is called {\bf unitary} if $aa\s=a\s a=1$. The set of all unitary elements of $A$ is a subgroup of $A^\times$ and is called the {\bf unitary group of $A$} and is denoted by $A_u$.
\item[(iv)] An element $a\in A$ is called {\bf idempotent} if $a^2=a$. The set of all idempotent elements of $A$ is denoted by $Idem(A)$.
\item[(v)] An element $a\in A$ is called {\bf projection} if $a^2=a=a\s$. The set of all projections of $A$ is denoted by $Proj(A)$.
\item[(vi)] Let $A$ be a \cs-algebra. An element $a\in A$ is called {\bf positive} if $a=a\s$ and $\si_A(a)\subseteq [0,\infty[$. This is denoted by $a\geq 0$ and the set of all positive elements of $A$ is denoted by $A_+$.
\item[(vii)] A subset $X$ of $A$ is called {\bf self adjoint} if it is closed under the involution of $A$.
\end{itemize}
\end{definition}

Most parts of the above definition are still well defined in more general settings. For example, Parts (i) make sense in every involutive algebra or Part (iv) is the definition of idempotent elements in any ring. The normal elements play a very important role in continuous functional calculus in the next section. The following remark shows how every element of an involutive algebra can be written as a linear combination of two self adjoint elements.

\begin{remark}
\label{rem:real-imaginary}
Let $A$ be a \cs-algebra or more generally an involutive Banach algebra. For given $a\in A$, define $a_1=Re(a):=\frac{a+a\s}{2}$ and $a_2=Im(a):=\frac{a-a\s}{2i}$. Then it is easy to see that both $a_1$ and $a_2$ are self-adjoint and $a=a_1+ia_2$. The elements $Re(a)$ and $Im(a)$ are called the {\bf real part} and the {\bf imaginary part of $a$}, respectively.
\end{remark}

This decomposition of an arbitrary element to a linear combination of two self-adjoint elements is particularly useful when we need to reduce the argument to self-adjoint elements. It is also used directly to state and prove some statements. For example, consider the following easy exercise:

\begin{exercise}
With the notation of Remark \ref{rem:real-imaginary}, show that $a\in A$ is normal if and only if its real and imaginary parts commute, i.e. $a_1a_2=a_2a_1$.
\end{exercise}

\begin{lemma}
\label{lem:sym}
Let $A$ be an involutive and commutative Banach algebra. Then the following statements are equivalent:
\begin{itemize}
\item [(i)] $A$ is symmetric.
\item [(ii)] $\widehat{a\s}=\overline{\hat{a}}$ for all $a\in A$.
\item [(iii)] $\ff(a)\in \r$ for all $\ff\in \Omega(A)$ and $a\in A_h$.
\end{itemize}
\end{lemma}
\begin{proof}
We only show that (iii) implies (i). The rest of the statements are clear. Let $a\in A$ and let $a=a_1+ia_2$ be the decomposition of $a$ as discussed in Remark \ref{rem:real-imaginary}. For every $\ff\in \Omega(A)$, we have
\[
\overline{\ff(a)}=\overline{\ff(a_1)+i \ff(a_2)}=\ff(a_1)-i\ff(a_2)=\ff(a_1-ia_2)=\ff(a\s).
\]
\end{proof}

\begin{proposition}
\label{prop:cs-sym}
Every commutative \cs-algebra is symmetric.
\end{proposition}
\begin{proof}
Assume $A$ is a unital commutative \cs-algebra. Let $\ff\in \oom(A)$ and $a\in A_h$. Consider $x,y\in \r$ such that $\ff(a)=x+iy$ and define $a_t:=a+it$ for all $t\in \r$. Then we have $a_t\s a_t=a^2+ t^2$ and $\ff(a_t)=x=i(y+t)$. Now, we have
\[
x^2+(y+t)^2=|\ff(a_t)|^2\leq \|a_t\|^2=\|a_t\s a_t\|=\|a^2+t^2\|\leq \|a\|^2 +t^2.
\]
Hence $x^2+y^2+1yt\leq \|a\|^2$ for all $t\in \r$. This is possible only if $y=0$. Therefore $\ff(a)$ is a real number. When $A$ is non-unital, the assertion follows from the above case by passing to $\tilde{A}$, the \cs-unitization of $A$.
\end{proof}

\begin{proposition}
\label{prop:spradius3}
Let $A$ be a \cs-algebra and let $a\in A$ be a normal element. Then $r(a)=\|a\|$.
\end{proposition}
\begin{proof}
For every self-adjoint element $x\in A$, we have $\|x^2\|=\|x\s x\|=\|x\|^2$. Hence we compute
\[
\|a^{2}\|^2=\|(a^{2})\s a^{2}\|=\|(a\s a)^2\|=\|a\s a\|^2=\|a\|^4.
\]
Therefore by induction, we get $\|a^{2^n}\|=\|a\|^{2^n}$, and consequently, we have
\[
r(a)=\lim_{n\ra \infty}\|a^n\|^{1/n}=\lim_{n\ra \infty}\|a^{2^n}\|^{1/2^n}=\|a\|.
\]
\end{proof}

\begin{corollary}
Let $A$ be a \cs-algebra and let $a\in A$. Then $\|a\|=r(a\s a)^{1/2}$.
\end{corollary}
This shows the norm of a \cs-algebra is completely determined by its algebraic structure. Therefore every \cs-algebra has only one \cs-norm. One can deduce this property also from Corollary \ref{cor:csinjection}.

\begin{theorem}
Let $A$ be a commutative \cs-algebra. The Gelfand transform is an isometric $\ast$-isomorphism from $A$ onto $C_0(\oom(A))$.
\end{theorem}
\begin{proof}
Since $A$ is commutative all elements of $A$ are normal. Hence for all $a\in A$, we have $\|a\|=r(a)=\|\hat{a}\|_{\sup}$. This shows that the Gelfand transform is an isometry and $A$ is semi-simple. It follows easily from Proposition \ref{prop:cs-sym} and Lemma \ref{lem:sym} that the Gelfand transform is a $\ast$-homomorphism. Finally, we note that $\hat{A}$ is closed and dense subalgebra of $C_0(\oom(A))$, because $\ggg$ is isometry and because of Propositions \ref{prop:dense-sym} and \ref{prop:cs-sym}. Hence the Gelfand transform must be onto.
\end{proof}
In the above discussion, we first associated a locally compact Hausdorff topological space, i.e. $\oom(A)$, to every commutative Banach algebra $A$, in particular every commutative \cs-algebra $A$. Then using the Gelfand transform, we proved that the \cs-algebra $C_0(\oom(A))$ is isometrically isomorphic to $A$. There is also a reverse procedure starting from a locally compact and Hausdorff topological space which is explained in the following proposition.

\begin{proposition}
\label{prop:inversegelfand}
Let $X$ be a locally compact and Hausdorff topological space. Then the map $\fff:X\ra \oom(C_0(X))$ defined as follows is an onto homeomorphism:
\[
x\mapsto \hat{x}, \quad \hat{x}(f):=f(x), \qquad \forall x\in X \,\text{and}\, f\in C_0(X).
\]
\end{proposition}
\begin{proof}
Clearly, $\hat{x}$ is a multiplicative homomorphism from $C_0(X)$ into $\c$ for all $x\in X$. Since $X^\infty$ is a normal topological space, see Theorems 2.4 and 3.1 of \cite{munkeres}, for every $x\in X$, there exists $f\in C_0(X)$ such that $f(x)\neq 0$. Hence $\widehat{x}\neq 0$ for all $x\in X$. This shows that $\fff$ is well-defined.

If $x_1$ and $x_2$ are two points in $X$ such that $\widehat{x_1}(f)=\widehat{x_2}(f)$ for every $f\in C_0(X)$, then $f(x_1)=f(x_2)$ for all $f\in C_0(X)$. Again, using the fact that $X^\infty$ is normal and using the Urysohn lemma, it is only possible when $x_1=x_2$. Therefore $\fff$ is one-to-one.

Let $\om$ be an element of $\oom(C_0(X))$. It is easy to see that $\om(f)\geq 0$ for every non-negative function $f\in C_0(X)$. Hence by the Riesz representation theorem, see Theorem \ref{thm:rieszrep}, there exists a positive Radon measure $\mu$ on $X$ such that $\om(f)=\int_X f(x)d\mu (x)$  for all $f\in C_0(X)$. Thus we have
\[
0=\om\left(\overline{(f-\om(f))} (f-\om(f))\right)=\int_X |f(x)-\om(f)|^2 d\mu(x).
\]
This means that, for every $f\in C_0(X)$, $f$ equals to the constant function  $\om(f)$ $\mu$-almost everywhere. Regarding Remark \ref{rem:concentrated}, there is a point $x_0\in X$ such that $\om(f)=f(x_0)$ for all $f\in C_0(X)$. In other words, $\om = \widehat{x_0}$ and this shows that $\fff$ is onto.

Let $(x_i)$ be a net in $X$ convergent to a point $x_0\in X$. Then for all $f\in C_0(X)$, we have $f(x_i)\ra f(x_0)$, which implies that $\widehat{x_i}\ra \widehat{x_0}$ in the weak $\ast$-topology. This means $\fff$ is continuous. One easily extends $\fff$ to a continuous and bijective map from $X^\infty$ onto $\oom\pr(C_0(X))$, see Exercise \ref{ex:ommprime}. Since $X^\infty$ is compact, this extension is a homeomorphism and so is $\fff$.
\end{proof}

Let $X$ be a set and consider $P(X)$, the power set of $X$, as the $\si$-algebra over $X$. For a given point $x_0\in X$, the {\bf Dirac measure} or {\bf point mass at $x_0$} is the measure $\delta_{x_0}$ defined by $\delta_{x_0}(E):=1$ if $x_0\in E$ and $\delta_{x_0}(E):=0$ otherwise. The same names are also applied for smaller $\si$-algebras than $P(X)$.

\begin{remark}
\label{rem:concentrated}
Let $X$ be a locally compact and Hausdorff topological space and let $\mu$ be a positive Radon measure on $X$. We say $x\in X$ is a {\bf concentration point of $\mu$} if every open set containing $x$ has non-zero measure. We say $\mu$ is {\bf concentrated at a point $x_0\in X$} if $x_0$ is the only concentration point of $\mu$. It is easy to see that if $\mu$ is concentrated at a point $x_0\in X$, then $\mu$ is equal to a positive multiple of the Dirac measure or point mass at the point $x_0$. Now, if for every $f\in C_0(X)$, $f$ is constant $\mu$-almost everywhere, $\mu$ must be concentrated at a point $X_0\in X$. The reason is that if $\mu$ has two concentration points, say $x_0\neq x_1$, then using the fact that $X^\infty$ is normal and using the Urysohn lemma, there exist a function $f\in C_0(X)$ such that $f$ takes two different values over disjoint neighborhoods of $x_0$ and $x_1$ and this contradicts with our assumption. Also, a similar argument excludes the case that $\mu$ has no concentration point. Finally, we note that if $\mu$ is concentrated at $x_0$, then $\int_X f(x)d\mu(x) = f(x_0)$ for all $f\in C_0(X)$.
\end{remark}
\begin{remark}
\label{rem:whyproper}
Let $\psi: X\ra Y $ be a continuous map between two compact topological spaces. Define $\psi\s: C(Y)\ra C(X)$ by $\psi\s(f)=f o \psi$. It is shown that it is an \ss-homomorphism. It is worthwhile to note that not every \ss-homomorphism $C(Y)\ra C(X)$ comes from a continuous map from $X$ into $Y$. For example, the zero homomorphism cannot be obtained in this way. Because, for every $y$ in the image of $\psi$, one can define a continuous function $f:Y\ra \c$ such that $f(y)=1$. Then $\psi\s(f)\neq 0$. One also notes that the compactness of $X$ is important here. To see this, consider the exponential map $e:\r\ra\t$, $t\mapsto e^{2\pi i}$. It is continuous, but $e\s:C(\t)\ra C_0(\r)$ is not well defined. Because it sends the constant function $1_\t$ to the constant function $1_\r$ which belongs to $C_b(\r)$ not $C_0(\r)$. The zero homomorphism $\ff:A\ra B$ between two commutative \cs-algebras cause another problem. Because $\ff\s(\om)=\om \ff=0$ for every $\om\in \oom(B)$, so $\ff\s=0$. Hence $\ff\s$ is not even a well defined map from $\oom(B)$ into $\oom(A)$. Therefore in order to obtain a bijective correspondence between continuous functions from a locally compact and Hausdorff space $X$ into another locally compact and Hausdorff space $Y$ and \ss-homomorphism from $C_0(Y)$ into $C_0(X)$, we have to impose some restrictions both on continuous maps and on \ss-homomorphisms.
\end{remark}
\begin{definition}
\label{def:propermaps}
\begin{itemize}
\item [(i)] Let $A$ be a \cs-algebra. A net $(h_\la)$ in $A$ is called an {\bf approximate unit for $A$} if every $h_\la$ is positive and $(h_\la)$ is an  approximate unit for $A$ as a Banach algebra, namely $\|h_\la\|\leq 1$ for all $\la$ and both nets $(ah_\la)$ and $(h_\la a)$ converge to $a$ for all $a\in A$, see also Definition \ref{def:approxtypes}.
\item [(ii)] A \ss-homomorphism $\ff:A\ra B$ between two \cs-algebras is called {\bf proper} if the image of every approximate unit in $A$ under $\ff$ is an approximate unit in $B$.
\item [(iii)] A continuous map $\psi:X\ra Y$ between two topological spaces is called {\bf proper} if the preimage of every compact subset of $Y$ is compact in $X$.
\end{itemize}
\end{definition}

\begin{example}
\label{exa:apprunit}
Let $\Sigma$ be the collection of all compact subsets of a locally compact and Hausdorff space $X$. $\Sigma$ is a directed set with respect to inclusion. For $K\in \Sigma$, pick a continuous $f_K: X\ra [0,1]$ vanishing at infinity such that $f_K(x)=1$ for all $x\in K$. The reader easily verifies that the net $(f_K)_{K\in \Sigma}$ is an approximate unit for $C_0(X)$. One notes that since $X^\infty$ is a normal topological space, the elements of this net can be chosen from compact support function if needed.
\end{example}

\begin{proposition}
\label{prop:contmapspec}
Let $\psi:X\ra Y$ be a proper continuous map between two locally compact and Hausdorff topological spaces. The map $\psi\s: C_0(Y)\ra C_0(X)$ defined by $\psi\s(f)=fo\psi$ is a proper \ss-homomorphism. When $\psi$ is a homeomorphism, $\psi\s$ is an isometric isomorphism.
\end{proposition}
\begin{proof}
For given $f\in C_0(Y)$ and for every $\ep>0$, let $K$ be a compact subset of $Y$ such that $|f(y)|<\ep$ for all $y\in Y\backslash K$. Set $K\pr :=\psi\inv (K)$. Then $K\pr$ is compact and $|\psi\s(f)(x)|<\ep$ for all $x\in X\backslash K\pr$. This shows that $\psi\s(f)\in C_0(X)$. It is straightforward to check that $\psi\s$ is a $\ast$-homomorphism. One also easily checks that when $\psi$ is onto, $\psi\s$ is an isometry and when $\psi$ is a homeomorphism, $\psi\s$ is an isomorphism.

Now, Let $(f_\la)_{\la\in \Lambda}$ be an approximate unit in $C_0(Y)$ and let $g\in C_0(X)$. For given $\ep>0$, let $K$ be a compact subset of $X$ such that $|g(x)|<\ep$ for every $x\in X\backslash K$. Pick a continuous function $\alpha:Y\ra [0,1]$ vanishing at infinity such that $\psi\s(\alpha)(x)=\alpha(\psi (x))=1$ for all $x\in K$. By definition, there is some $\la_0\in \Lambda$ such that $\|f_\la \alpha -\alpha \|_{\sup} < \frac{\ep}{ \|g\|_{\sup} }$ for all $\la\geq \la_0$. Then for $\la\geq \la_0$, we have
\begin{eqnarray*}
\|\psi\s(f_\la)g-g\|_{\sup} &\leq& \|\psi\s(f_\la)\psi\s(\alpha) g-\psi\s(\alpha)g\|_{\sup}\\
 &+& \|\psi\s(f_\la) g-\psi\s(f_\la)\psi\s(\alpha)g\|_{\sup}+ \|\psi\s(\alpha)g-g\|_{\sup}\\
&\leq& \|\psi\s(f_\la)\psi\s(\alpha)-\psi\s(\alpha)\|_{\sup}\|g\|_{\sup}\\
 &+& 2\sup\{|g(x)|; x\in X\backslash K \}\\
&<& \frac{\ep}{\|g\|_{\sup}} \|g\|_{\sup} +2 \ep=3\ep.
\end{eqnarray*}
This shows that the net $(\psi\s(f_\la))$ is an approximate unit for $C_0(X)$. Hence $\psi\s$ is proper.
\end{proof}

\begin{exercise}
\label{ex:contmapspec} Complete the gaps in the proof of Proposition \ref{prop:contmapspec}.
\end{exercise}

\begin{proposition}
\label{prop:properhom}
Let $\psi:A\ra B$ be a proper \ss-homomorphism between two \cs-algebras. Then $\psi\s:\oom(B)\ra \oom(A)$ defined by $\om\mapsto \om \psi$ is a proper continuous map. It is a homeomorphism if $\psi$ is an isomorphism.
\end{proposition}
\begin{proof}
Let $(a_i)$ be an approximate unit in $A$ and let $\om\in \oom(A)$. Pick an element $a$ in $A$ such that $\om(a)\neq 0$. Then $\om(a)=\om(\lim_i a_i a)= \lim_i\om(a_i) \om(a)$. This implies that $\lim_i\om(a_i)=1$. On the other hand, since $\psi$ is proper $\psi(a_i)$ is an approximate unit for $B$, and so $\lim_i \om\psi(a_i)=1$ for all $\om\in \oom(B)$. This implies that $\om\psi\neq 0$ for all $\om\in\oom(B)$. Therefore by applying Lemma \ref{lem:alghombanach}, we conclude that $\psi\s$ is a continuous map. When $\psi $ is an isomorphism, $\psi\s$ is bijective, and so is homeomorphism.

If $A$ is unital, then $B$ is unital too. It is clear that $\psi\s$ is proper in this case. Assume $A$ is non-unital. Extend $\psi$ to a unital \ss-homomorphism $\tilde{\psi}:\tilde{A}\ra \tilde{B}$. It is clear that $\tilde{\psi}$ is still a proper \ss-homomorphism. Therefore $\tilde{\psi}\s:\oom(\tilde{B})\ra \oom(\tilde{A})$ is a proper continuous map. But $\psi\s$ is the restriction of this map to $\oom(B)$, so it proper too.
\end{proof}
We will continue the above results and discussion in Section \ref{sec:Gelfandduality}, where we will explain the Gelfand duality.

\section{The continuous functional calculus}
\label{sec:confunctionalcal}
The continuous functional calculus is one of the most important tools in the theory of \cs-algebras. It is an immediate application of the Gelfand transform and inspires many similar results in the theory of \cs-algebra.

Let $A$ be a unital \cs-algebra and let $a$  be a normal element of $A$. Then the \cs-algebra generated by $\{ a, 1\}$ is a unital commutative \cs-algebra, which we denote it by $C\s(a,1)$.

\begin{lemma}
\label{lem:normalspec}
Let $A$ and $a\in A$ be as above and let $B=C\s(a,1)$. Then the map $\theta: \oom(B)\ra \si_A(a)$ defined by $\theta(\om):=\om(a)$ is an onto homeomorphism.
\end{lemma}
\begin{proof}
First, we show that the map $\theta: \oom(B)\ra \si_B(a)$ is an onto homeomorphism. Let $\om_1$ and $\om_2$ be two elements of $\oom(B)$ such that $\om_1(a)=\om_2(a)$. We also know that $\om_1(1)=\om_2(1)=1$. Hence $\om_1$ and $\om_2$ are equal over every complex polynomial with two variables $a$, $a\s$. The set of all these polynomials is dense in $B$. Therefore $\om_1=\om_2$ on $B$, and so $\theta$ is one-to-one. It is easy to see that $\theta$ is continuous. Since $\oom(B)$ is compact $\theta$ is a homeomorphism onto its image. Finally, it follows from Part (ii) of Theorem \ref{thm:gelfand1} that $\th$ is onto.

By Proposition \ref{prop:contmapspec}, the homeomorphism $\th:\oom(B)\ra \si_B(a)$ gives rise to an isometric $\ast$-isomorphism $\th\s:C(\si_B(a)) \ra C(\oom(B))$. We define $\Phi_a:C(\si_B(a))\ra B$ by
\[
\Phi_a:=\ggg\inv o \th\s,
\]
where $\ggg$ is the Gelfand transform from $B$ onto $C(\oom(B))$. For all $f\in C(\si_B(a))$, we denote $\Phi_a(f)$ by $f(a)$. Since the Gelfand transform is an isometry too, $\Phi_a$ is an isometric $\ast$-isomorphism from $C(\si_B(a))$ onto $B$, in particular, $\|f\|_{\sup}=\|f(a)\|$ for all $f\in C(\si_B(a))$.

Now, we prove $\si_A(a)=\si_B(a)$. We know from Proposition \ref{prop:subalgspec1} that  $\si_A(a)\subseteq \si_B(a)$. Assume that there exists some $\la\in \si_B(a)\backslash \si_A(a)$. Then $a-\la I$ has an inverse in $A$, say $b$. Pick a real number $s> \|b\|$ and define $f:\c \ra \c$ by the following formula:
\[
f(z):=\left\{ \begin{array}{lll} s& \, \text{if} \,& |z-\la|\leq 1/s \\
\frac{1}{|z-\la|}& \, \text{if} \,& |z-\la|\geq 1/s  \end{array} \right.
\]
Using $f$, also define $g(z):=(z-\la)f(z)$ for all $z\in \c$. Considering $f$ and $g$ as elements of $C(\si_B(a))$, we have $\|f\|_{\sup}\leq s$ and $\|g\|_{\sup}\leq 1$. We compute
\begin{eqnarray*}
\|b\|&<& s= f(\la)\leq \|f\|_{\sup}=\|f(a)\|\\
&=& \|b(a-\la) f(a)\|=\|b g(a)\|\\
&\leq & \|b\|\|g(a)\|=\|b\|\|g\|_{\sup}\\
&\leq& \|b\|.
\end{eqnarray*}
This is a contradiction, so $\si_A(a)=\si_B(a)$.
\end{proof}

\begin{corollary}
\label{cor:samespec} If $C$ is a unital \cs-subalgebra of a unital \cs-algebra $A$ and $a\in C$, then $\si_C(a)=\si_A(a)$. If $A$ or $C$ are not necessarily unital, then $\si_C(a) \cup \{0 \}=\si_A(a)  \cup \{0 \}$.
\end{corollary}
\begin{proof}
First, assume $A$ and $C$ are unital. When $a$ is normal, the statement follows immediately from the above lemma. For general $a\in A$, let $a-\la$ be invertible in $A$. Then both $(a-\la)\s(a-\la)$ and $(a-\la)(a-\la)\s$ are invertible in $A$. Since they are self adjoint, They are invertible in $C$ too. This proves that $(a-\la)$ has right and left inverses in $C$, and so it is invertible in $C$. Therefore $\si_C(a)\subseteq \si_A(a)$. The converse inclusion follows from Proposition \ref{prop:subalgspec1}. For general $A$ and $C$, the statement follows from the above case and the definition of the spectrum in a non-unital \cs-algebra.
\end{proof}
An immediate corollary of the above result is given in the following:
\begin{corollary}
\label{cor:additionspec}
Let $A$ be a unital \cs-algebra and let $a,b\in A$. If $ab=ba$, then we have
\begin{itemize}
\item [(i)] $\si_A(ab)\subseteq \si_A(a)\si_A(b)$,
\item [(ii)] $\si_A(a+b)\subseteq \si_A(a)+ \si_A(b)$.
\end{itemize}
\end{corollary}

\begin{proof}
Let $B$ be the commutative \cs-algebra generated by the set $\{ a,b, 1\}$. Using Corollary \ref{cor:samespec}, it is enough to prove these statements in $B$. By the Gelfand transform $B\simeq C(\oom(B))$. Therefore there are $f,g\in  C(\oom(B))$ such that $\hat{a}=f$ and $\hat{b}=g$. Now, the statements follow easily from the fact that $\si_A(x)=\hat{x}(\oom(B))$ for every $x\in B$.
\end{proof}

The above corollary shows how the Gelfand transform can be used to reduce some abstract problems involving commuting elements of a \cs-algebra to problems about function algebras. This idea is the essence of the continuous functional calculus and will be frequently used in some of the proofs and exercises in the future.

\begin{theorem}
\label{thm:confuncal} [The continuous functional calculus] Let $A$ be a unital \cs-algebra and let $a\in A$ be a normal element. There exists a unique isometric $\ast$ -homomorphism $\Phi_a: C(\si_A(a)) \ra C\s(a,1)\subseteq A$ such that $\Phi_a(1_{\si_A(a)})=1_A$ and $\Phi_a(id_{\si_A(a)})=a$.
\end{theorem}
\begin{proof} We already defined $\Phi_a$ in Lemma \ref{lem:normalspec}. Since $\th\s(1_{\si_A(a)})=1_{\oom(B)}= \ggg(1_A)$, by definition, we have $\Phi_a(1_{\si_A(a)})=1_A$. Similarly, for every $\om\in \oom(B)$, we have $\th\s(id_{\si_A(a)}) (\om)=id_{\si_A(a)} ( \th (\om))= \om (a)=\hat{a}(\om)=\ggg(a)(\om)$, which means $\Phi_a(id_{\si_A(a)})=a$.

To prove the uniqueness of $\Phi_a$, we note that any other $\ast$-homomorphism with the above properties is equal to $\Phi_a$ over the complex algebra of all complex polynomials of two variables $z$ and $\overline{z}$ over $\si_A(a)$. By the Stone-Weierstrass theorem, see Theorem A.10.1 of \cite{deitmar-echterhoff}, this algebra is dense in $C(\si_A(a))$. Now, continuity of such a $\ast$-homomorphism implies the uniqueness of $\Phi_a$.
\end{proof}

\begin{exercise}
Using the continuous functional calculus show that the spectrum of every self adjoint element of a \cs-algebra lies in $\r$. See also Proposition \ref{prop:specunitary} for another proof.
\end{exercise}

\begin{remark}
The continuous functional calculus is consistent with the holomorphic functional calculus. To see this, let $a$ be a normal element of a unital \cs-algebra $A$ and let $f$ be a holomorphic function on a neighborhood $U$ containing $\si_A(a)$. Consider a smooth simple closed curve $C$ in $U$ enclosing $\si_A(a)$. For every $\la\in C$, the function $id-\la$ is a non-zero function over $C$, and so it has a continuous inverse over $C$. For every $\om\in \oom(B)$, we compute
\begin{eqnarray*}
\ggg(f(a)) (\om)&=&\widehat{f(a)} (\om)=\om(f(a))\\
&=& \om \left( \frac{1}{2\pi i} \int_C \frac{f(\la)}{\la-a} d\la\right)\\
&=& \frac{1}{2\pi i} \int_C\om\left( (\la-a)\inv \right) f(\la) d\la\\
&=& \frac{1}{2\pi i} \int_C (\la-\om(a))\inv f(\la) d\la\\
&=& f(\om(a))= f o \th(\om)\\
&=&\th\s (f) (\om)=\ggg o \Phi_a (f) (\om).
\end{eqnarray*}
In the above computation, $\th$ is the map defined in Lemma \ref{lem:normalspec}. Since $\ggg$ is an isomorphism, one observes that $\Phi_a(f)$ equals the holomorphic functional calculus of $f$.
\end{remark}

The above remark and the holomorphic spectral mapping theorem, see Theorem \ref{thm:specmap}, suggest the following proposition:

\begin{proposition}
\label{prop:specmapcont}
[The continuous spectral mapping theorem] Let $a$ be a normal element of a unital \cs-algebra $A$ and let $f$ be a continuous function on $\si_A(a)$. Then $f(a)$ is normal and $\si_A({f(a)})=f(\si_A(a))$. Moreover, if $g$ is a continuous function over $\si_A({f(a)})$, then $g(f(a))=gof(a)$.
\end{proposition}

\begin{proof}
Since $C\s(a,1)$ is a commutative \cs-algebra, all of its elements are normal. For convenient, let us denote $C(\si_A(a))$ by $C$. Then it is clear that $f\in C$ is invertible if and only if $\Phi_a(f)=f(a)$ is invertible in $C\s(a,1)$. Therefore we have
\[
\si_A(f(a))=\si_C(f)=f(\si_C(id_{\si_A(a)}))=f(\si_A(a)).
\]
Since $\Phi_a$ is a $\ast$-isomorphism, the equality $g(f(a))=gof(a)$ is true when $g$ is any complex polynomial of two variable. The general case follows from the continuity of $\Phi_{f(a)}$ and the fact that the algebra of these polynomials is dense in $C(\si(f(a))$.
\end{proof}
The following proposition illustrates some applications of the above proposition:
\begin{proposition}
\label{prop:specunitary}
Let $A$ be a unital \cs-algebra.
\begin{itemize}
\item[(i)] If $u\in A_u$, then $\si(u)\subseteq \t =\{\la\in \c; |\la|=1\}$.
\item[(ii)] If $a\in A_h$, then $\si(a)\subseteq \r$.
\end{itemize}
\end{proposition}

\begin{proof}
\begin{itemize}
\item[(i)] We know that $\|u\|^2=\|uu\s\|=\|1\|=1$, so $|\la|\leq 1$ for all $\la\in \si(u)$. On the other hand, $\|u\s\|=1$. Hence using the continuous spectral mapping theorem, we have $|\mu|\leq 1$ for all $\mu \in \si(u\s)=\si(u\inv)=(\si(u))\inv=\{\la\inv;\la \in \si(u)\}$. Thus $|\la\inv|\leq 1$ for all $\la\in \si(u)$, by . Combining these inequalities, we get $|\la|=1$ for all $\la\in \si(u)$.
\item[(ii)] If $a\in A_h$, then it is easy to see that $e^{ia}$ is a unitary element. Assume $\la=\alpha+i\beta\in \si(a)$, where $\alpha$ and $\beta$ are real numbers and $\beta\neq 0$. Then by the continuous spectral mapping theorem, $e^{i\la}\in \si(e^{ia})$, but $|e^{i\la}|=|e^{i\alpha}e^{-\beta}| = |e^{-\beta}| \neq1$. This contradicts with Part (i).
\end{itemize}
\end{proof}
\begin{proposition}
Let $\ff:A\ra B$ be a unital \ss-homomorphism between two unital \cs-algebra and let $a\in A$ be normal. Then for every $f\in C(\si_A(a))$, we have $\ff(f(a))=f(\ff(a))$.
\end{proposition}
\begin{proof}
First, we note that $\ff(a)$ is normal and $\si_B(\ff(a))\subseteq \si_A(a)$, and so the restriction of $f$ to $\si_B(a)$ is continuous. Define two unital \ss-homomorphisms $\Phi_1, \Phi_2:C(\si_a(a)) \ra B$ by
\[
\Phi_1(f):=\ff(\Phi_a(f))=\ff(f(a)),
\]
and
\[
\Phi_2(f):= \Phi_{\ff(a)}(f|_{\si_B(\ff(a))})=f|_{\si_B(\ff(a))}(\ff(a)).
\]
It is easy to check that they both map $1_{\si_A(a)}$ and $id_{\si_A(a)}$ to $1_B$ and $\ff(a)$, respectively. Thus they agree on all polynomials of two variables $z$ and $\overline{z}$ over $\si_A(a)$ and since they are continuous, they agree on all of $C(\si_A(a))$.
\end{proof}
\begin{proposition}
Let $A$ be a \cs-algebra and let $K$ be compact subset of $\c$. Let $A_K$ denote the set of all normal elements of $A$ whose spectrum is contained in $K$. If $f$ is a continuous function on $K$, then the mapping $A_K\ra A$ defined by $a\mapsto f(a)$ is continuous.
\end{proposition}
\begin{proof}
For given $f\in C(K)$ and for every $\ep>0$, by Stone-Weierstrass theorem, there is some polynomial $P$ of two variables $z$ and $\overline{z}$ such that $\sup_{\la\in K} |P(\la,\overline{\la}) - f(\la)| <\ep$. For every $a\in A_K$, we use the continuous functional calculus of $a$ to conclude that $\|P(a,a\s)-f(a)\|<\ep$.  Set $M:=\sup\{|\la|; \la\in K\}$. Then for every $a,b\in A_K$, we have $\|a\|\leq M$. Using this, one can find $\d >0$ such that $\|P(a,a\s)-P(b,b\s)\|< \ep$ for every $a,b\in A_K$ provided that $\|a-b\|\leq \d$. Hence we have
\[
\|f(a)-f(b)\| \leq \|f(a)- P(a,a\s)\| + \|P(a,a\s)-P(b,b\s)\| +\|P(b,b\s)-f(b)\|< 3\ep.
\]
\end{proof}

\begin{proposition}
\label{prop:normdecreasing}
Let $\ff: A\ra B$ be an algebraic \ss-homomorphism from an involutive Banach algebra $A$ into a \cs-algebra $B$. Then it is norm decreasing, i.e. $\|\ff(a)\|\leq \|a\|$ for all $a\in A$, and therefore $\ff$ is continuous.
\end{proposition}
\begin{proof}
If $A$ is not unital, we can extend $\ff$ to a unital algebraic \ss-homomorphism from $A_1$ into $\widetilde{B}$. Therefore we can assume $A$ is unital. If $a-\la$ is invertible in $A$, then $\ff(a)-\la$ is invertible in $B$, and therefore
\[
\si_B(\ff(a))\subseteq \si_A(a),\qquad  \forall a\in A.
\]
This implies that $r_B(\ff(a))\leq r_A(a)$ for all $a\in A$. Then we have
\begin{eqnarray*}
\|\ff(a)\|^2 &=& \|\ff(a)\s \ff(a)\|= \|\ff(a\s a)\|\\
&=& r_B(\ff(a\s a))\leq r_A(a\s a)\\
&\leq& \|a\s a\|\leq \|a\s\| \|a\|=\|a\|^2.
\end{eqnarray*}
\end{proof}
\begin{proposition}
\label{prop:normincreasing}
Let $\ff: A\ra B$ be a one-to-one algebraic \ss-homomorphism from a \cs-algebra $A$ into an involutive Banach algebra $B$. Then it is norm increasing, i.e. $\|\ff(a)\|\geq \|a\|$ for all $a\in A$.
\end{proposition}
\begin{proof}
Consider a self adjoint element $a\in A$ and set $b:=\ff(a)$. Then $b$ is self adjoint too. Let $E$ denote the \cs-unitization of the \cs-subalgebra generated by $a$ in $A$ and let $F$ denote the unitization of the involutive Banach subalgebra generated by $b$ in $B$. Let $\phi:E\ra F$ denote the unital \ss-homomorphism mapping $a$ into $b$. These algebras are both commutative and their spectrums are compact. We define $\phi\s: \oom(F)\ra \oom(E)$ by $\om\mapsto \om  \phi$ for all $\om\in \oom(F)$.
We claim that $\phi\s$ is onto. If not there exist $\om_0\in \oom(E)- \phi\s(\oom(F))$. Since $\phi\s(\oom(F))$ is compact, there exist two not identically zero functions $f, g\in C(\oom(E))$ such that $fg=0$, $f(\om_0)=1$ and $g=1$ over $\phi\s(\oom(F))$. By using the inverse of the Gelfand transform, we obtain two non-zero elements $c,d\in E$ such that $cd=0$, $\om(\phi(d))=1$ for all $\om\in \oom(F)$. This implies that $\phi(d)$ does not belong to any maximal ideal of $F$, so it has to be invertible. But this contradicts with the facts that $\phi(c)\phi(d)=\phi(cd)=0$ and $\phi(c)\neq 0$. Therefore $\phi$ is onto. Now, we compute
\begin{eqnarray*}
\|\ff(a)\|&=&  \|b\|  \geq  r(b)\\
&=&\sup \{ |\widehat{b}(\om)|; \om\in\oom(F)\}\\
&=&\sup \{ |\widehat{\phi(a)} (\om)|; \om\in\oom(F)\}\\
&=&\sup \{ |\om(\phi(a))|; \om\in\oom(F)\}\\
&=&\sup \{ |\phi\s\om(a)|; \om\in\oom(F)\}\\
&=&\sup \{ |\om(a)|; \om\in\oom(E)\}\\
&=& r(a)=\|a\|.
\end{eqnarray*}
For arbitrary $a\in A$, we have
\[
\|a\|^2=\|a\s a\| \leq \|\phi(a)\s \phi(a)\| \leq \|\phi(a)\|^2.
\]
\end{proof}

\begin{corollary}
\label{cor:csinjection}
Every injective \ss-homomorphism between two \cs-algebras is an isometry.
\end{corollary}
\begin{exercise}
\label{exe:compactificationunit}
\begin{itemize} Let $X$ be a locally compact and Hausdorff topological space.
\item [(i)] When $X$ is not compact, show that the map $\iota : \widetilde{C_0(X)} \ra C(X^\infty)$ defined in Proposition \ref{prop:compactificationunit} is a \ss-isomorphism. Thus by the above corollary, it is an isometric isomorphism between these \cs-algebras.
\item [(ii)] When $X$ is compact, find an isometric isomorphism between $\widetilde{C(X)}$ and $C(X^\infty)$.
\end{itemize}
\end{exercise}
There is a non-unital version of the continuous functional calculus that appears useful for dealing with non-unital \cs-algebras.

\begin{remark}
\label{rem:nonunitalcfc}
Let $A$ be \cs-algebra and $a\in A$ be a normal element. If it is necessary, we add a unit to $A$. First, assume that $0$ belongs to the spectrum of $a$. Let $C$ be the \cs-subalgebra of $C(\si(a))$ defined as follows:
\[
C:=\{f\in C(\si(a)); f(0)=0 \}.
\]
Consider $\Phi_a:C(\si(a)) \ra \tilde{A}$, the continuous functional calculus over $a$. Then the image of the restriction of $\Phi_a$ to $C$ is exactly the \cs-algebra $C\s(a)$ generated by $a$ and one can directly write $\Phi_a:C \ra C\s(a)\subseteq A$ without any reference to the unit element of $A$ or $\tilde{A}$. This is called the {non-unital continuous functional calculus over $a$}. One notes that if $A$ is non-unital, then we always have $0\in \si(a)$. Now, assume $A$ is unital and $0\notin \si(a)$. Then $C=C(\si(a))$ and also $C\s(a,1)=C\s(a)$. In other words, the non-unital continuous functional calculus is the same as the original continuous functional calculus in this case.
\end{remark}

We conclude this section with some applications of the continuous functional calculus.

\begin{proposition}
\label{prop:nroots}
Let $a$ be a self adjoint element of a \cs-algebra $A$.
\begin{itemize}
\item [(i)] For given an odd natural number $n$, there is a unique self adjoint element $b\in A$ such that  $b^n=a$.
\item [(ii)] Let also $a$ is positive. Then for given an even natural number $n$, there is a unique positive element $b\in A$ such that  $b^n=a$.
\end{itemize}
In both cases, $b$ is called the {$n$th root of $a$} and is denoted by $a^{1/n}$ or $\sqrt[n]{a}$.
\end{proposition}

\begin{proof}
\begin{itemize}
\item [(i)] Since the function $f(t)=t^{1/n}$ is continuous over the real line and so over $\si(a)$, we can define $b:=f(a)$. Then $f(t)^n=id$ implies that $b^n=a$. Assume $c\in A_h$ satisfies the same equality. Then $ca=c(c^n)=(c^n)c=ac$, namely $c$ and $a$ commute. Since $b$ is a limit of a sequence of polynomials in $a$, $b$ commutes with $c$ as well. Therefore the \cs-algebra $C\s(b,c)$ generated by $b$ and $c$ is commutative and contains $a$. By the Gelfand transform, we arrive to two equations of real valued functions; $\hat{a}=\hat{b}^n$ and $\hat{a}=\hat{c}^n$. These equations imply $\hat{b}=\hat{c}$, and so $b=c$.
\item [(ii)] The proof is similar to the Item (i), except one should note that when $n$ is even, the function $f(t)=t^{1/n}$ is defined and is continuous only over $[0,\infty)$. Therefore we had to restrict this case to positive elements of $A$.
\end{itemize}
\end{proof}
The following exercises are among many problems that can be easily solved by the continuous functional calculus.
\begin{exercise}
\label{ex:normalabsolute}
Let $a\in A$ be a normal element of a \cs-algebra $A$. Show that $a\s a \geq 0$. Define the {\bf absolute value of $a$} by $|a|:=(a\s a )^{1/2}$. Then prove the following statements:
\begin{itemize}
\item [(i)] $a$ is positive if and only if $a=|a|$,
\item [(ii)] $\|a\|=\| |a|\|$,
\item [(iii)] $a$ is invertible if and only if $|a|$ is invertible. In this case, $a|a|\inv$ is a unitary element in $A$.
\end{itemize}
\end{exercise}
In Proposition \ref{prop:positivecone}, we shall show that the statement $a\s a\geq 0$ is true for all elements of a \cs-algebra. Therefore one can extend the definition of the absolute value to all elements of a \cs-algebra.
\begin{exercise}
Let $a$ be a positive element of a unital \cs-algebra $A$. Show that $a\leq \|a\|1$.
\end{exercise}
\begin{proposition}
\label{prop:4unitary}
Every element of a unital \cs-algebra $A$ can be written as a linear combination of four unitary elements of $A$.
\end{proposition}
\begin{proof}
Let $b$ be a self adjoint element of $A$ such that $\|b\|\leq 1$. Then $b^2$ is positive and $\|b^2\|\leq 1$. These imply that $1-b^2$ is positive too. Therefore we can define $u:=b+i(1-b^2)^{1/2}$. It is easy to see that $u$ is a unitary element and $b=\frac{u+u\s}{2}$. This means that every self adjoint element of $A$ is a linear combination of two unitary elements.  Now, the desired statement follows from the fact that every element of a \cs-algebra is a linear combination of two self adjoint elements.
\end{proof}
More applications of the continuous functional calculus will be also discussed in Chapter \ref{ch:basics}.


\section{The Gelfand duality}
\label{sec:Gelfandduality}

The Gelfand duality along with quantum physics and the general developments in index theory is one of the main motivations of noncommutative geometry. Therefore we use this theory as our guide in studying those parts of the theory of \cs-algebras which are necessary to understand noncommutative geometry. We begin this section with briefly recalling basic definitions of category theory. Our treatment of category theory, here, is rather informal and we content ourselves to the minimum amount of the theory that is going to be used in this book. The interested reader is referred to \cite{rotman, weibel} for further details. Afterwards, we explain how the Gelfand transform justifies the idea of considering \cs-algebras as the noncommutative analogues of point-set topological spaces. This analogy is based on the duality between locally compact and Hausdorff topological spaces and their \cs-algebras of continuous functions vanishing at infinity. This duality is called the Gelfand duality. Besides motivating some of the developments of noncommutative geometry, the Gelfand duality has found many generalizations to non-commutative \cs-algebras. These generalized correspondences between topological spaces and \cs-algebras are widely applied to facilitate and explain many interactions of the theory of \cs-algebras with other mathematical areas such as representations theory as well as with quantum physics. This duality also explains how $K$-theory of \cs-algebras generalizes topological $K$-theory.

In a {\bf category} $\mathcal{C}$, we have a class of objects, which we denote it by $obj(\mathcal{C})$. Then for every ordered pair of objects $(A,B)\in obj(\mathcal{C})^2$, there is a set of {\bf morphisms}, which we denote it by $Hom_{\mathcal{C}}(A,B)$. Moreover, for every triple $(A,B,C)\in obj(\mathcal{C})^3$, we have a {\bf composition law} $Hom_{\mathcal{C}}(A,B)\times Hom_{\mathcal{C}}(B,C) \ra Hom_{\mathcal{C}}(A,C)$ which we denote it by  $(f,g)\mapsto gf$. These ingredients are subject to the following axioms:

\begin{itemize}
\item [(i)] The sets of morphisms are pairwise disjoint.
\item [(ii)] For every object $A$, there is a unique morphism $1_A\in Hom_{\mathcal{C}}(A,A)$ such that $f1_A=f=1_B f$ for all $f\in Hom_{\mathcal{C}}(A,B)$.
\item [(iii)] The composition law is associative, namely, for given $A,B, C\in obj(\mathcal{C})$ and for all $f\in Hom_{\mathcal{C}}(A,B)$, $g\in Hom_{\mathcal{C}}(B,C)$ and $h\in Hom_{\mathcal{C}}(C,D)$, we have $h(gf)=(hg)f$.
\end{itemize}
A morphism $f\in Hom_{\mathcal{C}} (A,B)$ is called an {\bf isomorphism} if there exists a morphism $g\in Hom_{\mathcal{C}} (B,A)$ such that $fg=1_B$ and $gf=1_A$. In this case, one easily checks that $g$ is unique and is called the {\bf inverse} of $f$ and usually is denoted by $f\inv$. In the following, we introduce some of the categories that we are going to deal with in this book.
\begin{example}
\begin{itemize}
\item [(i)] Let $\mathcal{S}$ be a category whose objects are sets and, for every ordered pair $(A,B)$ of sets , the set $Hom_{\mathcal{S}}(A,B)$ of morphisms is the set of all functions from $A$ into $B$. This category is called the {\bf category of sets}. The composition law is the composition of functions.
\item [(ii)] The objects of the {\bf category of commutative \cs-algebras} are all commutative \cs-algebras and, for every ordered pair  $(A,B)$ of commutative \cs-algebras the set of morphism is all proper \ss-homomorphisms from $A$ into $B$. The composition law is the composition of two \ss-homomorphisms. We denote this category by $\mathcal{CCA}$.
\item [(iii)] The {\bf category of locally compact and Hausdorff topological spaces} is defined similarly. The objects of this category are locally compact Hausdorff topological spaces. For two objects $X$ and $Y$ in this category, the set of morphisms from $X$ to $Y$ consists of all proper continuous maps from $X$ into $Y$. This category is denoted by $\mathcal{LCS}$.
\item [(iv)] In Part (ii), if we only consider unital \cs-algebras and unital \ss-homomorphisms, then we obtain a new category called the {\bf category of unital commutative \cs-algebras} and is denoted by $\mathcal{UCA}$. One notes that every proper \ss-homomorphism between two unital \cs-algebra is automatically unital.
\item [(v)] Similarly, in Part (iii), if we consider only compact and Hausdorff topological spaces, then the obtained category is called the {\bf category of compact Hausdorff topological spaces} and is denoted by $\mathcal{CS}$. One notes that every continuous map from a compact topological space is automatically proper.
\item[(vi)] The objects of the {\bf category $\mathcal{AG}$ of abelian groups} are all abelian groups. The set of morphisms between two abelian group are all group homomorphisms between them and the composition law is the composition of homomorphisms. The {\bf category of groups} is defined similarly and is denoted by $\mathcal{GR}$.
\item[(vii)] Let $X$ be a set and let $\leq$ be a partial order on $X$. The set $X$ admits the structure of a category. Its objects are are elements of $X$. For $x,y\in X$, the set of morphisms $Hom(x,y)$ has only one element $\iota_y^x$ if $x\leq y$ and it is empty when $x\nleq y$. The composition law is defined using the transitivity of the relation $\leq$, namely if $x\leq y$ and $y\leq z$, then $\iota^x_z=\iota^y_z \iota^x_y$.
\item[(viii)] Let $G$ be a group. It enjoys the structure of a category too. Here, we have only one object which is usually denoted by $\ast$ and the morphisms from this object to itself are all elements of the group. The composition law is the multiplication of the group.
\end{itemize}
\end{example}

Those categories whose objects are some sets with (or without) some structures and the morphisms are functions preserving those structures and composition law is the composition of functions underlying the morphisms are called {\bf concrete}. For instance, Items (i) to (vi) of the above examples are concrete categories. A category whose class of objects is actually a set is called a {small category}. Items (vii) and (viii) are examples of a small category. In Item (viii), one notes that every morphism is an isomorphism. This motivates another definition for groupoids. A {\bf groupoid} is a small category all whose  morphisms are isomorphism. 

A {\bf subcategory} $\mathcal{S}$ of a category $\mathcal{C}$ is a category such that $obj(\mathcal{S})\subseteq obj(\mathcal{C})$, for every pair $A,B\in obj(\mathcal{S})$, we have $Hom_{\mathcal{S}}(A,B)\subseteq Hom_{\mathcal{C}}(A,B)$, the composition law in $\mathcal{S}$ coincides with the composition law in $\mathcal{C}$, and finally, for every $A\in obj(\mathcal{S})$, the identity morphism $1_A \in Hom_{\mathcal{S}}(A,A)$ is equal to the identity morphism $1_A \in Hom_{\mathcal{C}}(A,A)$. The category $\mathcal{S}$ is called a {\bf full subcategory} of $\mathcal{S}$ if $Hom_{\mathcal{S}}(A,B)= Hom_{\mathcal{C}}(A,B)$ for all $A,B\in obj(\mathcal{S})$.

\begin{exercise}
\label{exe:category}
\begin{itemize}
\item[(i)] Show that the category of abelian groups is a full subcategory of the category of groups.
\item[(ii)] Show that the category of compact topological spaces is a full subcategory of the category of locally compact topological space.
\item[(iii)] Assume $H$ be subgroup of a group $G$. Show that $H$ is a subcategory of $G$ and it is full if and only if $G=H$.
\item[(iv)] Let $n$ be an integer. For $1\leq i,j\leq n$, let $E_{ij}$ be the $(i,j)$-th elementary $n\times n$ matrix, namely the matrix whose $(i,j)$-th entry is one and the rest of its entries are zero. Consider a small category whose objects are the elements of the standard basis of $\r^n$, which we denote it by $\{e_1,\cdots, e_n\}$. For every pair $(e_i, e_j)$, set $Hom(e_i, e_j):=\{E_{i,j} \}$ and define the composition law by the multiplication of matrices. Show that this a groupoid. Describe some full (and non-full) subcategories of this category.
\end{itemize}
\end{exercise}

Similar to other mathematical structures, there are certain maps between categories named functors which preserve the structure of categories. A {\bf covariant functor} $F$ from a category $\mathcal{C}$ into a category $\mathcal{D}$ associates an object $F(A)\in obj(\mathcal{D})$ to every object $A\in obj(\mathcal{C})$. Moreover, $F$ maps every $f\in Hom_{\mathcal{C}}(A,B)$ to some $F(f) \in Hom_{\mathcal{D}}(F(A),F(B))$. A covariant functor $F:\mathcal{C}\ra \mathcal{D}$ also satisfies the following conditions:
\begin{itemize}
\item[(i)] For every $f\in Hom_{\mathcal{C}}(A,B)$  and $g\in Hom_{\mathcal{C}}(B,C)$, we have $F(gf)=F(g) F(f)$.
\item[(ii)] For every $A\in obj(\mathcal{C})$, we have $F(1_A)=1_{F(A)}$.
\end{itemize}

A {\bf contravariant} functor is defined similarly, except it reverses the arrows. In other words, if $F$ is a contravariant functor from a category $\mathcal{C}$ into a category $\mathcal{D}$, then  $F(f) \in Hom_{\mathcal{D}}(F(B),F(A))$ for every $f\in Hom_{\mathcal{C}}(A,B)$ and, for  $f$ and $g$ as above, we have $F(gf)=F(f)F(g)$. Sometimes a covariant functor is called simply a {\bf functor} while a contravariant functor is called a {\bf cofunctor}.

\begin{example}
\begin{itemize}
\item [(i)] The most obvious example of a covariant functor is the {\bf forgetful functor} from a concrete category into the category of sets or into another category with less structures. For instance, consider the category of compact topological spaces. The forgetful functor assigns the underlying set of a compact topological space to it and sends every continuous map to itself as a map between two sets without any structure.
\item [(ii)] Let $G$ be a finite abelian group and let $\t$ denote the unit circle in $\c$ considered as the subgroup of the multiplicative group of $\c$. By definition, the {\bf Pontryagin dual of $G$} is the group $\widehat{G}$ of all {\bf characters} of $G$, which is the set of all group homomorphism from $G$ into $\t$. The multiplication in $\widehat{G}$ is defined by
    \[
    (\rho \tau)(g):=\rho(g)\tau(g),\quad \forall \rho, \tau\in \widehat{G}, \, g\in G.
    \]
    Let $\mathcal{FAB}$ denote the category of abelian finite groups. Then we define a contravariant functor $P:\mathcal{FAB}\ra\mathcal{FAB}$ as follows:
    \begin{itemize}
\item [$\bullet$] $P(G):=\widehat{G}$, for all $G\in obj(\mathcal{FAB})$.
\item [$\bullet$] For two given finite abelian groups $G$ and $H$, $P(\ff):=\ff\s$ for all group homomorphism $\ff:G\ra H$, where $\ff\s:\widehat{H}\ra \widehat{G}$ is defined by $\ff\s(\rho):=\rho \ff$ for all $\rho\in \widehat{H}$.
    \end{itemize}
    One easily checks that this is a contravariant functor. This functor is defined on the bigger category of all locally compact abelian groups and it is called the {\bf Pontryagin duality}. Using the fact that every finite abelian group is a direct sum of cyclic finite groups, one easily observes that $G$ and $\widehat{G}$ are isomorphic, and so we have $G\simeq \widehat{\widehat{G}}$. This latter isomorphism still holds for the Pontryagin duality for locally compact groups.
\end{itemize}
\end{example}

Now, we explain the Gelfand duality as a contravariant functor.

\begin{definition}
The {\bf Gelfand functor} $D$ is a contravariant functor from the category of commutative \cs-algebras into the category of locally compact topological spaces. It sends every commutative \cs-algebra $A$ to its spectrum $\oom(A)$ and sends every proper \ss-homomorphism $\ff:A\ra B$ between two commutative \cs-algebras to the proper continuous map  $\ff\s:\oom(B)\ra \oom(A)$ defined by $\ff\s(\om):=\om\ff$ for all $\om\in \oom(B)$.
\end{definition}

Using the details we have already presented in Sections \ref{sec:Gelfandtrans}, it is easy to check that $D$ is a contravariant functor. As it was already explained, the Gelfand functor is extremely useful. For instance, it represents an abstract commutative \cs-algebra as a \cs-algebra of continuous functions equipped with a specific, and more importantly, computable \cs-norm. However, the second half of the Gelfand duality, which consist of the inverse of the Gelfand functor, completes the theory. In order to explain the meaning of an equivalence of two category and the inverse of a functor, we need some more definitions.

Let $F,G:\mathcal{C}\ra \mathcal{D}$ be two covariant functors. A {\bf natural transformation} $\eta:F\Rightarrow G$ is a one parameter assignment of morphisms $\eta=\left\{\eta_A:F(A)\ra G(A)\right\}_{A\in obj(\mathcal{C})}$ in $\mathcal{D}$ such that the following diagram is commutative for every morphism $f:A\ra B$ in $\mathcal{C}$:
\[
\xymatrix{
F(A)\ar[d]_{F(f)} \ar[r]^{\eta_A} &G(A)\ar[d]^{G(f)}   \\
F(B)\ar[r]_{\eta_B}& G(B)
}
\]

If $\eta_A$ is an isomorphism for every $A\in obj(\mathcal{C})$, then $\eta$ is called a {\bf natural isomorphism}. A natural transformation (resp. isomorphism) between two contravariant functors is defined similarly. We denote the identity functor from a category $\mathcal{C}$ into itself by $1_{\mathcal{C}}$. It is clearly a covariant functor. A covariant (resp. contravariant) functor $F:\mathcal{C}\ra \mathcal{D}$ between two categories is called an {\bf equivalence between $\mathcal{C}$ and $\mathcal{D}$} if there is a covariant (resp. contravariant) functor  $G:\mathcal{D}\ra \mathcal{C}$ such that there exist natural isomorphisms between $GF$ and $1_{\mathcal{C}}$ and between $FG$ and $1_{\mathcal{D}}$. In this case, we call two categories $\mathcal{C}$ and $\mathcal{D}$ {\bf equivalent}. The functor $G$ plays the role of an inverse in this definition, so it is called the {\bf inverse of $F$}. Our aim is to show that the Gelfand functor is an equivalence between the category of commutative \cs-algebras and the category of locally compact Hausdorff spaces. Therefore we need an inverse for the Gelfand functor. This inverse is nothing but the functor which sends every locally compact and Hausdorff space $X$ to its \cs-algebras $C_0(X)$ of continuous functions vanishing at infinity. It also sends every proper continuous map $f:X\ra Y$ to the proper \ss-homomorphism $f\s:C_0(Y)\ra C_0(X)$, where $f\s(g)(x):=g(f(x))$ for all $g\in C_0(Y)$, see Proposition \ref{prop:contmapspec}. We call this functor the {\bf inverse of the Gelfand functor} and denote it by $E$.

\begin{theorem}
\label{thm:Gelduality} The Gelfand transform $D$ is an equivalence of categories with the inverse $E$. Moreover, by restricting the Gelfand transform to the category of unital commutative \cs-algebras, we obtain an equivalence between this category and the category of compact Hausdorff topological spaces.
\end{theorem}
\begin{proof}
The natural isomorphism $1_{\mathcal{CCA}}\Rightarrow ED$ is the Gelfand transform. To see this, we only need to check that the following diagram commutes for every proper \ss-homomorphism $\ff:A\ra B$ between two commutative \cs-algebras:
\[
\xymatrix{
A\ar[d]_{\ff} \ar[rr]^{\ggg_A} &&C_0(\oom(A))\ar[d]^{\tilde{\ff}}   \\
B\ar[rr]_{\ggg_B}&& C_0(\oom(B))
}
\]
where $\tilde{\ff}=ED(\ff)$. By definition of $E$ and $D$, for all $\om\in\oom(B)$ and $a\in A$, we have
\begin{eqnarray*}
\tilde{\ff} (\ggg_A(a)) (\om)&=& \tilde{\ff}(\hat{a})(\om)=[ED(\ff)(\hat{a})] (\om)\\
&=& [E(D(\ff))(\hat{a})](\om)= \hat{a}(D(\ff)(\om))\\
&=&[D(\ff)(\om)](a)=\om(\ff(a))\\
&=&\widehat{\ff(a)} (\om)=\ggg_B(\ff(a)) (\om).
\end{eqnarray*}
The natural isomorphism $1_{\mathcal{LCS}}\Rightarrow DE$ at a space $X\in \mathcal{LCS}$ is the homeomorphism $\fff_X:X\ra \oom(C_0(X))$ defined by $x\mapsto \hat{x}$, where $\hat{x} (f)=f(x)$ for all $f\in C_0(X)$. Again, we need to check that the following diagram commutes for every proper continuous map $\psi:X\ra Y$ between two locally compact Hausdorff spaces $X$ and $Y$:
\[
\xymatrix{
X\ar[d]_{\psi} \ar[rr]^{\fff_X} &&\oom(C_0(X))\ar[d]^{\tilde{\psi}}   \\
Y\ar[rr]_{\fff_Y}&&\oom(C_0(Y))
}
\]
where $\tilde{\psi}=DE(\psi)$. For all $f\in C_0(Y)$ and $x\in X$, we have
\begin{eqnarray*}
\tilde{\psi}(\fff(x))(f)&=& \tilde{\psi} (\hat{x})(f)=[D(E(\psi))(\hat{x})](f)\\
&=& \hat{x}(E(\psi)(f))=\hat{x}(f\psi)=f\psi(x)\\
&=&\widehat{\psi(x)} (f)=\fff_Y(\psi (x)) (f).
\end{eqnarray*}
\end{proof}
Immediate corollaries of this theorem are (i) every proper \ss-homomorphism $\ff:A\ra B$ between two commutative \cs-algebras is induced by a proper continuous map from $\oom(B)$ into $\oom(A)$, and similarly (ii) every proper continuous map $f:X\ra Y$ between two locally compact and Hausdorff topological spaces is induced by a proper \ss-homomorphism from $C_0(Y)$ into $C_0(X)$. A slightly different formulation of this fact is given in the following corollary:
\begin{corollary}
The Gelfand functor and its inverse are bijective maps on the sets of morphisms. In other words,
\begin{itemize}
\item [(i)] for every two commutative \cs-algebras $A$ and $B$, the following map is bijective:
\[
D: Hom_{\mathcal{CCA}}(A,B)\ra Hom_{\mathcal{LCS}} (\oom(B), \oom(A)),
\]
\item [(ii)] and, for every two locally compact and Hausdorff topological spaces $X$ and $Y$, the following map is bijective:
\[
E: Hom_{\mathcal{LCS}}(X,Y)\ra Hom_{\mathcal{CCA}}(C_0(Y),C_0(X)).
\]
\end{itemize}
\end{corollary}
\begin{proof}
The natural isomorphism $1_{\mathcal{CCA}}\Rightarrow ED$ implies that $D$ is one-to-one and $E$ is onto. Similarly, the natural isomorphism $1_{\mathcal{LCS}}\Rightarrow DE$ shows that $E$ is one-to-one and $D$ is onto.
\end{proof}
\section{Problems}

\begin{e}
\label{e:3-1}
Define a new involution on the Banach algebra $\ell^1(\z)$ by the following formula:
\[
f\s(n):=\overline{f(n)}, \qquad \forall f\in \ell^1(\z), \, n\in \z.
\]
Show that this is an involution. Show that $\ell^1(\z)$ with this involution is not symmetric.
\end{e}
\begin{e}
\label{e:3-2}
Let $a$ be an element of a \cs-algebra $A$. Show that $\si(a)=\{0\}$ if and only if $a=0$. Use this to verify Remark \ref{rem:radical}.
\end{e}
\begin{e}
\label{e:3-3}
Let $a$ be a normal element of a \cs-algebra $A$. Show that if $\si(a)\subseteq \r$, then $a$ is self adjoint.
\end{e}
\begin{e}
\label{e:3-4}
Let $a$ be a normal element of a \cs-algebra $A$. Show that if $\si(a)\subseteq \t$, then $a$ is unitary.
\end{e}
\begin{e}
\label{e:3-5}
Let $a$ be a normal element of a \cs-algebra $A$. Show that $\si(a)\subseteq \{0,1\}$ if and only if $a$ is a projection. Conclude that every projection is a positive element.
\end{e}
\begin{e}
\label{e:3-13}
Let $a$ be a self adjoint element of a \cs-algebra $A$.
\begin{itemize}
\item [(i)] Show that if $a^3=a^2$, then $a$ is a projection.
\item [(ii)] Assume $a$ is positive. Show that if $a^n=a^m$ for some integers $0<m<n$, then $a$ is a projection.
\end{itemize}
\end{e}
\begin{e}
\label{e:3-6}
Show that if $a$ is a positive element of a \cs-algebra $A$ then $a=bb\s$ for some $b\in A$. The converse of this statement is also true, which will be proved in Proposition \ref{prop:positivecone}(ii).
\end{e}
\begin{e}
\label{e:3-7}
Let $A$ be a unital Banach algebra. Assume $x,y\in A$ and $xy=yx$. Show that
\[
r(xy)\leq r(x)r(y), \qquad r(x+y)\leq r(x)+r(y).
\]
\end{e}
\begin{e}
\label{e:3-8}
Assume $A$ is a Banach algebra. Show that if $e\neq f$ are two idempotents in $A$ that commute with each other, then $\|e-f\|\geq 1$. (Hint: Use the previous exercise.)
\end{e}
\begin{e}
\label{e:3-9}
Let $A$ be a commutative \cs-algebra. Assume $L:A\ra \c$ be a linear functional such that $L(a\s a)\geq 0$ for all $a\in A$. Prove that $L$ is bounded.
\end{e}
\begin{e}
\label{e:3-10}
Let $A$ and $B$ be two commutative \cs-algebras and let $\ff:A\ra B$ be a proper linear map such that $\ff(ab)=\ff(a)\ff(b)$ for all $a,b \in A$. Show that $\ff$ is a \ss-homomorphism. (Hint: consider $\ff\s:\oom(B)\ra \oom(A)$ defined by $\ff\s(\om):=\om \ff$.)
\end{e}
\begin{e}
\label{e:3-11}
Find an approximate unit for $C_0(\r)$.
\end{e}
\begin{e}
\label{e:3-12}
Let $(u_\la)_{\la\in \Lambda}$ be an approximate unit for $C_0(\r)$. Show that there exists a countable subnet $(u_{\la_n})_{n\in \n}$ such that it is an approximate unit for $C_0(\r)$ too.
\end{e}
\begin{e}
\label{e:3-14}
Let $A$ be a \cs-algebra, $a,b\in A_+$ and $c,d\in A$. Show that if $a^n=cb^nd$ for all $n\in\n$, then $a^{1/2}=c b^{1/2} d$.
\end{e}
\begin{e}
\label{e:3-15}
Let $A$ be a commutative unital Banach algebra. Show that the Gelfand transform $\ggg: A\ra C_0(\oom(A))$ is an isometry if and only if $\|a\|^2=\|a^2\|$ for all $a\in A$.
\end{e}
\begin{e}
\label{e:3-16}
Let $u$ be a unitary element of a unital \cs-algebra $A$.
\begin{itemize}
\item [(i)] Show that if $\|1-u\|<2$, then $\si(u)\neq \t$.
\item [(ii)] Show that if $\si(u)\neq \t$, then there exists a self adjoint element $a\in A$ such that $u=e^{ia}$.
\end{itemize}
\end{e}




\chapter{Basics of the theory of $C^\ast$-algebras}
\label{ch:basics}
Positive elements of a \cs-algebra and their properties are the special feature of the theory of abstract \cs-algebras among other topological algebras. This notion plays a key role in the realization of abstract \cs-algebras as subalgebras of algebras of bounded operators on Hilbert spaces. Positivity also facilitate many applications of the theory of \cs-algebras in quantum physics. Therefore we study positivity in \cs-algebras as the first step towards the abstract theory of \cs-algebras in Section \ref{sec:positivity}. The continuous functional calculus allows us to imitate the decomposition of every complex function to a linear combination of four non-negative real functions and write every element of a \cs-algebra as a linear combination of four positive elements. One will notice the application of this easy trick in many proofs in up coming topics.

Approximate units in \cs-algebras play a indispensable role in the theory of \cs-algebras too. We briefly discuss this notion in Section \ref{sec:approximateunit}.

Section \ref{sec:ideals}, is devoted to the basic results about ideals of \cs-algebras.  Some of the unique features of \cs-algebras among other topological algebras appear in their ideal structure. For instance, every closed two sided ideal of a \cs-algebra is automatically an involutive subalgebra, see Proposition \ref{prop:closedidealcs}, or the image of every \ss-homomorphism between two \cs-algebra is always a \cs-algebra, see Corollary \ref{cor:firstiso}. Afterwards, we study the close relationship between hereditary \cs-subalgebras of \cs-algebras and the ideal structure of \cs-algebras. For instance, we study a bijective correspondence between the family of all closed left ideals of a \cs-algebra and the family of all its hereditary \cs-subalgebras, see Theorem \ref{thm:herecorres}. It is also shown that every hereditary \cs-subalgebra of a simple \cs-algebra is simple too, see Proposition \ref{prop:heredsimple}. Multiplier algebra of a \cs-algebra is introduced and studied also in this section.

In this chapter, we refer to the continuous functional calculus briefly by CFC.

\section{Positivity}
\label{sec:positivity}

In this section $A$ is always a \cs-algebra. Recall that a self adjoint element $a\in A$ is called positive if $\si(a)\subseteq [0,\infty)$. We denote this by $a\geq 0$ (or equivalently $0\leq a$). This gives rise to an order relation between elements of $A$ by defining $a\leq b$ if $b-a\geq 0$. Although it is well defined among all elements of $A$, we usually use this order to compare self adjoint elements of $A$. An important feature of this partial order that follows from Corollary \ref{cor:samespec} is that the relation $a\leq b$ between two self adjoint elements $a,b$ is independent of the \cs-subalgebra containing $a, b$. The set of all positive elements of $A$ is denoted by $A_+$. This section is devoted to this set and various properties of positive elements of a \cs-algebra that will be useful in the rest of this book.
\begin{exercise}
Let $a$ be a self adjoint element of a unital \cs-algebra $A$. Then prove the following statements:
\begin{itemize}
       \item [(i)] The relation $\leq$ is a partial order in $A_h$, namely it is reflexive, anti-symmetric and transitive.
       \item [(ii)] $a\leq |a|$ and $-a\leq |a|$,
       \item [(iii)] $|a|\leq \|a\| 1_A$, and so $a\leq \|a\| 1_A$
\end{itemize}
\end{exercise}

\begin{example}
Let $X$ be a locally compact and Hausdorff space. A function $f\in C_0(X)$ is a positive element of the \cs-algebra $C_0(X)$ if and only if $f(x)\geq 0$ for all $x\in X$, namely it is a positive (non-negative) function on $X$. We can rephrase this by saying that an element $a$ in a commutative \cs-algebra $A$ is positive if and only $\om(a)\geq 0$ for all $\om\in \oom(A)$.
\end{example}
Continuing the idea discussed in the above example, we note that every real valued function $f$ can be written as the difference of two positive functions $f_+:=max \{f, 0\}$ and $f_-=max\{-f, 0\}$, i.e. $f=f_+ - f_-$. The continuous functional calculus allows us to use this phenomenon to find a similar decomposition for self adjoint elements of a \cs-algebra.
\begin{proposition}
\label{prop:twopositive}
Let $a$ be a self adjoint element of a \cs-algebra $A$. Then there are two unique positive elements $a^+$ and $a^-$ in $A$ with the property that $a=a_+ - a_-$ and $a_+ a_-=a_- a_+=0$. Moreover, we have $|a|=a_+ + a_-$.
\end{proposition}
\begin{proof}
Define $a_+:=\frac{|a|+a}{2}$ and $a_-:=\frac{|a|-a}{2}$. Then they are positive and one easily checks that they satisfy the above equalities. To prove the uniqueness of this decomposition, assume $(x_+, x_-)$ is another ordered pair of positive elements of $A$ with the properties that $a=x_+ - x_-$ and $x_+ x_-=x_- x_+=0$. Then we have
\[
|a|^2=a^2=x_+^2 +x_-^2=(x_+ + x_-)^2.
\]
By uniqueness of the squared root of positive elements, see Proposition \ref{prop:nroots}, we obtain $x_+ + x_- = |a|$. If we solve the system of equations
\[
\left\{ \begin{array}{l} x_+ + x_- = |a|\\  x_+ - x_- = a \end{array} \right.
\]
for $x_+$ and $x_-$, we get $x_+=a_+$ and $x_-=a_-$.
\end{proof}

The above decomposition of every self adjoint element $a\in A$ to the difference of two positive elements is called the {\bf Jordan decomposition of $a$}.

\begin{remark}
\label{rem:jordandecom} Let $a\in A$ be self adjoint. Using the non-unital continuous functional calculus of $a$, i.e. $\Phi_a$, we define
\[
b_+:=\Phi_a(\max\{id_{\si(a)}, 0\}), \quad \text{and}\quad b_-:=\Phi_a(\max\{-id_{\si(a)}, 0\}).
\]
They possesses the properties mentioned in Proposition \ref{prop:twopositive}, and so we have $a_+=b_+\in C\s(a)$ and $a_-=b_-\in C\s(a)$, where $a_+$ and $a_-$ are as Proposition \ref{prop:twopositive}.
\end{remark}
\begin{corollary}
\label{cor:4positive}
Every element $a$ of a \cs-algebra $A$ can be written as a linear combination of four positive elements of $A$.
\end{corollary}
Note that the existence of a unit element is not part of the assumption in the above corollary, in contrast with Proposition \ref{prop:4unitary}.

Some properties of positive elements are stated in the following proposition. One notes that most of the proofs are based on CFC.
\begin{proposition}
\label{prop:basicpositive}
Let $A$ be a \cs-algebra and $a,b \in A$.
\begin{itemize}
\item[(i)] If $a\geq 0$ and $-a\geq 0$, then $a=0$.
\item[(ii)] If $a\geq0$, then $\|a\|=\max \{ \la; \la\in \si(a)\}$. More generally, If $a=a\s$, then $\|a_+\|=\max \{ \la; \la\in \si(a)\}$ and $\|a_-\|=\min \{ \la; \la\in \si(a)\}$.
\item[(iii)] If $a,b\geq0$ and $ab=ba$, then $ab\geq0$ and $a+b\geq0$.
\item[(iv)] Let $A$ be unital. If $a=a\s$ and $\|a\|\leq 2$, then $a\geq 0$ if and only if $\|a-1\|\leq 1$.
\item[(v)] Let $A$ be unital. Then $a+\la \geq 0$ if and only if $a=a\s$ and $\la\geq\|a_-\|$.
\item[(vi)] Let $A$ be unital and let $a=a\s$. Then $a\geq 0$ if and only if $\|1-\frac{a}{\|a\|} \|\leq 1$.
\end{itemize}
\end{proposition}
\begin{proof}
\begin{itemize}
\item[(i)] It is an immediate application of CFC.
\item[(ii)] When $\si(a)\subseteq [0,\infty)$, we have $r(a)=\max\{\la; \la\in \si(a)\}$. Hence the first statement follows from Proposition \ref{prop:spradius3}. To prove the general statement, apply Remark \ref{rem:jordandecom}.
\item[(iii)] It follows from Corollary \ref{cor:additionspec}.
\item[(iv)] Since $a=a\s$, the inequality $\|a\|\leq 2$ means that $\si(a)\subseteq [-2,2]$. If $a\geq 0$, then we have $\si(a)\subseteq [0,2]$, and so $\si(a-1)\subseteq [-1,1]$. Hence we get $\|a-1\|=r(a-1)\leq 1$. The converse is proved by a similar argument.
\item[(v)] One can apply the ideas discussed in Remark \ref{rem:jordandecom} and Item (ii) to prove this part.
\item[(vi)] If $a\geq 0$, then $\si(\frac{a}{\|a\|})\sub [0,1]$. Hence $\si(1- \frac{a}{\|a\|})\sub [0,1]$, and so $\|1-\frac{a}{\|a\|} \|=r(1-\frac{a}{\|a\|})\leq 1$. For the converse, a similar argument based on the continuous spectral mapping theorem shows that $\si(a)\sub [0, 2\|a\|]$.
\end{itemize}
\end{proof}

\begin{definition}
Let $E$ be a vector space. A subset $C$ of $E$ is called a {\bf cone in $E$} if it is closed under addition and scalar multiplication by $\r_+:=[0,\infty)$. Also, it is often assumed that $C\cap (-C)=\{0\}$.
\end{definition}

\begin{proposition}
\label{prop:positivecone}
\begin{itemize}
\item [(i)] $A_+$ is a closed cone in $A$.
\item [(ii)] For every $a\in A$, we have $a\s a\geq 0$.
\item [(iii)] For every $a\in A$ and $b\in A_+$, we have $a\s b a\geq 0$.
\end{itemize}
\end{proposition}
\begin{proof}
\begin{itemize}
\item [(i)] It is clear that $\r_+ A_+=A_+$. Using Proposition \ref{prop:basicpositive}(iv), one easily observes that
\[
A_+ \cap (A)_1=A_h\cap (A)_1 \cap\{a\in A; \|1-x\|\leq 1\},
\]
where $(A)_1$ is the closed unit ball of $A$. Since all three sets appearing in the right hand side of this equality are closed and convex, so is the set in the left hand side. This clearly implies that $A_+$ is closed. Assume $a,b\in A_+$ and pick some $r>0$ such that $a/r$ and $b/r$ lie in $A_+ \cap (A)_1$. Hence $a/2r+b/2r \in A_+ \cap (A)_1$. Now, the equality $\r_+ A_+=A_+$ implies that $a+b=2r(a/2r +b/2r)$ belongs to $A_+$.
\item [(ii)] Let $a$ be an arbitrary element of $A$. Using Proposition \ref{prop:speccomm}, we know that $a\s a \geq 0$ if and only if $a a\s\geq 0$. On the other hand, if $a=x+iy$ is the decomposition of $a$ to the linear combination of two self adjoint elements $x,y$, then we have
    \begin{equation}
    \label{eqn:positive1}
    a\s a +aa\s=2(x^2+y^2)\geq 0
    \end{equation}
    by (i) and Exercise \ref{ex:normalabsolute}. Now, let $a\s a=c_+-c_-$ be the Jordan decomposition of $a\s a$ and set $b:=ac_-$. Then $-b\s b=-c_- a\s a c_-=c_-^3\geq 0$. Hence using (\ref{eqn:positive1}) and (i), we conclude that $bb\s=(b\s b + b b \s) + (-b\s b) \geq 0$. Therefore $b\s b \geq 0$. By Proposition \ref{prop:basicpositive}(i), we have $c_-=-(b\s b)^{1/3}= 0$. Hence $a\s a \geq 0$.
    \item [(iii)] Write $a\s b a=(b^{1/2} a)\s (b^{1/2} a)$ and apply (ii).
\end{itemize}
\end{proof}
\begin{example}
\label{exa:positiveityHilbert}
\begin{itemize}
\item[(i)] Let $(H, \lan - , -\ran)$ be a Hilbert space. If an operator $T\in B(H)$ is positive, then $T=S\s S$ for some $S\in B(H)$. Hence we have
\[
\langle Tx, x\rangle =\langle S\s S x,x\rangle =\langle Sx , Sx\rangle=\|Sx\|^2\geq 0, \quad \forall x\in H.
\]
This means that $T\geq 0$ implies that $\langle Tx,x\rangle\geq 0$ for all $x\in H$. The converse is also true. Let $T\in B(H)$ and let $\lan Tx,x\ran \geq 0$ for all $x\in H$. Then for every $x\in H$, we have $\langle Tx,x\rangle= \langle x,Tx\rangle$, because it is real. Since the adjoint operator is unique, see Corollary \ref{cor:adjointoperator}, $T=T\s$. Consider the Jordan decomposition of $T$, i.e. $T=T_+ -T_-$. For every $x\in H$, we have
\[
\langle T_-x, T T_- x \rangle=\langle T_-x, (T_+ -T_-) T_- x \rangle= \langle T_- x, -T_-^2 x\rangle= \langle x, -T_-^3 x\rangle.
\]
    The left hand side of the above equality is non-negative ($\geq 0$), because $T$ is positive and the right hand side is non-positive ($\leq 0$), because $T_-^3$ is positive. Therefore $\langle x, T_-^3 x\rangle=0$ for all $x\in H$. It follows from Problem \ref{e:5-18} that $T_-^3=0$, and so $T_-=0$. This means that $T$ is positive. Now that we established a new characterization of positive operators on a Hilbert space, without using their spectrum, it is a good exercise to show (without using CFC) that if a positive operator $T$ is invertible, then its inverse is positive too.
\item[(ii)] Now, let $H=\c^n$ and let $\langle- ,- \rangle$ be the ordinary inner product on $H$, namely
\[
\langle x, y\rangle=\sum_{i=1}^n x_i\overline{y_i}, \quad \forall x=(x_1,\cdots,x_n), y=(y_1,\cdots, y_n)\in \c^n.
\]
Then $T=(t_{ij})\in M_n(\c)=B(H)$ is self adjoint if and only if $t_{ij}=\overline{t_{ji}}$ for all $i,j=1,\cdots, n$, see Example \ref{exa:adjointmatrix}. Since the spectrum of a $T\in M_n(\c)$ is exactly the set of all eigenvalues of $T$, $T$ is positive if and only if $T$ is self adjoint and all its eigenvalues are positive (non-negative). For example, one checks that the operator defined by the matrix $\left( \begin{array}{cc} 2& -1\\-1&1 \end{array}\right)$ (in the standard basis) is positive.
\end{itemize}
\end{example}

The partial order $a\leq b$ between elements of a \cs-algebra is clearly translation invariant, that is $a+c\leq b+c$ for every $c\in A$. Using Proposition \ref{prop:positivecone}, we can extend this property as follows:
\begin{exercise}
\label{exe:positive3}
Let $a,b,c,d$ be elements of a \cs-algebra $A$.
\begin{itemize}
\item [(i)] If $a\leq b$ and $c\leq d$, show that $a+c\leq b+d$.
\item [(ii)] If $a\leq b$, then $c\s a c\leq c\s b c$.
\item [(iii)] If $a\geq 0$, then $c\s ac\leq \|a\| c\s c$.
\end{itemize}
\end{exercise}
For every element $a$ of a \cs-algebra $A$, using Proposition \ref{prop:positivecone}(ii), we can define the {\bf absolute value of $a$} by
\[
|a|:=(a\s a )^{1/2}.
\]
Of course, there is an alternative definition using $aa\s$ in stead of $a\s a$. But we choose the above option.

\begin{proposition}
\label{prop:positiveroots} Let $a$ be a positive element of a \cs-algebra $A$. For every positive real number $r$, there is a unique positive element $a^r\in C\s (a)$ such that this definition is consistent with the definition of $n$th-root in Proposition \ref{prop:nroots} when $r\in \q$, namely, if $r=m/n$, then $a^r=(a^{1/n})^m$. Moreover,  \begin{itemize}
\item [(i)] $a^{r+s} = a^r a^s$, for all $r,s\geq 0$,
\item [(ii)] the map $[0,\infty)\ra A_+$ defined by $r\mapsto a^r$ is continuous, and
\item [(iii)] when $A$ is unital and $a$ is invertible, $a^r$ is also defined for all $r\leq 0$. In this case the map defined in (ii) is continuous over $\r$.
\end{itemize}
\end{proposition}
\begin{proof} For $r\geq 0$, define $g_r(t):=t^r$. It is continuous over $[0,\infty)$ and so we can use the non-unital CFC over $a$ to define $a^r:=g_r(a)$. This definition is consistent with the definition of $n$th root in Proposition \ref{prop:nroots} and so the uniqueness is proved similarly.
\begin{itemize}
\item [(i)] It follows from the facts that $g_{r+s}(t)=g_r(t) g_s(t)$ and CFC is a \ss-homomorphism.
\item [(ii)] Due to the fact that CFC is an isometry, we only need to show that the map $[0,\infty)\ra C(\si(a))$, defined by $r\mapsto g_r(t)$ is continuous. This follows from the continuity of exponential map and boundedness of $\si(a)$.
\item [(iii)] When $r\leq 0$, define $a^r:= (a\inv)^{-r}$.
\end{itemize}
\end{proof}
Let $a,b,x$ be elements of a \cs-algebra $A$. A simple calculation proves the following identities:
\begin{itemize}
\item [(i)] The {\bf polarization identity:}
\begin{equation}
\label{eqn:polidentity1}
4b\s a=(a+b)\s(a+b)-(a-b)\s(a-b)+i(a+ib)\s(a+ib)-i(a-ib)\s(a-ib).
\end{equation}
\item [(ii)] The {\bf generalized polarization identity:}
\begin{equation}
\label{eqn:polidentity2}
4b\s x a=\sum_{k=0}^3 i^k (a+i^k b)\s x (a+i^kb).
\end{equation}
\end{itemize}
Several useful inequalities are given in the following propositions:
\begin{proposition} Let $a,b$, $a_1,\cdots,a_n$, and $b_1,\cdots,b_n$ be elements of a \cs-algebra $A$. Then we have
\label{prop:positiveineq}
\begin{itemize}
\item [(i)] $-(a\s a +b\s b )\leq a\s b +b\s a \leq a\s a +b\s b$,
\item [(ii)] $(a_1+\cdots + a_n)\s(a_1+\cdots + a_n) \leq n(a_1\s a_1 +\cdots+a_n\s a_n)$,
\item [(iii)] if $a_i\s a_j=0$ for $1\leq i\neq j\leq n$, then $\left\| \sum_{i=1}^n a_ib_i\right\|^2\leq \sum_{i=1}^n \|a_ib_i\|^2$.
\end{itemize}
\end{proposition}
\begin{proof}
\begin{itemize}
\item [(i)] It easily follows from expanding inequalities $0\leq (a+b)\s (a+b)$ and $0\leq (a-b)\s (a-b)$ and then regrouping them appropriately.
\item [(ii)] Let $\om$ be a primitive $n$th root of unity. Then $\sum_{j=1}^n \om^{j(l-k)}=0$ for all $k\neq l$. Hence we compute
    \begin{eqnarray*}
    n \sum_{k=1}^n a_k\s a_k&=& n\sum_{k=1}^n a_k\s a_k+ \sum_{\begin{array} {c} k,l=1\\ k\neq l\end{array}}^n  \sum_{j=1}^n \om^{j(l-k)}a_k\s a_l\\
&=&  \sum_{j=1}^n \sum_{k=1}^n a_k\s a_k+ \sum_{j=1}^n \sum_{\begin{array} {c} k,l=1\\ k\neq l\end{array}}^n  \om^{j(l-k)}a_k\s a_l\\
&=&  \sum_{j=1}^n\left[ \sum_{k=1}^n a_k\s a_k+  \sum_{\begin{array} {c} k,l=1\\ k\neq l\end{array}}^n  \om^{-j(k-1)} a_k\s  \om^{j(l-1)}a_l \right]\\
&=&  \sum_{j=1}^n \left[ \left( \sum_{k=1}^n \om^{j(k-1)}a_k\right)\s \left( \sum_{l=1}^n  \om^{j(l-1)} a_l \right) \right].\\
    \end{eqnarray*}
In the last line, all terms inside the summation over $j$ are positive. Thus the whole summation is greater than or equal the term corresponding to $j=n$. Therefore we obtain
\[
n \sum_{k=1}^n a_k\s a_k\geq \left( \sum_{k=1}^n a_k\right)\s \left( \sum_{l=1}^n  a_l \right).
\]
\item [(iii)] We simply compute
\begin{eqnarray*}
\left\| \sum_{i=1}^n a_i b_i\right\|^2&=& \left\| \left(\sum_{i=1}^n a_i b_i\right)\s \left( \sum_{i=1}^n a_i b_i\right)\right\|\\
&=& \left\| \sum_{i=1}^nb_i\s a_i\s  a_i b_i\right\|\\
&\leq& \sum_{i=1}^n\|b_i\s a_i\s  a_i b_i\|\\
&=& \sum_{i=1}^n\|a_i b_i\|^2.
\end{eqnarray*}
\end{itemize}
\end{proof}
\begin{lemma}
\label{lem:ineqnorm}
Assume $a$ and $b$ are two elements of a \cs-algebra $A$ such that $0\leq a\leq b$. Then $\|a\|\leq \|b\|$.
\end{lemma}
\begin{proof} Without loss of generality, we can assume $A$ is unital. If $a$ and $b$ commute, we consider the commutative \cs-algebras $C\s(a,b)\simeq C_0(\oom(C\s(a,b)))$. Then $a$ and $b$ are associated to two positive functions $f$ and $g$ in  $C_0(\oom(C\s(a,b)))$, respectively, such that $f\leq g$. Then it is clear that $\|a\|=\|f\|_{\sup} \leq \|g\|_{\sup}=\|b\|$.

For the general case, one first notes that $0\leq a\leq b \leq \|b\| 1$. Now, the statement follows from the above case and the fact that $a$ and $\|b\| 1$ commute.
\end{proof}

\begin{exercise}
\label{exe:middleineq}
Let $a,b,c,d$ be four self adjoint elements in a \cs-algebra $A$. Show that if $a\leq b \leq c\leq d$, then $c-b \leq d-a$.
\end{exercise}
\begin{definition}
A linear map $\ff:A\ra B$ between two \cs-algebra is called {\bf positive} if $\ff(a)\geq 0$ whenever $a\geq 0$, in other words, if $\ff$ maps positive elements of $A$ to positive elements of $B$.
\end{definition}
\begin{example}
\begin{itemize}
\item [(i)] Every \ss-homomorphism $\ff:A\ra B$ between two \cs-algebra is positive.
\item [(ii)] Let $tr:M_n(\c)\ra \c$ be the trace map. By Example \ref{exa:positiveityHilbert}(ii), $tr$ is positive, because the trace of every matrix $T\in M_n(\c)$ is the sum of its eigenvalues.
\end{itemize}
\end{example}
We use Part (i) of the above example in what follows. More specifically, we use the fact that the Gelfand transform and the \ss-isomorphism defining CFC preserve inequalities between elements of \cs-algebras.

\begin{proposition} Let $a$ and $b$ be elements of a \cs-algebra $A$ and let $0\leq a\leq b$.
\label{prop:powerspositive}
\begin{itemize}
\item [(i)] If $A$ is unital and $a$ is invertible, then $b$ is invertible too and we have $0\leq b\inv \leq a\inv$.
\item [(ii)] If $0<r \leq 1$, then $a^r \leq b^r$.
\end{itemize}
\end{proposition}
For given $0<r$ such that $0\leq a\leq b$ implies $a^r\leq b^r$, we say that the function $t\ra t^r$ is {\bf operator monotone}.
\begin{proof}
\begin{itemize}
\item [(i)] If $a$ is invertible, then $0\notin \si(a)$. Thus $a\geq 0$ implies that there exists a $\ep>0$ such that $\si(a)\sub (\ep, \infty)$. Hence $\ep 1\leq a$, and so $\ep 1\leq b$. This implies that $y$ is invertible too. By CFC, it is clear that $a\inv, b\inv$ are both positive. Furthermore, if $a$ and $b$ commute, we consider the commutative \cs-algebra $C\s(a,b,1)$. Then using the Gelfand transform, $a$ and $b$ are corresponded to two strictly positive (so invertible) functions $f,g\in C(\oom(C\s(a,b,1)))$, respectively, such that $f\leq g$. It is clear that $0\leq g\inv \leq f\inv$. Since the inverse of the Gelfand transform is a \ss-homomorphism, we obtain the desired inequality from this. For the general case, we note that $a\leq b$ implies that $b^{-1/2}a b^{-1/2}\leq b^{-1/2} b b^{-1/2} =1$, see Exercise \ref{exe:positive3}. Since $1$ commutes with every element, we have $1\leq (b^{-1/2}a b^{-1/2})\inv$ and this implies $b\inv\leq a\inv$.
\item [(ii)] Without loss of generality, we assume $A$ is unital. Define
\[
S:=\{r\in (0,\infty ); t\mapsto t^r \, \text{is operator-monotone} \}.
\]
We follow the following steps to prove $(0,1]\sub S$:
    \begin{itemize}
    \item [(a)] Clearly, $1\in S$ and $S$ is closed under multiplication.
    \item [(b)] The set $S$ is a closed subset of $(0,\infty)$:\\
Let $\{r_n\}$ be a sequence of elements of $S$ convergent to some $r_0\in (0,\infty)$. Then $\{b^{r_n}- a^{r_n}\}$ is a sequence in $A_+$ convergent to $b^{r_0}- a^{r_0}$, see Proposition \ref{prop:positiveroots}(ii). Since $A_+$ is closed, $b^{r_0}- a^{r_0}\geq 0$, and so $r_0\in S$.
    \item [(c)] A positive real number $r$ belongs to $S$ if and only if $0\leq a \leq b$ and $b\in A^\times$ imply $a^r\leq b^r$:\\
We assume the special case and prove the general case. If $0\leq a \leq b$, then $0\leq a \leq b+\ep1$ and $b+\ep1\in A^\times$ for all $\ep>0$. Thus $a^r\leq (b+\ep 1)^r$ for all $\ep>0$. On the other hand, $(b+\ep 1)^r\ra b^r$ as $\ep\ra 0$. To see this, we should look at CFC of $b$. Define $f_n(z):=(z+\frac{1}{n})^r$. Since the spectrum of $b$ is compact the sequence $\{f_n\}\sub C(\si(b))$ of functions is uniformly convergent to $f(z):=z^r$. Now, since $A_+$ is a closed set of $A$, we conclude that $a^r\leq b^r$.
    \item [(d)] If $0\leq a\leq b$, $b\in A^\times$ and $\| b^{-r/2}a^r  b^{-r/2}\|\leq 1$, then $r\in S$:\\
Using Exercise \ref{exe:positive3}(ii), one easily observes that $a^r\leq b^r$ if and only if $b^{-r/2}a^r  b^{-r/2}\leq 1$. Also, $b^{-r/2}a^r  b^{-r/2}\leq 1$ if and only if $\| b^{-r/2}a^r  b^{-r/2}\|\leq 1$. These latter implications follow from Lemma \ref{lem:ineqnorm} and the fact that $b^{-r/2}a^r  b^{-r/2}$ is positive.
    \item [(e)] $1/2\in S$:\\
Assume $0\leq a\leq b$ and $b\in A^\times$, then $b^{-1/2} a b^{-1/2}\leq 1$. Thus we have
\[
1\geq \|b^{-1/2} a b^{-1/2}\|= \|b^{-1/2} a^{1/2} a^{1/2} b^{-1/2}\|=\|b^{-1/2} a^{1/2}\|^2,
\]
and so $\|b^{-1/2} a^{1/2}\|\leq 1$. Using the fact that $r(xy)=r(yx)$ for all $x,y$ in a \cs-algebra, we get
\[
r(b^{-1/4} a^{1/2}b^{-1/4})=r(b^{-1/2} a^{1/2})\leq \|b^{-1/2} a^{1/2}\|\leq 1.
\]
Since $b^{-1/4} a^{1/2}b^{-1/4}$ is positive, this inequality implies that $b^{-1/4} a^{1/2}b^{-1/4}\leq 1$, and so $\|b^{-1/4} a^{1/2}b^{-1/4}\|\leq 1$. Using the above step we conclude $1/2\in S$.
    \item [(f)] If $r,s\in S$, then $t:=\frac{r+s}{2}\in S$:\\
The proof of this step is similar to the previous step. Assume $0\leq a\leq b$ and $b\in A^\times$. Then we have
\begin{eqnarray*}
r(b^{-t/2} a^t b^{-t/2})&=&r(b^{-t} a^t)=r(b^{-r/2} a^t b^{-s/2})\\
&\leq& \| (b^{-r/2} a^{r/2}) (a^{s/2}b^{-s/2})\|\\
&\leq& \| b^{-r/2} a^{r/2}\| \|a^{s/2}b^{-s/2}\|\leq 1.
\end{eqnarray*}
This implies $\|b^{-t/2} a^t b^{-t/2}\|\leq 1$. So, by Step (d), we have $t\in S$.
    \item[(g)] Applying the above steps, we observe that the set
\[
X=\{ \frac{m} {2^n}; m, n\in \n, m\leq 2^n\}
\]
lies in $S$. Since $S$ is closed and $X$ is dense in $(0,1]$, we have $(0,1]\sub S$.
    \end{itemize}
\end{itemize}
\end{proof}
One notes that the proof of Part (ii) of the above proposition would be very easy if $a$ and $b$ commute with each other. In fact, in this case, for all $r\in(0,\infty)$, $0\leq a \leq b$ implies $a^r\leq b^r$. The following exercise gives a counterexample for the latter statement in general.
\begin{exercise}
Let $a=\left( \begin{array} {cc}2 & 2\\2 &2 \end{array}\right)$ and $b=\left( \begin{array} {cc} 3&0 \\0 &6 \end{array}\right)$. Verify that $0\leq a\leq b$ and that $a^r \nleq b^r$ for all $r>1$.
\end{exercise}

\section{Approximate units}
\label{sec:approximateunit}

We have already seen examples of approximate units in Banach algebras, see Remark \ref{rem:diracnet} and Lemma \ref{lem:diracnet}, and in \cs-algebras, see Example \ref{exa:apprunit}. They facilitate many proofs in the lack of unit elements. Many applications of approximate units will be given in Section \ref{sec:ideals}. We begin this section with the definition of various types of approximate units in \cs-algebras. Afterwards, we prove the existence of each one these types.

Let $(\Lambda, \leq)$ and $(\Sigma, \sqsubseteq)$ be two directed sets. An {\bf isotone from $\Lambda$ into $\Sigma$} is a map $\ff:\Lambda \ra \Sigma$ preserving the order structure, that is $\ff(\la)\sqsubseteq\ff(\th)$ for all $\la\leq \th \in \Lambda$. An isotone is called an {\bf isomorphism of directed sets} if it is a bijective map.

\begin{definition}
\label{def:approxtypes}
Let $A$ be a \cs-algebra. An {\bf approximate unit for $A$} is a net $(u_\la)$ of positive elements of $A$ indexed by a directed set $\Lambda$ such that $\|u_\la\|\leq 1$ for all $\la \in \Lambda$ and we have
\begin{equation}
\label{eqn:approx1}
\|u_\la a-a\|\ra 0 \quad\text{and}\quad \|au_\la -a\|\ra 0, \qquad \forall a\in A.
\end{equation}
Some of the varieties of approximate units are as follows:
\begin{itemize}
\item [(i)] An approximate unit $(u_\la)$ is called {\bf increasing} if $u_\la\leq u_\th$ if $\la \leq \th$.
\item [(ii)] An approximate unit $(u_\la)$ is called {\bf idempotent} if $u_\la$ is a projection for all $\la \in \Lambda$.
\item [(iii)] An approximate unit $(u_\la)$ is called {\bf countable} if $\Lambda$ is countable.
\item [(iv)] An approximate unit $(u_\la)$ is called {\bf sequential} if $\Lambda$ is the same directed set as $\n$ (up to isomorphism of directed sets).
\item [(v)] An approximate unit $(u_\la)$ is called {\bf continuous} if $\Lambda$ is equipped with a topology and there is a continuous isomorphism $\ff:(0,\infty)\ra \Lambda$ of the directed sets.
\end{itemize}
\end{definition}
In fact, each one of the limits in (\ref{eqn:approx1}) implies the other one, by considering $a\s$ in stead of $a$ and using the continuity of the involution. It is also worthwhile to note that if $(u_\la)_{\la\in \Lambda}$ is an approximate unit for a unital \cs-algebra $A$, then it eventually consists of invertible elements in the unit ball of $A$ converging to $1_A$. Even when $A$ is not unital, we sometimes use the expressions $(1-u_\la)a\ra 0$ and $a(1-u_\la)\ra 0$ (or $\|(1-u_\la)a\|\ra 0$ and $\|a(1-u_\la)\|\ra 0$) in lieu of limits in (\ref{eqn:approx1}). One notes that they are have the same meaning in $\tilde{A}$ and since the inclusion $A\hookrightarrow \tilde{A}$ is an isometry, they make sense even in $A$.  An important feature of approximate units that plays a key role in studying various properties of \cs-algebras is the minimum cardinality of the index sets of approximate units. Before proving the existence of approximate units in general, we discuss some examples, notions and properties of approximate units.
\begin{exercise}
Show that every countable approximate unit $(u_\la)$ for a \cs-algebra $A$ has a subnet $(u_{\la_n})$ which is a sequential approximate unit for $A$.
\end{exercise}
\begin{definition}
A \cs-algebra $A$ is called {\bf $\si$-unital} if it possesses a countable (or equivalently sequential) approximate unit.
\end{definition}
\begin{example}
\begin{itemize}
\item [(i)] Unital \cs-algebras are trivial examples of $\si$-unital \cs-algebras.
\item [(ii)] For $n\in \n$ define $f_n:\r \ra \c$ by the following formula:
\[
f_n(t):=\left\{ \begin{array} {ll} 1 & |t|\leq n\\ n-|t|+1& n\leq |t|\leq n+1\\ 0& |t|\geq n+1 \end{array} \right.
\]
One easily verifies that $\{f_n\}$ is an increasing sequential approximate unit for $C_0(\r)$.
\end{itemize}
\end{example}
\begin{proposition} Let $(u_\la)$ be an approximate unit for a \cs-algebra $A$. Then we have the following statements:
\begin{itemize}
\item [(i)] For every $a\in A$, we have $u_\la a u_\la \ra a$.
\item [(ii)] For every $a\in A_+$, we have $a^{1/2} u_\la a^{1/2} \ra a$.
\item [(iii)] For every $r>0$, $u_\la^r$ is an approximate unit.
\end{itemize}
\end{proposition}
\begin{proof}
\begin{itemize}
\item [(i)] One notes that
\[
\|u_\la a u_\la -a\|\leq \|u_\la a u_\la -au_\la \|+\|a u_\la -a\|\leq\|u_\la a -a \|+\|a u_\la -a\|.
\]
This inequality implies (i).
\item [(ii)] When $a\geq 0$, we have $u_\la a^{1/2}\ra a^{1/2}$ and $ a^{1/2}u_\la\ra a^{1/2}$. Using the continuity of the multiplication, (ii) follows from these limits.
\item [(iii)] Fix $a\in A$. We note that $\|u_\la x - x u_\la \|\ra 0$ as $\la \ra \infty$ for all $x\in A$. Hence by setting $x=u_\la a$, we get
\[
\|u_\la^2 a - a \|\leq \|u_\la^2 a - u_\la a u_\la \|+ \|u_\la a u_\la  -a \| \ra 0.
\]
By induction, we obtain $\|u_\la^{2^n} a - a \|\ra 0$ for all $n\in \n$. For $r>0$, pick $n$ big enough such that $2^n\geq 2r$. Inequalities
\[
a\s u_\la^{2^n} a \leq a\s u_\la^{2r} a \leq a\s u_\la^{r} a \leq a\s a,
\]
follow from Exercise \ref{exe:positive3}(ii) and the fact that $u_\la^\alpha \leq 1$ for all $\alpha>0$. Regarding this inequalities and using Lemma \ref{lem:ineqnorm} and Exercise \ref{exe:middleineq}, we have
\[
\|a\s u_\la^{2r} a -a\s u_\la^{r} a\|\leq \|a\s u_\la^{2^n} a -a\s a\|\ra \infty,
\]
and
\[
\|a\s u_\la^{r} a -a\s  a\|\leq \|a\s u_\la^{2^n} a -a\s a\|\ra \infty.
\]
Therefore we get
\begin{eqnarray*}
\| a- u_\la^r a\|^2 &=& \|(a\s- a\s u_\la^r )(a- u_\la^r a)\|\\
&=&\|a\s a -2 a\s u_\la^r a+ a\s u_\la^{2r} a\|\\
&\leq&\|a\s a - a\s u_\la^r a\| + \| a\s u_\la^r a- a\s u_\la^{2r} a\|\ra \infty.
\end{eqnarray*}
\end{itemize}
\end{proof}
The following exercise is used in the next theorem:
\begin{exercise}
Let $I$ be an ideal (left, right or two sided) of a complex algebra $A$ and let $A_1$ be the algebraic unitization of $A$. Show that $I$ is an ideal (left, right or two sided) of $A_1$ as well. See also Proposition \ref{prop:idealideal}.
\end{exercise}

\begin{theorem}
\label{thm:generalapprox}
Every dense two sided ideal $I$ of a \cs-algebra $A$ contains an approximate unit for $A$. More precisely, Define
\[
\Lambda_I:=\{ a\in A_+ \cap I; \|a\|<1 \}.
\]
The set $\Lambda_I$ with the order structure inherited from $A_+$ is a directed set and is an increasing approximate identity, whose index set is itself.
\end{theorem}
\begin{proof}
Let us denote the set of all positive elements of $I$ by $I_+$. Define $\ff:\Lambda_I \ra I_+$ by $u\mapsto (1-u)\inv -1$. One easily sees that the map $x\mapsto 1-(1+x)\inv$ is the inverse of $\ff$. In the definition of $\ff$ and $\ff\inv$, we used the unit element which belongs to $\tilde{A}$ in general. So, we have to check that $\ff$ and $\ff\inv$ are both well defined. In other words, we must show that $\ff(u)\in I_+$ for all $u\in \Lambda_I$ and $\ff\inv(x)\in \Lambda_I$ for all $x\in I_+$.

For given $u\in \Lambda$, since $\|u\|<1$ we have
\[
\ff(u)=\sum_{n=1}^\infty u^n.
\]
Therefore $\ff(u)\in C\s(u)\sub A$. We also note that $\ff(u)=u+u \sum_{n=1}^\infty u^n$, and so $\ff(u)\in I$. Moreover, since $u^n$ is positive for all $n\in \n$ and $A_+$ is a closed cone, we have $\ff(u)\in A_+$.

Let $x\in I_+$. One easily checks that $\|\ff\inv(x)\|<1$  and $\ff\inv(x)$ is positive. Now, we note that $\ff\inv(x)= x(x+1)\inv$. Consider the continuous functional calculus of $x$, i.e. $\Phi_x:C(\si_{\tilde{A}}(x)) \ra C\s(x,1)$. Then we have $\ff\inv(x)=\Phi_x(f)$, where $f(z):=\frac{z}{z+1}$. Clearly,  $f(0)=0$, so $\ff\inv(x)=\Phi_x(f)\in C\s(x)\sub A$. On the other hand, the equality $\ff\inv(x)= x(x+1)\inv$ implies that $\ff\inv(x)\in I$. Therefore $\ff\inv(x)\in \Lambda_I$.

It follows from Proposition \ref{prop:powerspositive}(i) that $\ff$ preserves the order. Therefore $\Lambda_I$ is a directed set.

Now, we want to show that $(1-u)a\ra 0$ as $u\ra \infty$ for all $a\in A$. First, we prove this statement for a given $a\in \Lambda_I$. Let $\ep>0$. By Proposition \ref{prop:positiveroots}(ii), there exists $n\in \n$ such that $\|a\s (1-a^{1/n}) a\|= \|a^2 -a^{2+\frac{1}{n}}\| <\ep$. For $u\geq a^{1/n}$ in $\Lambda_I$, we have $1-u\leq 1-a^{1/n}$, and so $a\s(1-u)a \leq a\s(1-a^{1/n})a$. This implies $\|a\s (1-u)a\|\leq \|a\s(1-a^{1/n})a\|<\ep$. Hence when $u\ra \infty$ (in the directed set $\Lambda_I$), we have
\[
\|(1-u) a\|^2= \|a\s (1-u)^2 a\|\leq \|a\s (1-u) a \| <\ep.
\]
This proves the first step. The first step clearly implies the same statement for $a\in I_+$ as well. Now, let $a\in A_+$ and let $\ep>0$. Assume $\{b_n\}$ is a sequence in $I$ such that $b_n\ra a^{1/2}$. Then $b_n\s b_n\ra a$, and so $(1-u)b_n\s b_n\ra (1-u)a$ for all $u\in \Lambda_I$. This shows that $(1-u)a\ra 0$ as $u\ra \infty$ for all $a\in A_+$. Since every element of $A$ is a linear combination of four positive elements, this latter statement implies that $\Lambda_I$ is an approximate unit for whole $A$.
\end{proof}

\begin{definition}
A \cs-algebra $A$ is called {\bf separable} if it possesses a countable and dense subset.
\end{definition}
\begin{proposition}
\begin{itemize}
\item [(i)] Every separable \cs-algebra $A$ is $\si$-unital.
\item [(ii)] Every $\si$-unital \cs-algebra $A$ has a continuous approximate unit.
\end{itemize}
\end{proposition}
\begin{proof}
\begin{itemize}
\item [(i)] Let $\{a_k\}$ be a dense sequence in $A$ and let $(u_\la)$ be an approximate unit for $A$ obtained in Theorem \ref{thm:generalapprox}. Choose $u_{\la_1}$ arbitrarily. For a natural number $n\geq 2$, choose $u_{\la_n}$ such that $u_{\la_n}\geq u_{\la_{n-1}}$ and $\|(1-u_{\la_n}) x_k\|<1/n$ for all $k=1,\cdots,n$. Then by the construction, it is easy to see that the sequence $\{u_{\la_n}\}$ is a sequential approximate unit for $A$. One notes that this sequential approximate unit is increasing too.
\item [(ii)] Let $\{u_i\}$ be a sequential approximate unit in a \cs-algebra $A$ as obtained in (i). For every $n\in \z$ and $t\in [n,n+1]$, we define
\[
u_t:=(n+1-t) u_n + (t-n)u_{n+1}.
\]
Then $\{u_t\}_{t\in \r}$ is a continuous approximate unit for $A$. To see this, assume $a\in A$ and $\ep>0$ are given. Choose $n_0\in\n$ such that for all $n\geq n_0$, we have $\|(1-u_n)a\|<\ep$ and $\|(1-u_{n+1})a\|<\ep$. Then for all $t\geq n_0$, there is some $n\in \n$ such that $t\in [n,n+1]$ and we have
\begin{eqnarray*}
\|(1-u_t)a\|&=&\|(n+1-t)(1-u_n)a+ (t-n)(1-u_{n+1})a\|\\
&\leq& \|(n+1-t)(1-u_n)a\|+\|(t-n)(1-u_{n+1})a\|\\
&<& (n+1-t)\ep +(t-n)\ep=\ep.
\end{eqnarray*}
\end{itemize}
\end{proof}
\begin{exercise}
\begin{itemize}
\item [(i)] Complete the proof of Part (i) of the above proposition.
\item [(ii)] In Part (ii) of the above proposition, show that if  $\{u_i\}$ is an increasing sequential approximate unit, then $\{u_t\}_{t\in \r}$ is increasing too.
\end{itemize}
\end{exercise}

\section{Ideals and homomorphisms}
\label{sec:ideals}

In this section, $A$ is always a \cs-algebra. By an ideal of $A$, we mean a two sided ideal $I$ of $A$ which is closed under scalar multiplication, or equivalently, it is a subalgebra of $A$. Left or right ideals will be specified.

\begin{proposition}
\label{prop:idealapprox}
Let $I$ be a closed left ideal of a \cs-algebra $A$. Then there is a net $(u_\la)$ of positive elements of $I$ in the unit ball of $I$ such that $\lim_\la (a-a u_\la)=0$ for all $a\in I$. Moreover, if $I$ is a two sided ideal of $A$, then $(u_\la)$ is an approximate unit for $I$.
\end{proposition}
\begin{proof}
Set $J:=I\cap I\s$. One notes that $J$ is a \cs-subalgebra of $A$. Therefore by  Theorem \ref{thm:generalapprox}, there exists an approximate unit $(u_\la)$ for $J$ in the unit ball of $J$. For every $a\in I$, we have $a\s a \in J$, and consequently, $\lim_\la \|a\s a(1-u_\la)\|=0$. Thus we have
\begin{eqnarray*}
\lim_\la \|a(1-u_\la)\|^2 &=& \lim_\la \|(1-u_\la)a\s a(1-u_\la)\|\\
&\leq& \lim_\la \|a\s a(1-u_\la)\|=0.
\end{eqnarray*}
When $I$ is a two sided ideal, we use the fact that $aa\s\in J$ for all $a\in I$, and so $\lim_\la \|(1-u_\la)aa\s\|=0$. Then it follows from a similar argument that $\lim_\la \|(1-u_\la)a\|=0$ for all $a\in I$.
\end{proof}
\begin{proposition}
\label{prop:closedidealcs}
Every closed ideal $I$ of $A$ is closed under the involution, and so is a \cs-subalgebra of $A$.
\end{proposition}
\begin{proof}
Let $(u_\la)$ be the approximate unit for $I$ obtained in Proposition \ref{prop:idealapprox} and let $a\in I$. We have $a=\lim_\la a u_\la$, so $a\s =\lim_\la u_\la a\s$. Thus $a\s\in I$, because $u_\la a\s\in I$ for all $\la$ and $I$ is closed.
\end{proof}
\begin{proposition}
\label{prop:idealideal}
Let $I$ be a closed ideal in a \cs-algebra $A$ and let $J$ be a closed ideal of $I$. Then $J$ is an ideal of $A$.
\end{proposition}
\begin{proof}
Since $J$ is a \cs-algebra, every element of $J$ is a linear combination of four positive elements. Therefore to prove that $J$ is an ideal of $A$, it is enough to show that $ab,ba\in J$ for all $a\in A$ and $b\in J_+$. Let $(u_\la )$ be an approximate unit for $I$. Then for all $a\in A$, $b\in J_+$ and $\la$, we have $u_\la b^{1/2}\in I$, so $au_\la b^{1/2}\in I$. On the other hand, $b^{1/2}\in J$, so $au_\la b^{1/2}b^{1/2}\in J$ for all $\la$. Therefore $ab=\lim_\la au_\la b \in J$. Similarly, we have $a\s b\in J$. Since $J$ is a \cs-algebra, it implies $ba\in J$.
\end{proof}

\begin{proposition}
\label{prop:csquotient}
Let $I$ be a closed ideal of a \cs-algebra $A$.
\begin{itemize}
\item [(i)] The quotient norm on $A/I$ satisfies the following identities for every approximate unit $(u_\la)$ for $I$:
\[
\|a+I\|=\lim_\la \|a-u_\la a\|=\lim_\la \|a- au_\la\|, \qquad \forall a\in A.
\]
\item [(ii)]  The quotient algebra $A/I$ equipped with the quotient norm is a \cs-algebra.
\end{itemize}
\end{proposition}
\begin{proof}
Let $(u_\la)$ be an approximate unit for $I$.
\begin{itemize}
\item [(i)]  Let $a\in A$. For given $\ep>0$, let $b$ be an element of $I$ such that $\|a+b\|< \|a+I\|+\ep/2$. Pick $\la_0$ such that $\la\geq \la_0$ implies $\|(1-u_\la)b\|<\ep/2$. Then we have
\begin{eqnarray*}
\|(1-u_\la)a\|&\leq& \|(1-u_\la)(a+b)\|+ \|(1-u_\la)b\|\\
&\leq& \|a+b\|+\|(1-u_\la)b\|\\
&<& \|a+I\|+\ep/2+\ep/2\\
&\leq& \|(1-u_\la)a\|+\ep.
\end{eqnarray*}
When $\ep\ra 0 $, we obtain the the first equality. The second equality is proved similarly.
\item [(ii)] It is straightforward to check that $A/I$ is an involutive algebra. We also know that $A/I$ is a Banach algebra. For every $a\in A$, we have $\|a\s +I\|=\lim_\la \|(1-u_\la)a\s\|=\lim_\la \|a(1-u_\la)\|=\|a+I\|$. Therefore $A/I$ is an involutive Banach algebra. Now, we prove the \cs-identity for the quotient norm. Let $a\in A$ and $b\in I$. We note that the net $(\|(a\s a+b) (1-u_\la)\|)_\la$ is bounded by $\|a\s a+b\|$, so we have
\begin{eqnarray*}
\|a+I\|^2&=& \lim_\la \|(1-u_\la)a\|^2=\lim_\la \|(1-u_\la)a\s a (1-u_\la)\|\\
&\leq& \lim_\la \|(1-u_\la)(a\s a+b) (1-u_\la)\|+ \lim_\la \|(1-u_\la)b(1-u_\la)\|\\
&\leq& \|a\s a+ b\| + \lim_\la \|(1-u_\la)b\|= \|a\s a+ b\|.
\end{eqnarray*}
Thus we have
\[
\|a+I\|^2 \leq \|a\s a+I\| \leq \|a\s+I\| \|a+I\|= \|a+I\|^2.
\]
The \cs-identity follows from this.
\end{itemize}
\end{proof}
\begin{corollary}
\label{cor:firstiso} The image of every \ss-homomorphism $\ff:A\ra B$ between two \cs-algebras is a \cs-algebra.
\end{corollary}
\begin{proof}
Let $I$ be the kernel of $\ff$. Then the induced \ss-homomorphism $A/I\ra B$ is an injective \ss-homomorphism, and so an isometry by Corollary \ref{cor:csinjection}. Since $A/I$ is complete, $\ff(A)$ is closed in $B$, and so is complete. This proves that $\ff(A)$ is a \cs-algebra.
\end{proof}
The essence of the above corollary is the fact that the first isomorphism (in algebra setting) $A/I\simeq \ff(A)$ is an isometry. It leads us to the second isomorphism and its consequences.
\begin{corollary}
\label{cor:secondiso}
Let $I$ be a closed ideal of a \cs-algebra $A$ and let $B$ be a \cs-subalgebra of $A$. Then $B+I$ is a \cs-subalgebra of $A$ and there is an isometric \ss-isomorphism
\[
\frac{B+I}{I}\simeq \frac{B}{B\cap I }.
\]
This isometry implies
\[
\inf\{\|b+ a\|; a\in I\} =\inf \{ \|b+c\| ; c\in B\cap I\}, \qquad \forall b\in B.
\]
In particular, If $J$ is another closed ideal of $A$, then $I+J$ is also a closed ideal of $A$.
\end{corollary}
\begin{proof}
Let $\ff: B\ra A/I$ be the composition of the inclusion map $\iota:B\hookrightarrow A$ and the quotient map $\pi:A\ra A/I$. Clearly, $\ff$ is a \ss-homomorphism, so its image is a \cs-subalgebra of $A/I$. One notes that $\pi\inv(\ff(B))=B+I$. Hence $B+I$ is closed in $A$, and so is complete. The isomorphism $\frac{B+I}{I}\simeq \frac{B}{B\cap I }$, which is called the second isomorphism in algebra, is clearly a \ss-isomorphism, so is an isometry by Corollary \ref{cor:csinjection}.
\end{proof}
Every \cs-algebra $A$ is an ideal of $\widetilde{A}$. Also, Proposition \ref{prop:compactoperators} implies that $K(H)$ is an ideal of the \cs-algebra $B(H)$ of bounded operators on a Hilbert space $H$. Let us characterize closed ideals of commutative \cs-algebras.
\begin{example}
\label{exa:opensubsets} Let $X$ be locally compact and Hausdorff topological space. A subset $I\sub C_0(X)$ is an ideal of $C_0(X)$ if and only if $I=C_0(U)$ for some open subset $U\sub X$. For given ideal $I$ of $C_0(X)$, set
\[
C:=\{x\in X; f(x)=0, \, \forall f\in I\}.
\]
Let $(x_i)$ be a net in $C$ converging to some point $x\in X$. Using Proposition \ref{prop:inversegelfand}, we have $f(x_i)=\widehat{x_i} (f)\ra \widehat{x}(f)=f(x)$ for all $f\in C_0(X)$ and this shows that $C$ is a closed subspace of $X$. Now, set $U:=X- C$. Clearly, $I\sub C_0(U)$. Let $g\notin I$. Then $\pi(g)\neq 0$, where $\pi: C_0(X)\ra C_0(X)/I$ is the quotient map. Since $C_0(X)/I$ is a \cs-algebra by Proposition \ref{prop:csquotient}(ii), there exists a $\ff\in \oom(C_0(X)/I)$ such that $\ff(\pi(g))\neq 0$. Assume $x$ is the point in $X$ such that $\hat{x}=\pi\s(\ff)= \ff \pi$. Then one easily checks that $x\in C$. But $g(x)=\hat{x}(g)\neq 0$. Therefore $g\notin I$. This shows that $C_0(U)\sub I$, and so $I=C_0(U)$. The other implication is easy.
\end{example}
\begin{exercise}
\label{exa:closedsubsets}
Let $X$ be locally compact and Hausdorff topological space. Prove that there is a bijective correspondence between quotients of $C_0(X)$ and closed subsets of $X$.
\end{exercise}
Let $S_1,\cdots, S_n$ be subsets of a \cs-algebra $A$. The closure of the linear span of all elements of the form $s_1\cdots s_n$, where $s_i\in S_i$ for all $i=1,\cdots,n$, is denoted by $S_1\cdots S_n$.
\begin{proposition}
\label{prop:idealproduct}
Let $I$ and $J$ be closed ideals of a \cs-algebra $A$. Then we have $I\cap J=IJ$.
\end{proposition}
\begin{proof}
It is clear that $IJ\sub I\cap J$. To prove the reverse inclusion, it is enough to show that $a\in IJ$ for all $a\in (I\cap J)_+$. Since $I\cap J$ is a \cs-algebra and $a$ is positive, $a^{1/2}\in I\cap J$. Thus if $(u_\la)$ be an approximate unit for $I$, $u_\la a^{1/2}\in I$ for all $\la$ and $a^{1/2}\in J$. Therefore $a=\lim_\la u_\la a^{1/2} a^{1/2}\in IJ$.
\end{proof}

\begin{definition}
\label{def:hereditary}
Let $A$ be a \cs-algebra. A \cs-subalgebra $B$ of $A$ is called a {\bf hereditary \cs-subalgebra of $A$} if the inequality $a\leq b$ for $a\in A_+$ and $b\in B_+$ implies $a\in B$.
\end{definition}
Let $A$ be a \cs-algebra. Any intersection of hereditary \cs-subalgebras of $A$ is a hereditary \cs-subalgebra of $A$ as well. Therefore we can define the {\bf hereditary \cs-subalgebra generated by a subset $S\sub A$} to be the smallest hereditary \cs-subalgebra of $A$ containing $S$.
\begin{example}
\label{exa:hereditary}
Let $p$ be a projection in a \cs-algebra $A$, that is $p^2=p=p\s$. It is clear that $pAp=\{ pap; a\in A\}$ is a \cs-subalgebra of $A$. Moreover, $pAp$ is a hereditary \cs-subalgebra of $A$ which is called a corner in $A$. To show this, let $b\in A_+$ and let $b\leq pap$ for some $pap\in (pAp)_+$. Then by Exercise \ref{exe:positive3}(ii), we have $0\leq (1-p) b (1-p) \leq (1-p)pap(1-p)=0$. Therefore we get
\[
0=\|(1-p) b(1-p)\|=\|(1-p) b^{1/2} b^{1/2}(1-p)\|=\|b^{1/2}(1-p)\|^2,
\]
and so $b^{1/2}=b^{1/2} p$. Hence $b=b^{1/2} b^{1/2} =(b^{1/2})\s b^{1/2}=p b^{1/2} b^{1/2} p=pbp\in pAp$.

The corner $(1-p) A (1-p)$ is the {\bf complementary corner of $pAp$}, which is defined by $1-p$, the {\bf complementary projection of $p$}.
\end{example}
\begin{exercise}
Find an example of a \cs-algebra $A$, a projection $p\in A$, and $a\in A$ such that $a$ is not positive or even self adjoint, but $pap$ is positive.
\end{exercise}
\begin{theorem} Let $A$ be a \cs-algebra.
\label{thm:herecorres}
\begin{itemize}
\item [(i)] For every left closed ideal $J$ of $A$, $J\cap J\s$ is a hereditary \cs-subalgebra of $A$
\item [(ii)] The map $\th:J\mapsto J\cap J\s$ is a bijective correspondence between the set of all left closed ideals of $A$ and the set of all hereditary \cs-subalgebras of $A$. In fact, the inverse of this correspondence is given by the map $\th\inv:B\mapsto J_B$, where
\[
J_B:=\{a\in A; a\s a \in B \}.
\]
\item [(iii)] These correspondences preserve inclusions. In other words, for every two closed left ideals $J_1$ and $J_2$ of $A$, $J_1\sub J_2$ if and only $J_1\cap J_1\s\sub J_2\cap J_2\s$.
\end{itemize}
\end{theorem}
\begin{proof}
\begin{itemize}
\item [(i)] Clearly, for every closed left ideal $J$ of $A$, $J\cap J\s$ is a \cs-subalgebra of $A$. Let $a\in A_+$ and $0\leq a\leq b$ for some $b\in J\cap J\s$. Let $(u_\la)$ be the net in the open unit ball of $J$ obtained in Proposition \ref{prop:idealapprox}. Then we have $b=\lim_\la bu_\la$. The inequalities $0\leq a\leq b$ implies $0\leq (1-u_\la)a (1-u_\la) \leq (1-u_\la)b(1-u_\la)$ for all $\la$. Hence we have
\[
\|a^{1/2}(1-u_\la)\|^2=\|(1-u_\la) a (1-u_\la)\|\leq \|(1-u_\la) b (1-u_\la)\|\leq \|b (1-u_\la)\|\ra 0.
\]
    This shows that $a^{1/2}=\lim_\la a^{1/2} u_\la$. Hence $a^{1/2}\in J$, and so $a\in J$.
\item [(ii)] First, we have to show that $\th\inv$ is well defined. Let $B$ be a hereditary \cs-subalgebra of $A$ and let $x,y \in J_B$. Then $x\s x, y\s y\in B$. Hence we have
\[
(x+y)\s (x+y) \leq  (x+y)\s (x+y) + (x-y)\s (x-y)=2x\s x +2 y\s y\in B,
\]
    and so $x+y\in J_B$. Now, let $a\in A$ and $b\in J_B$. Then we get $(ab)\s ab=b\s a\s ab\leq \|a\|^2 b\s b \in B$, so $ab\in J_B$. Similarly, one shows that $J_B$ is closed under scalar multiplication. Since $B$ is closed, one easily observes that $J_B$ is closed too. Therefore $J_B$ is a closed left ideal of $A$.

    Now, we show that the map $\th \th\inv$ is the same as the identity map on the set of hereditary \cs-subalgebras of $A$. For every $b\in B$, we have $b\s b, bb\s\in B$, so $b, b\s \in J_B$. This shows that $b\in J_B\cap J_B\s$. Hence $B\sub J_B\cap J_B\s$. Let $b\in (J_B\cap J_B\s)_+$. Then $b^2\in B$, and since $B$ is a \cs-algebra, $b\in B$. This shows that $(J_B\cap J_B\s)_+\sub B_+$, and so $(J_B\cap J_B\s)\sub B$.

    Finally, we prove that the map $\th\inv \th$ is the same as the identity map on the set of closed left ideals of $A$.  Let $J$ be a closed left ideal of $A$ and set $B:=J\cap J\s$. If $x\in J$, then $x\s x\in J\cap J\s =B$, and so $x\in J_B$. Hence $J\sub J_B$. Let $x$ be a positive element of $J_B$. Then $x^2=x\s x\in J\cap J\s=B$. Since $B$ is a \cs-algebra, we have $x\in B\sub J$. Hence $(J_B)_+\sub J_+$, and consequently $J_B\sub J$.
\item [(iii)] Let $J_1$ and $J_2$ be two closed left ideals of $A$. If $J_1\sub J_2$, then we have $J_1\cap J_1\s \sub J_2\cap J_2\s$. Conversely, let $J_1\cap J_1\s \sub J_2\cap J_2\s$ and let $a\in J_1$. Consider an approximate unit $(u_\la)$ for $J_1\cap J_1\s$. One notes that
\[
\lim_\la \|a(1-u_\la)\|^2 =\lim_\la \|(1-u_\la) a\s a (1-u_\la)\| \leq \lim_\la \|a\s a (1-u_\la) \|=0,
\]
because $a\s a\in J_1\cap J_1\s$. So, $a=\lim_\la au_\la$. On the other hand, for every $\la$, we have $u_\la \in J_1\cap J_1\s \sub J_2\cap J_2\s \sub J_2$. Hence $au_\la \in J_2$ for all $\la$, and consequently $a\in J_2$. Therefore $J_1\sub J_2$.
\end{itemize}
\end{proof}

\begin{corollary} Let $I$ be a closed ideal of a \cs-algebra $A$.
\label{cor:herecorres}
\begin{itemize}
\item [(i)] $I$ is a hereditary \cs-subalgebra of $A$.
\item [(ii)] Considering $I$ as a hereditary \cs-subalgebra of $A$, we have $I=J_I= \{a\in A; a\s a \in I \}$.
\item [(iii)] For every $a\in A$, we have $a\in I$ if and only if $a\s a\in I$ if and only if $aa\s\in I$.
\end{itemize}
\end{corollary}
\begin{proof}
\begin{itemize}
\item [(i)] Because $I=I\s=I\cap I\s$. In other words, $\th(I)=I$ with the notation of the above theorem.
\item [(ii)] It follows from Part (ii) of the above theorem and (i), because $\th\inv (I)=I$.
\item [(iii)] It follows from (i) and (ii).
\end{itemize}
\end{proof}
Here is another characterization of hereditary \cs-subalgebras:
\begin{proposition}
\label{prop:eqhereditary}
A \cs-subalgebra $B$ of a \cs-algebra $A$ is hereditary if and only if $bab'\in B$ for all $b,b'\in B$ and $a\in A$.
\end{proposition}
\begin{proof}
Assume $B$ is hereditary, then $B=J\cap J\s$ for some closed left ideal $J$ of $A$. If $b,b'\in B$ and $a\in A$, then $(ba)b'\in J$, because $b'\in J$ and $(bab')\s=(b'\s a\s) b\s\in J$, because $b\in J\s$.  Therefore $bab' \in J\cap J\s=B$.

Conversely, assume $B$ is a \cs-subalgebra of $A$ such that $ba b' \in B$ for all $b,b'\in B$ and $a\in A$. Let $c \in A_+$, $d \in B_+$, and $c\leq d$. Consider an approximate unit $(u_\la )$ for $B$. Then $0\leq c\leq d$ implies $0\leq (1-u_\la) c ( 1-u_\la ) \leq (1-u_\la) d (1-u_\la)$, and so
\[
\|c^{1/2}(1-u_\la)\| \leq \|d^{1/2}(1-u_\la)\|, \qquad \forall\, \la.
\]
Hence $c^{1/2}=\lim_\la c^{1/2} (1-u_\la)$, because $d^{1/2}=\lim_\la d^{1/2} (1-u_\la)$. Since $u_\la cu_\la\in B$ for all $\la$, this shows that $c=\lim_\la u_\la c u_\la\in B$. Thus $B$ is hereditary.
\end{proof}
\begin{exercise}
\label{exe:idealplushere}
Let $B$ be a hereditary \cs-subalgebra of a \cs-algebra $A$ and let $I$ be a closed ideal of $A$. Show that $B+I$ is a hereditary \cs-subalgebra of $A$.
\end{exercise}
\begin{corollary}
Let $a$ be a positive element of a \cs-algebra $A$. Then the closure of $aAa$ is the hereditary \cs-subalgebra of $A$ generated by $a$.
\end{corollary}
\begin{proof}
Set $B:=\overline{aAa}$. By the above proposition, it is straightforward to check that $B$ is a hereditary \cs-subalgebra of $A$. Let $(u_\la)$ be an approximate unit for $A$. Then $a^2=\lim_\la au_\la a$, so $a^2\in B$. Since $B$ is a \cs-algebra, $a\in B$ as well. Again, it follows from the above proposition that $B$ is the smallest hereditary \cs-subalgebra of $A$ containing $a$.
\end{proof}
All separable hereditary \cs-subalgebras are of the form described in the above proposition:
\begin{proposition}
Let $B$ be a separable hereditary \cs-subalgebra of a \cs-algebra $A$. Then there exists some $a\in A$ such that $B=\overline {aAa}$.
\end{proposition}
\begin{proof}
Let $(u_n)_{n\in \n}$ be a sequential approximate unit for $B$ and set
\[
a:=\sum_{n=1}^\infty \frac{u_n}{2^n}.
\]
Clearly, $a\in B_+$, and so $\overline{aAa} \sub B$ by the above corollary. On the other hand, for every $n\in \n$, $\frac{u_n}{2^n}\leq a$, so $u_n\in \overline{aAa}$. For every $b\in B$, we have $b=\lim_{n\ra \infty} u_n b u_n$. Therefore by Proposition \ref{prop:eqhereditary}, we have $b\in \overline{aAa}$. Hence $B\sub \overline{aAa}$. This completes the proof.
\end{proof}
\begin{proposition}
Let $B$ be a hereditary \cs-subalgebra of a unital \cs-algebra $A$ and let $a\in A_+$. Assume for every $\ep>0$, there exists some $b\in B_+$ such that $a\leq b+\ep$. Then $a\in B$.
\end{proposition}
\begin{proof} For given $\ep>0$, pick $b_\ep\in B_+$ such that $a\leq b_\ep^2 +\ep^2$. This implies $a\leq b_\ep^2 +\ep^2 +2\ep b_\ep=(b_\ep +\ep)^2$. We also note that $b_\ep + \ep$ is an invertible element in $B$. Therefore $ (b_\ep +\ep)\inv a (b_\ep +\ep)\inv\leq 1$. We also observe that $1-b_\ep (b_\ep +\ep)\inv = \ep(b_\ep +\ep)\inv$. Using these facts, we have
\begin{eqnarray*}
\|a^{1/2} (1-b_\ep (b_\ep +\ep)\inv\|^2&=& \ep^2 \|a^{1/2} (b_\ep +\ep)\inv \|^2\\
&=& \ep^2 \|(b_\ep +\ep)\inv a (b_\ep +\ep)\inv\| \leq \ep^2.
\end{eqnarray*}
Hence we have  $a^1/2=\lim_{\ep\ra 0} a^{1/2} b_\ep (b_\ep +\ep)\inv$, and similarly, $a^1/2=\lim_{\ep\ra 0} (b_\ep +\ep)\inv b_\ep a^{1/2}$. Since all parts of this latter limit are positive elements, by taking adjoint, we get $a= \lim_{\ep\ra 0}(b_\ep +\ep)\inv b_\ep a b_\ep (b_\ep +\ep)\inv$.  Since $B$ is hereditary, by Proposition \ref{prop:eqhereditary}, we have $(b_\ep +\ep)\inv b_\ep a b_\ep (b_\ep +\ep)\inv\in B$, and so $a\in B$.
\end{proof}
The following theorem is a helpful tool to examine the ideal structure of \cs-algebras:
\begin{proposition}
\label{prop:ideashered}
Let $B$ be a hereditary \cs-subalgebra of a \cs-algebra $A$. A subset $J\sub B$ is a closed ideal of $B$ if and only if the exists a closed ideal $I$ of $A$ such that $J=B\cap I$.
\end{proposition}
\begin{proof} Let $J$ be a closed ideal of $B$ and set $I:=AJA$. Then $I$ is a closed ideal of $A$. Using an approximate unit for $J$, one easily sees that $J=J^3$. On the other hand, since $B$ is hereditary, using an approximate unit for $B$ and by applying Proposition \ref{prop:eqhereditary}, we have $B\cap I=BIB$. Again, we apply Proposition \ref{prop:eqhereditary} to obtain
\[
B\cap I = BIB = B(AJA)B= BAJ^3AB\sub BJB=J.
\]
The reverse inclusion is trivial. Also the converse implication is clear.
\end{proof}
To illustrate an application of the above proposition in ideal structure of \cs-algebras, we need a definition:
\begin{definition}
A \cs-algebra $A$ is called {\bf simple} if $0$ and $A$ are its only closed ideals.
\end{definition}
\begin{exercise}
Prove that every non-zero \ss-homomorphism from $\ff:\c\ra \c$ equals identity. Conclude that $\c$ is a simple \cs-algebra.
\end{exercise}
In the above example, it is enough to assume $\ff$ is a non-zero algebraic homomorphism.
\begin{proposition}
\label{prop:simplematrixalg}
For all $n\in \n$, the \cs-algebra $M_n=M_n(\c)$ is simple.
\end{proposition}
\begin{proof}
Let $n\in \n$ is given. Using the standard basis of $\c^n$, we can consider $M_n$ as the algebra of all $n\times n$ matrices with entries in $\c$. Then $M_n$ as a complex vector space is generated by elementary matrices $E_{ij}$ for all $1\leq i,j \leq n$, see Exercise \ref{exe:category}(iv). Let $I\neq 0$ be an ideal of $M_n$. Pick $0\neq T=(t_{ij})\in I$ and assume $t_{rs}\neq0$ for some $1\leq r,s \leq n$. Then it is straightforward to check that $E_{11}= \frac{1}{t_{rs}} E_{1r} T E_{s1} \in I$. Similarly, one checks that
\[
E_{ij}=  E_{i1} E_{11} E_{1j} \in I, \qquad \forall 1\leq i,j \leq n,
\]
and therefore $M_n=I$.
\end{proof}
Again, $M_n$ is simple even as an algebra. We can also rephrase the above proposition by saying that every \ss-homomorphism (or just algebraic homomorphism) from $M_n$ into another \cs-algebra (or just complex algebra) is either one-to-one or zero.
\begin{proposition}
\label{prop:heredsimple}
Hereditary \cs-subalgebras of a simple \cs-algebra are simple.
\end{proposition}
\begin{proof}
Assume $A$ is a simple \cs-algebra, $B$ is a hereditary \cs-subalgebra of $A$, and $J$ is a closed ideal of $B$. Then $J=B\cap I$ for some closed ideal of $A$, by Proposition \ref{prop:ideashered}. Since $A$ is simple, $I=A$ or $I=0$. Therefore $J=B$ or $J=0$.
\end{proof}

We conclude this section with a brief discussion of multiplier algebras of \cs-algebras. In the rest of this section, $A$ is a \cs-algebra. The \cs-unitization of $A$ studied in Section \ref{sec:topalgebras} is the smallest \cs-algebra that contains $A$ as an ideal. However, we have to impose a certain condition to be able to determine the biggest unitization for a \cs-algebra. The phrase `` the biggest unitization'' will be explained shortly.

\begin{definition}
Let $R$ be a ring. A two sided ideal $I$ of $R$ is called an {\bf essential ideal} if every other non-zero ideal of $R$ has a non-zero intersection with $I$.
\end{definition}
One easily sees that $A$ is always an essential ideal of $\widetilde{A}$. We shall show that the ideal $F(H)$ of finite rank operators on a Hilbert space $H$ is an essential ideal of $B(H)$, so is $K(H)$, see Proposition \ref{prop:finiterank3}. In the following example, we describe essential ideals of commutative \cs-algebras.
\begin{example}
\label{exa:opendense}
Let $X$ be locally compact and Hausdorff topological space. An ideal $I$ of $C_0(X)$ is essential if and only if $I=C_0(U)$ for some open and dense subset $U\sub X$. The correspondence between ideals of $C_0(X)$ and open subsets of $X$ has already been discussed in Example \ref{exa:opensubsets}. Assume $U$ is an open but not dense, subset of $X$, so there is an open subset $O\sub X$ such that $O\cap U=\emptyset$. One observes that $C_0(O)\cap C_0(U)= 0$. Hence $C_0(U)$ is not an essential ideal of $C_0 (X)$. Conversely assume $U$ is an open and dense subset of $X$. If $C_0(U')$ is a non-zero ideal of $C_0(X)$, then $U'$ must be non-empty, and so $O=U\cap U'\neq \emptyset$. One observes that $0\neq C_0(O)= C_0(U)\cap C_0(U')$. Therefore $C_0(U)$ is an essential ideal of $C_0(X)$.
\end{example}

\begin{definition}
\label{def:doublecentralizer}
An ordered pair $(L,R)$ of bounded operators on $A$ is called a {\bf double centralizer for $A$} if for every $a,b\in A$, we have
\[
L(ab)=L(a)b, \quad R(ab)=aR(b), \quad\text{and}\quad R(a)b=aL(b).
\]
The set of all double centralizers for $A$ is denoted by $M(A)$.
\end{definition}
\begin{example}
\label{exa:doublecentralizer}
For every $c\in A$, define $L_c, R_c\in B(A)$ by $L_c(a):=ca$ and $R_c(a):=ac$. then the pair $(L_c,R_c)$ is a double centralizer for $A$ and one easily checks that $\|L_c\|=\|R_c\|=\|c\|$.
\end{example}
This suggests the following proposition:
\begin{proposition}
\label{prop:doublecentralizer1}
\begin{itemize}
\item [(i)] If $(L,R)$ is a double centralizer for $A$, then $\|L\|=\|R\|$.
\item [(ii)] If $A$ is a unital \cs-algebra, then every double centralizer for $A$ is of the form $(L_c,R_c)$ for some $c\in A$.
\item [(iii)] $M(A)$ is a closed subspace of $B(A)\oplus B(A)$, where the norm on $B(A)\oplus B(A)$ is defined by  $\|(T, S)\|:=\max\{\|T\|,\|S\|\}$ for all $(T, S)\in B(A)\oplus B(A)$.
\end{itemize}
\end{proposition}
\begin{proof}
\begin{itemize}
\item [(i)] For every $a,b\in A$, we have $\|aL(b)\|=\|R(a)b\|\leq \|R\|\|a\|\|b\|$. Hence
\begin{eqnarray*}
\|L(b)\|&=&\sup \{\|aL(b)\|; \|a\|\leq 1\}\\
&\leq& \sup \{\|R\|\|a\|\|b\|; \|a\|\leq 1\}\\
&\leq& \|R\|\|b\|.
\end{eqnarray*}
This shows that $\|L\|\leq\|R\|$. The reverse inequality follows from a similar argument.
\item [(ii)] Let $A$ be unital and let $(L,R)$ be a double centralizer for $A$. Set $c:=L(1)=R(1)$. Then for every $a\in A$, we have  $L_c(a)=ca=L(1)a=L(a)$ and similarly $R_c(a)=R(a)$.
\item [(iii)] It follows from the fact that all three equations in Definition \ref{def:doublecentralizer} pass the limit by continuity of product.
\end{itemize}
\end{proof}
Regarding the above proposition, it makes sense to define the norm of a double centralizer $(L,R)$ for $A$ by
\[
\|(L,R)\|:=\|L\|=\|R\|.
\]
The scalar product is defined by $\la(L,R):=(\la L, \la R)$ for all $\la\in \c$ and we define the product of two double centralizers $(L_1,R_1), (L_2, R_2)\in M(A)$ by
\[
(L_1,R_1) (L_2, R_2):=(L_1L_2,R_1R_2).
\]
We also define an involution on $M(A)$ by
\[
(L,R)\s:= (R\s, L\s), \qquad \forall (L,R)\in M(A),
\]
where $L\s(a):=L(a\s)\s$ and $R\s(a):=R(a\s)\s$ for all $a\in A$.
\begin{exercise}
Prove that $M(A)$ is an involutive Banach algebra with the above operations.
\end{exercise}
\begin{proposition}
\label{prop:multiplier1}
\begin{itemize}
\item [(i)] The algebra $M(A)$ is a \cs-algebra.
\item [(ii)] The double centralizer $(id_A, id_A)$ is the unit element of $M(A)$.
\item [(iii)] The map $a\mapsto (L_a, R_a)$ is an injective \ss-homomorphism from $A$ into $M(A)$.
\end{itemize}
\end{proposition}
\begin{proof}
\begin{itemize}
\item [(i)] We only need to check the \cs-identity for $M(A)$. Let $(L,R)\in M(A)$. For every $a\in A$ such that $\|a\|\leq 1$, we have
\begin{eqnarray*}
\|L(a)\|^2&=&\|L(a)\s L(a)\|= \|L\s(a\s) L(a)\|\\
&=&\| R( L\s(a\s)) a\|\leq \|R L\s (a\s)\|\\
&\leq& \|RL\s\|=\|(LR\s,RL\s)\|\\
&=& \|(L,R) (R\s, L\s)\|=\|(L,R)(L,R)\s\|
\end{eqnarray*}
Hence
\begin{eqnarray*}
\|(L,R)\|^2&=&\|L\|^2= \sup \{\|L(a)\|^2; \|a\|\leq 1 \} \\
&\leq&  \|(L,R)(L,R)\s\|\leq  \|(L,R)\|^2.
\end{eqnarray*}
Therefore $\|(L,R)\|^2=\|(L,R)(L,R)\s\|$.
\item [(ii)] It is straightforward to check.
\item [(iii)] One easily checks that the map $a\mapsto (L_a, R_a)$ is a \ss-homomorphism and, by Example \ref{exa:doublecentralizer}, it is an isometry. Hence it is injective.
\end{itemize}
\end{proof}
The \cs-algebra $M(A)$ is called the {\bf multiplier algebra of $A$}. By identifying $A$ with its image in $M(A)$, we often consider $A$ as a \cs-subalgebra of $M(A)$.
\begin{proposition}
\label{prop:aessentialinmofa} Every \cs-algebra $A$ is an essential ideal of $M(A)$.
\end{proposition}
\begin{proof}
First, we show that $A$ is an ideal of $M(A)$. For given $(L,R)\in M(A)$ and $c\in A$, set $\alpha:=R(c)$. We claim $(L_\alpha, R_\alpha)=(L_c,R_c)(L,R)$. For every $a\in A$, we compute $L_\alpha(a)=R(c)a=cL(a)=L_c(L(a))$, so $L_\alpha=L_cL$, or equivalently $\|L_\alpha-L_cL\|=0$. This implies $\|R_\alpha-R_cR\|=0$, and so $R_\alpha=R_c R$. This proves our claim and shows that $A$ is a right ideal of $M(A)$. Since $A$ is an involutive subalgebra of $M(A)$, it is a left ideal of $M(A)$ as well.

Now, let $I$ be a non-zero ideal of $M(A)$ and let $0\neq (L,R)\in I$. So there is $a\in A$ such that $x:=L(a)\neq0$. The double centralizer $(L_{x\s}, R_{x\s})(L,R)$ belongs to both $A$ and $I$, since both are ideals of $M(A)$. On the other hand, $L_{x\s} (L(a))= x\s L(a)=x\s x\neq 0$. This shows that $I\cap A\neq 0$.
\end{proof}
\begin{proposition}
\label{prop:multiplier2} Let $I$ be a closed ideal of $A$. Then there is a \ss-homomorphism $\ff:A\ra M(I)$ extending the inclusion $I\hookrightarrow M(I)$. Moreover, $\ff$ is one-to-one if and only if $I$ is essential in $A$. In particular, if a \cs-algebra $B$ contains $A$ as an essential ideal, then there is a one-to-one \ss-homomorphism from $B$ into the multiplier algebra $M(A)$.
\end{proposition}
\begin{proof} For every $a\in A$, the pair $(L_a,R_a)$ is a double centralizer for $I$. Hence we define $\ff: a\mapsto (L_a,R_a)$. It is clear that $\ff$ extends the inclusion $I\hookrightarrow M(I)$.

Assume $\ff$ is one-to-one, then $A$ is identified with a \cs-subalgebra of $M(I)$. Let $J$ be a non-zero ideal of $A$. Since $I$ is contained in $A$, the same argument as the proof of Proposition \ref{prop:aessentialinmofa} shows that $I\cap J\neq 0$.

Conversely, assume that $I$ is essential in $A$. Then the kernel of $\ff$ is a closed ideal in $A$ whose intersection with $I$ is zero. Therefore it has to be zero.
\end{proof}
Since $M(A)$ is unital, the above proposition explains our earlier statement about $M(A)$ being the biggest unitization of $A$. Of course, we have to impose the condition that every unitization contains $A$ as an essential ideal.

\begin{remark}
\label{rem:multiplier}
The above proposition proves a universal property for the multiplier algebra of a \cs-algebra. This universal property is often used to define the multiplier algebras abstractly. Then the algebra of all double centralizers for a \cs-algebras becomes a model for the abstract multiplier algebra. There are other models for multiplier algebras.
\end{remark}
\begin{example}
\label{exa:multipliercommutative}
Let $X$ be a locally compact and Hausdorff topological space. One easily observes that $C_0(X)$ is a closed ideal of $C_b(X)$. We claim that $C_b(X)$ contains $C_0(X)$ as an essential ideal. Let $I$ be an ideal of $C_b(X)$ and $f\in I$ be a non-zero bounded and continuous function from $X$ into $\c$. Pick $x_0\in X$ such that $f(x_0)\neq 0$. Hence there exists an open neighborhood $U$ around $x$ such that $f(x)\neq 0$ for all $x\in U$. Since $X^\infty$ is a normal space, by Urysohn's lemma, there is a continuous function $g:X\ra [0,1]$ such that $g(x_0)=1$ and $supp(g)\sub U$. Clearly, $0\neq gf\in C_0(X)\cap I$. This proves our claim.

Therefore by Proposition \ref{prop:multiplier2}, there is a one-to-one \ss-homomorphism $\ff:C_b(X)\ra M(C_0(X)$. We prove that $\ff$ is onto, and so an isomorphism. To prove this it is enough to show that, for every $c\in M(C_0(X))_+$, there is $g\in C_0(X)$ such that $\ff(g)=c$. Let $(u_\la)$ be an increasing approximate unit for $C_0(X)$. then for every $x\in X$, the net $(cu_\la(x))$ lies in $[0,1]$ and we have $cu_\la(x)\leq \|cu_\la\|_{\sup} \leq \|c\|_{\sup}$ for all $\la$, so this net is bounded above. Similarly, one checks that this net is increasing. This shows that this net is convergent, and consequently we can define $g:X\ra \c$ by $g(x):=\lim_\la cu_\la (x)$. Clearly, $g$ is a non-negative bounded function on $X$. Moreover, for every $f\in C_0(X)$, we have
\[
gf(x)=\lim_\la cu_\la(x) f(x)=c\lim_\la u_\la(x) f(x)= (c\lim_\la u_\la f)(x)=cf(x).
\]
Hence $gf=cf\in C_0(X)$.

Next, we show that it is also continuous and so a member of $C_b(X)$. Let $(x_\mu)$ be a net in $X$ convergent to a point $x_0$. We choose a compact neighborhood $K$ around $x_0$ and assume $(x_\mu)$ lies in $K$. By Urysohn's lemma, there is a continuous function $h\in C_0(X)$ such that $h=1$ on $K$. Since $hg=gh\in C_0(X)$, we have
\[
g(x_0)=hg(x_0)= hg(\lim_\mu x_\mu)=\lim_\mu hg(x_\mu)=\lim_\mu g(x_\mu).
\]
Therefore $g\in C_b(X)$. For every $f\in C_0(X)$, we have
\[
\ff(g)f=\ff(g)\ff(f)= \ff(gf)=gf= cf,
\]
and similarly $f\ff(g)=fc$. These two equalities show that
\begin{equation}
\label{eqn:multipliercomm}
(\ff(g)- c)C_0(X)=0=C_0(X)(\ff(g)-c).
\end{equation}
If $\ff(g)-c\neq 0$, let $I$ be the non-zero ideal of $M(C_0(X))$ generated by $\ff(g)-c$. It follows from (\ref{eqn:multipliercomm}) that $IC_0(X)=0$. But this contradicts with the fact that $C_0(X)$ is an essential ideal of $M(C_0(X))$. Therefore $c=\ff(g)$. This shows that $\ff$ is onto.
\end{example}
As another example for multiplier algebras, we shall show that $B(H)$ is the multiplier algebra of the \cs-algebra $K(H)$ of compact operators on a Hilbert space $H$, see Proposition \ref{prop:multipliercompact}.

\section{Problems}

\begin{e}
\label{e:4-1}
\begin{itemize}
\item [(i)] Find an ideal of the commutative \cs-algebra $C_0(\r)$ that is not closed.
\item [(ii)] Find an ideal of the commutative \cs-algebra $C_0(\c)$ that is not self adjoint.
\end{itemize}
\end{e}
\begin{e}
\label{e:4-2}
Let $X$ be a locally compact and Hausdorff topological space.
\begin{itemize}
\item [(i)] Show that $X$ is {\bf $\si$-compact}, namely $X$ can be covered by a sequence of its compact subsets, if and only if there exists some $f\in C_0(X)$ such that $f(x)>0$ for all $x\in X$.
\item [(ii)] Assume $X$ is $\si$-compact and $f$ is a function as described in the above.  Show that $\{f^{1/n}\}_{n\in \n}$ is a sequential approximate unit for $C_0(X)$. Describe a condition on $f$ which implies that $\{f^{1/n}\}_{n\in \n}$ is an increasing approximate unit for $C_0(X)$.
\item [(iii)] Show that $C_0(X)$ is $\si$-unital if and only if $X$ is $\si$-compact.
\end{itemize}
\end{e}
\begin{e}
\label{e:4-3}
Let $\ff:A\ra B$ be an isometric linear map between two unital \cs-algebras such that $\ff(a\s)=\ff(a)\s$ for all $a\in A$ and $\ff(1)=1$. Show that $\ff$ is a positive map.
\end{e}
\begin{e}
\label{e:4-4}
Let $A$ be a \cs-algebra. A semi-norm $N$ on $A$ is called a {\bf \cs-semi-norm} if for every $a,b\in A$, we have
\[
N(a)\leq \|a\|, \quad N(ab)\leq N(a)N(b), \quad N(a\s a)=N(a)^2.
\]
The set of all \cs-semi-norms on $A$ is denoted by $\mathcal{N}(A)$.
\begin{itemize}
\item [(i)] Equip $\mathcal{N}(A)$ with the point-wise convergence topology, namely a net $(N_i)$ of \cs-semi-norms is convergent to a \cs-semi-norm $N$ if and only if $N_i(a)\ra N(a)$ for all $a\in A$. Show that $\mathcal{N}(A)$ is compact in this topology.
\item [(ii)] For a closed ideal $I$ of $A$, define a \cs-semi-norm $N_I$ by $N_I(a):=\|\pi (a)\|$ for all $a\in A$, where $\pi:A\ra A/I$ is the natural quotient map. Show that the correspondence $I\rightsquigarrow N_I$ is a bijective correspondence between the set all closed ideals of $A$, which we denote it by $\mathcal{I}(A)$, and $\mathcal{N}(A)$.
\item [(iii)] For $I,J\in \mathcal{I}(A)$, show that $N_{I\cap J}=\sup\{N_I, N_J\}$.
\item [(iv)] A \cs-semi-norm $N$ is called {\bf extremal} if $N_I\leq N$ and $N_J\leq N$ for $I,J\in \mathcal{I}(A)$ imply that either $N_I=N$ or $N_J=N$. An ideal $I$ is called {\bf prime} if $J_1 J_2\sub I$ for two ideals $J_1, J_2$ of $A$ implies that either $J_1\sub I$ or $J_2\sub I$. Show that a closed ideal $I$ of $A$ is prime  if and only if the \cs-semi-norm $N_I$ is non-zero and extremal.
\end{itemize}
\end{e}






\chapter{Bounded operators on Hilbert spaces}
\label{ch:Hilbertspaces}

Every closed involutive subalgebra of the algebra $B(H)$ of bounded operators on a Hilbert space $H$ is a \cs-algebra. These \cs-algebras are known as {\bf concrete \cs-algebras}, in contrast with abstract \cs-algebras which are involutive Banach algebras whose norms satisfy \cs-identity. In this chapter, we present basic definitions and results concerning concrete \cs-algebras. Naturally, our discussion intersects with the general theory of operator algebras on Hilbert spaces, but we avoid a comprehensive, or even a moderate, study of operator algebras here.

We begin with Hilbert spaces in Section \ref{sec:Hilbertspaces} and cover basic notions and materials about Hilbert spaces necessary for our purposes. This includes various identities and inequalities in Hilbert spaces, orthogonality, various examples and constructions of Hilbert spaces as well as weak topology in Hilbert spaces. We also discuss orthonormal bases for Hilbert spaces. In section \ref{sec:boundedoperators}, we study the elementary topics about bounded operators on Hilbert spaces such as sesquilinear forms, adjoint operators, invertibility and finite rank operators. We conclude this section with introducing the commutant of a subset of $B(H)$ and definition of a von Neumann algebra.

In Section \ref{sec:concreteexamples}, we discuss three important examples of concrete \cs-algebras; the reduced group \cs-algebra of a locally compact group $G$, the \cs-algebra $L^\infty(X,\mu)$ acting on $L^2(X,\mu)$, where $(X,\mu)$ is a measure space, and the Toeplitz algebra. There are many locally convex topologies, besides the norm topology, on the \cs-algebra $B(H)$ which reveal different features of this \cs-algebra. In Section \ref{sec:topologies}, we study three major topologies on $B(H)$; the strong, weak and strong-$\s$ operator topologies. They are compared to each other and many results concerning convergence in these topologies are proved.

The Borel functional calculus in $B(H)$ is presented in Section \ref{sec:Borelfunctionalcalculus}. Section \ref{sec:projections} is devoted to projections in $B(H)$. After presenting basic materials about projections, we prove the polar decomposition of elements of $B(H)$. In Section \ref{sec:compactoperators}, \cs-algebras of compact operators are studied briefly. Finally, we conclude this chapter with a short section about von Neumann algebra, which is devoted to the bicommutant theorem.

In this chapter all vector spaces are over the field $\c$ of complex numbers.

\section{Hilbert spaces}
\label{sec:Hilbertspaces}

Hilbert spaces are characterized by the cardinality of their orthonormal bases. Our main goal in this section is to cover enough basic materials from the theory of Hilbert spaces to prove this statement. Along the way, many useful results, techniques and examples are presented too. We also explain briefly the weak topology of Hilbert spaces.

\begin{definition}
Let $E$ be a vector space. A {\bf sesquilinear form on $E$} is a function $\inner: E\times E\ra \c$ such that, for all $x,y, z\in E$ and $\la\in \c$, we have
\begin{itemize}
\item [(i)] $\lan \la x+ y, z\ran=\la \lan x, z\ran+ \lan y, z\ran$, ($\inner$ is linear in its first variable), and
\item [(ii)] $\lan y, x\ran=\overline{\lan x, y \ran}$, ($\inner$ is {\bf conjugate-symmetric}).
\end{itemize}
It is called a {\bf pre-inner product} if it is also a {\bf positive form}, that is
\begin{itemize}
\item [(iii)] $\lan x, x\ran\geq 0$, for all $x\in E$.
\end{itemize}
A pre-inner product is called an {\bf inner product} if it is also {\bf definite}, that is
\begin{itemize}
\item [(iv)] $\lan x, x\ran=0$ if and only if $x=0$.
\end{itemize}
\end{definition}
A pre-inner product comes with many tools that are helpful in this chapter.
\begin{proposition}
\label{prop:pre-inner}
Let $\inner$ be a  pre-inner product on a vector space $E$. Define a function $\|-\|:E\ra [0,\infty)$ by $\|x\|:=\lan x, x \ran^{1/2}$ for all $x\in E$. $\|-\|$ is a semi-norm on $E$ and it possesses the following properties, for all $x, y \in E$:
\begin{itemize}
\item [(i)] $|\lan x, y\ran|\leq \|x\| \|y\|$, (the {\bf Cauchy-Schwartz inequality} or briefly {\bf CS inequality}),
\item [(ii)] $\|x+y\|^2 + \|x- y \|^2 = 2(\|x\|^2 +\|y\|^2)$, ( the {\bf parallelogram law}),
\item [(iii)] $4\lan x,y\ran=\|x+y\|^2- \|x-y\|^2 +i \|x+iy\|^2 -i \|x-iy\|^2 $, ( the {\bf polarization identity}).
\item [(iv)] If $\inner$ is an inner product, then $\|-\|$ is a norm.
\end{itemize}
Moreover, we have $\|x\|=\sup\{|\lan x, y\ran |; \|y\|=1 \}$ for all $x\in E$.
\end{proposition}
\begin{proof} To show $\|-\|$ is a semi-norm, we only prove the {\bf triangle inequality}, that is $\|x+y\|\leq \|x\| +\|y\|$ for all $x,y\in E$. And to prove this, we have to use CS inequality.
\begin{eqnarray*}
\|x+y\|^2&=& \lan x+y, x+y\ran\\
&=& \lan x, x\ran+ \lan y, y\ran+ \lan x, y\ran + \lan y, x\ran\\
&\leq& \|x\|^2 + \| y\|^2 + |\lan x, y\ran| + |\lan y, x\ran|\\
&\leq& \|x\|^2 + \| y\|^2 + 2\|x\| \|y\|\\
&=& (\|x\| + \| y\|)^2.
\end{eqnarray*}
The rest of the properties of a semi-norm are left to the reader.
\begin{itemize}
\item [(i)] By expanding the inequality $\lan x+\alpha y, x+\alpha y\ran \geq 0$, we get
\[
\|x\|^2 +\overline{\alpha} \lan x,y \ran + \alpha \lan y, x \ran + |\alpha|^2 \|y\|^2\geq 0.
\]
By putting $\alpha=\frac{-\lan x,y \ran} {\|y\|^2}$ when $\|y\|\neq 0$, the desired inequality is obtained. When $\|y\|=0$ but $\|x\|\neq0$, a similar arguments works. When $\|y\|=\|x\|=0$, using Parts (ii) and (iii), one can show that $\lan x, y \ran=0$.
\item [(ii)] It is proved easily by expanding the left hand side of the identity.
\item [(iii)] It is proved easily by expanding the right hand side of the identity.
\item [(iv)] It is clear.
\end{itemize}
Finally, it is clear from CS inequality that $\sup\{|\lan x, y\ran |; \|y\|=1 \}\leq \|x\|$. If $\|x\|=0$, the reverse inequality is clear too. If $\|x\|\neq 0$, we have $\|x\|=\lan x, \frac{x}{\|x\|}\ran$ and $\|\frac{x}{\|x\|} \|=1$. This proves the reverse inequality.
\end{proof}
An immediate consequence of CS inequality is the following corollary:
\begin{corollary}
\label{cor:innercont}
Let $\inner$ be a pre-inner product on a vector space $E$. Then $\inner$ is jointly continuous.
\end{corollary}

\begin{corollary}
\label{cor:innerzero}
Let $\inner$ be an inner product on a vector space $E$. If $\lan x,y\ran=0$ for all $y\in E$, then $x=0$.
\end{corollary}
\begin{proof}
Use the equality  $\|x\|=\sup\{|\lan x, y\ran |; \|y\|=1 \}$  and the fact that $\|-\|$ is a norm.
\end{proof}
\begin{proposition}
\label{prop:csequality}
Let $\inner$ be an inner product on a vector space $H$ and let $x,y\in H$. Then the following statements are equivalent:
\begin{itemize}
\item [(i)] $\|x+y\|=\|x\|+\|y\|$.
\item [(ii)] $\lan x,y\ran =\|x\| \|y\|$.
\item [(iii)] One of $x$ and $y$ is non-negative scaler multiple of the other one.
\end{itemize}
\end{proposition}
\begin{proof} Assume (i) holds. We compute
\begin{eqnarray*}
\|x\|^2 +\|y\|^2+2 \|x\|\|y\| &=& (\|x\| +\|y\|)^2= \|x+y\|^2\\
&=& \lan x+y , x+y\ran= \|x\|^2+\|y\|^2+ \lan x,y\ran +\lan y,x\ran\\
&=& \|x\|^2+\|y\|^2+ 2Re\lan x,y\ran.
\end{eqnarray*}
Using CS inequality, we obtain
\[
|\lan x,y \ran| \leq \|x\| \|y\| = Re\lan x,y \ran \leq |\lan x,y \ran|.
\]
This proves (ii).

Let (ii) hold. then for $a,b\in \r$, we have
\begin{eqnarray*}
\|ax+by\|^2&=& a^2\|x\|^2 + 2abRe\lan x, y\ran +b^2\|y\|^2 \\
&=& a^2\|x\|^2 + 2ab\|x\| \|y\| +b^2\|y\|^2 \\
&=& (a\|x\|+b\|y\|)^2.
\end{eqnarray*}
By setting $a=\|y\|$ and $b=-\|x\|$, we obtain $\|y\|x -\|x\| y =0$. If $x=0$, then $x=0 y$. If $x\neq 0$, then $y=\frac{\|y\|}{\|x\|} x$. This shows (iii).

Assume (iii) holds, so $x=ay$ for some $a\geq0$, then $\|x+y\| = \|(a+1)y\| = (a+1)\|y\|= \|ay\| +\|y\|=\|x\|+\|y\|$. This shows (i).
\end{proof}

\begin{corollary}
\label{cor:csequality}
Let $\inner$ be an inner product on a vector space $H$ and let $x,y\in H$. Then $|\lan x,y\ran|=\|x\| \|y\|$ if and only if $x$ and $y$ are linearly dependent.
\end{corollary}
\begin{proof}
Assume $|\lan x, y \ran|=\|x\|\|y\|$, then there exits some $\la\in \c$ such that $|\la|=1$ and $\la\lan x, y\ran=\|x\|\|y\|$, or equivalently $\lan \la x, y\ran =\|\la x\|\|y\|$. By the above proposition, $x$ and $y$ are linearly dependent.

Conversely, assume $x=\la y$ for some $\la\in \c$, then
\[
|\lan x,y\ran|= |\la \lan y,y\ran|=|\la| \|y\|^2= \|x\| \|y\|.
\]
\end{proof}

\begin{definition}
\label{def:unitaryequ}
\begin{itemize}
\item [(i)] A vector space $H$ equipped with an inner product $\inner$ is called a {\bf pre-Hilbert space}. If $H$ is complete with respect to the norm defined by  $\inner$, then $H$ is called a {\bf Hilbert space}. Such a Hilbert space is often  denoted by the ordered pair $(H, \inner)$.
\item [(ii)] A {\bf unitary equivalence} from a Hilbert space $(H_1 , \inner_1)$ into another Hilbert space $(H_2 , \inner_2)$ is a linear isomorphism $u:H_1\ra H_2$ which preserves the inner product structures, that is
\[
\lan u(x), u(y)\ran_2= \lan x, y\ran_1, \quad \forall x,y \in H_1.
\]
If there is a unitary equivalence between two Hilbert spaces $(H_1 , \inner_1)$ and $(H_2 , \inner_2)$, we call them {\bf unitary equivalent} and denote this by $H_1\simeq H_2$. If $u$ as above preserves the inner products but is not necessarily onto, we call it a {\bf unitary injection}.
\end{itemize}
\end{definition}
\begin{exercise}
Show that a linear map $\ff:H_1\ra H_2$ between two Hilbert spaces is a unitary injection if and only if it is an isometry.
\end{exercise}
\begin{exercise}
Show that every unitary equivalence is continuous and its inverse is a unitary equivalence too.
\end{exercise}
\begin{example}
\label{exa:Hilbertspaces}
\begin{itemize}
\item [(i)] Every finite dimensional pre-Hilbert space is a Hilbert space, because every finite dimensional normed space is complete, see Theorem 1.21 od \cite{rudinfunctional}. Therefore $H=\c^n$ equipped with the standard inner product, that is $\langle x, y\rangle:=\sum_{i=1}^n x_iy_i$ for all $x=(x_1,\cdots,x_n), y=(y_1,\cdots, y_n)\in \c^n$, is a Hilbert space.
\item [(ii)] Let $(X, \mu)$ be a measure space. Consider the vector space $L^2(X,\mu)$ (or briefly $L^2(X)$) consisting of all {\bf square integrable functions on $X$} with respect to $\mu$, that is
\[
L^2(X,\mu):= \left\{ f:X\ra \c; \int_X |f(x)|^2 d\mu(x) <\infty\right\}.
\]
Define $\inner: L^2(X,\mu) \times L^2(X,\mu) \ra \c$ by
\[
\lan f, g \ran:=\int_X f(x) \overline{g(x)} d\mu(x), \qquad \forall f,g\in L^2(X,\mu).
\]
It is easy to show that $\inner$ is a pre-inner product. Set
\[
N:=\{f\in L^2(X,\mu); \lan f, f\ran=0\}.
\]
Then $N$ is a subspace of $L^2(X,\mu)$. For every pair $(f+N, g+N)$ of elements of the quotient space $L^2(X,\mu)/N$, we define $\lan f+N, g+N\ran:= \lan f,g \ran$. It is straightforward to check that this function is well defined, and in fact, it is an inner product on the quotient space. We do not change the notation and denote the pre-Hilbert space obtained in this way by $(L^2(X,\mu), \inner)$. It is actually shown that $L^2(X,\mu)$ is a Banach space with respect to the norm defined by $\inner$ and so is a Hilbert space, see Theorem 6.6 of \cite{folland-ra} or Problem \ref{e:5-2}. The norm defined by this inner product is usually denoted by $\norm_2$, that is
\[
\|f\|_2:= \left(\int_X |f(x)|^2 d\mu(x)\right)^{1/2}, \quad \forall f\in L^2(X,\mu).
\]

When $\mu$ is the counting measure, $N=0$ and we use the notation $\ell^2(X)$ in lieu of $L^2(X,\mu)$. In this case, we also use summation in stead of integral. For example, the inner product of $\ell^2(\n)$ is defined by
\[
\lan (a_1, a_2, \cdots),(b_1, b_2, \cdots)\ran:= \sum_{n=1}^\infty a_n\overline{b_n}.
\]
The Hilbert space $\ell^2(\n)$ has a specific feature and is denoted simply by $\ell^2$ in many books. We employ this simple notation too. One also easily observes that whenever $X=\{1, \cdots, n\}$, we obtain the Hilbert space introduced in Item (i).

We can generalize this class of Hilbert spaces further by replacing $\c$ with an arbitrary Hilbert space. Let $(H,\inner)$ be a Hilbert space and let $(X,\mu)$ be a measure space as before. A map $f:X\ra H$ is called {\bf weakly measurable} if the map $x\mapsto \lan f(x), h\ran$ is measurable for all $h\in H$. Define $L^2(X, \mu, H)$ (or simply $L^2(X,H)$) to be the vector space of all weakly measurable maps $f:X\ra H$ such that they are square integrable, that is
\[
\|f\|^2_2:=\int_X\|f(x)\nt d\mu(x)<\infty.
\]
The inner product for this Hilbert space is defined as follows:
\[
\lan f,g \ran:=\int_X\lan f(x), g(x)\ran d\mu(x), \qquad \forall f,g\in L^2(X, \mu, H).
\]
Of course, again, we have to consider the quotient of $L^2(X, \mu, H)$ modulo the subspace of all null functions with respect to this inner product and the measure $\mu$.

\item [(iii)] Let $A$ be a \cs-algebra and let $\ff:A \ra \c$ be a \ss-homomorphism. Then $\ff$ is positive, and so $\ff(aa\s)\geq 0$ for all $a\in A$. We define $\inner_\ff: A\times A\ra \c$ by $\lan a, b \ran_\ff:= \ff(ab\s)$. One checks that $\inner_\ff$ is a pre-inner product on $A$. Define
\[
N_\ff:=\{ a\in A; \ff(aa\s)=0\}.
\]
Then the form $\inner_\ff$ can be defined similarly over the quotient space $A/N_\ff$ and the pair $(A/N_\ff, \inner_\ff)$ is a pre-Hilbert space. The completion of $A/N_\ff$ with respect to the norm defined by $\inner_\ff$ is denoted by $H_\ff$ and the Hilbert space $(H_\ff, \inner_\ff)$ is called the {\bf Hilbert space induced by $\ff$}. The Hilbert spaces obtained in this way play the key role in GNS construction.
\end{itemize}
\end{example}
In the following, we explain some useful constructions on Hilbert spaces :

\begin{example}
\label{exa:constructions}
\begin{itemize}
\item[(i)] Let $(H_1,\inner_1)$ and $(H_2,\inner_2)$ be two Hilbert spaces. We endow the vector space $H_1\oplus H_2$ with the following inner product:
\[
\lan x_1\oplus x_2, y_1\oplus y_2\ran:=\lan x_1, y_1\ran_1 +\lan x_2, y_2\ran_2, \quad \forall x_1\oplus x_2, y_1\oplus y_2\in H_1\oplus H_2.
\]
One easily checks that $(H_1\oplus H_2, \inner)$ is a Hilbert space. It is called the {\bf direct sum of $H_1$ and $H_2$}.

Now, let $\{ (H_\la,\inner_\la); \la\in \Lambda\}$ be a collection of Hilbert spaces index by a set $\Lambda$. Let $\bigoplus_{\la\in \Lambda} H_\la$ denote the set of all elements $(h_\la)$ in the direct product $\prod_{\la\in \Lambda} H_i$ such that
\[
\sum_{\la\in \Lambda} \lan h_\la, h_\la \ran_\la < \infty.
\]
It is straightforward to check that $\bigoplus_{\la\in \Lambda} H_\la$ with the following inner product is a Hilbert space:
\[
\lan (x_\la), (y_\la) \ran:= \sum_{\la\in \Lambda} \lan x_\la, y_\la\ran_\la, \qquad \forall (x_\la), (y_\la)\in \bigoplus_{\la\in \Lambda} H_\la.
\]
This Hilbert space is called the {\bf direct sum of the family $\{ (H_\la,\inner_\la); \la\in \Lambda\}$ of Hilbert spaces}. If all the Hilbert spaces in this family are the same Hilbert space $H$, then this direct sum is called the {\bf amplification of $H$ by the cardinality of $\Lambda$} and it is denoted by $H^\Lambda$. If $|\Lambda|=n$ (resp. $\Lambda$ is countably infinite), $H^\Lambda$ is denoted by $H^n$ (resp. $H^\infty$). One easily checks that the algebraic direct sum of the vector spaces $\{ H_\la; \la\in \Lambda\}$ is dense in $\bigoplus_{\la\in \Lambda} H_\la$.

\item[(ii)] Let $(H_1,\inner_1)$ and $(H_2,\inner_2)$ be two Hilbert spaces. Denote the algebraic tensor product of $H_1$ and $H_2$ over $\c$ by $H_1\odot H_2$. In order to define inner product on this vector space, we have to take a closer look at its structure.
\[
H_1\odot H_2:=\frac{\lan h_1\otimes h_2;\, h_1\in H_1,\, h_2\in H_2\ran }{N},
\]
where the subspace $N$ is defined in a specific way to imply various properties of the tensor product. More precisely, $N$ is generated by the following types of elements:
\begin{itemize}
\item [(a)] $(a+b)\otimes c-[a\otimes c+ b\otimes c]$ for all $a,b\in H_1$ and $c\in H_2$,
\item [(b)] and similar relation for the second variable; $a\otimes (c+d) -[a\otimes c+ a\otimes d]$ for all $a\in H_1$ and $c,d\in H_2$,
\item [(c)] $(\la a)\otimes c-\la(a\otimes c)$ for all $a\in H_1$, $c\in H_2$ and $\la\in \c$,
\item [(d)] and similar relation for the second variable; $a\otimes (\la c)-\la(a\otimes c)$ for all $a\in H_1$, $c\in H_2$ and $\la\in \c$.
\end{itemize}
Elements of the form $h_1\otimes h_2$, where $h_1\in H_1$ and $h_2\in H_2$, are called {\bf simple tensors}. We first define inner product for simple tensors, then we extend it linearly to the vector space generated by them, and afterwards, we show briefly that the inner product preserves the relations defining $N$, and so it is well defined over tensor product. Finally, one easily checks that the inner product obtained in this way satisfies all the axioms of an inner product. We define
\[
\lan a\otimes c , b\otimes d\ran:= \lan a, b \ran_1 \lan c,d\ran_2, \qquad  \forall a,b\in H_1, c,d\in H_2.
\]
Let us check the relation described in Item (a) in the above. For all $a,b,h_1\in H_1$ and $c,h_2\in H_2$, we compute
\begin{eqnarray*}
\lan (a+b)\otimes c, h_1\otimes h_2\ran &=& \lan a+b, h_1\ran_1 \lan c, h_2\ran_2\\
&=& (\lan a, h_1\ran_1 + \lan b, h_1\ran_1 ) \lan c, h_2\ran_2\\
&=& \lan a, h_1\ran_1 \lan c, h_2\ran_2+\lan b, h_1\ran_1 \lan c, h_2\ran_2\\
&=& \lan a\otimes c, h_1\otimes h_2\ran + \lan b\otimes c, h_1\otimes h_2\ran\\
&=& \lan a\otimes c+ b\otimes c, h_1\otimes h_2\ran.
\end{eqnarray*}
One easily checks the rest of the relations as well as the fact that $\inner$ is an inner product on $H_1\odot H_2$. The completion of $H_1\odot H_2$ with respect to the norm defined by this inner product is called the {\bf tensor product of $H_1$ and $H_2$} and is denoted by $H_1\otimes H_2$.
\end{itemize}
\end{example}
The reader is strongly advised to check all the details of the above examples. There are also some exercises at the end of this chapter related to these examples and constructions.

The key concept in Hilbert spaces is {\bf Orthogonality} which determines the geometric, analytical, and (somehow) algebraic behavior of Hilbert spaces and their algebras of operators.
\begin{definition}
\label{def:basisHilbert}
Let $(H,\inner)$ be a Hilbert space, (or a pre-Hilbert space) and let $\|-\|$ denote the norm defined by the inner product.
\begin{itemize}
\item [(i)] A subset $S\sub H$ is called an {\bf orthogonal set} if $0\notin S$ and $\lan x, y\ran=0$ for every pair $x\neq y$ in $S$.
\item [(ii)] A subset $S\sub H$ is called {\bf normal} if $\|x\|=1$ for all $x \in S$. A subset $S\sub H$ is called {\bf orthonormal} if it is normal and orthogonal.
\item [(iii)] Let $S$ be a subset of $H$. A vector $y\in H$ is called {\bf orthogonal to $S$} if $\lan x,y \ran=0$ for all $x\in S$. This is denoted by $S\perp y$ (or equivalently $y\perp S$). The set of all vectors orthogonal to $S$ is called the {orthogonal complement of $S$} and is denoted by $S^\perp$. Two subsets $S$ and $R$ of $H$ are called {\bf orthogonal} if $\lan x, y\ran=0$ for all $x\in S$ and $y\in R$.
\end{itemize}
\end{definition}
For every subset $S\sub H$, one easily sees that $S^\perp$ is a closed subspace of $H$.
\begin{lemma}
\label{lem:orthogonalprojection}
Let $X$ be a closed subspace of a Hilbert space $H$. For every given $h\in H$, there is a unique vector $x_h\in X$ such that
\begin{equation}
\label{eqn:closestpoint}
\|h-x_h\|\leq \|h-y\|, \qquad \forall  y\in X.
\end{equation}
Moreover, we have
\[
Re\lan x_h, h-x_h\ran \geq Re\lan y, h-x_h\ran, \qquad \forall y\in X.
\]
\end{lemma}
\begin{proof}
Set $d:=\inf\{ \|h-y\|; y\in X\}$ and let $\{ y_n\}$ be a sequence in $X$ such that $\|h-y_n\|\ra d$. For all $m,n\in \n$, using parallelogram law for $h-y_m$ and $h-y_n$, we obtain
\begin{equation}
\label{eqn:orthoprojection}
\|2h-y_n - y_m\|^2 +\|y_n-y_m\|^2 =2(\|h-y_n\|^2 +\|h-y_m\|^2).
\end{equation}
Since $\frac{y_n+y_m}{2}\in X$, we have $4d^2\leq 4\|h- \frac{y_n+y_m}{2}\|^2$. Thus we obtain
\[
\|y_n-y_m\|\leq 2(\|h-y_n\|^2 +\|h-y_m\|^2) -4d^2.
\]
When $\min\{m,n\}\ra \infty$, the right hand side of the above inequality tends to zero. This shows that $\{ y_n\}$ is a Cauchy sequence. Since $X$ is closed, $\{ y_n\}$ converges to some $x_h\in h$. It is clear that $\|h-x_h\|= d$. Assume $x$ is another vector in $X$ such that $\|h-x\|= d$. If we replace $y_n$ and $y_m$ in (\ref{eqn:orthoprojection} ) by $x_h$ and $x$, then we get $\|x_h-x\|=0$. This proves the uniqueness of $x_h$.

Using (\ref{eqn:closestpoint}), for every $t\in (0,1)$ and for all $y\in X$, we have
\begin{eqnarray*}
\|h-x_h\|^2&\leq & \|h-x_h-t(y-x_h)\|^2\\
&=& \|h-x_h\|^2 - 2tRe\lan y-x_h, h-x_h\ran + t^2\|y-x_h\|^2.
\end{eqnarray*}
Hence
\[
2Re\lan y-x_h, h-x_h\ran \leq t\|y-x_h\|^2,
\]
for all $t\in (0,1)$. When $t\ra 0$, we get the desired inequality; $Re\lan x_h, h-x_h\ran \geq Re\lan y, h-x_h\ran$ for all $y\in X$.
\end{proof}
The unique vector $x_h$ associated to $h\in H$ in the above lemma is called the {\bf orthogonal projection of $h$ on $X$}.

The following notations are useful. Let $E$ be a topological vector space and let $S$ be subset of $E$. The set of all linear combination of elements of $S$ are denoted by $\lan S \ran$ and is called the {\bf span of $S$}. The {\bf closed span of $S$}, i.e. the closure of the span of $S$, is denoted by $[S]$. It is the smallest closed subspace containing $S$. A subset $S$ of $E$ is called a {\bf total set of $E$} if $E=[S]$.
\begin{corollary}
\label{cor:perp}
Let $(H,\inner)$ be a Hilbert space and let $X$ be a subset of $H$.
\begin{itemize}
\item [(i)] If $X$ is a closed subspace of $H$, then both $X$ and $X^\perp$ equipped with inner products inherited from $H$ are Hilbert spaces and $H\simeq X \oplus X^\perp$.
\item [(ii)] $X\sub (X^\perp)^\perp$. Moreover, $X$ is a closed subspace of $H$ if and only if $X= (X^\perp)^\perp$.
\item [(iii)] If $X$ is a closed subspace of $H$, then $X=H$ if and only if $X^\perp=\{0\}$. Generally, $X^\perp= \{0\}$ implies that $\lan X\ran$ is dense in $H$.
\item [(iv)] $[X]=(X^\perp)^\perp$.
\end{itemize}
\end{corollary}
\begin{proof}
\begin{itemize}
\item [(i)] Assume $X$ is a closed subspace of $H$. It is straightforward to check that $X$ and $X^\perp$ are Hilbert spaces and that $X\cap X^\perp=\emptyset$. For every $h\in H$, let $x_h$ be the orthogonal projection of $h$ on $X$ and set $y_h:=h-x_h$. then for every $x\in X$, we have $Re\lan x, y_h\ran\leq Re\lan x_h, y_h\ran$. For arbitrary $\la\in \c$, we replace $x$ with $\la x$ in the latter inequality and we get
\[
Re\, \la \lan x, y_h\ran\leq Re\lan x_h,y_h\ran, \qquad \forall x\in X, \la \in \c.
\]
Hence $\lan x, y_h\ran=0$ for all $x\in X$, i.e. $y_h\in X^\perp$. Therefore the decomposition $h=x_h+y_h$ shows that $H$ is the direct sum of $X$ and $X^\perp$.
\item [(ii)] A simple calculation shows that $X\sub (X^\perp)^\perp$. Assume that $X$ is a closed subspace of $H$ and $h\in (X^\perp)^\perp$. Then $h=x_h + y$, where $x_h$ is the orthogonal projection of $h$ on $X$ and $y=h-x_h\in X^\perp$. Since both $h$ and $x_h$ belong to $(X^\perp)^\perp$, $y\in (X^\perp)^\perp$. Then we have $y=0$ because $X^\perp\cap(X^\perp)^\perp=\{0\}$. This shows $h=x_h\in X$. Therefore $(X^\perp)^\perp\sub X$.
\item [(iii)] It is clear from (i).
\item [(iv)] It is clear from (ii).
\end{itemize}
\end{proof}
\begin{exercise}
\label{exe:Pythagorean}
Let $x,y$ be two elements of a Hilbert space $H$ and let $x\perp y$. Prove {\bf Pythagoras' identity};
\[
\|x+y\|^2=\|x\|^2+\|y\|^2.
\]
Extend this identity for summation of $n$ pairwise orthogonal elements of a Hilbert space.
\end{exercise}

\begin{exercise}
\label{exe:orthogonalproj}
Let $X$ be a closed subspace of a Hilbert space $H$ and let the map $H\ra X$, $h\mapsto x_h$ be the orthogonal projection on $X$.
\begin{itemize}
\item [(i)] Show that $x_h=0$ if and only if $h\in X^\perp$.
\item [(ii)] Show that $x_h=h$ if and only if $h\in X$.
\end{itemize}
\end{exercise}
A subset $S$ of a topological vector space is called {\bf linearly independent} if every finite subset of $S$ is linearly independent. Part (ii) of the following proposition can be regarded as the generalization of Pythagoras' identity for infinite sums:
\begin{proposition}
\label{prop:phyth}
Let $H$ be a Hilbert space and let $B$ be an orthogonal subset of $H$. Then
\begin{itemize}
\item [(i)] $B$ is a linearly independent set.
\item [(ii)] An infinite sum $\sum_{\la\in \Lambda} b_\la$ of elements of $B$ is convergent in $H$ if and only if $\sum_{\la\in \Lambda}\| b_\la\| < \infty$ . In this case, we have
\[
\left\| \sum_{\la\in \Lambda} b_\la\right\| =\sum_{\la\in \Lambda} \|b_\la\|.
\]
\end{itemize}
\end{proposition}
\begin{proof}
\begin{itemize}
\item [(i)] Let $\{b_1,\cdots, b_n\}$ be a finite subset of $B$ and let $\sum_{i=1}^n \alpha_i b_i=0$. Then for all $i=1,\cdots, n$, we have
\[
\alpha_i \|b_i\|^2=\lan \alpha_i b_i, b_i\ran = \left\langle \sum_{i=1}^n \alpha_i b_i, b_i\right\rangle =0.
\]
Thus $\alpha_i=0$ for all $i=1,\cdots, n$.
\item [(ii)] The meaning of the convergence of uncountable (or unordered) infinite sums was explained in Remark \ref{rem:uncountable sum}. Let $F$ be a finite subset of $\Lambda$. Using Pythagoras' identity, we have
\[
\left\| \sum_{\la\in F} b_\la\right\| =\sum_{\la\in F} \|b_\la\|.
\]
This shows that the summation $\sum_{\la\in \Lambda}\| b_\la\|$ is Cauchy if and only if the summation $\sum_{\la\in \Lambda} b_\la$ is Cauchy. When these sums are convergent, the above equality for every finite subset $F\sub \Lambda$ implies the desired equality.
\end{itemize}
\end{proof}

To find a basis for a Hilbert space, it is often more convenient to work with an orthonormal set.
\begin{proposition}
\label{prop:orthonormalpro}
Let $X$ be an orthonormal subset of a Hilbert space $H$. Let $h$ be an arbitrary vector of $H$.
\begin{itemize}
\item [(i)] Consider a function $\th:X\ra \c$.  The sum $\sum_{x\in X}\th(x) x$ converges in $H$ if and only if $\sum_{x\in X}|\th(x)|^2 <\infty$. In this case, we have
\[
\left\|\sum_{x\in X}\th(x) x\right\|^2 =\sum_{x\in X}|\th(x)|^2.
\]
\item [(ii)] $\sum_{x\in X} |\lan h,x\ran|^2\leq \|h\|^2$. This inequality is called {\bf Bessel's inequality}.
\item [(iii)] The sum $\sum_{x\in X} \lan h,x\ran x$ converges and we have
\[
h - \sum_{x\in X} \lan h,x\ran x \in X^\perp.
\]
\item [(iv)] We have
\[
\|h\|^2 =\left\|h - \sum_{x\in X} \lan h,x\ran x\right\|^2 + \sum_{x\in X} |\lan h,x\ran|^2.
\]
\item [(v)] The sum $\sum_{x\in X} \lan h, x\ran x\in X^\perp$ is exactly the orthogonal projection of $h$ on $[X]$.
\item [(vi)] The following conditions are equivalent:
\begin{itemize}
\item [(a)] $h\in [X]$.
\item [(b)] $h=\sum_{x\in X} \lan h, x\ran x$.
\item [(c)]$\|h\|^2=\sum_{x\in X} |\lan h, x\ran|^2$.
\end{itemize}
\end{itemize}
\end{proposition}
\begin{proof}
\begin{itemize}
\item [(i)] It is an immediate corollary of Proposition \ref{prop:phyth}(ii).
\item [(ii)] Let $F$ be an arbitrary finite subset of $X$ and define $\th(x):= \lan h,x\ran$. Then we have
\begin{eqnarray*}
\left\|h -\sum_{x\in F} \lan h, x\ran x\right\|^2&=&\left\langle h -\sum_{x\in F} \th(x) x,h -\sum_{z\in F} \th(z) z\right\rangle \\
&=&\lan h, h \ran + \sum_{x,z\in F} \th(x)\overline{\th(z)} \lan x,z\ran\\
&-&\sum_{x\in F} \th(x) \lan x,h\ran -\sum_{z\in F} \overline{\th(z)}\lan h, z\ran\\
&=&\|h\|^2 + \sum_{x\in F} | \th(x)|^2 -\sum_{x\in F} |\th(x)|^2 -\sum_{z\in F} |\th(z)|^2\\
&=&\|h\|^2 -\sum_{x\in F} |\lan h, x\ran|^2.
\end{eqnarray*}
This implies the following inequality for every finite subset $F$ of $X$:
\begin{equation}
\label{eqn:besselortho}
\sum_{x\in F} |\lan h, x\ran|^2=\|h\|^2 - \left\|h -\sum_{x\in F} \lan h, x\ran x\right\|^2\leq \|h\|^2.
\end{equation}
The convergence of $\sum_{x\in X} |\lan h, x\ran|^2$ and the Bessel's inequality follow from this.
\item [(iii)] The convergence of $\sum_{x\in X} \lan h,x\ran x$ follows from Parts (i) and (ii). Using the continuity and the linearity of the inner product, for every $x_0\in X$, we have
\begin{eqnarray*}
\left\langle h-\sum_{x\in X} \lan h,x\ran x, x_0\right\rangle &=& \lan h,x_0\ran -  \sum_{x\in X} \lan h,x\ran \lan x, x_0\ran\\
&=& \lan h, x_0\ran - \lan h, x_0\ran =0.
\end{eqnarray*}
This show that $h-\sum_{x\in X} \lan h,x\ran x\in X^\perp$.
\item [(iv)] This also follows from (\ref{eqn:besselortho}).
\item [(v)] It follows from (iii) and Corollary \ref{cor:perp}(i).
\item [(vi)] The equivalence between (a) and (b) is clear and the equivalence between (b) and (c) follows from (iv).
\end{itemize}
\end{proof}
\begin{theorem}
\label{thm:basisHilbert}
Let $X$ be a an orthonormal subset of a Hilbert space $H$. Then the following conditions are equivalent:
\begin{itemize}
\item [(i)] $H=[X]$.
\item [(ii)] $X^\perp=\{0\}$.
\item [(iii)] $X$ is a maximal orthonormal subset of $H$.
\item [(iv)] $h=\sum_{x\in X} \lan h, x\ran x$ for all $h\in H$.
\item [(v)] $\|h\|^2=\sum_{x\in X} |\lan h, x\ran|^2$ for all $h\in H$.
\item [(vi)] $\lan h, h'\ran = \sum_{x\in X} \lan h, x\ran\lan x, h'\ran$ for all $h,h' \in H$.
\end{itemize}
\end{theorem}
\begin{proof}
The equivalence between (i) and (ii) follows from Corollary \ref{cor:perp}(ii). The equivalence between (i), (iv) and (v) follows from Proposition \ref{prop:orthonormalpro}(vi).

To show (ii) implies (iii), assume $X$ is not a maximal orthonormal subset of $H$. Then there exist $h\in H$ such that $\|h\|=1$ and $\{h \}\cup X$ is orthonormal. Thus $h\perp X$ and so $X^\perp\neq \{0\}$. Conversely, assume there is $0\neq h\in H$ such that $h\perp X$. Then $X\cup\{\frac{h}{\|h\|}\}$ is a bigger orthonormal subset of $H$ than $X$.

Statement (v) follows immediately from (vi). Conversely, due to the continuity of inner product, (v) implies (vi).
\end{proof}
The equation in Condition (vi) in the above theorem is called {\bf Parseval's equation}.
\begin{definition}
A maximal orthonormal subset of $H$ is called an {\bf orthonormal basis} of $H$.
\end{definition}

An argument based on Zorn's lemma proves the following proposition.
\begin{proposition}
Every Hilbert space has an orthonormal basis.
\end{proposition}
\begin{remark}
\label{rem:cardinality}
Let $\aleph_0$ denotes the cardinality of $\n$ and let $m$ be the cardinality of an arbitrary infinite set. We remember from basic set theory that $\aleph_0 m=m$.
\end{remark}
\begin{exercise}
Let $X$ be an infinite (possibly uncountable) set and let $\alpha: x\mapsto \alpha_x$ be a function from $X$ into $[0,\infty)$. Show that if $\sum_{x\in X} \alpha_x< \infty$, then there are at most countably $x\in X$ such that $\alpha_x>0$.
\end{exercise}
\begin{proposition}
\label{prop:cardinalitybasis}
Let $X$ and $Y$ be two orthonormal basis of a Hilbert space $H$ with cardinalities $m$ and $n$ respectively. Then $m=n$.
\end{proposition}
\begin{proof}
If one of the cardinalities is finite, one can apply elementary linear algebra to show that $m=n$. So we assume both $m$ and $n$ are infinite and without loss of generality, we assume $m\leq n$. If we find an onto map from $X$ into $Y$, then we have $m\geq n$, and so the equality holds. By the above exercise, since $\sum_{y\in Y} |\lan x,y\ran|^2 =\| x\|^2=1$ for all $x\in X$, the set $Y_x\sub Y$ defined by
\[
Y_x:=\{ y\in Y; \lan x,y\ran\neq 0\}
\]
is countable and non-empty for all $x\in X$. On the other hand,
\[
Y=\bigcup_{x\in X} Y_x,
\]
because for every $y\in Y$, using a similar argument, there exists $x\in X$ such that $\lan x,y\ran\neq 0$.

Given $x\in X$, since $Y_x$ is countable, there exists a countable subset $X_x$ of $X$ equipped with an onto map $\ff_x:X_x\ra Y_x$. Let $\Sigma$ be the disjoint union of all $X_x$ for $x\in X$. Then because $Y=\bigcup_{x\in X}Y_X$ and using onto maps $\ff_x$ for all $x\in X$, we obtain a map from $\Sigma$ onto $Y$. Therefore $|\Sigma|\geq n$. But we know $|\Sigma|=\aleph_0 m=m$. This proves that $m=n$.
\end{proof}
\begin{definition}
A Hilbert space is called {\bf separable} if it has a countable orthonormal basis.
\end{definition}
\begin{exercise}
Let $H$ be a Hilbert space. Show that $H$ is a separable Hilbert space if and only if $H$ is a separable topological space.
\end{exercise}
\begin{corollary}
\label{cor:Hilbertcardinality} Assume $X_1$ and $X_2$ are orthonormal bases for two Hilbert spaces $(H_1, \inner_1)$ and $(H_2, \inner_2)$ respectively. These Hilbert spaces are unitary equivalent if and only if $X_1$ and $X_2$ have the same cardinality.
\end{corollary}
\begin{proof}
Assume $\th:X_1\ra X_2$ be a surjective map. We extend $\th$ linearly to a linear map $\lan X_1\ran \ra H_2$. Since it maps an orthonormal basis to an orthonormal basis it is straightforward to check that this map preserves the inner product and so is an isometry. Since $H_1=[X_1]$, we can extend this map to a unitary equivalence from $H_1$ onto $H_2$.

Conversely, assume $\ff:H_1\ra H_2$ be a unitary equivalence. One easily checks that $\ff(X_1)$ is an orthonormal basis for $H_2$. Hence $|X_1|=|\ff(X_1)|=|X_2|$.
\end{proof}

\begin{exercise}
Let $X$ be a set. For every $x\in X$, let $\d_x:X\ra \c$ be the characteristic map of the one point subset $\{x\}\sub X$, i.e. $\d_x(y):=\left\{ \begin{array}{cc} 1 & y=x\\ 0& y\neq x \end{array} \right.$. Show that the set $\{ \d_x; x\in X\}$ is an orthonormal subset of $\ell^2(X)$.
\end{exercise}

\begin{corollary}
Every Hilbert space is unitary equivalent to a Hilbert space of the form $\ell^2(X)$ for some set $X$. In particular, every separable Hilbert space is unitary equivalent to $\ell^2=\ell^2(\n)$.
\end{corollary}

We conclude this section with a discussion on the weak topology of Hilbert spaces.

\begin{definition}
Let $(H,\inner)$ be a Hilbert space. For every $h\in H$, define a semi-norm
\[
\rho_h(x):=|\lan x, h\ran|, \qquad \forall x\in H.
\]
The {\bf weak topology of $H$} is the locally convex topology defined by the semi-norms $\rho_h$ for all $h\in H$. The convergence in weak topology is called {\bf weak convergence}.
\end{definition}
Given a vector $x_0$ in a Hilbert space $H$, the family of all sets of the form
\[
\{x\in H; |\lan x-x_0, y_i\ran| <\ep, \forall i=1,\cdots,n\},
\]
where $n\in \n$, $y_1,\cdots, y_n\in H$ and  $\ep>0$, is a basis of open neighborhoods of $x_0$ in the weak topology of $H$.
\begin{exercise}
Let $H$ be a Hilbert space. Show that a net $(x_\la)$ in $H$ is weakly convergent to a vector $x_0\in H$ if and only if $\lan x_\la, y\ran\ra \lan x_0, y\ran$ for all $y\in H$.
\end{exercise}
A map $\psi:E\ra F$ between two complex vector spaces is called {\bf conjugate-linear} if
\[
\psi(\la x+ y)= \overline{\la} \psi(x)+ \psi(y), \quad \forall x,y\in E, \la\in \c.
\]
\begin{theorem}
\label{thm:selfdual} (The Riesz duality)
Let $H$ be a Hilbert space. The inner product induces an isometric conjugate-linear isomorphism from $H$ onto its dual $H\s$ by the following formula:
\begin{eqnarray*}
h&\mapsto& \ff_h, \,\qquad h\in H\\
\ff_h(x)&:=& \lan x, h\ran, \quad x\in H.
\end{eqnarray*}
We call this map the {\bf Riesz duality}.
\end{theorem}
\begin{proof}
Using CS inequality, for every $h\in H$, we have
\[
|\ff_h(x)|=|\lan x,h \ran|\leq \|x\| \|h\|, \quad \forall x\in X.
\]
One also notes that the above inequality becomes equality when $x=h$.  Hence $\|\ff_h\|=\|h\|$ and consequently $\ff_h$ is bounded. Now, it is straightforward to check that the Riesz duality is a conjugate-linear one-to-one map from $H$ into $H\s$.

It is clear that the Riesz duality maps the zero vector to the zero functional. Let $\ff\in H\s$ be a non-zero functional. Set $X:=ker\ff$. We have $X^\perp \neq \{0\}$, so we can find a unit vector $u$ in $X^\perp$. One checks that $\ff(u)h-\ff(h)u\in X$ for all $h\in H$. Hence we have $\lan \ff(u)h-\ff(h)u, u\ran=0$. by solving this equation for $\ff(h)$, we get
\[
\ff(h)=\ff(u) \lan h, u\ran=  \lan h, \overline{\ff(h)}u\ran, \qquad \forall h\in H.
\]
Thus $\ff=\ff_y$, where $y:=\overline{\ff(u)}u$. This proves that the Riesz duality is onto.
\end{proof}

\begin{corollary}
\label{cor:rieszduality}
The \ws topology on $H\s$ is consistent with the weak topology on $H$ under the Riesz duality. In other words, the Riesz duality is a homeomorphism even if we consider the \ws topology on $H\s$ and the weak topology on $H$. Therefore the closed unit ball in $H$ is weakly compact.
\end{corollary}
\begin{proof}
The first statement is clear. The second statement follows from the first and the Banach-Alaoghlu theorem, see Theorem \ref{thm:banachalaoghlo}.
\end{proof}
\begin{definition}
A net $(x_i)$ in a Hilbert space $H$ is called a {\bf weakly Cauchy net} if it is a Cauchy net with respect to every semi-norm defining the weak topology of $H$, namely $\lan x_i, h\ran$ is a Cauchy net in $\c$ for all $h\in H$.
\end{definition}
\begin{corollary}
Every norm bounded weakly Cauchy net $(x_i)$ in a Hilbert space $H$ is weakly convergent to a unique limit.
\end{corollary}
\begin{proof}
The weak topology is Hausdorff, so if a limit exists, it has to be unique. Since $(x_i)$ is norm bounded, it is contained in a positive multiple of the closed unit ball of $H$ which is weakly compact. Therefore a subnet of $(x_i)$ is weakly convergent, and so is $(x_i)$, see Proposition \ref{prop:subnetcompact}.
\end{proof}
\begin{proposition}
\label{prop:weaknormconv}
Let $H$ be a Hilbert space and let $(x_i)$ be a weakly convergent net to some vector $x\in H$. Then
\[
\|x\|\leq \liminf_i \|x_i\|.
\]
Moreover, $x_i \ra x$ in norm if and only if $\|x_i\| \ra \|x\|$.
\end{proposition}
\begin{proof}
By CS inequality, we have
\[
\|x\nt =\lan x, x \ran =\lim_i \lan x_i, x \ran=\liminf_i \lan x_i, x \ran\leq \liminf_i \|x_i\| \|x\|.
\]
Furthermore, one notes that
\[
\|x-x_i\nt=\lan x_i, x_i \ran- \lan x_i, x \ran - \lan x, x_i \ran +\lan x, x \ran.
\]
When $i\ra \infty$, the right hand side goes to zero if and only if $\|x_i\|\ra \|x\|$.
\end{proof}

\section{Bounded operators on Hilbert spaces}
\label{sec:boundedoperators}

In this section $H$ is always a Hilbert space. The \cs-algebra $B(H)$ of bounded operators on a Hilbert space $H$ is studied in details in this section.

When we discussed the algebra $B(H)$ of bounded operators on a Hilbert space as an example of a \cs-algebra in Section \ref{sec:topalgebras}, we assumed the existence of adjoint operator $T\s$ for every $T\in B(H)$. Therefore we begin this section with proving this statement.
\begin{definition}
Let $H_1$ and $H_2$ be two Hilbert spaces. A function
\[
(-,-): H_1\times H_2 \ra \c
\]
is called a {\bf sesquilinear form on $H_1\times H_2$} if it is linear in the first variable and conjugate-linear in the second variable. The form $(-,-)$ is called {\bf bounded} if there is some $c\in [0,\infty)$ such that
\[
|(x,y)|\leq c \|x\| \|y\|, \qquad \forall (x,y)\in H_1\times H_2.
\]
In this case, the {\bf norm} of $(-,-)$ is defined by
\[
\|(-,-)\|:= \inf\{ c\in[0,\infty); |(x,y)|\leq c \|x\| \|y\|, \forall (x,y)\in H_1\times H_2\}.
\]
\end{definition}
The inner products of Hilbert spaces determine the general form of bounded sesquilinear forms:
\begin{theorem}
\label{thm:sesquilinear}
Let $H_1$ and $H_2$ be two Hilbert spaces. For every bounded sesquilinear form $(-,-)$ on $H_1\times H_2$, there exists a unique bounded operator $T\in B(H_1, H_2)$ such that
\[
(x,y)=\lan Tx, y\ran, \qquad \forall (x,y) \in H_1\times H_2.
\]
Moreover, $\|(-,-)\|=\|T\|$.

Conversely, every bounded operator $T\in B(H_1,H_2)$ gives rise to a bounded sesquilinear form by defining $(x,y):=\lan Tx, y\ran$ and we have $\|(-,-)\|=\|T\|$.
\end{theorem}
\begin{proof}
For every $x\in H_1$, define $\ff_x: H_2\ra \c$ by $y\mapsto \overline{(x,y)}$. This is a bounded functional on $H_2$. Therefore  by the Riesz duality theorem, \ref{thm:selfdual}, there exists a unique $y_x\in H_2$ such that $\ff_x(y)=\lan y_x, y\ran$ for all $y\in H_2$. Due to the uniqueness of $y_x$, one easily sees that the assignment $x\mapsto y_x$ is linear. If we denote this map by $T$, then for every $x\in H_1$, we have
\begin{eqnarray*}
\|Tx\|&=&\|y_x\|= \|\ff_x\|\\
&=& \sup\{ |\overline{(x,y)}|; y\in H_2, \|y\|=1\}\\
&\leq& \sup\{ \|(-,-)\|\|x\|\|y\| ; y\in H_2, \|y\|=1\}\\
&=& \|(-,-)\|\|x\|.
\end{eqnarray*}
Thus $\|T\|\leq \|(-,-)\|$, and so $T\in B(H_1, H_2)$. On the other hand, for every $(x,y)\in H_1\times H_2$, we have
\begin{eqnarray*}
|(x,y)| &=& |\overline{\ff_x(y)}|=  |\overline{\lan y_x, y\ran}|\\
&=& |\overline{\lan Tx, y\ran}| \leq \|Tx\| \|y\|\\
&\leq & \|T\| \|x\| \|y\|.
\end{eqnarray*}
This shows that $\|(-,-)\|\leq \|T\|$. The uniqueness of $T$ follows from Problem \ref{e:5-12}.

It is straightforward to check the converse, and so it is left to the reader.
\end{proof}
\begin{corollary}
\label{cor:adjointoperator}
For every bounded operator $T\in B(H_1, H_2)$ between two Hilbert spaces, there exists a unique adjoint operator $T\s\in B(H_2, H_1)$ such that
\[
\lan Tx, y \ran=\lan x, T\s y\ran,\qquad \forall x\in H_1, y\in H_2.
\]
Moreover, $\|T\s\|=\|T\|$.
\end{corollary}
\begin{proof}
One checks that the form $(-,-)$ defined by $(y,x) :=\lan Tx,y\ran$ is a bounded sesquilinear form on $H_2\times H_1$ and $\|T\|=\|(-,-)\|$. Therefore by the above theorem, there exists a bounded operator $T\s:H_2\ra H_1$ such that $\lan T\s y , x\ran = \overline{\lan Tx, y \ran}$, or equivalently, $\lan x, T\s y \ran = \lan Tx, y \ran$.
Also, it follows from the above theorem that $\|T\s\|=\|(-,-)\|$. Hence $\|T\s\|=\|T\|$.
\end{proof}
\begin{proposition}
\label{prop:involutionoperator} Let $H$ be a Hilbert space. The adjoint operator defines an involution in $B(H)$ and $B(H)$ with this involution is a unital \cs-algebra.
\end{proposition}
\begin{proof} It is straightforward to check that $T\mapsto T\s$ is an involution on $B(H)$. We only check the \cs-identity. For every $T\in B(H)$ and $x\in H$, we have
\begin{eqnarray*}
\|Tx\|^2&=& \lan Tx,Tx\ran= |\lan x, T\s T x\ran| \\
&\leq& \|x\|\|T\s T x\| \leq \|T\s T\| \|x\|^2.
\end{eqnarray*}
Hence $\|T\|^2\leq \|T\s T\|$. On the other hand, We have $\|T\s T\| \leq \|T\s \|\|T\| = \|T\|^2$. Therefore $\|T\|^2=\|T\s T\|$.
\end{proof}
\begin{example}
\label{exa:adjointmatrix}
Consider the Hilbert space $H=\c^n$ equipped with the standard basis and the standard inner product and denote the identity operator by $I$. Let us denote the vectors in $H$ by $n\times 1$ matrices. Given a matrix $M=(m_{ij})$, the {\bf conjugate of $M$} is the matrix $\overline{M}:=(\overline{m_{ij}})$ and the {\bf transpose of $M$} is the matrix $M^t:=(m_{ji})$. Then the inner product on $H$ can be represented by a matrix product as follows:
\[
\lan x, y\ran= x^t I \overline{y}, \qquad \forall x,y\in H.
\]
Since we fixed a basis, every operator $T\in B(H)$ can be represented by an $n\times n$ matrix denoted by $T$, again. In other words, we use the realization $B(H)\simeq M_n(\c)$ coming from the standard basis. Then for every $x,y\in H$, we compute
\[
\lan Tx, y\ran=(Tx)^t I \overline{y}=x^t T^t I \overline{y}=x^t I \overline{\overline{T^t}y}= \lan x, \overline{T^t} y\ran.
\]
This shows that $T\s=\overline{T^t}$ for all $T\in B(H)$.
\end{example}
An immediate application of adjoint operators is seen in the following proposition:
\begin{proposition}
\label{prop:boundedweakly}
Let $T\in B(H_1, H_2)$ be a bounded operator between two Hilbert spaces. Then the following statements are true:
\begin{itemize}
\item[(i)] The operator $T$ is weakly continuous too, namely $T$ is continuous with respect to the weak topologies of $H_1$ and $H_2$.
\item[(ii)] The image of the closed unit ball of $H_1$ under $T$ is weakly compact in $H_2$.
\end{itemize}
\end{proposition}
\begin{proof}
\begin{itemize}
\item [(i)] Let $(x_i)$ be a net in $H_1$ weakly convergent to $x\in H_1$ and let $y\in H_2$. Then $\lan T(x_i-x), y\ran= \lan x_i-x, T\s y\ran\ra 0$. Hence $(T(x_i))$ is weakly convergent to $T(x)$.
\item [(ii)] It follows from (i) and Corollary \ref{cor:rieszduality}.
\end{itemize}
\end{proof}
A straightforward argument shows the following lemma:
\begin{lemma}
\label{lem:rangenull}
Let $T\in B(H_1, H_2)$ be a bounded operator between two Hilbert spaces. Then $R(T)^\perp=N(T\s)$.
\end{lemma}
The above lemma can be proved using Lemma \ref{lem:NRannihil}(i) too. We only need to interpret the relationship between adjoint operators (in Banach spaces) as described in Section \ref{sec:spectralcompact} and adjoint operators (in Hilbert spaces) as described in the present section.
\begin{remark}
\label{rem:adjointbanachcs}
Let $T\in B(H_1, H_2)$ be a bounded operator between two Hilbert spaces. If we denote the Riesz duality explained in Theorem \ref{thm:selfdual} for $H_1$ and $H_2$, by $\ff^1:H_1\ra H_1^\ast$ and $\ff^2:H_2\ra H_2^\ast$, respectively, then for every $x\in H_1$ and $y\in H_2$, we have
\[
\ff^2_y(Tx)=\lan Tx , y\ran = \lan x, T\s y\ran=\ff^1_{T\s y} (x).
\]
Now, if $T^\sharp:H_2\s \ra H_1\s$ denotes the adjoint of operator $T$ regarded as an operator between two Banach spaces, i.e. $T^\sharp(\rho) (h)=\rho(Th)$ for all $\rho\in H_2\s$ and $h\in H_1$, then we have
\[
T^\sharp(\ff^2_y)(x)=\ff^2_y(Tx), \qquad \forall x\in H_1, y\in H_2.
\]
Therefore $T^\sharp(\ff^2_y)=\ff^1_{T\s y}$ for all $y\in H_2$. In other words, the following diagram is commutative:
\[
\xymatrix{
H_2\ar[d]_{\ff^2}\ar[r]^{T\s}&H_1\ar[d]^{\ff^1}\\
H_2\s\ar[r]_{T^\sharp} &H_1\s
}
\]
Since the vertical arrows in the above diagram are conjugate-linear isomorphisms, we have $T\s=(\ff^1)\inv T^\sharp \ff^2$. This clears the relationship between these two notions of adjoint operators. Using this realization, we can translate all results that have already been proved for adjoint operators (in Banach spaces) in Section \ref{sec:spectralcompact} for adjoint operators (in Hilbert spaces). For instance, if $T$ is a compact operator, then $T\s$ is a compact operator too. However, most of the times, it is often easier to prove those results again using inner products and other tools of Hilbert spaces.
\end{remark}
\begin{exercise} Let $H$ be a Hilbert space and let $X$ be a subset of $H$. Describe the relationship between $X^\perp$ in Hilbert spaces and $X^\perp$ as the annihilator of $X$. Prove Lemma \ref{lem:rangenull} using Lemma \ref{lem:NRannihil}(i).
\end{exercise}
Let $H$ be a Hilbert space. Definitions of self adjoint, unitary, projection and normal elements in $B(H)$ are the same as in \cs-algebras. Positive elements in $B(H)$ have two equivalent definitions which were discussed in Example \ref{exa:positiveityHilbert}(i). Another concept for elements of $B(H)$ is isometry. A $T\in B(H)$ is called an {\bf isometry} if $\|Tx\|=\|x\|$ for all $x\in H$. For $T\in B(H)$, the condition $T\s T=1$ is equivalent to being an isometry and it can be generalized for to define isometry elements in abstract \cs-algebras.
\begin{exercise}
\label{exe:adjointopt}
Let $T$ be as above.
\begin{itemize}
\item [(i)] Prove that the above conditions on $T$ for being an isometry are equivalent. (Hint: use the polarization identity.)
\item [(ii)] Show that if $T$ is a unitary element, then both $T$ and $T\s$ are isometry.
\item [(iii)] Assume $T$ is invertible or normal. Show that if $T$ is isometry then $T$ is unitary.
\item [(iv)] Show that if $T$ is isometry, then $TT\s$ is a projection.
\item [(v)] Prove $N(T)=N(T\s T)$.
\end{itemize}
\end{exercise}
Let $H_1$ and $H_2$ be two Hilbert spaces. One notes that there are similar definitions for unitary and isometry operators in $B(H_1, H_2)$. In fact, an operator $T\in B(H_1, H_2)$ is called unitary if $T\s T=1_{H_1}$ and $TT\s=1_{H_2}$, and similarly, $T$ is called an isometry if $T\s T=1_{H_1}$. One easily observes that unitary (resp. isometry) operators are the same as unitary equivalences (resp. unitary injections) defined in Definition \ref{def:unitaryequ}(ii).
\begin{proposition}
\label{prop:normaloperator}
Let $T$ be a bounded operator on a Hilbert space $H$. Then $T$ is normal if and only if $\|Tx\|=\|T\s x\|$ for all $x\in H$. Moreover, when $T$ is normal, we have $N(T)=N(T\s)=R(T)^\perp$.
\end{proposition}
\begin{proof}
Assume $T$ is normal, then for every $x\in H$, we have
\[
\|Tx\|^2=\lan Tx, Tx \ran = \lan x, T\s T x\ran = \lan x, TT\s x\ran = \lan T\s x, T\s x\ran = \|T\s x \|^2.
\]
This also implies that $N(T)=N(T\s)=R(T)^\perp$.

Conversely, let $\|Tx\|=\|T\s x\|$ for all $x\in H$. Then using the polarization identity, we obtain $\lan Tx, Ty\ran = \lan T\s x, T\s y\ran$ for all $x,y\in H$. By Corollary \ref{cor:innerzero}, this implies that $T$ is normal.
\end{proof}
Spectral theory of bounded operators on $H$ is more concrete than abstract \cs-algebras, and therefore we provide more details here. By Proposition \ref{prop:invoperator}, $T\in B(H)$ is invertible if and only if it is one-to-one and onto. However, there is another condition which is useful to check whether $T$ is invertible.

When $T$ is not one-to-one, clearly $T$ cannot be bounded below. When $T$ is not bounded below, for every $n\in \n$, there is $0\neq x_n\in H$ such that $\|Tx_n\|<\frac{ \|x_n\|}{n}$. If $T$ is invertible, then
\[
\|x_n\|=\|T\inv Tx_n\|\leq \|T\inv\| \|Tx_n\|< \|T\inv\|\frac{ \|x_n\|}{n}, \quad \forall n\in \n.
\]
Apparently, this is a contradiction, and so $T$ cannot be invertible.
\begin{proposition}
\label{propo:boundedbelow}
Let $T\in B(H)$.
\begin{itemize}
\item [(i)] $T$ is invertible if and only if $R(T)$ is dense in $H$ and $T$ is bounded below.
\item [(ii)] When $T$ is normal, $T$ is invertible if and only if $T$ is bounded below.
\end{itemize}
\end{proposition}
\begin{proof}
\begin{itemize}
\item [(i)] Assume $R(T)$ is dense in $H$ and $T$ is bounded below, then the inverse of $T$, i.e. $T\inv:R(T)\ra H$ is bounded, because there is some $\ep>0$ such that $\|T\inv Tx\|=\|x\|\leq 1/\ep \|Tx\|$ for every $Tx\in R(T)$. Thus we can extend $T^\inv$ to $H$ by continuity. Let us denote the extension of $T\inv$ to $H$ by $S$ for a moment. Then it is clear $ST=1_H$ and $TS$ equals $TT\inv=1_{R(T)}$ on $R(T)$. But, since $R(T)$ is dense, we have $TS=1_H$. Therefore $T$ is invertible. The other implication follows from the above discussion.
\item [(ii)] When $T$ is normal, $N(T)=N(T\s)= R(T)^\perp$. If $T$ is bounded below, then $T$ is one-to-one, and so $R(T)$ is dense in $H$. This implies that $T$ is invertible.
\end{itemize}
\end{proof}
Now, we are ready to study the ideal $F(H)$ of finite rank operators on a Hilbert space $H$.
\begin{proposition}
\label{prop:finiterank1}
For every Hilbert space $H$, $F(H)$ is an involutive subalgebra of $B(H)$.
\end{proposition}
One notes that $F(H)$ is not closed in $B(H)$ unless $H$ is finite dimensional. Therefore to prove this proposition,  one cannot use Proposition \ref{prop:closedidealcs}.
\begin{proof} Let $T\in F(H)$. Since $R(T)^\perp = N(T\s)$, the kernel of $T\s$ has a finite codimension in $H$. Hence $T\s\in F(H)$. This proves that $F(H)$ is closed under the involution.
\end{proof}
It is useful to introduce a generating set for $F(H)$ consisting of very simple operators. For every $x,y\in H$, the operator defined by $h\mapsto x\lan h, y\ran=\lan h, y\ran x$ is bounded and its image is one dimensional. Hence it is a finite rank operator on $H$. We denote this operator by $x\otimes y$, in some books it is denoted by $\Theta_{x,y}$. The idea of these rank one operators comes from the elementary matrices $E_{ij}$ in $M_n(\c)$ for $1\leq i,j \leq n$, see Exercise \ref{exe:category}(iv). One easily checks that $E_{ij}=e_i\otimes e_j$.

\begin{proposition}
\label{prop:finiterank2}
Every operator of rank one in $B(H)$ is of the form $x\otimes y$ for some $x,y \in H$. The set $\{ x\otimes y; x,y\in H\}$ of all operators of rank one in $B(H)$ generates $F(H)$.
\end{proposition}
\begin{proof}
Assume the rank of $T\in B(H)$ is one and pick $x_0\in H$ such that $\|Tx_0\|=1$. Set $y:=Tx_0$. Since $R(T)$ is one dimensional, for every $h\in R(T)$, we have $h=\lan h, y \ran y$. Therefore $\th:R(T)\ra \c$, $h\mapsto \lan h, y\ran$ is an isomorphism whose inverse is $\la\mapsto \la y$. Since $\th T:H\ra \c$ is a linear functional, by the Riesz duality, Theorem \ref{thm:selfdual}, there is $x\in H$ such that $\th T k = \lan k, x\ran$ for all $k\in H$. Then for every $k\in H$, we have
\[
Tk = \lan Tk, y\ran y =(\th T k )y= \lan k, x\ran y = (y\otimes x)k.
\]
Therefore $T=y\otimes x$.

Let $T$ be an arbitrary finite rank operator on $H$ and let $B=\{x_1,\cdots, x_n\}$ be an orthonormal basis for the image of $T$. The equality $R(T)^\perp = N(T\s)$ implies that the rank of $T\s$ is $n$ too. So we can find an orthonormal basis $B'=\{y_1,\cdots, y_n\}$ for $R(T\s)=N(T)^\perp$. For every $1\leq i\leq n$, we have $T(y_i)= \la_1 x_{1i}+\cdots+\la_{ni} x_n$ for some $\la_{1i},\cdots, \la_{ni}\in \c$. Using this, one easily checks that
\[
T=\sum_{i,j=1}^n \la_{ji} x_j\otimes y_i.
\]
\end{proof}

\begin{exercise}
\label{exe:finiterank}
Let $H$ be a Hilbert space. For every $x,y,x',y'\in H$ and $T\in B(H)$, show the following statements:
\begin{itemize}
\item [(i)] $(x\otimes y)\s=y\otimes x$.
\item [(ii)] $\|x\otimes y\|=\|x\| \|y\|$.
\item [(iii)] $T(x\otimes y)=Tx\otimes y$.
\item [(iv)] $(x\otimes y) T= x\otimes T\s y$.
\item [(v)] $(x\otimes y)(x'\otimes y')=\lan y,x'\ran (x\otimes y')$.
\item [(vi)] The operator $x\otimes y$ is a projection (a {\bf rank one projection}) if and only if $x=y$ and $\|x\|=1$. Every rank one projection is of this form.
\end{itemize}
\end{exercise}
\begin{proposition}
\label{prop:finiterank3}
Every non-zero ideal of $B(H)$ contains $F(H)$.
\end{proposition}
Again, by an ideal, we always mean a two sided ideal.
\begin{proof}
Assume $J$ is a nonzero ideal of $B(H)$ and $0\neq T\in J$. Pick $x_0\in H$ such that $\|Tx_0\|=1$ and set $y_0:=Tx_0$. Then for every $x,y\in H$, we have
\[
(x\otimes y_0) T (x_0\otimes y)= (x\otimes y_0) (Tx_0\otimes y)= \lan y_0, y_0 \ran (x\otimes y ) =x\otimes y.
\]
This shows that all operators of rank one belong to $J$. Therefore by Proposition \ref{prop:finiterank2}, $F(H)\sub J$.
\end{proof}
It follows from the above proposition and Proposition \ref{prop:compactoperators} that $F(H)$ and $K(H)$ are essential ideals of $B(H)$. We conclude this section with an important concept which is useful to study the rich structure of the \cs-algebra $B(H)$ of bounded operators on a Hilbert space $H$.
\begin{definition}
Let $X$ be a subset of $B(H)$. The {\bf commutant of $X$} is the set
\[
X':= \{T\in B(H); TS=ST \, \forall S\in X\}.
\]
The {\bf bicommutant of $X$} is $X'':=(X')'$.  Similarly, we use the notation: $X''':=(X'')'$, $X'''':=(X''')'$, and so on.
\end{definition}
Many basic properties of commutants are summarized in the following proposition:
\begin{proposition}
\label{prop:basiccommutant}
Let $X, X_1$ and $X_2$ be subsets of $B(H)$. Then the following statements are true:
\begin{itemize}
\item [(i)] $X_1\sub X_2$ implies $X_2'\sub X_1'$.
\item [(ii)] $X'$ is a closed unital subalgebra of $B(H)$.
\item [(iii)] $X\sub X''=X''''=\cdots$ and $X'=X'''=X'''''=\cdots$.
\item [(iv)] If $X$ is a self adjoint subset of $B(H)$, then $X'$ is self adjoint, and consequently a unital \cs-subalgebra of $B(H)$.
\item [(v)] $X''=B(H)$ if and only if $X'=\c 1$.
\end{itemize}
\end{proposition}
\begin{proof}
\begin{itemize}
\item [(i)] This immediately follows from the definition.
\item [(ii)] It immediately follows from the definition that $X'$ is a unital subalgebra of $B(H)$. The commutant $X'$ is closed because the multiplication in $B(H)$ is continuous.
\item [(iii)] The inclusion $X\sub X''$ immediately follows from the definition. Applying this inclusion to $X'$, we get $X'\sub X'''$. Using Part (i) for the inclusion $X\sub X''$, we get $X'''\sub X'$. Hence $X'=X'''$. The rest of the equalities follow from this equality.
\item [(iv)] Let $X$ be self adjoint and let $T\in X'$. Then $TS=ST$ for all $S\in X$, or equivalently $S\s T\s=T\s S\s $ for all $S\in X$. This is equivalent to $ST\s = T\s S$ for all $S\s\in X$. Since $X$ is self adjoint, it is equivalent to say that $ST\s = T\s S$ for all $S\in X$. Therefore $T\s\in X'$.
\item [(v)] It is clear that $X''=B(H)$ whenever $X'=\c 1$. Conversely, assume there exists $T\in X'$ such that $T\notin \c 1$. This amounts to the existence of a non-zero vector $x\in H$ such that $y=Tx\neq \la x$ for all $\la\in \c$. Clearly, $y\neq 0$. If $\lan y,x\ran =0$, set $p:=x\otimes x$. Then $pTx=(x\otimes x)y =0$ and $Tpx= Tx \|x\|^2= y\|x\|^2\neq 0$, so $Tp\neq pT$. If $\lan y,x\ran\neq 0$, set $q:=x\otimes z$, where $z:= y- \frac{\lan y,x\ran}{\|x\|^2}x$. We note that $z\neq 0$ and we compute
\[
qTx=(x\otimes z)y= x\lan y,z\ran= x\left(  \right\|y\|^2 -\frac{\lan y,x\ran^2}{\|x\|^2})\neq 0,
\]
because of Proposition \ref{prop:csequality} and the fact that $x$ and $y$ are linearly independent. On the other hand, we compute
\[
Tqx=T\left(x (\lan y, x \ran -  \frac{\lan y,x\ran}{\|x\|^2}\lan x, x\ran) \right)=0.
\]
Hence $Tq\neq qT$. This shows that $X''\neq B(H)$.
\end{itemize}
\end{proof}

\begin{definition}
Let $H$ be a Hilbert space. An involutive subalgebra $M$ of $B(H)$ is called a {\bf von Neumann algebra on $H$} if $M=M''$. Let $S$ be a subset of $B(H)$. The {\bf von Neumann algebra generated by $S$} is $C\s(S)''$, the bicommutant of the \cs-subalgebra generated by $S$, and is denoted by $VN(S)$. If $A$ is a \cs-subalgebra of $B(H)$, then $A''$ is also called the {\bf enveloping von Neumann algebra of $A$}.
\end{definition}

One notes that every von Neumann algebra is necessarily unital. Although we defined von Neumann algebras on a Hilbert space $H$ using the algebraic notion of commutants, they have certain topological meaning in $B(H)$ too. We shall explain the topological viewpoint of von Neumann algebras in Section \ref{sec:vonneumann}.

\section{Concrete examples of \cs-algebras}
\label{sec:concreteexamples}
In this section, we describe some classes of examples for {\bf concrete \cs-algebras}, namely \cs-algebras embedded in $B(H)$ for some Hilbert space $H$. We begin with explaining the construction of reduced group \cs-algebras associated to locally compact groups. In what follows, for $1\leq p \leq \infty$, the $L^p$-norm is denoted by $\|-\|_p$. For $1<p<\infty$, the positive real number $q$ satisfying the identity $\frac{1}{p}+\frac{1}{q}=1$ is called the {\bf conjugate exponent of $p$}. When $p=1$ (resp. $p=\infty$) , it is reasonable to assume that $q$, the conjugate exponent of $p$, is $\infty$ (resp. $1$).
\begin{theorem}
\label{thm:lpduality} [Duality in $L^p$ spaces] Let $(X,\mu)$ be a measure space. For given $0<p<\infty$, let $q$ be its conjugate exponent. The following map is an isometric isomorphism from $L^p(\mu)$ onto $(L^q(\mu))\s$:
\begin{eqnarray*}
g&\mapsto& \ff_g,  \quad \forall g\in L^p(\mu)\\
\ff_g(f)&:=& \int f(x) g(x) d\mu(x), \quad \forall f\in L^q(\mu)
\end{eqnarray*}
\end{theorem}
For the proof of the above theorem see Theorem 6.15 in \cite{folland-ra}.
\begin{proposition}
\label{prop:minkowskiinequality} [Minkowski's inequality for integrals] Assume $(X, \mathcal{M}, \mu)$ and $(Y,\mathcal{N}, \nu)$ are two $\si$-finite measure spaces and $f$ is a $(\mathcal{M}\otimes \mathcal{N})$-measurable function on $X\times Y$.
\begin{itemize}
\item [(i)] If $f(x,y)\geq 0$ for all $(x,y)\in X\times Y$ and $1\leq p<\infty$, then
\begin{equation}
\label{eqn:minkowski}
\left(\int \left( \int f(x,y)d\nu (y)\right)^p d\mu(x)\right)^{1/p}\leq \int\left(  \int f(x,y)^p d\nu (y) \right)^{1/p} d\mu(x).
\end{equation}
\item [(ii)] If $1\leq p\leq \infty$, $f(-, y)\in L^p(\mu)$ for almost every $y$, and the function $y\mapsto \|f(-,y)\|_p$ is in $L^1(\nu)$, then $f(x,-)\in L^1(\nu)$ for almost every $x$, the function $x\mapsto \int f(x,y) d\nu(y)$ belongs to $L^p(\mu)$, and
\begin{equation}
\label{eqn:minkowskinorm}
\left\|\int f(-,y)d\nu(y)\right\|_p \leq \int \| f(-,y)\|_p d\nu(y).
\end{equation}
\end{itemize}
\end{proposition}
\begin{proof}
\begin{itemize}
\item [(i)] For $p=1$, (\ref{eqn:minkowski}) follows directly from the Fubini-Tonelli theorem, see Theorem 2.37 of \cite{folland-ra}. For $1<p<\infty$, let $q$ be the conjugate exponent of $p$. For every $g\in L^q(\mu)$, using the Fubini-Tonelli Theorem, we have
\begin{eqnarray*}
\int \left( \int f(x,y) d\nu(y)\right) |g(x)| d\mu(x)&=& \int \left[ \int f(x,y)|g(x)|d\mu(x)\right] d\nu(y)\\
&=& \int \left[ \ff_{f(-,y)} (|g(x)|) )\right] d\nu(y)\\
&\leq& \int \left[ \| f(-,y)\|_p \|g\|_q \right] d\nu(y)\\
&=& \|g\|_q \int \left( \int f(x,y)^p d\mu(x)  \right)^{1/p} d\nu(y),
\end{eqnarray*}
where $\ff_{f(-,y)}$ is defined in the duality between $L^p$ spaces, Theorem \ref{thm:lpduality}. Since this is true for every $g\in L^q(\mu)$, again Theorem \ref{thm:lpduality} implies the following inequality:
\[
\left\|\int f(x,y) d\nu(y) \right\|_p \leq \int \left( \int f(x,y)^p d\mu(x)  \right)^{1/p} d\nu(y).
\]
The latter inequality is equivalent to (\ref{eqn:minkowski}) because $f\geq 0$ on $X\times Y$.
\item [(ii)] For $1\leq p<\infty$, it follows from (i) by replacing $f$ by $|f|$. For $p=\infty$, it follows from the monotonicity of the integral.
\end{itemize}
\end{proof}

\begin{proposition}
\label{prop:younginequality} [Young's inequality] Let $G$ be an LCG  with a Haar measure $\mu$.  Let $1\leq p \leq \infty$, $f\in L^1(\mu )$, and $g\in L^p(\mu)$. Then $f\ast g(x)$ exists for almost every $x$, $f\ast g\in L^p$ and we have
\begin{equation}
\label{eqn:young}
\|f\ast g\|_p \leq \|f\|_1 \|g\|_p.
\end{equation}
\end{proposition}
\begin{proof} One notes that Minkowski's inequality for integrals relies on the Fubini-Tonelli theorem and this latter theorem is valid when the measure spaces are $\si$-finite. However, in Remark \ref{rem:sigmacompact}, we explained why we can apply the Fubini-Tonelli theorem and its consequences to integration on locally compact groups equipped with a Haar measure.

Now, define $F:G\times G\ra \c$ by $F(x,y):=f(y) g(y\inv x)$ for all $x,y \in G$. Then for every $y\in G$, $F(-, y)\in L^p(\mu)$ because $g\in L^p(\mu)$. On the other hand, for every $y\in G$, we have
\[
\|F(-, y)\|_p = \|f(y) g(y\inv -)\|_p= |f(y)|\|L_y(g)\|_p=|f(y)|\|g\|_p.
\]
Hence for every $y\in G$, the function $y\mapsto \|F(-, y)\|_p$ is in $L^1(\mu)$. By applying Minkowski's inequality for integrals to $F(x,y)$, see \ref{prop:minkowskiinequality}(ii), we observe that the function $x\mapsto \int_G f(y)g(y\inv x) d\mu(y)= f\ast g(x)$ belongs to $L^p(\mu)$ and we have
\begin{eqnarray*}
\|f\ast g\|_p &=& \left\| \int f(y) g(y\inv -) d\mu(y) \right\|_p\\
&=& \left\| \int_G F(-,y) d\mu(y) \right\|_p\\
&\leq&  \int_G \left\|F(-,y)  \right\|_p d\mu(y)\\
&=& \int_G  |f(y)|\|g\|_p d\mu(y)\\
&=& \|f\|_1\|g\|_p.
\end{eqnarray*}
\end{proof}

\begin{example}
\label{exa:reducedgroupcsalgebra}
Let $G$ be a LCG with a Haar measure $\mu$. The convolution product defines a \ss-homomorphism $\la: L^1(G)\ra B(L^2(G))$ as follows:
\[
\la(f)\xi (g) :=f\ast \xi (g)=\int_G f(h) \xi(h\inv g) d\mu(h),
\]
for all $f\in L^1(G)$, $\xi\in L^2(G)$, $g\in G$. This is called the {\bf left regular representation of $G$} (or $L^1(G)$). It follows from Young's inequality for $p=2$ that $f\ast \xi\in L^2(G)$. Moreover, $\la(f)$ is a bounded operator, in fact, we have $\|\la(f)\|\leq \|f\|_1$. Also, we need to show  that $\la$ is a \ss-homomorphism. It clearly preserves the addition and scalar multiplication. The associativity of the convolution product implies that $\la$ preserves the multiplication, more precisely, for all $f,g\in L^1(G)$ and $\xi\in L^2(G)$, we compute
\[
\la(f\ast g)(\xi)=(f\ast g)\ast \xi=f\ast (g\ast \xi) = \la(f)(\la(g)(\xi))= (\la(f)\la(g))(\xi).
\]
Now, we show that $\la$ preserves the involution. Let $\Delta$ be the modular function on $G$. Then for all $f\in L^1(G)$, $\xi\in L^2(G)$ and $g\in G$, we compute
\begin{eqnarray*}
\lan \eta , \la(f\s) \xi\ran &=& \int_G \eta\overline{f\s \ast \xi(g)} d\mu(g)\\
&=& \int_G \eta(g) \overline{\left( \int_G \Delta(h\inv) \overline{f(h\inv)} \xi(h\inv g)d\mu(h)\right)} d\mu(g).
\end{eqnarray*}
By using Lemma \ref{lem:intinvsub} for the integral over $h$, we obtain
\begin{eqnarray*}
\lan \eta , \la(f\s) \xi\ran &=& \int_G \eta(g) \overline{\left( \int_G \overline{f(h)} \xi(h g)d\mu(h)\right)} d\mu(g)\\
&=& \int_G \int_G \eta(g) f(h) \overline{\xi(hg)} d\mu(h) d\mu(g)\\
&=& \int_G \left(\int_G \eta(g) f(h) \overline{\xi(hg)} d\mu(g)\right)d\mu(h).
\end{eqnarray*}
If we substitute $hg$ by $k$, then $d\mu(g)=d\mu(hg)=d\mu(k)$ for all $h\in G$, $g=h\inv k$, and so we have
\begin{eqnarray*}
\lan \eta , \la(f\s) \xi\ran
&=& \int_G \left(\int_G \eta(h\inv k) f(h) \overline{\xi(k)} d\mu(k) \right)d\mu(h)\\
&=& \int_G \left(\int_G f(h) \eta(h\inv k)   d\mu(h)\right) \overline{\xi(k)}d\mu(k)\\
&=& \int_G f\ast\eta(k) \overline{\xi(k)}d\mu(k)\\
&=& \lan \la(f)\eta, \xi\ran.
\end{eqnarray*}
This completes the proof of the fact that $\la$ is a \ss-homomorphism.

Another crucial fact about $\la$ is that it is one-to-one. To show this, we use a Dirac net $(f_j)$ on $G$. Let $\la(f)=0$ for some $f\in L^1(G)$. Since $f_j\in C_c(G)\sub L^2(G)$ for all $j$, we have $0=\la(f)f_j=f\ast f_j$ for all $j$. Therefore using Lemma \ref{lem:diracnet}, we have $f=\lim_j f\ast f_j =0$.

The above discussions show that $L^1(G)$ can be embedded in the \cs-algebra $B(L^2(G))$ as a \ss-subalgebra. But, by Proposition \ref{prop:triviall1}, $L^1(G)$ is not a \cs-algebra unless $G$ is the trivial group. Therefore to obtain a \cs-algebra, we consider the closure of the image of $L^1(G)$ in $B(L^2(G))$. This \cs-algebra is called the {\bf reduced group \cs-algebra of $G$} and is denoted by $C_r\s(G)$. It clearly inherits many features of $L^1(G)$. For instance, $C_r\s(G)$ is commutative if and only if $G$ is abelian.
\end{example}

\begin{example}
\label{exa:linfty} Let $(X,\mu)$ be a measure space. We define a map
\begin{eqnarray*}
M:L^\infty (X)&\ra& B(L^2(X)),\\
f&\mapsto& M_f, \qquad \forall f\in L^\infty(X),
\end{eqnarray*}
where
\[
M_f\xi(x):= f(x)\xi (x) \qquad \forall \xi\in L^2(X), x\in X.
\]
For $f\in L^\infty(X)$ and $\xi\in L^2(X)$, using Remark \ref{rem:linfinity}, we have
\begin{eqnarray*}
\|M_f\xi\|_2^2&=&\int |f(x) \xi(x)|^2 d\mu(x)\\
&=& \int_{X\backslash P} |f(x) \xi(x)|^2 d\mu(x)+ \int_P |f(x) \xi(x)|^2 d\mu(x)\\
&=& \int_{X\backslash P} |f(x) \xi(x)|^2 d\mu(x)\\
&\leq& \int_{X\backslash P} \|f(x)\|_\infty^2 | \xi(x)|^2 d\mu(x)\\
&\leq& \|f(x)\|_\infty^2 \|\xi\|_2^2,
\end{eqnarray*}
where $P=\{x\in X; |f(x)|>\|f\|_\infty\}$. This shows that $M$ is well defined and $\|M_f\|_\infty \leq \|f\|_\infty $. It is straightforward to check that $M$ is actually a one-to-one \ss-homomorphism. Therefore by Corollary \ref{cor:csinjection} and Corollary \ref{cor:firstiso}, $M$ is an isometric embedding of $L^\infty(X)$ into $B(L^2(X))$ and its image is a \cs-subalgebra of $B(L^2(X))$. For every $f\in L^\infty(X)$, the operator $M_f$ is called the {\bf multiplication operator of $f$}.
\end{example}
\begin{example}
\label{exa:shiftoperator}
We define the {\bf unilateral shift operator} $S$ on $\ell^2=\ell^2(\n)$ by setting $S(\d_n):=\d_{n+1}$, where $\d_n$ is the characteristic function of $\{n\}$, and extending it linearly. Since $\{\d_n; n\in \n\}$ is an orthonormal basis of $\ell^2$, this operator is an isometry, and therefore it is bounded. One also easily checks that $S\s$, the adjoint of $S$, is given by the linear extension of the following map:
\[
S\s(\d_n):=\left\{ \begin{array}{lc} \d_{n-1} & n\geq 2\\ 0& n=1 \end{array} \right.
\]
It is also easy to see that $S\s S=1$, but $SS\s\neq 1$. However, $SS\s$ is a projection, in fact the projection on the closed subspace generated by $\{\d_n; n\geq 2 \}$ in $\ell^2$. The \cs-subalgebra of $B(\ell^2)$ generated by $\{S,1\}$ is called the {\bf Toeplitz algebra} and is denoted by $\mathcal{T}$.
\end{example}

\section{Locally convex topologies on $B(H)$}
\label{sec:topologies}

In this section $H$ is always a Hilbert space. There are many topologies on $B(H)$ besides the norm topology. Here, we content ourself to weak, strong, and strong-$\s$ operator topologies. We refer the interested reader to \cite{li} for a comprehensive list of topologies on $B(H)$ and various comparisons between them.

\begin{definition}
The {strong operator topology on $B(H)$} (or simply the strong topology on $B(H)$) is the topology of pointwise norm-convergence of elements of $B(H)$. In other words, a net $(T_i)$ in $B(H)$ strongly converges to $T\in B(H)$ if and only if the net $(T_ix)$ converges to $Tx$ in norm topology of $H$ for every $x\in H$.
\end{definition}

For every $x\in H$, the map $\rho_x:B(H)\ra \c$ defined by $\rho_x(T):=\|Tx\|$ for all $T\in B(H)$ is a semi-norm. One easily observes that the strong topology on $B(H)$ is the locally convex and Hausdorff topology defined by semi-norms $\rho_x$ for all $x\in H$.

\begin{definition}
The {weak operator topology on $B(H)$} (or simply the weak topology on $B(H)$) is the topology of pointwise weak-convergence of elements of $B(H)$. In other words, a net $(T_i)$ in $B(H)$ weakly converges to $T\in B(H)$ if and only if the net $(T_ix)$ converges to $Tx$ in weak topology of $H$ for every $x\in H$.
\end{definition}

For every $x,y\in H$, the map $\rho_{x,y}:B(H)\ra \c$ defined by $\rho_{x,y}(T):=|\lan Tx, y\ran |$ for all $T\in B(H)$ is a semi-norm. One easily observes that the weak topology on $B(H)$ is the locally convex and Hausdorff topology defined by semi-norms $\rho_{x,y}$ for all $x,y\in H$.

It is often useful to consider smaller families of semi-norms to define the above topologies.
\begin{proposition} Let $(T_i)$ be a norm bounded net in $B(H)$ and let $T\in B(H)$.
\label{prop:lessseminormstops}
\begin{itemize}
\item [(i)] $T_i\ra T$ strongly if and only if $T_ix\ra Tx$ in norm for all $x$ in a dense (or just a total) subset of $H$.
\item [(ii)] $T_i\ra T$ weakly if and only if $T_ix\ra Tx$ in weak topology of $H$ for all $x$ in a dense (or just a total) subset of $H$.
\end{itemize}
\end{proposition}
\begin{proof} Let $M$ be a positive number such that $\|T\|\leq M$ and $\|T_i\|\leq M$ for all $i$.
\begin{itemize}
\item [(i)] Let $E$ be a dense subset of $H$ and let $T_ix\ra Tx$ in norm for all $x\in E$. For given $y\in H$, assume $\{x_n\}$ is a sequence in $E$ such that $x_n\ra y$. For given $\ep>0$, pick $n_0\in \n$ such that $\|x_{n_0} - y\|<\frac{\ep}{3M}$. Then pick $i_0$ such that $\|T_i x_{n_0} -Tx_{n_0}\|<\ep/3$ for all $i\geq i_0$. Then for every $i\geq i_0$, we have
\begin{eqnarray*}
\|T_i y -Ty\|&\leq & \|T_i y - T_ix_{n_0}\|+ \|T_i x_{n_0} - Tx_{n_0} \| +\|Tx_{n_0} -Ty\|\\
&<& 2M \frac{\ep}{3M} + \ep/3 =\ep.
\end{eqnarray*}

When $E$ is a total set, it is easy to see that the convergence $T_ix\ra Tx$ in norm for all $x\in E$ extends to the same convergence for all $x\in \lan E \ran$ which is dense in $H$ by definition of a total set.
\item [(ii)] It follows from a similar argument.
\end{itemize}
\end{proof}
\begin{exercise}
Write the proof of Part (ii) of the above proposition.
\end{exercise}
The following example shows how strong and weak topology differ from norm topology and from each other. It also gives some hints for how to compare these topologies.
\begin{example}
\label{exa:shifttopologies}
Let $S$ be the unilateral shift operator described in Example \ref{exa:shiftoperator}.
\begin{itemize}
\item [(i)] The sequence $\{ S^n\}$ weakly converges to zero, but it is not strongly convergent to zero. Clearly, this sequence is norm bounded, so we can apply the above proposition. Since $\{\d_m; m\in \n\}$ is a total set in $\ell^2$, it is enough to show that $\lan S^n \d_m , x\ran \ra 0$ as $n\ra \infty$ for all $x\in \ell^2$ and $m\in \n$. In fact, $\{\d_m; m\in \n\}$ is an orthonormal basis of $\ell^2$. Hence for every $x\in \ell^2$, we have
\[
x=\sum_{m=1}^\infty x_m \d_m,
\]
where $x_m=\lan x, \d_m\ran \ra 0$ as $m\ra \infty$. This implies that $\lan S^n \d_m , x\ran =\lan \d_{n+m} , x\ran \ra 0$ as $n\ra \infty$. However, $S^n \d_m=\d_{n+m}$ does not approach to zero in norm as $n\ra \infty$.
\item [(ii)] Similar arguments show that the sequence $\{ (S\s)^n\}$ strongly converges to zero, but it is not convergent to zero in norm.
\end{itemize}
\end{example}
\begin{exercise}
Prove Part (ii) of the above example.
\end{exercise}
\begin{proposition}
\label{prop:strongweaktop}
\begin{itemize}
\item [(i)] The weak topology of $B(H)$ is weaker than the strong topology of $B(H)$.
\item [(ii)] The strong topology of $B(H)$ is weaker than the norm topology of $B(H)$.
\end{itemize}
\end{proposition}
\begin{proof} Let $(T_i)$ be a net in $B(H)$ and $T\in B(H)$.
\begin{itemize}
\item [(i)] It amounts to show that if $T_i\ra T$ strongly, then  $T_i\ra T$ weakly. But, this follows immediately from CS inequality as follows:
\[
|\lan (T-T_i) x, y\ran| \leq \|(T-T_i)x\| \|y\|, \qquad \forall x,y\in H.
\]
\item [(ii)] Similarly, it amounts to show that if $T_i\ra T$ in norm, then $T_i\ra T$ strongly. This follows immediately from the following inequality:
\[
\|(T-T_i)x\| \leq \|T-T_i\| \|x\|, \qquad  \forall x\in H.
\]
\end{itemize}
\end{proof}
\begin{exercise}
\label{exe:involutionweakcontinuous}
Show that the involution in $B(H)$ is continuous with respect to the weak operator topology.
\end{exercise}
Example \ref{exa:shifttopologies} shows that the involution is not strongly continuous. To remedy this situation, another locally convex topology is defined on $B(H)$.
\begin{definition}
The locally convex topology defined by semi-norms of the form $T\mapsto \|Tx\| + \|T\s x \|$ for all $T\in B(H)$, where $x$ varies in $H$, is called the {\bf strong-\ss\, operator topology of $B(H)$} (or shortly the strong-\ss\, topology).
\end{definition}
It is clear that the strong-\ss\, operator topology is stronger than the strong operator topology and weaker than the norm topology on $B(H)$.
\begin{remark}
\label{rem:continofoperations}
\begin{itemize}
\item [(i)] Since these topologies are defined by families of semi-norms, both addition and scalar multiplication are jointly continuous in weak, strong and strong-\ss\, operator topologies.
\item [(ii)] Multiplication is separately (on each variable) continuous in all these topologies.
\item [(iii)] Multiplication is also jointly continuous in strong and strong-\ss\, topology on bounded sets.
\item [(iv)] Multiplication is not jointly continuous in weak topology even on bounded sets. For instance, assume $S$ is the unilateral shift operator, then both sequences $\{ S^n\}$ and $\{{S\s}^n\} $ are bounded and weakly convergent to zero, see Example \ref{exa:shiftoperator}. But $(S\s)^n S^n = 1$ for all $n\in \n$, and so the multiplication of these sequences is not convergent to zero.
\item [(v)] The separate weak (and strong) continuity of multiplication implies that $S'$ is weakly (and strongly) closed for all subsets $S\sub B(H)$. Therefore every von Neumann algebra is weakly (and strongly) closed.
\end{itemize}
\end{remark}
\begin{exercise}
Prove Items (ii), (iii) and (v) in the above remark.
\end{exercise}
\begin{proposition}
\label{prop:strongunitball}
The (norm) closed unit ball of $B(H)$ is closed in strong operator topology.
\end{proposition}
\begin{proof} Assume $(T_i)$ is a net in closed unit ball of $B(H)$ such that it is strongly convergent to some $T\in B(H)$. Given $x\in H$, for every $\ep>0$, there is some $i_0$ such that $i\geq i_0$ implies that $\|Tx- T_i x\|<\ep$. Thus we have
\[
\|Tx\|<\|T_ix \| +\ep \leq \|T_i\|\|x\|+\ep\leq \|x\|+\ep.
\]
Therefore $\|Tx\|\leq \|x\|$ for every $x\in H$, and so $\|T\|\leq 1$.
\end{proof}
\begin{proposition}
\label{prop:weakunitball}
The (norm) closed unit ball of $B(H)$ is compact in weak operator topology.
\end{proposition}
\begin{proof}
For convenience, let $B$ denote the closed unit ball of $B(H)$ equipped with the weak operator topology. For given $x,y\in H$, let $D_{x,y}$ be the closed disk of radius $\|x\|\|y\|$ in $\c$. Define
\begin{eqnarray*}
\th:B&\ra& \prod_{x,y\in H} D_{x,y}\\
\th(T)&:=& (\lan Tx,y\ran )_{x,y}.
\end{eqnarray*}
where $\prod_{x,y\in H} D_{x,y}$ is considered with the product topology, and so it is compact by the Tychonoff theorem, see Theorem 1.1 in Chapter 5 of \cite{munkeres}. Let $X$ denote the image of $B$ under $\th$.

We first prove that $\th$ is a homeomorphism from $B$ onto $X$. It is easy to see that $\th$ is one-to-one. Therefore $\th:B\ra X$ is a bijective map. Let $\Sigma$ be the family of the subsets of $B$ of the following form:
\[
U_{T,x_0, y_0, \ep}:=\{S\in B\,; |\lan (T-S)x_0,y_0\ran|<\ep, \},
\]
for some $T\in B$, $x_0, y_0\in H$ and $\ep>0$. Similarly, let $\Delta$ be the family of subsets of $X$ of the following form:
\[
O_{T,x_0,y_0,\ep}:=\{(z_{x,y} )\in X\,; \exists S\in B; z_{x_0, y_0}=\lan Sx_0, y_0\ran,\, |\lan (T-S)x_0,y_0\ran|<\ep \},
\]
for some $T\in B$, $x_0, y_0\in H$ and $\ep>0$. One easily checks that $\th$ is a bijective correspondence between $\Sigma$ and $\Delta$.  Since $\Sigma$ generated the weak operator topology in $B$  and $\Delta$ generates the product topology in $X$, $\th$ and $\th\inv$ are both continuous. Hence $\th$ is a homeomorphism.

Now, it is enough to show that $X$ is closed in  $\prod_{x,y\in H} D_{x,y}$. Let $f=(f_{x,y})$ be a limit point in $X$. Define $F:H\times H\ra \c$ by $x,y\mapsto f_{x,y}$. We want to show that $F$ is a sesquilinear map. Assume $x_1,y_1, x_2,y_2 \in H$, and $\la\in \c$ are given. For every $\ep_0>0$, set $\ep:=\min\{ \ep_0, \ep_0/|\la|\}$. Then for every $T$ in the open neighborhood $U$ around $f$ defined by
\[
U:=\left(\bigcap_{i,j=1}^2 U_{f,x_i,y_j,\ep}\right) \bigcap \left(\bigcap_{j=1}^2 U_{f, \la x_1+x_2 , y_j,\ep}\right) \bigcap \left(\bigcap_{i=1}^2 U_{f, x_i,\la y_1+y_2,\ep}\right),
\]
we have
\[
|F(x_i, y_j) - \lan Tx_i, y_j\ran |<\ep_0, \qquad \forall i,j=1,2,
\]
\[
|\la F(x_i, y_j) -\la \lan Tx_i, y_j\ran |<\ep_0, \qquad \forall i,j=1,2,
\]
\[
| F(\la x_1+x_2, y_j) - \lan T(\la x_1+x_2), y_j\ran |<\ep_0, \qquad \forall j=1,2,
\]
\[
| F(x_i, \la y_1+y_2) - \lan Tx_i, \la y_1+y_2 \ran |<\ep_0, \qquad \forall i=1,2.
\]
Hence we have
\[
|F(\la x_1+x_2, y_1)-\la F(x_1,y_1) -F(x_2,y_1)| <3\ep_0,
\]
\[
|F(x_1, \la y_1+y_2) -\overline{\la} F(x_1, y_2)- F(x_1, y_2)|<3\ep_0.
\]
Since $\ep_0$ is arbitrary, these show that $F$ is sesquilinear. On the other hand, since for every $x,y\in H$, $F(x,y)\in D_{x,y}$, we have
\[
|F(x,y)|\leq \|x\|\|y\|, \qquad \forall x,y\in H,
\]
and so $F$ is bounded, in fact $\|F\|\leq 1$. Therefore by Theorem \ref{thm:sesquilinear}, there exists $T_f\in B(H)$ such that $F(x,y)=\lan T_f x,y\ran$ for all $x,y\in H$ and $\|T_f\|\leq 1$. Hence $f=\th(T_f)\in X$. This shows that $X$ is closed, and consequently compact.
\end{proof}
\begin{exercise}
Let $T\in B(H)$. Show that the sesquilinear form defined by $(x,y)\mapsto \lan Tx, y\ran$ is positive if and only if $T$ is positive.
\end{exercise}
\begin{lemma}
\label{lem:positivenumerical}
Let $T\in B(H)$ be positive and set
\[
M:=\sup \{ \lan T x,x\ran; \|x\|=1\}.
\]
Then for every $x\in H$, we have
\begin{equation}
\label{eqn:positivenumerical}
\|T x\|^2\leq M\lan Tx, x\ran \leq M^2 \|x\|^2.
\end{equation}
\end{lemma}
\begin{proof}
The sesquilinear form  defined by $(x,y):=\lan Tx, y \ran$ is positive. Therefore by Proposition \ref{prop:csequality}, the Cauchy--Schwartz inequality holds for $(-,-)$. Hence for every $x,y\in H$, we have
\begin{eqnarray*}
|\lan T x,y\ran|^2&=&|(x,y)|^2\leq (x,x) (y,y)\\
&=&\lan T x,x\ran\lan T y,y\ran\leq \lan T x,x\ran M\|y\|^2\\
&\leq& M^2\|x\|^2\|y\|^2.
\end{eqnarray*}
Setting $y:=Tx$ and dividing through by $\|Tx\|^2$, we obtain (\ref{eqn:positivenumerical}) for every $x\in H$.
\end{proof}
\begin{proposition}
\label{prop:suppositiveopt}
Let $(T_i)$ be an increasing and bounded net of positive operators in $B(H)$. Then there exists a positive operator $T\in B(H)$ such that $T_i\ra T$ strongly,
\[
\|T\|=\sup_i \|T_i\|,
\]
and $T$ is the least upper bound of $(T_i)$ in the directed set $(B(H)_+,\leq)$.
\end{proposition}
For the definition of the least upper bound of a subset in a directed set see Definition \ref{def:leastgreatest}.
\begin{proof}
Using CS inequality, it is clear that the net $(\lan T_ix, x\ran)$ is increasing and bounded in $\r$ for every $x\in H$, so it has a limit in $\r$. Using the polarization identity, one observes that the limit $\lim_i \lan T_i x,y\ran$ exists as well for every $x,y\in H$. Therefore we can define a map as follows:
\begin{eqnarray*}
F:H\times H&\ra &\c,\\
(x,y)&\mapsto& \lim_i \lan T_i x,y\ran.
\end{eqnarray*}
It is straightforward to check that $F$ is a sesquilinear form on $H$. Moreover, for every $x,y\in H$, we have
\[
|F(x,y)|=|\lim_i \lan T_i x,y\ran|\leq \lim_i \|T_i\|\|x\|\|y\|.
\]
Hence $F$ is bounded too, and consequently by Theorem \ref{thm:sesquilinear}, there exists an operator $T\in B(H)$ such that $\lim_i \lan T_i x,y\ran=F(x,y)=\lan Tx,y\ran$ for all $x,y\in H$ and $\|T\|\leq \lim_i \|T_i\|$. Since $\lan T_i x,x\ran\leq \lan Tx , x \ran$ for all $i$ and $x\in H$, $T$ is positive and $T_i\leq T$ for all $i$. This also shows that $\|T\|=\lim_i \|T_i\|$.

For every $i$, set $M_i:=\sup \{ \lan (T-T_i) x,x\ran; \|x\|=1\}$. Then by Lemma \ref{lem:positivenumerical}, we have
\[
\|(T-T_i)x\|^2\leq M_i^2 \|x\|^2, \qquad  \forall  x\in H.
\]
Since $\lim_i M_i=0$, the above inequality implies that $T_i\ra T$ strongly.
\end{proof}
\begin{corollary}
\label{cor:strongdecreasing}
Every decreasing net $(T_i)$ of positive operators converges strongly to a positive operator.
\end{corollary}
\begin{proof}
Pick an $i_0$ from the index set and define $S_i:=T_{i_0} -T_i$ for $i\geq i_0$. The net $(S_i)_{i\geq i_0}$ is a bounded increasing net of positive operators. Moreover, $T_0$ is an upper bound for this net. Hence by the above proposition, it converges strongly to some positive operator $S$. One checks that $T_{i_0}-S$ is a positive operator and $T_i\ra T_{i_0}-S$ strongly.
\end{proof}
The last part of the proof of Proposition \ref{prop:suppositiveopt} worths to be considered as well:
\begin{proposition}
\label{prop:weakincreasing} Let $(T_i)$ be an increasing (resp. decreasing) net of self adjoint operators. If $(T_i)$ is weakly convergent to some operator $T\in B(H)$, then $T$ is the least upper bound (resp. greatest lower bound) of $(T_i)$ and $T_i\ra T$ strongly.
\end{proposition}
\begin{proof}
One notes that $T$ is self adjoint. Also, when $(T_i)$ is decreasing, one should consider $(-T_i)$ instead.
\end{proof}
\begin{proposition}
\label{prop:weakstrongnet1}
Let $(T_i)$ be a net in $B(H)$ such that $T_i\s T_i\ra 0$ weakly. Then
\begin{itemize}
\item [(i)] $T_i\ra 0$ strongly, and
\item [(ii)] if $(T_i)$ is bounded, then $T_i\s T_i\ra 0$ strongly.
\end{itemize}
\end{proposition}
\begin{proof}
\begin{itemize}
\item [(i)] For every $x\in H$, we have $\|T_i x\|^2= \lan  T_i x, T_i x \ran= \lan T_i\s T_i x, x \ran\ra 0$.
\item [(ii)] Let $M>0$ be a positive number such that $\|T_i\|\leq M$ for all $i$. Then for every $x\in H$, we have $\|T_i\s T_i x\|\leq M\|T_i x\|\ra 0$.
\end{itemize}
\end{proof}
\begin{corollary}
\label{cor:weakstrongzero}
Let $(T_i)$ be a bounded net of positive operators in $B(H)$ such that $T_i\ra 0$ weakly. Then $T_i\ra 0$ strongly.
\end{corollary}
\begin{proof}
Set $S_i:=T_i^{1/2}$ and apply Part (ii) of the above proposition.
\end{proof}
\begin{proposition}
\label{prop:unitaryweakstrong}
The weak, strong, and strong-$\s$ topologies coincide on the group $U(H)$ of unitary operators on $H$ and make $U(H)$ into a topological group.
\end{proposition}
\begin{proof}
Assume $(T_i)$ be a net of unitary operators converging weakly to a unitary operator $T$. Then for every $x\in H$, $T_i x\ra T x$ in weak topology of $H$. By Proposition \ref{prop:weaknormconv}, $T_i x\ra Tx$ in norm if and only if $\|T_i x\|\ra \|Tx\|$. But the latter convergence is obvious due to the fact that $\|T_ix\|=\|x\|=\|Tx\|$ for all $i$. Therefore $T_i\ra T$ strongly. The strong-$\s$ convergence of $T_i\ra T$ is proved similarly. The converses of these implications follow from Proposition \ref{prop:strongweaktop}(i) and definition of strong-$\s$ operator topology.
\end{proof}

\begin{theorem}
\label{thm:weakstrongfunctional}
Let $\ff:B(H)\ra \c$ be a bounded linear functional. Then the following statements are equivalent:
\begin{itemize}
\item [(i)] $\ff(T)=\sum_{k=1}^n \lan Ty_k, z_k\ran$ for some $y_1,\cdots, y_n,z_1,\cdots, z_n\in H$.
\item [(ii)] $\ff$ is weakly continuous.
\item [(iii)] $\ff$ is strongly continuous.
\end{itemize}
\end{theorem}
\begin{proof}
Implications (i)$\Rightarrow$(ii)$\Rightarrow$(iii) are clear. Assume $\ff$ is strongly continuous, then $\ff\inv (\{\la\in \c; |\la|<1\})$ is open in $B(H)$. Therefore there are $h_1,\cdots,h_n\in H$ and $\ep>0$ such that $\ff(U_{h_1,\cdots,h_n,\ep})\sub \{\la\in \c; |\la|<1\}$, where
\[
U_{h_1,\cdots,h_n,\ep}=\{T\in B(H); \|Th_k\|<\ep, \, \forall k=1,\cdots,n\}.
\]
Set $y_k:=\frac{h_k}{\ep}$ for all $k=1,\cdots,n$. Then for every $T\in B(H)$, if $\|Ty_k\|<1$ for all $k=1,\cdots,n$, then $|\ff(T)|<1$. This implies that if $T\in B(H)$ and $\|Ty_k\|\leq 1$ for all $k=1,\cdots,n$, then $|\ff(T)|\leq 1$. Thus
\[
|\ff(T)|\leq \max \{\|Ty_k\|, k=1,\cdots, n\}\leq \left(\sum_{k=1}^n \|Ty_k\|^2\right)^2, \quad \forall T\in B(H).
\]
Consider $H^n=\oplus_{k=1}^n H$ and define an operator $D:B(H)\ra B(H^n)$ by
\[
D(T)(x_1,\cdots,x_n):=(Tx_1,\cdots, Tx_n), \quad \forall T\in B(H), (x_1,\cdots,x_n)\in H^n.
\]
Set $y:=(y_1,\cdots, y_n) \in H^n$, $X:=\{ D(T)y; T\in B(H)\}$, and $K:=\overline{X}$. Then $\psi:X\ra \c$ defined by $\psi(D(T)y):=\ff(T)$ is a bounded linear functional, and therefore it extends to a bounded linear functional on $K$. Since $K$ is a Hilbert space, by Theorem \ref{thm:sesquilinear}, there exists $z=(z_1,\cdots,z_k)\in H^n$ such that $\psi(x)=\lan x,z\ran$ for all $x\in K$. For every $T\in B(H)$, put $x:=D(T)y=(Ty_1,\cdots, Ty_n)$. Then we obtain
\[
\ff(T)=\psi(D(T))=\lan (Ty_1,\cdots, Ty_n),(z_1,\cdots,z_k) \ran = \sum_{k=1}^n \lan Ty_k, z_k\ran.
\]
This proves (i).
\end{proof}

\begin{corollary}
\label{cor:weakstrongconvexclosed} Let $X$ be a convex set in $B(H)$. Then $X$ is strongly closed if and only if it is weakly closed.
\end{corollary}
\begin{proof} Assume $X$ is strongly closed. It is enough to show that if $S\in B(H)-X$, then $S$ does not belong to the weak closure of $X$. Since $X$ is convex, by Proposition \ref{prop:locallyconvexfunctional}, there exists a strongly continuous functional $\rho\in B(H)\s$ and $b\in \r$ such that $Re\rho(S)>b$ and $Re\rho(T)\leq b$ for all $T\in X$.
By the above theorem, $\rho$ is weakly continuous too. Thus the set
\[
C:=\{ T\in B(H); Re\rho(T)\leq b \}
\]
is weakly closed in $B(H)$, contains $X$ and $S\notin C$. Therefore $X$ is weakly closed. The other implication is clear.
\end{proof}
\section{The Borel functional calculus}
\label{sec:Borelfunctionalcalculus}

For the Borel functional calculus, we follow Nik Weaver's book, \cite{weaver},  where it was explained more clearly. It is based on Theorem \ref{thm:rieszrep}, see Theorem 7.17 of \cite{folland-ra} for the proof. In order to formulate the theorem, we need to recall the norm of a complex measure. Let $X$ be a locally compact and Hausdorff topological space and let $M(X)$ denote the space of all complex Radon measures on $X$. For every complex measure $\mu$ on $X$, the {\bf total variation of $\mu$} is the positive measure $|\mu|$ which is defined as follows: If $\mu=g_\mu d\nu$, where $\nu$ is a positive measure on $X$ and such a decomposition exists by Theorem 3.12 of \cite{folland-ra}, then $|\mu|:=|g_\mu |\nu$. It was proved in Page 93 of \cite {folland-ra} that $|\mu|$ is well defined. Then the norm of $\mu$ is defined by $\|\mu\|:=|\mu|(X)$. It was shown in Proposition 7.16 of \cite{folland-ra} that this defines a norm on $M(X)$.

\begin{theorem}
\label{thm:rieszrep} [The Riesz representation theorem] Let $X$ be as above. For every $\mu\in M(X)$, we define
\begin{eqnarray*}
I_\mu :C_0(X)&\ra& \c\\
I_\mu(f)&:=&\int_X f(x) d\mu(x), \qquad \forall f\in C_0(X).
\end{eqnarray*}
Then the map $\mu\mapsto I_\mu$ is an isometric isomorphism from $M(X)$ onto $C_0(X)\s$.
\end{theorem}
In our discussion, $X$ is the spectrum of a normal operator on a Hilbert space. Thus, from now on, we assume $X$ is compact. By Theorem 7.8 of \cite{folland-ra}, every finite Borel measure on $X$ is Radon. On the other hand, a complex measure never takes an infinite value, and so a positive measure is complex if and only if it is finite, see page 93 of \cite{folland-ra}. So, we can drop the finiteness condition as well and say that $M(X)$ is the space of all complex Borel measures on $X$. One also notes that every complex measure is a linear combination of four positive measures, so in many situations, we can restrict ourself to positive measures.

\begin{remark}
\label{rem:borelfc}
Now, let $B(X)$ denote the space of all bounded Borel functions on $X$, i.e. all bounded functions $f:X\ra \c$ such that, for every open set $U\sub \c$, $f\inv(U)$ belongs to the $\si$-algebra generated by open subsets of $X$. The vector space $B(X)$ is equipped with supremum norm. Consider the map $\th:B(X)\ra M(X)\s$ defined by
\[
\th(f)(\mu):=\int_X f(x) d\mu(x),\qquad \forall f\in B(X).
\]
Then for $f\in B(X)$, we have
\begin{eqnarray*}
\left| \int_X f(x) d\mu(x)\right| &\leq& \|f\|_{\sup} \left| \int_X g_\mu d\nu(x)\right|\\
&\leq& \|f\|_{\sup} \int_X |g_\mu|d\nu(x) \\
&=&\|f\|_{\sup} \int_X d|\mu|(x)\\
&=& \|f\|_{\sup}\, |\mu|(X)\\
&=& \|f\|_{\sup}\, \|\mu\|.
\end{eqnarray*}
Therefore $\th$ is bounded, and in fact, $\|\th\|\leq 1$. For $\ep>0$, pick $x_0\in X$ such that $\|f\|_{\sup} -\ep \leq |f(x_0)|$ and let $\d_{x_0}$ denote the Dirac measure at $x_0$. Then $\|\d_{x_0}\|=1$ and we have
\begin{eqnarray*}
\|f\|_{\sup} -\ep &\leq& |f(x_0)|\\
&=& \left| \int_X f(x) d\d_{x_0}(x) \right|\\
&=& |\th(f)(\d_{x_0})|\\
&\leq& \|\th\|\|f\|_{\sup}.
\end{eqnarray*}
This shows that $\th$ is an isometry from $B(X)$ into $M(X)\s\simeq C(X)^{\ast \ast}$. Therefore one can consider $B(X)$ as a subspace of the dual space $M(X)\s$. We equip $B(X)$ with the \ws-topology on $M(X)\s$, that is a net $(f_\la)$ of elements of $B(X)$ converges to some $f\in B(X)$ if
\[
\int_X f_\la(x) d\mu(x)\ra \int_X f(x) d\mu(x), \quad \forall \mu\in M(x).
\]
The next step is to show $C(X)$ is \ws dense in $B(X)$. We use the fact that $C(X)$ is dense in $L^1(X,\mu)$ for every $\mu\in M(X)$, see Proposition 7.8 of \cite{folland-ra}. Let $f\in B(X)$. Since integration is linear, without loss of generality, we assume that $f$ is non-negative. Recall that a basis of neighborhoods of $f$ in \ws topology of $C(X)^{\ast \ast}$ is given by sets of the form
\[
U_{f, \mu_1, \cdots, \mu_k, \ep}:=\{ g\in C(X)^{\ast \ast}; |g(\mu_i) - f(\mu_i)|<\ep, \forall 1\leq i\leq k\},
\]
for some $k\in \n$, $\mu_1,\cdots,\mu_k\in M(X)$ and $\ep>0$. So, fix $\mu_1,\cdots,\mu_k\in M(X)$. For every $1\leq i\leq k$, there exists a sequence $\{f_{i,n}\}$ in $C(X)$ such that
\[
\int_X f_{i,n}(x)d\mu_i(x)\ra \int_X f(x)d\mu_i(x).
\]
We define a new sequence $\{f_n\}$ by setting $f_n:=f_{i,m}$, where $1\leq i\leq k$ and $n=(m-1)k +i$. Clearly, we have \[
\int_X f_{n}(x)d\mu_i(x)\ra \int_X f(x)d\mu_i(x), \quad \forall 1\leq i\leq k
\]
This shows that the sequence $\{f_n\}$ lies in $U_{f, \mu_1, \cdots, \mu_k, \ep}$ eventually. Therefore $f$ belongs to the closure of $C(X)$ in \ws topology.
\end{remark}
\begin{theorem}
\label{thm:borelfunctionalcalculus} [The Borel functional calculus] Let $H$ be a Hilbert space. For every normal operator $T\in B(H)$, there exists a unique one-to-one \ss-homomorphism $\Psi_T: B(\si(T))\ra B(H)$ extending the continuous functional calculus of $T$.

Moreover, $\Psi_T$ is continuous with respect to the \ws topology of $B(\si(T))$ and the weak operator topology of $B(H)$.
\end{theorem}
\begin{proof} Consider the \ss-isomorphism defining the continuous functional calculus of $T$, i.e.  $\Phi_T:C(\si(T))\ra B(H)$. For convenience, let us denote the restriction of its double adjoint $\Phi_T^{\ast \ast}:C(\si(T))^{\ast \ast} \ra B(H)^{\ast\ast}$ to $B(\si(T))$ by $\Theta$. By using Exercise \ref{exe:dualspaces}(iv) twice, we have
\begin{equation}
\label{eqn:borelfc1}
\|\Theta\|=\|\Phi_T\|=1.
\end{equation}
For every $x,y\in H$, let $\rho_{x,y}\in B(H)\s$ be the linear functional defined by $S\mapsto \lan Sx, y\ran$ for all $S\in B(H)$. One checks that
\begin{equation}
\label{eqn:borelfc2}
\|\rho_{x,y}\|= \|x\|\|y\|,
\end{equation}
see Exercise \ref{exe:borelfc1}(i). For every $f\in B(\si(T))$, we define a map
\begin{eqnarray*}
& &\{-,-\}_f:H\times H\ra \c\\
& &\{x,y\}_f:= \Theta(f)(\rho_{x,y}).
\end{eqnarray*}
It is straightforward to check that $\{-,-\}_f$ is a sesquilinear map. Also, using Equations (\ref{eqn:borelfc1}) and (\ref{eqn:borelfc2}), we have
\[
|\{x,y\}_f|\leq \|\Theta(f)\|\|\rho_{x,y}\|\leq \|f\|_{\sup} \|x\|\|y\|.
\]
Hence $\{ -,-\}$ is bounded. Therefore by Theorem \ref{thm:sesquilinear}, there exists a bounded operator in $B(H)$, which we denote it by $\Psi_T(f)$, such that $\{x,y\}_f=\lan \Psi_T(f) x, y\ran$. The map $\Psi_T:B(\si(T))\ra B(H)$ is called the {\bf Borel functional calculus of $T$}. Let us check that $\Psi_T$ extends the continuous functional calculus of $T$. For every $f\in C(\si(T))$ and $x,y\in H$, we have
\[
\lan \Psi_T(f) x,y\ran =\{ x,y\}_f=\Theta(f)(\rho_{x,y})=\rho_{x,y}(\Phi_T(f))= \lan \Phi_T(f) x,y\ran.
\]
Therefore $\Psi_T(f)=\Phi_T(f)$ for all $f\in C(\si(T))$. It is straightforward to check that $\Psi_T$ is a linear map, see Exercise \ref{exe:borelfc1}(ii). Let $(f_\la)$ be a net in $B(\si(T))$ convergent to $f\in B(\si(T))$ in \ws-topology. Then for every $x,y\in H$, we have
\begin{eqnarray*}
\lan \Psi_T(f_\la) x,y\ran&=& \{x,y\}_{f_\la}\\
&=& \Theta(f_\la)(\rho_{x,y})\longrightarrow  \Theta(f)(\rho_{x,y})\\
&=& \{x,y\}_{f}=\lan \Psi_T(f) x,y\ran.
\end{eqnarray*}
Therefore $\Psi_T$ is continuous with respect to the \ws topology of $B(\si(T))$ and the weak operator topology of $B(H)$. This also justifies our next computations, wherein  all limits are taken with respect to the weak operator topology of $B(H)$.

For given $f\in B(\si(T))$, let $(f_\la)$ be a net in $C(\si(T))$ such that $f_\la\ra f$ in \ws topology. Then for every $g\in C(\si(T))$, we compute
\begin{eqnarray*}
\Psi_T(fg)&=&\lim_\la \Psi_T(f_\la g)=\lim_\la \Phi_T(f_\la g)\\
&=&\lim_\la \Phi_T(f_\la) \Phi_T(g)=\lim_\la \Psi_T(f_\la) \Psi_T(g)\\
&=&\Psi_T(f) \Psi_T(g).
\end{eqnarray*}
For every $f\in B(\si(T))$ and $g\in B(\si(T))$, let $(g_\mu)$ be a net in $C(\si(T))$ such that $g_\mu\ra g$ in \ws topology. Then using the above computation, we have
\[
\Psi_T(fg)=\lim_\mu \Psi_T(fg_\mu) = \Psi_T(f) \lim_\mu \Psi_T (g_\mu) = \Psi_T(f)\Psi_T(g).
\]
This shows that $\Psi_T$ is a multiplicative map. One notes that we had to break the argument in two steps, because the multiplication of $B(H)$ is only separately continuous in the weak operator topology, see Remark \ref{rem:continofoperations}(ii).

We know that $\Phi_T$ is a \ss-homomorphism, so $\Phi_T(\overline{f})=\Phi_T(f)\s$ for all $f\in C(\si(T))$. Let $f\in B(\si(T))$ and let $(f_\la)$ be a net in $C(\si(T))$ such that $f_\la\ra f$ in \ws topology. Then $\overline{f_\la}\ra \overline{f}$ and, using Exercise \ref{exe:involutionweakcontinuous}, we compute
\begin{eqnarray*}
\Psi_T(\overline{f})&=&\lim_\la \Psi_T(\overline{f_\la})= \lim_\la \Phi_T(\overline{f_\la})\\
&=& \lim_\la\Phi_T (f_\la)\s =\lim_\la\Psi_T (f_\la)\s \\
&=&\left( \lim_\la\Phi_T (f_\la)\right)\s =\Psi_T(f)\s.
\end{eqnarray*}
This shows that $\Psi_T$ is a \ss-homomorphism as well.

Finally, the uniqueness of $\Psi_T$ follows from the fact that $C(\si(T))$ is \ws dense in $B(\si(T))$, see \ref{rem:borelfc}.
\end{proof}
Similar to the continuous functional calculus, for every $f\in B(\si(T))$, $\Psi_T(f)$ is denoted by $f(T)$.
\begin{exercise} Assume the notations of the above proof.
\label{exe:borelfc1}
\begin{itemize}
\item [(i)] Verify Equation (\ref{eqn:borelfc2}).
\item [(ii)] Check that $\Psi_T$ is a linear map.
\end{itemize}
\end{exercise}

\section{Projections and the polar decomposition}
\label{sec:projections}
In this section $H$ is a Hilbert space. We first study elementary topics about projections in $B(H)$. afterwards, we prove the polar decomposition of elements of $B(H)$.

Assume $X$ be a closed subspace of $H$. Define a map $P:H\ra X$ by defining $Px$ to be the orthogonal projection of $x$ on $X$. By Corollary \ref{cor:perp}(ii), $H=X\oplus X^\perp$. It follows from this decomposition of $H$ that $P$ is a linear map. For every $x\in H$, assume $x=x_1+x_2$, where $x_1=Px\in X$ and $x_2=x-Px\in X^\perp$. Then using Exercise \ref{exe:orthogonalproj}, we have
\[
\|Px\|=\|P(x_1+x_2)\|=\|Px_1\|= \|x_1\|\leq \|x_1+x_2\|=\|x\|.
\]
This shows that $P$ is also bounded. It is clear that $P^2=P$. It also follows from the above decomposition of $H$ that $P\s=P$. Therefore $P$ is a projection in $B(H)$. This projection is determined completely by $X$ and usually is denoted by $P_X$. One also notes that $1-P_X=P_{X^\perp}$. For every projection $P\in B(H)$, $1-P$ is a projection again and is called the {\bf complement of $P$} and is denoted by $P^\perp$.

Conversely, let $P$ be a projection in $B(H)$. Then $R(P)=N(1-P)$, and so $X=R(P)$ is a closed subspace of $H$ and $P=P_X$. This shows that there is a bijective correspondence between projections in $B(H)$ and closed subspaces of $H$.

The proof of the following proposition is easy and is left as an exercise.
\begin{proposition}
\label{prop:projection1}
Let $X$ and $Y$ be two closed subspaces of $H$. The following statements are equivalent:
\begin{itemize}
\item [(i)] $P_X\leq P_Y$.
\item [(ii)] $P_X\leq \la P_Y$ for some $\la>0$.
\item [(iii)] $X\sub Y$.
\item [(iv)] $P_X P_Y=P_Y P_X=P_X$.
\item [(v)] $P_Y-P_X$ is a projection in $B(H)$, (In fact, $P_Y-P_X=P_{Y\cap X^\perp}$).
\end{itemize}
\end{proposition}
Two projections $P, Q\in B(H)$ are called {\bf orthogonal} if $PQ=QP=0$. This is denoted by $P\perp Q$, or equivalently $Q \perp P$. For example $P$ and $1-P$ are orthogonal.
\begin{exercise}
Let $P, Q\in B(H)$ be two projections. Show that $P\perp Q$ if and only if $Q\leq 1-P$.
\end{exercise}
Given projections $P_{X_i}$ associated with closed subspaces $X_i$ for $i\in I$, where $I$ is an arbitrary index set, we define $\wedge_{i} P_{X_i} := P_{\cap_i X_i}$ and $\vee_i P_{X_i} :=P_{\overline{\sum_i X_i}}$. Then for every $i_0\in I$, we have
\[
\wedge_i P_{X_i} \leq P_{X_{i_0}},\quad \text{and}\quad P_{X_{i_0}}\leq \vee_i P_{X_i}.
\]
When $I=\{1,2\}$, $\wedge_i P_{X_i}$ and $\vee_i P_{X_i}$ are denoted by $P_{X_1}\wedge P_{X_2}$ and $P_{X_1}\vee P_{X_2}$, respectively. These properties are understood better using the notion of a lattice.

\begin{definition}
\label{def:leastgreatest}
Let $(S,\leq)$ be a partially ordered set.
\begin{itemize}
\item [(i)] A {\bf lower bound of a subset $T\sub S$} is an element $l\in S$ such that $l\leq t$ for all $t\in T$. The {\bf greatest lower bound} (shortly, {\bf g.l.b.}) {\bf of  $T$} is a lower bound $g$ of $T$ such that $l\leq g$ for every lower bound $l$ of $T$. (The uniqueness of g.l.b. easily follows from the definition.)
\item [(ii)] An {\bf upper bound of a subset $T\sub S$} is an element $u\in S$ such that $t\leq u$ for all $t\in T$. The {\bf least upper bound} (shortly, {\bf l.u.b.}) {\bf of  $T$} is an upper bound $l$ of $T$ such that $l\leq u$ for every upper bound $u$ of $T$. (The uniqueness of l.u.b. easily follows from the definition.)
\item [(iii)] A {\bf lattice} is a partially ordered set, say $(S,\leq)$, such that, for every subset $T\sub S$ with two elements, there exist the greatest lower bound and the least upper bound of $T$.
\item [(iv)] A {\bf complete lattice} is a partially ordered set, say $(S,\leq)$, such that, for every subset $T\sub S$, there exist the greatest lower bound and the least upper bound of $T$.
\end{itemize}
\end{definition}
One can learn more about lattices in \cite{grillet}.
\begin{exercise}
Show that the set $P(H)$ of all projections in $B(H)$ equipped with partial order $\leq$ is a complete lattice. Describe the greatest lower bound and the least upper bound of a subset of $P(H)$ in terms of the above notations.
\end{exercise}
\begin{exercise}
Let $P$ and $Q$ be two projections in $B(H)$. If $P$ and $Q$ commute, then show that $P\wedge Q=PQ$ and $P\vee Q=P+Q-PQ$. Conclude that if $P\perp Q$, then $PQ$ and $P+Q$ are both projections.
\end{exercise}
\begin{exercise}
\label{exe:orthogonalwedge} Let $P$ and $Q$ be two projections in $B(H)$.
\begin{itemize}
\item [(i)] Show that $P\leq Q$ if and only if $1-Q\leq 1- P$.
\item [(ii)] Show that $1-(P\vee Q)=(1-P)\wedge (1-Q)$.
\item [(iii)] Show that $1-(P\wedge Q) =(1-P)\vee (1-Q)$.
\item [(iv)] Generalize the above statements for arbitrary family of projections in $B(H)$.
\end{itemize}
\end{exercise}
\begin{proposition}
\label{prop:incdecstrongly}
\begin{itemize}
\item [(i)] Assume $(P_i)$ is an increasing sequence of projections in $B(H)$, then $P_i\ra \vee_i P_i$ strongly.
\item [(ii)] Assume $(P_i)$ is an decreasing sequence of projections in $B(H)$, then $P_i\ra \wedge_i P_i$ strongly.
\end{itemize}
\end{proposition}
\begin{proof} For $i\in \n$, set $X_i=R(P_i)$. Then we have $P_i=P_{X_i}$.
\begin{itemize}
\item [(i)] Set
\[
X:=\overline{\sum_{i=1}^\infty X_i}
\]
and set $P:=P_X$. Since $\{P_i\}$ is an increasing sequence,
\[
X= \overline{\bigcup_{i=1}^\infty X_i}.
\]
For every $x\in X$, we can write $x=\sum_{i=1}^\infty x_i$, where $x_i\in X_i$. Thus $\sum_{i=1}^n x_i \ra x$ and $\sum_{i=1}^n x_i\in X_n$. Therefore for every $\ep>0$, there exists $n_\ep\in \n$ such that $\|P_{X_m} x -x\|<\ep$ for all $m\geq n$. Now, for every $h\in H$ and $\ep>0$, if $i\geq n_\ep$, then $\|Ph-P_ih\|= \|Ph- P_iPh\|<\ep$, because $Ph\in X$. This proves (i).
\item [(ii)] It follows from (i) and Exercise \ref{exe:orthogonalwedge}.
\end{itemize}
\end{proof}
\begin{proposition}
\label{prop:orthomodular}
Let $P,Q, R$ be projections in $B(H)$ such that $P\perp Q$ and $P\leq R$. Then we have
\begin{equation}
\label{eqn:orthomodular}
(P+Q)\wedge R=P\vee (Q\wedge R).
\end{equation}
\end{proposition}
\begin{proof} Let $X, Y, Z$ be closed subspaces of $H$ associated with $P,Q, R$, respectively. Then the above assumptions are equivalent to $Y\sub X^\perp$ and $X\sub R$ and Equality (\ref{eqn:orthomodular}) is equivalent to $(X+Y)\cap R= X+ (Y\cap R)$, which is easy to check.
\end{proof}
\begin{proposition}
\label{prop:topologiesprojections}
The strong, weak, and strong-$\s$ operator topologies coincide on the set of projections in $B(H)$.
\end{proposition}
\begin{proof}
Let $P_i\ra P$ weakly. Then for every $x\in H$, we have
\[
\|P_ix\|^2=\lan P_ix, P_i x \ran=\lan P_i\s P_ix, x \ran= \lan P_i x ,x\ran\ra \lan Px,x\ran=\|Px\|^2.
\]
Hence $P_i\ra P$ strongly. Since every projection is self adjoint, this implies that $P_i\ra P$ in strong-$\s$ operator topology as well. The converses of these implications follows from Proposition \ref{prop:strongweaktop}(i) and definition of strong-$\s$ operator topology.
\end{proof}
\begin{definition}
Let $\Delta= \{ P_i; i\in I\}$ be a family of projections in $B(H)$.
\begin{itemize}
\item [(i)] Elements of $\Delta$ are called {\bf pairwise orthogonal} if $P_iP_j=0$.
\item [(ii)] Let $I$ be finite and $X_i=R(P_i)$, i.e. $P_{X_i}=P_i$. Then $\oplus_{i\in I}P_i:=P_{\oplus_{i\in I} X_i}$ is called the sum of $\Delta$.
\end{itemize}
A family $\{T_\la; \la\in \Lambda\}$ of elements of $B(H)$ is called {\bf summable} in strong (resp. weak, strong-$ \s$) operator topology if $\sum_{\la\in \Lambda} T_\la$ is convergent in strong (resp. weak and strong-$\s$) operator topology.
\end{definition}
\begin{proposition}
\label{prop:summablefamilyprojections}
Every family $\{ P_i; i\in I\}$ of pairwise orthogonal projections in $B(H)$ is summable to the projection $P:=\vee_{i\in I} P_i$ in strong operator topology.

Moreover, we have
\[
\|Px\|=\left( \sum_{i \in I} \|P_i x\|^2 \right)^{1/2}, \qquad \forall x\in H.
\]
When $P=1$, the map $U:H\ra \oplus_{i\in I} P_i(H)$ defined by $x\mapsto (P_i(x))$ is a unitary operator.
\end{proposition}
For every $i\in I$, $P_i(H)$ is a Hilbert space and $\oplus_{i\in I} P_i(H)$ is the direct sum of these Hilbert spaces, see Example \ref{exa:constructions}(i).
\begin{proof}
Let $(\mathcal{F}, \leq)$ be the collection of all finite subsets of $I$ directed by inclusion. For every $F\in \mathcal{F}$, $P_F:=\sum_{i\in F} P_i$ is a projection and the net $(P_F)_{F\in \mathcal{F}}$ is an increasing net of projections. It follows from proposition \ref{prop:incdecstrongly}(i) that the net $(P_F)_{F\in \mathcal{F}}$ converges to $P$ strongly, i.e. $P=\sum_{i\in I} P_i$ in strong topology.

Moreover, since the elements of $\{ P_i; i\in I\}$ are pairwise orthogonal projections by applying Pythagoras' identity, Exercise \ref{exe:Pythagorean}, for every $x\in H$ and $F\in \mathcal{F}$, we have
\[
\|Px\|^2 = \lim_F \|P_Fx\|^2=\lim_F \sum_{i\in F} \|P_i x\|^2= \sum_{i\in I} \|P_i x\|^2.
\]
One easily checks that if $P=1$, then $U\inv:\oplus_{i\in I} P_i(H) \ra H$ is given by
\[
U\inv (x_i)_{i\in I}=\sum_{i\in I} x_i, \qquad \forall (x_i)_{i\in I}\in \bigoplus_{i\in I} P_i(H).
\]
Hence for every $(x_i)_{i\in I}\in \oplus_{i\in I} P_i(H)$ and $x\in H$, we compute
\begin{eqnarray*}
\lan U\inv (x_i) , x\ran &=& \sum_{i\in I} \lan x_i , x\ran =\sum_{i\in I} \lan x_i , P_ix\ran \\
&=& \lan (x_i), (P_i x)\ran = \lan (x_i), Ux \ran = \lan U\s (x_i), x \ran .
\end{eqnarray*}
Therefore $U\inv=U\s$.
\end{proof}

\begin{definition}
Let $H_1$ and $H_2$ be two Hilbert spaces and let $T\in B(H_1,H_2)$. The orthogonal projection on the closed subspace $N(T)^\perp$ of $H_1$ is called the {\bf right support projection of $T$} and is denoted by $P_T$. The orthogonal projection on the closed subspace $\overline{R(T)}$ of $H_2$ is called the {\bf left support projection of $T$} and is denoted by $Q_T$.
\end{definition}
The basic properties of the left and right support projections are listed in the following propositions:
\begin{proposition}
\label{prop:supportprojections1}
Let $H, H_1,H_2$ be Hilbert spaces. For given $A\in B(H_1,H_2)$ and $T\in B(H)$, the following statement hold:
\begin{itemize}
\item [(i)] The left support projection of $A$ is the right support projection of $A\s$, i.e. $Q_A=P_{A\s}$.
\item [(ii)]
\begin{itemize}
\item [(a)] $TP_T=T$.
\item [(b)] For every $S\in B(H)$, if $TS=0$, then $P_T S=0$.
\end{itemize}
    Moreover, the projection $P_T$ with these two properties is unique.
\item [(iii)]
\begin{itemize}
\item [(a)] $Q_T T=T$.
\item [(b)] For every $S\in B(H)$, if $ST=0$, then $SQ_T=0$.
\end{itemize}
    Moreover, the projection $Q_T$ with these two properties is unique.
\item [(iv)] The right support projection of $T$ is the smallest projection $P\in B(H)$  such that $TP=T$.
\item [(v)] If $T$ is normal, then $P_T=Q_T$.
\item [(vi)] $P_T=P_{T\s T}=Q_{T\s T}$ and $Q_T=Q_{TT\s}=P_{TT\s}$.
\end{itemize}
\end{proposition}
\begin{proof}
\begin{itemize}
\item [(i)] By Lemma \ref{lem:rangenull}, $N(A\s)=R(A)^\perp$. Therefore the orthogonal projection on the complement subspace of $N(A\s)$, i.e. $Q_{A\s}$, equals to the orthogonal projection on $\overline{R(A)}$, i.e. $P_A$. Similarly, $Q_A=P_{A\s}$.
\item [(ii)] Conditions (a) and (b) are immediate consequences of the definition. Assume $P$ is a projection in $B(H)$ satisfying (a) and (b). Let $X$ be the closed subspace of $H$ such that $P=P_X$, i.e. $X=R(P)$. Then $1-P=P_{X^\perp}$. Condition (a) implies that $T(1-P)=0$, or equivalently $X^\perp\sub N(T)$. Hence $N(T)^\perp\sub X$. Condition (b) means that if $R(S)\sub N(T)$, then $R(S)\sub N(P)=X^\perp$. Hence $N(T)\sub X^\perp$, or equivalently $X\sub N(T)^\perp$. Therefore $N(T)^\perp=X$, and so $P=P_T$.
\item [(iii)] It follows from (i) and (ii).
\item [(iv)] Assume $P=P_X$ for some closed subspace $X$ of $H$ and $TP=T$. In item (ii), we already proved that $N(T)^\perp\sub X$. Hence $P_T\leq P_X$.
\item [(v)] By Proposition \ref{prop:normaloperator}, when $T$ is normal, $N(T)^\perp=N(T\s)^\perp$. Hence $P_T=P_{T\s}=Q_T$.
\item [(vi)] The first equality follows easily from the fact that $N(T)=N(T\s T)$. The second equality follows from the first one and (i).
\end{itemize}
\end{proof}
\begin{proposition}
\label{prop:rightsupportstrong} Let $T\in B(H)$ be a positive operator and let $t\in ]0,\infty[$. Then we have the following limits in the strong operator topology:
\begin{itemize}
\item [(i)] $P_T=\lim_{t\ra 0} T^t$, and therefore $P_T$ is a strong limit of polynomials in $T$. This also shows that $P_T\in C\s(T)^{\prime \prime}$.
\item [(ii)] $Q_T=\lim_{t\ra 0} (T\s T)^t$ and $P_T=\lim_{t\ra 0} (TT\s)^t$.
\item [(iii)] $P_T=\lim_{t\ra 0} T(T+ t)\inv$.
\end{itemize}
\end{proposition}
\begin{proof}
\begin{itemize}
\item [(i)] Given $t>0$, the function $f(x):=x^t$ is a limit of polynomials whose constant terms are zero. So, by the non-unital continuous functional calculus, if $Tx=0$, then $T^t x=0$ for all $t>0$. This proves $T^tx\ra P_T x$ for all $x\in N(T)$. Since $T$ is positive (normal), by Proposition \ref{prop:normaloperator}, $\overline{R(T)}=N(T)^\perp$. Hence
\begin{equation}
\label{eqn:NRdirectsum1}
H=N(T)\oplus \overline{R(T)}.
\end{equation}
Therefore it is enough to show that $T^t x \ra P_T x$ for all $x\in R(T)$. Since $T$ is positive, by Proposition \ref{prop:positiveroots}(ii), $T^{t+1}\ra T$. For given $x\in R(T)$, $x=T(y)$ for some $y\in H$. Therefore we have \[
Tx=T^{t+1} y \ra Ty=x=Q_T x=P_T x, \qquad \forall x\in R(T).
\]
By Proposition \ref{prop:lessseminormstops}(i), this convergence holds for all $x\in \overline{R(T)}$ as well.
\item [(ii)] It follows from (i) and Part (vi) of the above proposition.
\item [(iii)] Similar to (i), we use the equality (\ref{eqn:NRdirectsum1}) and prove the convergence in two steps. Since $T$ is positive, $T+t$ is invertible for all $t\in ]0,\infty[$. On the other hand, for every $t\in ]0,\infty[$, $(T+t)\inv\in C\s(T,1)$, and so it commutes with $T$. This implies that $N(T)\sub N(T(T+t)\inv)$ for all $t\in ]0,\infty[$. Hence $T(T+t)\inv x= P_Tx=0$ for all $x\in N(T)$. For every $x\in R(T)$, we have
\[
T(T+t)\inv x = (T+t)\inv (T+t) x + tx = x+tx\ra x=Q_T x =P_T x, \quad \text{as}\, t\ra 0.
\]
\end{itemize}
\end{proof}
\begin{definition}
An element $a$ in a \cs-algebra $A$ is called a {\bf partial isometry} if $a\s a$ is a projection in $A$, which is called the {\bf support projection of $a$}. In the context of operators, an operator $T\in B(H_1, H_2)$ between two Hilbert spaces is called  a {\bf partial isometry} if $T\s T$ is a projection in $B(H_1)$, which is called the {\bf support projection of $T$}.
\end{definition}
One easily sees that the above notions of partial isometry coincide on $B(H)$. If $P_0$ is a projection in $B(H)$, then $P_{P_0}=Q_{P_0}=P_0$. Therefore for every partial isometry $T\in B(H)$, we have $P_T=P_{T\s T} = T\s T$. This justifies the name ``support projection''.
\begin{proposition}
\label{prop:partialisometry} Let $T\in B(H_1,H_2)$ be a bounded operator between two Hilbert spaces. Then the following conditions are equivalent:
\begin{itemize}
\item [(i)] $T$ is a partial isometry.
\item [(ii)] $TT\s$ is a projection, ($T\s$ is a partial isometry).
\item [(iii)] $T=TT\s T$.
\item [(iv)] The restriction of $T$ on $N(T)^\perp$ is an isometry.
\end{itemize}
\end{proposition}
\begin{proof}
Assume (i), then $(TT\s)^3=(TT\s)^2$. Therefore by Problem \ref{e:3-13}, $TT\s$ is a projection. Hence (ii) holds. Similarly, (ii) implies (i).

Assume (iii), then $T\s T=(T\s T)^3$. Therefore by Problem \ref{e:3-13}(ii), $T\s T$ is a projection. Hence (i) holds. Conversely, in the above, we observed that (i) implies that $P_T=T\s T$. Therefore $T=TP_T=TT\s T$. Hence (iii) follows from (i).

Assume (iv). For every $x\in N(T)^\perp$, we compute
\[
\lan P_T x, x\ran =\lan x, x\ran= \|x\|^2= \|Tx\|^2= \lan Tx, Tx\ran=\lan T\s T x, x\ran.
\]
Also, for every $x\in N(T)$, we have $\lan P_T x,x\ran =0=\lan T\s T x,x\ran$. Using the decomposition $H=N(T)^\perp \oplus N(T)$, we obtain $\lan P_T x,x\ran =\lan T\s T x,x\ran$ for all $x\in H$. Therefore by Problem \ref{e:5-18}, $P_T=T\s T$, and so (i) holds. If (i) holds, then $P_T=T\s T$. Therefore for every $x\in N(T)^\perp$, we compute
\[
\|Tx\|^2= \lan Tx, Tx \ran= \lan T\s Tx, x\ran=\lan P_T x,x\ran= \lan x,x\ran=\|x\|^2.
\]
Hence (iv) holds.
\end{proof}
\begin{exercise}
Let $T\in B(H)$. Prove that $T$ is a partial isometry if and only if $P_T=T\s T$.
\end{exercise}
Similar to polar decomposition of elements of $\c$, every element $T$ of the algebra $B(H)$ has a left (resp. right) polar decomposition $T=U|T|$ (resp. $T=|T\s|U$), where $U$ is a partial isometries. These decompositions have many applications in the theory of \cs-algebras.
\begin{theorem}
\label{thm:polardecomp} [Polar decomposition] Let $T\in B(H_1, H_2)$ be a bounded operator between two Hilbert spaces. Then there exists a unique partial isometry $U\in B(H_1,H_2)$ such that $T=U|T|$, where $|T|:=(T\s T)^{1/2}$, and $N(T)=N(U)$. Furthermore, $U\s T=|T|$. This decomposition of $T$ is called the {\bf left polar decomposition of $T$}.

Moreover, when $H_1=H_2$, we have $U\in C\s(T,T\s)^{\prime \prime}\sub B(H_1)$.
\end{theorem}
\begin{proof}
By Exercise \ref{exe:adjointopt}(v), $N(T)=N(T\s T)$. Hence the restriction of $T\s T$ to $N(T)^\perp=P_T H_1$ is one-to-one. For every $x\in N(T)^\perp$, we have $T\s Tx=P_T T\s Tx\in N(T)^\perp$, so the map $T\s T :N(T)^\perp\ra N(T)^\perp$ is a well defined bounded operator. Also, it follows from Lemma \ref{lem:rangenull} that the image of this map is dense in $N(T)^\perp$. These facts show that $|T|=(T\s T)^{1/2}:N(T)^\perp \ra N(T)^\perp$ is one-to-one and its image is dense in $N(T)^\perp$. It also follows easily from definition that $T:N(T)^\perp \ra R(T)$ is one-to-one and onto. One also notes that $R(T)$ is dense in $Q_TH_2$.

Now, we define a map $R(|T|)\ra R(T)$ by $|T|x \mapsto Tx$. One easily checks that this is a linear map. Also, by Problem \ref{e:5-20}, we have $\| |T|x\|=\|Tx\|$ for all $x\in H_1$, so it is an isometry. It extends to an isometry $U:N(T)^\perp \ra Q_TH_2$. Since $U$ is an isometry between two Banach spaces and its image is dense, it has to be onto. In fact, by Problem \ref{e:5-19}, $U$ is a unitary operator in $B( N(T)^\perp, Q_TH_2)$. Hence $U\s Tx= |T|x $ for all $x\in H_1$. We can extend $U$ to $H_1$ by setting $Ux=0$ for all $x\in N(T)$. Then by Proposition \ref{prop:partialisometry}(iv), $U\in B(H_1,H_2)$ is a partial isometry. It is clear that $T=U|T|$ and $N(T)=N(U)$. The uniqueness of $U$ follows from these equalities.

When $H_1=H_2$, in order to prove that $U\in C\s(T,T\s)^{\prime \prime}$, it is enough to show that $T(T\s T + t)^{-1/2}\ra U$ strongly as $t\ra 0$ (for $t\in ]0,\infty[$), see Remark \ref{rem:continofoperations}(v). Since $T\s T$ is positive, for every $t>0$, the operator $T\s T +t$ is invertible and positive, and so $(T\s T+t)^{-1/2}$ is well defined. For $x\in N(T)^\perp$, we compute
\begin{eqnarray*}
T(T\s T+t)^{-1/2} |T|x&=& T(T\s T+t)^{-1/2} (T\s T+t) x +T t^{-1/2}x\\
&=& Tx +T t^{-1/2}x \ra Tx= Ux, \qquad \text{(in norm) as}\quad t\ra 0.
\end{eqnarray*}
For $x\in N(T)$, we have $T(T\s T+t)^{-1/2} |T|x= 0 = Ux$. This proves the required convergence.
\end{proof}
\begin{corollary}
\label{cor:rightpolardecomp} Assume $T$ is as the above theorem and $T=U|T|$ is its left polar decomposition. Then $T=|T\s|U$. This decomposition of $T$ is called the {\bf right polar decomposition of $T$}.
\end{corollary}
\begin{proof}
It is clear that $(TT\s)^n= U(T\s T)^nU\s$ for all $n\in \n$. Therefore by Problem \ref{e:3-14}, we have $|T\s|= U|T|U\s$. Since $U\s U=P_T$ and $|T|P_T=|T|$, see Problem \ref{e:5-21}, we obtain $T=U|T|=|T\s|U$.
\end{proof}
Some of the easy properties of the polar decomposition is listed in the following exercise:
\begin{exercise}
\label{exe:polardecomp} Assume $T$ is as the above theorem.
\begin{itemize}
\item [(i)] Show that if $T\s T$ is invertible, then $U$ is an isometry and $U=T(T\s T)^{-1/2}$.
\item [(ii)] Show that if $T$ is invertible (or more generally, $T$ is one-to-one and its image is dense), then $U$ is a unitary operator.
\item [(iii)] Show that if $H_1=H_2$ and $T\s T$ is invertible, then $U\in C\s(T,T\s)$.
\end{itemize}
\end{exercise}


\section{Compact operators}
\label{sec:compactoperators}
In this section, $H$ is a Hilbert space. By Proposition \ref{prop:compactoperators}, the algebra $K(H)$ is a closed two sided ideal of $B(H)$. Therefore by Proposition \ref{prop:closedidealcs}, $K(H)$ is closed under involution, and so is a \cs-subalgebra of $B(H)$. It is called the {\bf \cs-algebra of compact operators on $H$}. When $H$ is an infinite dimensional separable Hilbert space, or equivalently $H\simeq \ell^2$, $K(H)$ is briefly called the {\bf \cs-algebra of compact operators} and is denoted by $\mathcal{K}$. The quotient \cs-algebra $B(H)/K(H)$ is called the {\bf Calkin algebra of $H$} and is denoted by $Q(H)$. When $H\simeq \ell^2$, it is briefly called the {\bf Calkin algebra} and is denoted by $\mathcal{Q}$.
\begin{definition}
An operator $T\in B(H)$ is called {\bf diagonalisable} if there exists an orthonormal basis for $H$ consisting of eigenvectors of $T$.
\end{definition}
\begin{exercise}
Show that every diagonalisable operator $T\in B(H)$ is normal.
\end{exercise}
The converse of the above exercise is not generally true, see the following example:
\begin{example}
Define an operator $S:\ell^2(\z)\ra \ell^2(\z)$ by $S(\d_n):=\d_{n+1}$ and extend this rule linearly, where as usual $\d_n$ is the characteristic function of $\{n\}$. This operator is bounded and is called the {\bf bilateral shift operator}. It is a unitary operator, and so normal. One can checks that $S$ has no eigenvalues, and so is not diagonalisable.
\end{example}
\begin{proposition}
\label{prop:compactdiag}
Every normal compact operator $T\in K(H)$ is diagonalisable.
\end{proposition}
\begin{proof}
Let $E$ be a maximal orthonormal set of eigenvectors of $T$, which exists by Zorn's lemma. Let $H_0$ be the closed span of $E$. Then $H=H_0\oplus H^\perp_0$. One observes that the restriction of $T$ to $H_0$ is a compact operator which we denote it by $T'$. $T'$ is compact and normal. By maximality of $E$, $T'$ has no non-zero eigenvalues. Hence by Theorem \ref{thm:compactoptspec}(ii), $\si(T')$ has no non-zero element. Since $T'$ is normal, this implies that $\|T'\|=r(T')=0$, and so $T'=0$. Therefore $H_0^\perp$ is the eigenspace of the eigenvalue $0$ of $T$ and the union of $E$ with every orthonormal basis of $H_0^\perp$ is an orthonormal basis for $H$ consisting of eigenvectors of $T$. This contradicts with the maximality of $E$ unless $H_0^\perp=0$. In this case, $H=H_0$ and the proof is complete.
\end{proof}
In Proposition \ref{prop:compactoperators}, for every Banach space $E$, we proved that the closure of $F(E)$ is a subalgebra of $K(E)$. When $E$ is a Hilbert space, we can say more as the following proposition:
\begin{proposition}
\label{prop:compactoptfiniterank}
The \cs-algebra $K(H)$ is the closure of $F(H)$.
\end{proposition}
\begin{proof}
Let $T$ be a compact operator. Without loss of generality, using the fact that every element of a \cs-algebra is the linear combination of four positive elements, we can assume that $T$ is positive. Therefore by Proposition \ref{prop:compactdiag}, there is an orthonormal basis $B$ for $H$ consisting of eigenvectors of $T$. Using Theorem \ref{thm:compactoptspec}, we can arrange the set of eigenvalues of $T$ as a decreasing sequence $\la_1>\la_2>\cdots$ of non-negative real numbers. For every $n\in \n$, let $B_n$ be the subset of $B$ consisting of eigenvectors of $\la_n$ and define
\[
T_n(u):=\left\{
\begin{array}{ll}
  \la_n u & u\in B_n \\
  0 &  u\in B- B_n
\end{array} \right.
\]
and extend $T_n$ linearly to $H$. Clearly, $T_n$ is a finite rank operator, and so is $\sum_{m=1}^n T_m$ for every $n\in \n$. Using Theorem \ref{thm:compactoptspec}, we have
\[
\|T-\sum_{m=1}^n T_m \|=\la_{n+1}\ra 0, \qquad \text{when} \quad n\ra \infty.
\]
Therefore $T$ is the limit of the sequence $(\sum_{m=1}^n T_m)_{n\in\n}$ of finite rank operators.
\end{proof}
\begin{corollary}
\label{cor;compactsimple}
The \cs-algebra $K(H)$ is simple.
\end{corollary}
\begin{proof}
Let $I$ be a non-zero closed ideal of $K(H)$. Then by Proposition \ref{prop:finiterank3}, $F(H)\sub I$. Since $I$ is closed, $K(H)=\overline{F(H)}\sub I$.
\end{proof}
\begin{proposition}
\label{prop:multipliercompact}
$M(K(H))=B(H)$.
\end{proposition}
In the following proof, we use some parts of Exercise \ref{exe:finiterank}.
\begin{proof}
By Proposition \ref{prop:finiterank3}, $K(H)$ is an essential ideal of $B(H)$. Therefore by Proposition \ref{prop:multiplier2}, there is a one-to-one \ss-homomorphism $\ff:B(H)\ra M(K(H))$. To show $\ff$ is onto, let $(L,R)\in M(K(H))$. Fix a unit vector $u\in H$ and define
\begin{eqnarray*}
T:H&\ra& H,\\
x&\mapsto& L(x\otimes u)(u), \quad \forall x\in H.
\end{eqnarray*}
Clearly, $T$ is linear and we also have
\[
\|Tx\|\leq \|L(x\otimes u)\|\leq \|L\| \|s\otimes\|=\|L\|\|x\|, \quad \forall x\in H.
\]
Hence $T\in B(H)$. For every $x,y,z\in H$, we compute
\begin{eqnarray*}
[L_T(x\otimes y)] z&=& (Tx\otimes y)x\\
&=&\lan z,y\ran Tx\\
&=& \lan z, y \ran [L(x\otimes u)]u\\
&=& [L(x\otimes u)](\lan z, y\ran u)\\
&=& [L(x\otimes u)] (u\otimes y) z\\
&=&[L((x\otimes u)(u\otimes y)) ]z\\
&=& [L(x\otimes y)]z.
\end{eqnarray*}
This shows that $L_T=L$ over $F(H)$, and since $F(H)$ is dense in $K(H)$ and both $L$ and $L_T$ are bounded, $L_T=L$ over $K(H)$. This amounts to $0=\|L_T-L\|=\|R_T-R\|$. Hence $\ff(T)=(L_T,R_T)=(L,R)$, and therefore $\ff$ is onto.
\end{proof}

\begin{exercise}
Show that when $H$ is an infinite dimensional Hilbert space, $K(H)$ is not unital. Therefore $K(H)$ is not a von Neumann algebra.
\end{exercise}
The following proposition follows from elementary properties of von Neumann algebras:
\begin{proposition}
\label{prop:voncompact}
$K(H)''=B(H)$.
\end{proposition}
\begin{proof}
Let $T\in B(H) -\c 1$. In the proof of Proposition \ref{prop:basiccommutant}(v), we showed that $T$ does not commute with some finite rank operator. Hence $T\notin K(H)'$. This implies $K(H)'=\c 1$, and consequently $K(H)''=B(H)$ by Proposition \ref{prop:basiccommutant}(v).
\end{proof}

\section{Elements of von Neumann algebras}
\label{sec:vonneumann}
In this section, $H$ is always a Hilbert space. We mainly follow Gert K. Pedersen's \cite{pedersen1} book to prove the bicommutant theorem. Using this theorem, we observe that the image of the Borel functional calculus of an operator $T$ lies in $VN(T)=C\s(T,T\s)^{\prime \prime}$, the von Neumann algebra generated by $T$.
\begin{definition}
We say a {\bf \cs-subalgebra $A$ of $B(H)$ acts non-degenerately on $H$} if $x\in H$ and $Tx=0$ for all $T\in A$ implies $x=0$.
\end{definition}
\begin{theorem}
\label{thm:bicommutant} [The von Neumann bicommutant theorem] Let $M$ be a \cs-subalgebra of $B(H)$ acting non-degenerately on $H$. Then the following statements are equivalent:
\begin{itemize}
\item [(i)] $M=M^{\prime \prime}$.
\item [(ii)] $M$ is weakly closed.
\item [(iii)] $M$ is strongly closed.
\end{itemize}
\end{theorem}
\begin{proof} The implications (i) $\Rightarrow$ (ii) $\Leftrightarrow$ (iii) follow from Proposition \ref{prop:strongweaktop}(i), Remark \ref{rem:continofoperations}(v), and Corollary \ref{cor:weakstrongconvexclosed}.

Assume (iii) holds. For given $x_0\in H$, let $X$ be the closure of the vector space $Mx_0:=\{Tx_0; T\in M\}$ and set $P:=P_X$. One checks $PTP=TP$ for all $T\in M$. Thus
\[
TP=(PT\s P)\s=(T\s P)\s= PT, \quad \forall T\in M,
\]
and so $P\in M'$. On the other hand, for every $T\in M$, we have $T(1-P)x_0=(1-P)Tx_0=0$. Since $M$ acts non-degenerately on $H$, $(1-P)x_0=0$. Hence $Px_0=x_0$. For given $S\in M^{\prime\prime}$, we have $SP=PS$, so $Sx_0 = SPx_0 =PSx_0 \in PH=X$. Thus for every $\ep_0>0$, there is $T\in M$ such that $\|(S-T)x_0\|<\ep_0$. Let
\[
U_{S, x_1,\cdots,x_n, \ep}:= \{ T\in B(H); \|(T-S)x_k\|<\ep, \, \forall k=1,\cdots, n \}
\]
be an arbitrary basic neighborhood in the strong operator topology of $B(H)$ containing $S$. We need to show that $U_{S, x_1,\cdots,x_n, \ep}$ contains an element of $M$.

Set $x:=(x_1,\cdots, x_n)\in H^n$ and define $D:B(H)\ra B(H^n)$ by
\[
D(T):=(Tx_1,\cdots, Tx_n), \quad \forall T\in B(H).
\]
Then using the isomorphism $B(H^n)\simeq M_n(B(H))$, we have
\[
D(M)'=\left\{ (T_{ij})\in B(H^n); T_{ij}\in M' \, \forall i,j=1,\cdots,n\right\},
\]
see Problem \ref{e:5-22}. Thus $D(S)\in D(M)^{\prime \prime}$. Now, apply the first part of the proof with $D(M)$, D(S), $x$, $\ep$, and $H^n$ in place of $M$, $S$, $x_0$, $\ep$, and $H$. Then there is some $T\in B(H)$ such that $\|(D(S)-D(T)) x\|<\ep$. Using this, for every $m=1,\cdots, n$, we have
\[
\|(S-T)x_m\|\leq \left(\sum_{k=1}^n \|(S-T)x_k\|^2\right)^{1/2} = \|(D(S)-D(T)) x\|<\ep.
\]
Therefore $T\in U_{S, x_1,\cdots,x_n, \ep}$.
\end{proof}
The following corollary is an immediate consequence of the bicommutant theorem:
\begin{corollary}
Let $M$ be a von Neumann algebra on $H$ and Let $T\in M$ be a normal element. Then for every Borel function $f\in B(\si(T))$, we have $f(T)\in M$.
\end{corollary}
\section{Problems}

\begin{e}
\label{e:5-1}
Describe all inner products on $\c^n$. Show that, for every natural number $n$, there is only one Hilbert space of dimension $n$ up to unitary equivalence.
\end{e}
\begin{e}
\label{e:5-2}
Let $(E,\|-\|)$ be a normed vector space. Prove that $E$ is a Banach space if and only if every absolutely convergent series in $E$ is convergent.
\end{e}
\begin{e}
\label{e:5-3}
Let $(X,\mu)$ be a measure space. Use the above exercise to show that $L^2(X,\mu)$ equipped with the norm defined by $\|f\|_2^2:=\int_X|f(x)|^2d\mu(x)$ for all $f\in L^2(X,\mu)$ is a Banach space.
\end{e}
\begin{e}
\label{e:5-4}
[The {\bf Gram-Schmidt orthogonalization process}] Assume
\[
X=\{x_n; n\in \n\}
\]
is a linearly independent subset of a Hilbert space $H$. Show that there is an orthonormal subset $\{u_n; n\in \n\}$ in $H$ such that $[\{x_1,\cdots, x_n\}]=[\{u_1,\cdots, u_n\}]$ for all $n\in \n$.
\end{e}
\begin{e}
\label{e:5-5}
Prove that the Hilbert space $L^2(\r^n,m)$ is separable for all $n\in \n$, where $m$ is the Lebesgue measure. More generally, let $(X,\mu)$ be a measure space such that the topology of $X$ has a countable basis (in other words, $X$ is {\bf second countable}) and $\mu$ is a Borel measure. Show that $L^2(X, \mu)$ is separable.
\end{e}
\begin{e}
\label{e:5-6}
Let $H$ be an infinite dimensional Hilbert space. Show that the weak topology on $H$ is not first countable. (Remember; a topological space is called {\bf first countable} if each point has a countable basis of neighborhoods.)
\end{e}
\begin{e}
\label{e:5-7}
Using the Uniform boundedness theorem, see \ref{thm:uniboundedness}, show that every weakly convergent sequence in a Hilbert space is norm bounded. On the contrary, find an example to show that a weakly convergent net in a Hilbert space need not be norm bounded.
\end{e}
\begin{e}
\label{e:5-23}
Let $X$ be a set and let $H$ be a Hilbert space. Show that Hilbert spaces $\ell^2(X) \otimes H$ and $H^X$ are unitary equivalent.
\end{e}
\begin{e}
\label{e:5-8}
Let $(X, \mu)$ be a measure space and let $H$ be a Hilbert space. Show that Hilbert spaces $L^2(X) \otimes H$ and $L^2(X, H)$ are unitary equivalent.
\end{e}
\begin{e}
\label{e:5-9}
Let $H_1$, $H_2$ and $H_3$ be Hilbert spaces. Show that
\[
(H_1\oplus H_2)\otimes H_3 \simeq (H_1\otimes H_3)\oplus (H_2\otimes H_3).
\]
\end{e}
\begin{e}
\label{e:5-10}
Assume $X=\{x_i; i\in I\}$ and $Y=\{y_j; j\in J\}$ be orthonormal bases for Hilbert spaces $H_1$ and $H_2$, respectively. Prove that $\{ x_i\otimes y_j; i\in I, \, j\in J\}$ is an orthonormal basis for $H_1\otimes H_2$. Show that $H_1\otimes H_2$ is unitary equivalent to $H_1\otimes \ell^2(J)$. Also, use $X$ and $Y$ to find an orthonormal basis for $H_1\oplus H_2$.
\end{e}
\begin{e}
\label{e:5-11}
Let $G$ be an LCG with a Haar measure $\mu$.
\begin{itemize}
\item [(i)] Using a Dirac net on $G$ find and approximate unit for $C_r\s(G)$.
\item [(ii)] If the topology of $G$ is first countable, show that $C_r\s(G)$ is $\si$-unital.
\item [(iii)] Assume $G$ is discrete. Show that $\la(\d_e)=1\in B(\ell^2(G))$ and $\la(\d_g)$ is a unitary element in $C_r\s(G)$ for all $g\in G$, where $\d_g$ is the characteristic function of the one point subset $\{ g\}$ of $G$.
\end{itemize}
\end{e}
\begin{e}
\label{e:5-12}
Let $B$ be an orthonormal basis for a Hilbert basis $H$ and let $T,S\in B(H)$. Show that $T=S$ if and only if $\lan Tu,v\ran=\lan S u, v\ran$ for all $u, v\in B$.
\end{e}
\begin{e}
\label{e:5-13} Assume $H$ is a Hilbert space. Prove that $F(H)$ is generated by projections of rank one.
\end{e}
\begin{e}
\label{e:5-14}
Let $H$ be a Hilbert space. Show that the set $B(H)_h$ of self adjoint operators is weakly, and consequently strongly, closed.
\end{e}
\begin{e}
\label{e:5-15}
Prove that the Borel functional calculus agrees with the holomorphic functional calculus on holomorphic functions.
\end{e}
\begin{e}
\label{e:5-16}
Let $M:L^\infty (X,\mu)\ra B(L^2(X))$ be the map defined in Example \ref{exa:linfty} and let $f\in L^\infty(X)$.
\begin{itemize}
\item [(i)] By definition, the {\bf essential range of $f$} is the set
\[
\left\{ \la\in \c; \mu(f\inv(O))>0 \, \text{for all open subsets $O\sub \c$ containing $\la$ } \right\}.
\]
     Show that $\si(M_f)$ is exactly the essential range of $f$.
\item [(ii)] Show that if $|f|=1$ almost every where, then $M_f$ is a unitary operator.
\item [(iii)] Show that if $f=0$ or $f=1$ almost every where, then $M_f$ is a projection.
\item [(iv)] When is $M_f$ self adjoint or positive? Justify your answer.
\item [(v)] Let $g\in C(\si(M_f))$. Show that $g(M_f)=M_{g o f}$.
\item [(vi)] Let $g\in B(\si(M_f))$, see Remark \ref{rem:borelfc} and Theorem \ref{thm:borelfunctionalcalculus}. Show that
\[
g(M_f)=M_{g o f}.
\]
\end{itemize}
\end{e}
\begin{e}
\label{e:5-17}
Let $P$ and $Q$ be two projections on a Hilbert space $H$. Show that $P\leq Q$ if and only if $\|Px\|\leq \|Qx\|$ for all $x\in H$.
\end{e}
\begin{e}
\label{e:5-18}
Let $H$ be a Hilbert space and let $T, S\in B(H)$. Show that if $\lan Tx, x\ran=\lan Sx, x\ran$ for all $x\in H$, then $T=S$.
\end{e}
\begin{e}
\label{e:5-19}
Let $T\in B(H_1,H_2)$  be an onto isometry between two Hilbert spaces. Prove that $T$ is a unitary.
\end{e}
\begin{e}
\label{e:5-20}
Let $T\in B(H_1,H_2)$ be a bounded operator between two Hilbert spaces and set $|T|:=(T\s T)^{1/2}\in B(H_1)$. Prove that $\| |T|x\|=\|Tx\|$ for all $x\in H_1$.
\end{e}
\begin{e}
\label{e:5-21}
Let $H$ be a Hilbert space, $T\in B(H)$, and let $P_T$ and $Q_T$ be the left and right support projections of $T$, respectively. For every $S\in C\s(T,T\s)$, show that $SP_T=S$ and $Q_T S=S$.
\end{e}
\begin{e}
\label{e:5-22}
Let $H$ be a Hilbert space and let $H^n$ be the (orthogonal) direct sum of $n$ copies of $H$.
\begin{itemize}
\item [(i)] Show that $B(H^n)\simeq M_n(B(H))$.
\item [(ii)] Using the above isomorphism, for all $T\in B(H)$, define $D(T)\in B(H^n)$ by
\[
D(T):=(T_{ij})=\left\{ \begin{array} {ll} T& i=j\\ 0 & i\neq j \end{array} \right.
\]
Show that if $X$ is a \ss-subalgebra of $B(H)$, then we have
\[
D(M)'=\{ (T_{ij})\in B(H^n); T_{ij}\in M' \, \forall 1\leq i,j\leq n \}.
\]
\end{itemize}
\end{e}






\end{document}